\documentclass[a4paper,12pt,reqno,twoside]{amsart}
\usepackage{amsmath}
\usepackage{amsfonts}
\usepackage{amssymb}
\usepackage{amsthm}
\usepackage{color}
\usepackage{ifpdf}
\usepackage{array}
\usepackage{url}
\usepackage{graphicx}
\usepackage{float}
\usepackage{multirow,bigdelim}
\usepackage{hyperref}
\usepackage{cleveref}
\usepackage{lipsum}
\usepackage{chngcntr}

\numberwithin{equation}{section}
\newtheorem{definition}{Definition}[section]
\newtheorem{theorem}[definition]{Theorem}
\newtheorem{lemma}[definition]{Lemma}
\newtheorem{proposition}[definition]{Proposition}
\newtheorem{corollary}[definition]{Corollary}
\newtheorem{remark}[definition]{Remark}
\newtheorem{conjecture}[definition]{Conjecture}

\newcommand*{\C}{\mathbb{C}}
\newcommand*{\R}{\mathbb{R}}
\newcommand*{\Z}{\mathbb{Z}}
\newcommand*{\N}{\mathbb{N}}

\newcommand{\comment}[1]{}
\title[On the zeros of Riemann $\Xi(z)$ function]%
      {On the zeros of Riemann $\Xi(z)$ function} 
\author[Y. SHI]{Y. SHI}
\date{Version of \today}
\AtBeginDocument{%
\begin{abstract}

The Riemann $\Xi(z)$ function (even in $z$) admits a Fourier transform of an even kernel $\Phi(t)=4e^{9t/2}\theta''(e^{2t})+6e^{5t/2}\theta'(e^{2t})$. Here $\theta(x):=\theta_3(0,ix)$ and $\theta_3(0,z)$ is a Jacobi theta function, a modular form of weight $\frac{1}{2}$.  (A) We discover a family of functions $\{\Phi_n(t)\}_{n\geqslant 2}$ whose Fourier transform on compact support $(-\frac{1}{2}\log n, \frac{1}{2}\log n)$, $\{F(n,z)\}_{n\geqslant2}$, converges to $\Xi(z)$ uniformly in the critical strip $S_{1/2}:=\{|\Im(z)|< \frac{1}{2}\}$. (B) Based on this we then construct another family of functions $\{H(14,n,z)\}_{n\geqslant 2}$ and show that it uniformly converges to $\Xi(z)$ in the critical strip $S_{1/2}$. (C) Based on this we construct another family of functions $\{W(n,z)\}_{n\geqslant 8}:=\{H(14,n,2z/\log n)\}_{n\geqslant 8}$ and show that if all the zeros of $\{W(n,z)\}_{n\geqslant 8}$ in the critical strip $S_{1/2}$ are real, then all the zeros of $\{H(14,n,z)\}_{n\geqslant 8}$ in the critical strip $S_{1/2}$ are real.  (D) We then show that $W(n,z)=U(n,z)-V(n,z)$ and  $U(n,z^{1/2})$ and $V(n,z^{1/2})$ have only real, positive and simple zeros.  And there exists a positive integer $N\geqslant 8$ such that for all $n\geqslant N$, the zeros of  $U(n,x^{1/2})$ are strictly left-interlacing with those of $V(n,x^{1/2})$. Using an entire function equivalent to Hermite-Kakeya Theorem for polynomials we show that $W(n\geqslant N,z^{1/2})$ has only real, positive and simple zeros. Thus $W(n\geqslant N,z)$ have only real and imple zeros. (E) Using a corollary of Hurwitz's theorem in complex analysis we prove that $\Xi(z)$ has no zeros in $\left(S_{1/2}\setminus \R\right)$, i.e., $\left(S_{1/2}\setminus \R\right)$ is a zero-free region for $\Xi(z)$.  Since all the zeros of $\Xi(z)$ are in $S_{1/2}$, all the zeros of $\Xi(z)$ are in $\R$, i.e., all the zeros of $\Xi(z)$ are real.
\\
\\
\end{abstract}
\maketitle
}
\begin{document}
%
%
\footnotetext{yaoming\_shi@yahoo.com}

{\centering{\selectfont{\textbf{Pre-introduction}}}}\\

We diligently follow Scott Aaronsen's guideline to get rid of ``Ten Signs a Claimed Mathematical Breakthrough is Wrong" in an effort to make our manuscript readable by professional mathematicians.

 This manuscript is written from a view point of amateurs like us. We want this manuscript to be as rigorous as possible so that for any theorem/lamma/corollary, we either find a proof reference for it in a timely manner from the literature or we supply a proof for it. There may be a lot of materials in this paper that seem standard to experts in this field.  But  they are certainly new to amateurs like us. Therefore we do not claim originality to any individual theorem/lemma/corollary in this manuscript that does not have a reference. We only claim the originality of putting them together to reach our final goal.

\section{\textbf{Introduction}}
\counterwithin{equation}{section}

Let $s,z$ be two complex variables, 
$\zeta(s)$ be the Riemann $\zeta$-function, and

\begin{equation}\label{xisdef}
\xi(s)=\frac{1}{2}s(s-1)\pi^{-s/2}\Gamma\left(\frac{s}{2}\right)\zeta(s)
\end{equation}
be the Riemann (lower case) $\xi$-function,  which is an entire function~\cite{E1974,T1986} 
satisfying the functional equation $\xi(s)=\xi(1-s)$ and the equation $\xi(\overline{s})=\overline{\xi(1-s)}$. Riemann hypothesis~\cite{B2001,C2003} is then equivalent to the statement that all the zeros of $\xi(s)$ are on the critical line $\{\Re(s)=1/2\}$.

Let $N(T)$ be the total zeros in the section $0<\Im(s)<T$ of the critical strip $\{0<\Re(s)<1\}$ and $N_0(T)$ be the total zeros in the section $0<\Im(s)<T$ of the critical line.
Let 
\begin{equation}
\kappa=\lim\limits_{T\to\infty}\frac{N_0(T)}{N(T)}.
\end{equation}
\indent Riemann~\cite{R1859} in 1859 proved that all the zeros of $\xi(s)$ are in the strip $\{0\leqslant \Re(s)\leqslant 1\}$. Hadamard ~\cite{H1896} and de la Vall\'{e}e-Poussin ~\cite{VP1896} in 1896 independently proved that no zeros could lie on the line $\Re(s) = 1$ (and thus also no zeros on the line $\Re(s)=0$). Hardy~\cite{H1914} in 1914 and Hardy and Littlewood~\cite{HL1921} in 1921 showed there are infinitely many zeros on the critical line. Selberg~\cite{S1946} in 1946 proved that at least a (small) positive proportion of zeros lie on the line ($0<\kappa<1$). Levinson~\cite{L1974} in 1974 improved this to one-third of the zeros ($1/3<\kappa<1$). Conrey~\cite{C1989} in 1989 improved this further to two-fifths ($2/5<\kappa<1$). This number is now slowly creeping up.

 De la Vall\'{e}e-Poussin ~\cite{VP1999} in 1899-1900 proved that if $\sigma + it$ is a zero of the Riemann zeta function, then 
 
 \begin{equation}
 1-\sigma\geqslant \frac{C}{\log t}
 \end{equation}
 for some positive constant $C$. In other words, zeros cannot be too close to the line $\sigma = 1$: there is a zero-free region close to this line. This zero-free region has been enlarged by several authors using methods such as Vinogradov's mean-value theorem. Ford ~\cite{F2002} in 2002 gave a version with explicit numerical constants: $\zeta(\sigma + it) \not= 0$ whenever $|t| \geqslant 3$ and
 
 \begin{equation}
 \sigma\geqslant 1-\frac{1}{57.54(\log |t|)^{2/3}(\log\log|t|)^{1/3}}
 \end{equation}

Let
\begin{equation}
\Xi(z)=\xi(iz+1/2)
\end{equation}
be the Riemann (upper case) $\Xi$-function, which is an entire function~\cite{E1974,T1986} 
satisfying the functional equation $\Xi(z)=\Xi(-z)$ and the equation $\Xi(\bar{z})=\overline{\Xi(z)}$. The critical line $\{\Re(s)=1/2\}$ then becomes the line of $\{\Im(z)=0\}$ and the critical strip $\{0<\Re(s)<1\}$ becomes $S_{1/2}:=\{|\Im(z)|<1/2\}$. Riemann hypothesis~\cite{R1859, B2001, C2003} is then the statement that $\Xi(z)$ has only real zeros.
\\
\\
\indent Riemann $\Xi(z)$ function can be expressed as a Fourier transformation.

\begin{equation}\label{Xizdef0}
\aligned
\Xi(z)&=\int_{-\infty}^{\infty}\Phi(t)\exp(izt)\mathrm{ d}t,\\
\Phi(t)&=\sum_{k=1}^{\infty}\phi_k(t)=\Phi(-t)\\
\phi_k(t)&=\left(4(\pi k^2)^2 e^{9t/2}-6\pi k^2e^{5t/2}\right)\exp\left(-\pi k^2 e^{2t}\right).
\endaligned
\end{equation}

\indent P\'{o}lya ~\cite{P1926,P1927} noticed that $\lim\limits_{t\to\infty}e^{2t}=2\cosh(2t)$. So he approximated $\Phi(t)$ with $\Phi_{P}(t)$ and $\Phi_{P2}(t)$ by keeping only the leading ($k=1$) term and replacing $e^{at}$ with $(e^{at}+e^{-at})=2\cosh(at)$:
\begin{equation}\label{PhiP}
\Phi_{P}(t)=4\pi^2\cosh(9t/2)\exp\left(-2\pi \cosh (2t)\right),
\end{equation}
\begin{equation}\label{PhiP2}
\Phi_{P2}(t)=\left(4\pi^2\cosh(9t/2)-6\pi\cosh(5t/2)\right)\exp\left(-2\pi \cosh (2t)\right).
\end{equation}
Thus  $\lim\limits_{t\to\infty}\Phi_{P}(t)=\Phi(t),\lim\limits_{t\to\infty}\Phi_{P2}(t)=\Phi(t)$.
\noindent P\'{o}lya proved that $\int_{-\infty}^{\infty}\Phi_P(t)\exp(izt)\mathrm{ d}t$ and $\int_{-\infty}^{\infty}\Phi_{P2}(t)\exp(izt)\mathrm{ d}t$ have only real zeros.

\indent De Bruijn ~\cite{dB1950} approximated $\Phi(t)$ with $\Phi_{dB}(t)$:
\begin{equation}\label{PhidB}
\aligned
\Phi_{dB}(t)&=\exp\left(-2\pi \cosh (2t)\right)\\
&\times\left(4\pi^2\cosh(t/2)+(4\pi^3-6\pi)\cosh(5t/2)+4\pi^2\cosh(9t/2)\right).
\endaligned
\end{equation}
De Bruijn proved that the function $\int_{-\infty}^{\infty}\Phi_{dB}(t)\exp(izt)\mathrm{ d}t$ has only real zeros.

\indent Hejhal ~\cite{H1990} approximated $\Phi(t)$ with $\Phi_{H,m}(t)$:
\begin{equation}\label{PhiH}
\Phi_{H,m}(t)=\sum_{n=1}^{m}\left(4\pi^2n^4\cosh(9t/2)-6\pi n^2\cosh(5t/2)\right)\exp\left(-2\pi n^2 \cosh (2t)\right)
\end{equation}
Hejhal proved that almost all the zeros of the function $\int_{-\infty}^{\infty}\Phi_{H,m}(t)\exp(izt)\mathrm{ d}t$ are real. 

We notice (\cite{S2012}, theorem 2) that there is one thing in common in P\'{o}lya's approximation $\Phi_{P}(t), \Phi_{P2}(t)$ of \eqref{PhiP} and \eqref{PhiP2}, de Bruijin's approximation $\Phi_{dB}(t)$ of \eqref{PhidB}, and Hejhal's approximation $\Phi_{H,m}(t)$ of \eqref{PhiH}, they all captured the contribution of the tail part (as $t\to\infty$) of $\Phi(t)$ in the Fourier transformation. But none of them converges to $\Phi(t)$ at $t=0$.  This aspect is clearly shown in Figure 1. Thus they can hardly capture the contribution of the head part (at $t=0$) of $\Phi(t)$ in the Fourier transformation. 
%
\begin{figure}[H]
\centering
\includegraphics[scale=0.8,keepaspectratio]{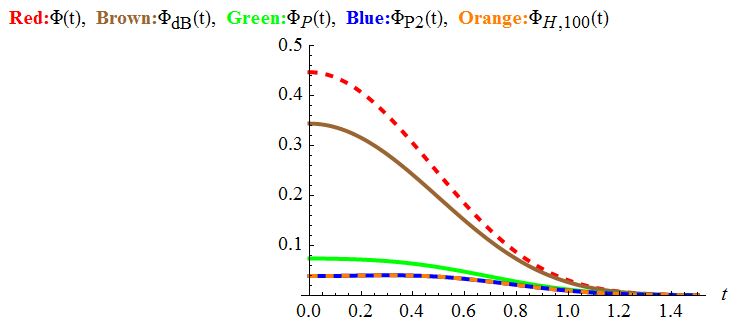}
\caption{Plots of various Phi functions vs. $t$. This includes $\Phi(t)$(Red), de bruijin's $\Phi_{D}(t)$(Brown), P\'{o}lya's $\Phi_{P}(t)$(Green) and $\Phi_{P2}(t)$(Blue), and Hejhal's $\Phi_{H,100}(t)$(Black). There is no visible difference between $\Phi_{H,10}(t)$ and $\Phi_{H,100}(t)$. }\label{figure1}
\end{figure}

A natural question then arises: Is it possible to find approximations to $\Phi(t)$ such that they converge to $\Phi(t)$ at $t\to\infty$ and $t=0$, and the corresponding Fourier transforms have only real zeros ?
We gave a positive answer to this question in \cite{S2012}.  
Let $\alpha_{-n}=\alpha_n>0,0<\beta_{-n}=\beta_n<1$ and 
\begin{equation}\label{FdB}
K(t)=\Phi_P(t)+\exp(-2\pi \cosh (2t))\sum_{n=-N}^{N}\alpha_n \exp(9\beta_n t/2).
\end{equation}
So $\lim\limits_{t\to\infty}K(t)= \Phi_{P}(t)$.
The actual values of parameters $\alpha_n$ and $\beta_n$ are then used to satisfy the other two criteria; namely (ii) $K(0)=\Phi(0)$, (iii) $\int_{-\infty}^{\infty}K(t)\exp(i z t)\mathrm{d}t$ has only real zeros.

Here is one such example (\cite{S2012},theorem 2) where we approximate $\Phi(t)$ with $\Phi_{S2}(m,a;t)$.
\begin{theorem}
	Let $0<a<1, b=(4\pi^2)^{-1}(e^{2\pi}\Phi(0)-1)\approx 5.059069$,
	\begin{equation}
c =\frac{(e^{2\pi}\Phi(0)-4\pi^2)}{4\pi^2}\frac{(1-a)}{(m(1-a)-a(1-a^m))}
\end{equation}
	\begin{equation}
		g(m,a;t)=\cosh\left(9t/2\right)+c \sum_{k=0}^{m-1}(1-a^{k+1})\cosh\left(9kt/(2m)\right),
	\end{equation}
	\begin{equation}\label{PhiS2m}
	\Phi_{S2}(m,a;t)=4\pi^2 g(m,a;t)\exp\left(-2\pi \cosh (2t)\right),
	\end{equation}
	If $\mu$ satisfies the equation: 
	\begin{equation}\label{mu}
		\mu (1-a^\mu)=b,
	\end{equation}
and
	\begin{equation}\label{mceil}
		m \ge \lceil\mu\rceil,
	\end{equation}
	Then 
	
	(A)$\lim\limits_{t\to\infty}\Phi_{S2}(m,a;t)= \Phi(t),\text{ when}\quad t\to\infty$;
	
	(B)$\Phi_{S2}(m,a;0)= \Phi(0)$; 
	
	(C)the entire function $\int_0^{\infty}\exp(izt)\Phi_{S2}(m,a;t)\mathrm{d}t$ has only real zeros.
\end{theorem}

\begin{figure}[H]
	\centering
	\includegraphics[scale=0.8,keepaspectratio]{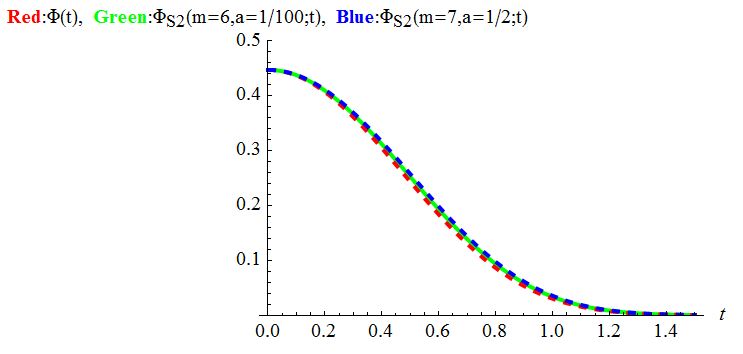}
	\caption{Plots of various Phi functions vs. $t$. This includes $\Phi(t)$(Red), $\Phi_{S2}(m=6,a=1/100;t)$ (Green), $\Phi_{S2}(m=7,a=1/2;t)$(Blue). }\label{figure2}
\end{figure}

But our approximation to $\Phi(t)$ in ~\eqref{PhiS2m} is clearly a better resemblance to $\Phi(t)$ in shape but not in essence.

As pointed out by Dimitrov and Rusev ~\cite{DR2011} and Dimitrov~\cite{D2013}:``A natural approach to resolving the Riemann hypothesis is to establish criteria for an entire function, or more specifically, a Fourier transform of a kernel, to possess only real zeros and to apply them to the Riemann $\Xi(z)$ function. There is no doubt that this was the main reason why so many celebrated mathematicians have been interested in the zero distribution of entire functions and, in particular, of Fourier transforms. Among them are
such distinguished masters of the Classical Analysis as Hurwitz, Jensen, P\'{o}lya, Hardy, Tichmarsh, de Bruijn, Newman, Obrechkoff, Tchakaloff etc."\\

\indent``The efforts to settle the Riemann hypothesis establishing the fact that $\Xi[z]$ define as Fourier transform of $\Phi(t)$ has (so far) failed despite the efforts of P\'{o}lya ~\cite{P1926,P1927}, de Bruijn ~\cite{dB1950}, Newman ~\cite{N1976}  and many other celebrated mathematician for two reasons. The first one is that P\'{o}lya's question `What properties of the function $K(t)$ are sufficient to secure that its Fourier transform $\int_{-\infty}^{\infty}\exp(izt)K(t)dt$ has only real zeros' remains open. The second is that the sufficient conditions for the kernels that have been proved either are extremely difficult to be verified for $\Phi(t)$ or simply do not hold for it."

For complete review of the research in this direction we refer the readers to the excellent 108 page review paper, {\it Zeros of Entire Fourier Transforms}, by Dimitrov and Rusev ~\cite{DR2011} and references therein. See also the review paper by Ki ~\cite{K2006} and Hallum's Master Thesis ~\cite{H2014}(an easy-to-read reference).

See also recent work by  Csordas and Varga ~\cite{CV1989}, Csordas and Dimitrov ~\cite{CD2000}, Cardon ~\cite{C2004}, Adam and Cardon ~\cite{AD2007}, and Csordas and Yang ~\cite{CY2012} etc.
\\
Din ~\cite{D2010} and Din and Moneta~\cite{DM2011} studied the zeros of the Mellin transform of various kernel $A(x)$:
\begin{equation}
\tilde{\xi}(s)=\int_1^{\infty}A(x)\left(x^{s/2}+x^{(1-s)/2}\right)\frac{\mathrm{d}x}{x}
\end{equation}
\\
\indent Our current approach is greatly influenced by the comments of Din ~\cite{D2010} and Din and Moneta~\cite{DM2011}. They pointed out in ~\cite{DM2011} that``There exists a great number of different attempts to prove the Riemann Hypothesis and very often they
result in the statement of another equivalent hypothesis, frequently concerning a field of mathematics
more or less unrelated to complex analysis [for a review see ~\cite{C2003}]. Concerning approaches using
complex analysis, it is somewhat surprising that historically there has been no serious attempt to
investigate the validity of the Riemann Hypothesis in terms of the well-known Hurtwitz theorem of
complex analysis [see ~\cite{M1999}], which in a quite manifestly (way) concerns the issue at hand about zeros of
analytic functions. This theorem states that if an analytic function is a limit of a sequence of analytic
functions, uniformly convergent on a compact subset of the complex plane, then any zero of the
function in the subset must be a limit of the zeros of the sequence functions. In particular, if the zeros
of the approximants were all real, then the same would be true for the limit function.
\indent One explanation for the apparent lack of interest in trying to apply the Hurtwitz theorem may be related
to the fact that, apart from knowing that a certain entire function may theoretically be represented by a
product involving its zeroes, then it is in general not obvious how to construct explicitly an appropriate
sequence of analytic functions converging to the given function. For the problem of the Riemann
Hypothesis, of course `appropriate' here means that the approximating functions should have only real
zeros so as to be able to infer that the same would be the case for the limit function."
\\
\\
\indent What is new in this paper is that we stick to the following 3 guidelines.\\
\noindent (A) \textbf{Simplifying Kernel while insuring uniform convergence}: \\
We directly simplify $\Phi(t)$ itself and in the beginning do not worry about the locations of the zeros of the corresponding Fourier transform.  However we do keep in mind that when the simplifying nob (e.g., $n\in\N$) turns to one extreme (e.g., $n\to\infty$) the simplified kernel function (e.g., $\Phi_n(t)$) should go back to the original kernel function; i.e., $   \lim\limits_{n\to\infty}\Phi_n(t)\to\Phi(t)$; its corresponding Fourier transform on compact support $(-\frac{1}{2}\log n, \frac{1}{2}\log n)$, $F(n,z)$,  should uniformly converge to $\Xi(z)$ in the critical strip $S_{1/2}$. 
\\
\\
\noindent (B) \textbf{Simplifying $F(n,z)$ while insuring uniform convergence until all the zeros of the resulting functions are real}: \\
We try to prove that there exists a positive integer $n_0$ such that all the zeros of $F(n\geqslant n_0,z)$ are real. If not, we keep simplifying $F(n,z)$ while insuring its uniform convergence to $\Xi(z)$ in the critical strip $S_{1/2}$. In the end, we discover two families of functions,  $\{W(n,z)\}_{n\geqslant n_0}$, and $\{H(14,n,z)\}_{n\geqslant n_0}$ such that (B.1) $\{H(14,n,z)\}_{n\geqslant n_0}$ uniformly converges to $\Xi(z)$ in the critical strip $S_{1/2}$; (B.2) $W(n,z)$ has only real zeros; (B.3) That $W(n,z)$ has only real zeros implies $H(14,n,z)$ has only real zeros.
\\
\\
\noindent (C) \textbf{Applying a corollary of Hurwitz theorem in complex analysis to complete the analysis}:\\
Now using a corollary of Hurwitz theorem in complex analysis we can prove that $\left(S_{1/2}\setminus\R\right)$ is a zero-free region for $\Xi(z)$. Thus all the zeros of $\Xi(z)$ in $S_{1/2}$ are on $\R$. Combining this with Riemann's result ($\{|\Im (z)|>1/2\}$ is a zero-free region for $\Xi(z)$) and the Prime Number Theory result ($\{|\Im (z)|=1/2\}$ is a zero-free region for $\Xi(z)$), we finally conclude that $\Xi(z)$ has only real zeros.
\\
\\
\\

\indent In section (2) we construct a family of functions $\Phi_n(t)$ ($n\in\N$) by truncating the infinite summation up to the first $n$ terms and then symmetrizing the result with respect to $t$ to make it an even function of $t$ (cf. ~\eqref{Phindef}). We then truncate the support of Fourier transform from $(-\infty,\infty)$ to $(-\omega_n,\omega_n)$ where $\omega_n=(1/2)\log n$ (cf. ~\eqref{Fnzdef}) such that $F(n,z)$ uniformly converges to $\Xi(z)$ in the critical strip $S_{1/2}$ (cf. \cref{FzUniformConvergence}). We also find that 
\begin{equation}\label{FnzdefSec1}
\aligned
F(n,z)=\sum_{j\geqslant 1} (-1)^{j}c_{n,j}\left(\mathrm{sinc}(\omega_n z-i\varphi_{n,j})+\mathrm{sinc}(\omega_n z+i\varphi_{n,j})\right).\\
\endaligned
\end{equation}
where $\mathrm{sinc}(x):=(1/x)\sin x$.
\\
\\
\indent The infinite summation of $j$ in ~\eqref{FnzdefSec1} is still not convenient for our further analysis.
So in section (3), by truncating the summation of $j$ up to the first $m$ terms , we construct another family of functions, $G(m,n,z)$ ($m\in\N$) \eqref{Gmnzdef}. 
\begin{equation}\label{GmnzdefSec1}
\aligned
G(m,n,z)=\sum_{1\leqslant j\leqslant m} (-1)^{j}c_{n,j}\left(\mathrm{sinc}(\omega_n z-i\varphi_{n,j})+\mathrm{sinc}(\omega_n z+i\varphi_{n,j})\right).\\
\endaligned
\end{equation}
We prove that $G(m,n,z)$ uniformly converges to $F(n,z)$ in the critical strip $S_{1/2}$ (cf. \cref{GmnzUniformConvergence}). We also prove that $H(l,n,z):=G(2+ln^3,n,z)$, $l\in\N_0+8$, uniformly converges to $\Xi(z)$ in the critical strip $S_{1/2}$ (cf. \cref{GqnnUniformConvergence}).  \\

Because $\omega_8=(1/2)\log 8\approx 1.03972>1$, to prove that there exists a positive inetger $N\geqslant 8$ such that all the zeros of $H(l,n\geqslant N,z)$ in the critical strip $S_{1/2}$ are real, it is then suffice to prove that all the zeros of $H(l,n\geqslant N,z/\omega_n)$ in the critical strip $S_{1/2}$ are real.  
\\
\\
\indent In section (4), we set $l=14$ and separate the even terms from the odd terms in ~\eqref{GmnzdefSec1} and define
\begin{equation}\label{WnzdefSec1}
\aligned
W(n,z):&=H(14,n,z/\omega_n)=U(n,x)-V(n,x)\\
U(n,x):&=\sum_{0\leqslant j\leqslant 7n^3} c_{n,2j+2}\left(\mathrm{sinc}( z-i\varphi_{n,2j+2})+\mathrm{sinc}(z+i\varphi_{n,2j+2})\right)\\
V(n,x):&=\sum_{0\leqslant j\leqslant 7n^3} c_{n,2j+1}\left(\mathrm{sinc}( z-i\varphi_{n,2j+1})+\mathrm{sinc}(z+i\varphi_{n,2j+1})\right).
\endaligned
\end{equation}
Note that functions $U(n,z),V(n,z), W(n,z)$ are all even in $z$.
We discover that functions $U(n,z),V(n,z)$ have only real and simple zeros and $U(n,0)>0,V(n,0)>0$.  And there exists a sufficiently large and positive integer $N\geqslant 8$ such that for all $n\geqslant N$ the positive real and simple zeros of $U(n,z)$ are strictly left-interlacing with those of $V(n,z)$;the negative real and simple zeros of $U(n,z)$ are strictly right-interlacing with those of $V(n,z)$.  
\\
\\
\indent In section (5), we construct Hadmard factorizations for entire functions of order one, $U(n,z)$ and $V(n,z)$. Let $U_{p(n)}(n,z),V_{p(n)}(n,z)$ be the $2p(n)$-th order polynomials ($p(n)\in\N(n)$ and $p$ does not particularly stand for prime in this paper) that we obtain by keeping only the first $2p(n)$ terms in the infinite products of $U(n,z)$ and $V(n,z)$. We prove that their difference, $W_{p(n)}(n,z):=U_{p(n)}(n,z)-V_{p(n)}(n,z)$, uniformly converges to $W(n,z)$ in the compact disk $D(\pi n):=\{|z|< \pi n\}$. In addition, $W(n\geqslant N, z)$ has only real zeros. Using a corollary of Hurwitz's theorem in complex analysis we prove that $\Xi(z)$ has no zeros in $(S_{1/2}\setminus\R)$, i.e., $(S_{1/2}\setminus\R)$ is a zero-free region for $\Xi(z)$. Since all the zeros of $\Xi(z)$ are in $S_{1/2}$, therefore all the zeros of $\Xi(z)$ are in $\R$, i.e., all the zeros of $\Xi(z)$ are real. 
\\
\\
\indent Our proof of ~\cref{w2gtw1prop} turned out to be extremely lengthy. It takes about two thirds of the space of the whole paper (it spans section (6-8) and Appendices A and B). In section (6-7) and Appendices A and B, we present supporting lemmas and theorems. In section (8) we present the detailed proof of ~\cref{w2gtw1prop} in ~\cref{w3gtw2gtw1}.

\section{\textbf{A family of entire functions $F(n,z)$ that uniformly converge to $\Xi(z)$} }
Riemann $\Xi(z)$ function can be expressed as  ~\cite{R1859,E1974,T1986, DR2011p10}:
\begin{equation}\label{Xizdef2}
\Xi(z)=2\int_1^{\infty}x^{1/4}\Psi(x)\cos((z/2)\log x)\frac{\mathrm{ d}x}{x},
\end{equation}
\noindent where
\begin{equation}\label{Psixdef}
\aligned
\Psi(x)&=(1/2)\left(2x^2\theta''(x)+3x\theta'(x)\right)\\
&=\sum_{k=1}^{\infty}\psi_k(x)\\
\psi_k(x)&=\left(2(\pi k^2)^2 x^2-3\pi k^2x\right)\exp\left(-\pi k^2 x\right).
\endaligned
\end{equation}

\noindent The function $\theta(x)$ above is a special case of the Jacobi theta function:
\begin{equation}
\theta(x):=\theta_3(0, i x) =\sum_{k=-\infty}^{\infty}\exp(-\pi k^2 x ), \qquad \Re(x) > 0.
\end{equation}

\noindent It satisfies the periodic relation and the self-inverse relation 
\begin{equation}\label{thetaS1}
\theta(x-i) = \theta(x+i)\qquad\quad\text{(periodic)},
\end{equation}
\begin{equation}\label{theta_self_inverse}
(1/x)^{1/4}\theta(1/x) =x^{1/4}\theta(x)\qquad\text{(self-inverse)}. 
\end{equation}
\begin{remark}
	Because of the factors $x^{2}, x$ appeared in ~\eqref{Pshidef2}, $\Psi(x+i) \not = \Psi(x-i)$. Thus unlike $\theta(x)$,  $\Psi(x)$ is no longer periodic anymore .
\end{remark}

\noindent Using the ~\eqref{theta_self_inverse}, it can be shown ~\cite{DR2011p10}  that,
\begin{equation}\label{Pshidef2}
\aligned
(1/x)^{1/4}\Psi(1/x)&=(1/2)(1/x)^{1/4}\left(2(1/x)^2\theta''(1/x)+3(1/x)\theta'(1/x)\right)\\
&=(1/2)x^{1/4}\left(2x^2\theta''(x)+3x\theta'(x)\right)\\
&=x^{1/4}\Psi(x)
\endaligned
\end{equation}
Thus ~\eqref{Xizdef2} becomes
\begin{equation}\label{Xizdef3}
\aligned
\Xi(z)&=\int_0^{\infty}x^{iz/2}x^{1/4}\Psi(x) \frac{\mathrm{ d}x}{x}\\
&=\frac{1}{2}\int_0^{\infty}x^{iz/2}\left(x^{1/4}\Psi(x)+(1/x)^{1/4}\Psi(1/x)\right) \frac{\mathrm{ d}x}{x}\\
\endaligned
\end{equation}

\noindent And ~\eqref{Xizdef2} can also be written as a Mellin transform ($s=iz+1/2$)
\begin{equation}\label{xisdef3}
\aligned
\xi(s)=\Xi(z)&=\int_0^{\infty}x^{s/2}\Psi(x) \frac{\mathrm{ d}x}{x}\\
&=\int_1^{\infty}\Psi(x)\left(x^{s/2}+x^{(1-s)/2}\right) \frac{\mathrm{ d}x}{x}\\
\endaligned
\end{equation}

If we substitute ~\eqref{Psixdef} into ~\eqref{Xizdef2}, exchange the order of integration with summation, then we obtain

\begin{equation}\label{XizDef4}
\aligned
\Xi(z)&=
\sum_{1\leqslant k\leqslant \infty} \left(\frac{\gamma(\frac{9}{4}+\frac{i}{2}z, \pi k^2)}{(\pi k^2)^{\frac{1}{4}+\frac{i}{2}z}}+\frac{\gamma(\frac{9}{4}-\frac{i}{2}z, \pi k^2)}{(\pi k^2)^{\frac{1}{4}-\frac{i}{2}z}}\right)\\
&-\frac{3}{2}\sum_{1\leqslant k\leqslant \infty} \left(\frac{\gamma(\frac{5}{4}+\frac{i}{2}z, \pi k^2)}{(\pi k^2)^{\frac{1}{4}+\frac{i}{2}z}}+\frac{\gamma(\frac{5}{4}-\frac{i}{2}z, \pi k^2)}{(\pi k^2)^{\frac{1}{4}-\frac{i}{2}z}}\right)\\
\endaligned
\end{equation}

Setting $\frac{1}{4}+\frac{i}{2}z=\frac{s}{2},\frac{1}{4}-\frac{i}{2}z=\frac{1-s}{2}$ we obtain
\begin{equation}\label{xisdef2}
\aligned
\xi(s)=&\sum_{1\leqslant k\leqslant \infty} \left(\frac{\gamma(2+\frac{s}{2}, \pi k^2)}{(\pi k^2)^{s/2}}+\frac{\gamma(2+\frac{1-s}{2}, \pi k^2)}{(\pi k^2)^{(1-s)/2}}\right)\\
-\frac{3}{2}&\sum_{1\leqslant k\leqslant \infty} \left(\frac{\gamma(1+\frac{s}{2}, \pi k^2)}{(\pi k^2)^{s/2}}+\frac{\gamma(1+\frac{1-s}{2}, \pi k^2)}{(\pi k^2)^{(1-s)/2}}\right)\\
\endaligned
\end{equation}

Set $x=e^{2t}$ in ~\eqref{Xizdef2}, then $\frac{\mathrm{d}x}{x}=2\mathrm{d}t$ and Riemann $\Xi(z)$ function can be expressed as a Fourier transformation ~\cite{R1859, E1974, T1986, DR2011p10}:
\begin{equation}\label{Xizdef}
\Xi(z)=2\int_0^{\infty}\Phi(t)\cos(zt)\mathrm{ d}t,
\end{equation}
\noindent where
\begin{equation}\label{Phitdef2}
\aligned
\Phi(t):&=2e^{t/2}\Psi(e^{2t})=e^{t/2}\left(2e^{4t}\theta''(e^{2t})+3e^{2t}\theta'(e^{2t})\right)\\
&=\sum_{1\leqslant k\leqslant \infty}\phi_k(t),\\
\phi_k(t)&=\left(4(\pi k^2)^2 e^{9t/2}-6\pi k^2e^{5t/2}\right)\exp\left(-\pi k^2 e^{2t}\right).
\endaligned
\end{equation}
\noindent Among the properties that $\Phi(t)$ satisfies, we would like to emphasize the following three that are used in this paper:
\begin{lemma}[~\cite{LM2011},  Lemma 3.1]\label{Phiprop}
(1) $\Phi(t)$ of \eqref{Phitdef2} is an even function: $\Phi(t)=\Phi(-t)$

(2) $\phi_k(t)>0,t>0,k\geqslant 1$ 
\end{lemma}

In order to simplify the analysis, we, follow ~\cite{H2010}, first truncate the infinite summation for kernel $\Phi(t)$ in \eqref{Phitdef2} up to the first $n$ terms

\begin{equation}
\aligned
\tilde{\Phi}_n(t)&=\sum_{k=1}^{n}\phi_k(t)\\
\endaligned
\end{equation}

But $\tilde{\Phi}_n(t)$ is no longer an even function of $t$. This violates the property that $\Phi(t)$ is even in $t$. As pointed out by P\'{o}lya that the Fourier transform of an asymmetric (in $t$) kernel $K(t)$ will have infinite non-real zeros.

So we mend this problem by symmetrizing $\tilde{\Phi}_n(t)$ and consider the Kernel

\begin{equation}
\aligned
\Phi_n(t)=(1/2)(\tilde{\Phi}_n(t)+\tilde{\Phi}_n(-t))=(1/2)\sum_{k=1}^{n}\left(\phi_k(t)+\phi_k(-t)\right).
\endaligned
\end{equation}

We then find out that the Fourier transform of $\Phi_n(t)$ does not converges to $\Xi(z)$ uniformly  in the critical strip $S_{1/2}=\{|\Im(z)|< 1/2\}$. After further analysis, we find out that if we truncate the support of Fourier transform from $(-\infty,\infty)$ to $(-(1/2)\log n, (1/2)\log n)$, then the corresponding Fourier transform on finite support does converge to $\Xi(z)$ uniformly in $S_{1/2}$. All the details are summarized in ~\cref{EnbnzUniformConvergence} and ~\cref{FzUniformConvergence} below.

\begin{theorem} \label{EnbnzUniformConvergence}
	Let $n\in\N_0+2$, $0<\beta_0\leqslant1/2$, $0<\beta_0\leqslant\beta\leqslant 1-\beta_0<1$, $b_n=n^{2\beta}>1$,  and 
	\begin{equation}\label{Phindef}
	\aligned
	\Phi_n(t)&=\frac{1}{2}\sum_{k=1}^{n}(\phi_k(t)+\phi_k(-t))\\
	\endaligned
	\end{equation}
	
	\noindent where $\phi_k(t)$ is defined in \eqref{Phitdef2}.
	\noindent Let
	\begin{equation}\label{Enbnzdef}
	\aligned
	E(n,b_n,z)&=2\int_{0}^{(1/2)\log b_n}\Phi_n(t)\cos(zt)\mathrm{d}t\\
	&=\int_{-(1/2)\log b_n}^{(1/2)\log b_n}\Phi_n(t)\exp(izt)\mathrm{d}t,\\
	\endaligned  
	\end{equation}
	
	\noindent Then $E(n,b_n,z)$ converges to $\Xi(z)$ uniformly in $S_{1/2}$.
	
\end{theorem}

\begin{proof}

	Let $z=x+i y, x\in\R,y\in\R$. Then in the critical strip $S_{1/2}$ we have $|y|\le 1/2$ and consequently
	\begin{equation}\label{coszbound}
	\aligned
	2|\cos(zt)|&=|\exp(ixt-yt)+\exp(-ixt+yt)|\\
	&\leqslant |\exp(ixt-yt)|+|\exp(-ixt+yt)|\\
	&\leqslant 2\exp(|y|t)\\
	&< 2e^{t/2} \quad (\because |y|<1/2).
	\endaligned
	\end{equation}
	
	\noindent We now calculate the supreme of $\left|\Xi(z)-E(n,b_n,z)\right|$ for $z\in S_{1/2}$.
	
	\begin{equation}
	\aligned
	\sup_{z\in S_{1/2}}\left|\Xi(z)-E(n,b_n,z)\right|
	\leqslant&\left|\int_0^{(1/2)\log b_n}2\cos(zt)\left(\Phi(t)-\Phi_n(t)\right){\rm d}t\right|\\
	+&\left|\int_{(1/2)\log b_n}^{\infty}2\cos(zt)\left(\Phi(t)\right){\rm d}t\right|\\
	\leqslant&\left|\int_0^{(1/2)\log b_n}\cos(zt)
	\left(\sum_{k=n+1}^{\infty}(\phi_k(t)+\phi_k(-t))\right)\mathrm{d}t\right|\\
	+&\int_{(1/2)\log b_n}^{\infty}2|\cos(zt)|\left(\sum_{k=1}^{\infty}\phi_k(t)\right)\mathrm{d}t\quad \because \phi_k(t)>0\\
	\leqslant&\sum_{k=n+1}^{\infty}\int_{0}^{(1/2)\log b_n}|\cos(zt)|\left|(\phi_k(t)+\phi_k(-t))\right|\mathrm{d}t\\
	+&\sum_{k=1}^{\infty}\int_{(1/2)\log b_n}^{\infty}2|\cos(zt)|\phi_k(t)\mathrm{d}t\\
	\leqslant
	&\sum_{k=1}^{\infty}\left(\int_{(1/2)\log b_n}^{\infty}2e^{t/2}\phi_k(t)\mathrm{d}t\right)\\
	&+\sum_{k=n+1}^{\infty}\left(\int_{0}^{(1/2)\log b_n}e^{t/2}(\phi_k(t))\mathrm{d}t+\int_{0}^{(1/2)\log b_n}e^{t/2}|\phi_k(-t)|\mathrm{d}t\right)\\
	=:&\delta_1(b_n)+\delta_2(n,b_n)\\
	\endaligned
	\end{equation}
	\noindent where
	
	\begin{equation}
	\aligned
	\delta_1(b_n)
	&:=\sum_{k=1}^{\infty}\int_{(1/2)\log b_n}^{\infty}2e^{t/2}\phi_k(t)\mathrm{d}t\\
	\delta_2(n,b_n)
	&:=\sum_{k=n+1}^{\infty}\int_{0}^{(1/2)\log b_n}e^{t/2}(\phi_k(t))\mathrm{d}t\\
	&+\sum_{k=n+1}^{\infty}\int_{0}^{(1/2)\log b_n}e^{t/2}|\phi_k(-t)|\mathrm{d}t\\
	\endaligned
	\end{equation}
	
	\noindent The interchange in integration and summation above will be justified after we show that $0<\delta_1(b_n)<\infty$,$0<\delta_2(n,b_n)<\infty$.
	
	We will now calculate $\delta_1(b_n)$. For $0\leqslant t<\infty,k \in\N$, we have
	
	\begin{equation}
	\aligned
	&\int_{(1/2)\log b_n}^{\infty}2e^{t/2} \phi_{k}(t)\mathrm{d}t\\
	&=\int_{(1/2)\log b_n}^{\infty}2e^{t/2} \left(4\pi^2k^4e^{9t/2}-6\pi k^2e^{5t/2}\right)\exp\left(-\pi k^2 e^{2t}\right)\mathrm{d}t\\
	&=4b_n^{3/2}\pi k^2 \exp(- b_n \pi k^2)\\
	&=:h_1(k,b_n)
	\endaligned
	\end{equation}
	
	\noindent where the last line defined function $h_1(k,b_n)$. Because 
	\begin{equation}
	-\partial_k h_1(k,b_n)= 8(b_n \pi k^2-1)b_n^{3/2}\pi k \exp(- b_n \pi k^2)>0, \quad \because b_n \pi k^2>\pi>1
	\end{equation}
	
	therefore $h_1(k,b_n)$is a positive and monotonic decreasing function  for $k \in \N$ . We then have ~\cite{D2014}
	\begin{equation}
	\aligned
	h_1(k,b_n)&\leqslant\int_{k-1}^{k}h_2(x,b_n)\mathrm dx \qquad (k\geqslant 2)\\
	\delta_2(b_n)&=\sum_{k=1}^{\infty}h_2(k,b_n)\leqslant h_2(1,b_n)+\int_{1}^{\infty}h_2(x,b_n)\mathrm dx\\
	&=2\sqrt{b_n}(1+2\pi b_n)\exp(-b_n\pi)
	+\mathrm{erfc}\left(\sqrt{b_n\pi}\right)\\
	&<2\sqrt{b_n}(1+2\pi b_n)\exp(-b_n\pi)+\exp(-b_n\pi)\\
	&=(2\sqrt{b_n}(1+2\pi b_n)+1)\exp(-b_n\pi)\\
	\endaligned
	\end{equation}
	
	\noindent In the above derivation, we make use of ~\cref{lemmaerfc}.
	
	\noindent Define
	\begin{equation}
	\aligned
	g(b)
	&=5\pi b^{3/2}-(2\sqrt{b}(1+2\pi b)+1)\\
	\endaligned
	\end{equation}
	Then
	\begin{equation}
	\aligned
	g(1)&=\pi-3>0\\
	\partial_b g(b)&=\frac{3\pi b-2}{2\sqrt{b}}>0, \quad b>1
	\endaligned
	\end{equation}
	
	Thus $g(b)>0, b>1$. Consequently we have
	
	\begin{equation}
	\aligned
	\delta_1(b_n)
	&<(2\sqrt{b_n}(1+2\pi b_n)+1)\exp(-b_n\pi)\\
	&<5\pi b_n^{3/2}\exp(-b_n\pi)\\
	\endaligned
	\end{equation}

	Next we calculate $\delta_2(n,b_n)$. For $0\leqslant t<\infty,k \in\N$, we have
	
	\begin{equation}
	\aligned
	&\int_{0}^{(1/2)\log b_n}e^{t/2} \phi_{k}(t)\mathrm{d}t\\
	&=\int_{0}^{(1/2)\log b_n}e^{t/2} \left(4\pi^2k^4e^{9t/2}-6\pi k^2e^{5t/2}\right)\exp\left(-\pi k^2 e^{2t}\right)\mathrm{d}t\\
	&=2\pi k^2 \left(\exp(-\pi k^2)-b_n^{3/2}\exp(-\pi k^2 b_n)\right)\\
	&=:h_{2a}(k,b_n)
	\endaligned
	\end{equation}
	
	\begin{equation}
	\aligned
	&\int_{0}^{(1/2)\log b_n}e^{t/2} |\phi_{k}(-t)|\mathrm{d}t\\
	&=\int_{0}^{(1/2)\log b_n}e^{t/2} \left|4\pi^2k^4e^{-9t/2}-6\pi k^2e^{-5t/2}\right|\exp\left(-\pi k^2 e^{-2t}\right)\mathrm{d}t\\
	&<\int_{0}^{(1/2)\log b_n}e^{t/2} \left(4\pi^2k^4e^{-5t/2}\right)\exp\left(-\pi k^2 e^{-2t}\right)\mathrm{d}t\\
	&=2\pi k^2 \left(\exp(-\pi k^2 /b_n)-\exp(-\pi k^2)\right)\\
	&=:h_{2b}(k,b_n)
	\endaligned
	\end{equation}
	
	\begin{equation}
	\aligned
	h_{2a}(k,b_n)+h_{2b}(k,b_n)&=2\pi k^2 \left(\exp(-\pi k^2 /b_n)-b_n^{3/2}\exp(-\pi k^2 b_n)\right)\\
	&<2\pi k^2 \left(\exp(-\pi k^2 /b_n)\right)\\
	&=:h_{2}(k,b_n)
	\endaligned
	\end{equation}
	
	Because $k^2\geqslant (n+1)^2>n^2>n^{2\beta}=b_n$, so
	
	\begin{equation}
	-\partial_k h_2(k,n,b_n)= 4 b_n^{-1} (\pi k^2-b_n)\pi k \exp(- \pi k^2/b_n )>0. 
	\end{equation}
	
	\noindent therefore $h_2(k,b_n)$ is a positive and monotonically decreasing function for $k\in[n+1,\infty)$. We have 
	\begin{equation}
	\aligned
	h_2(k,b_n)&\leqslant\int_{k-1}^{k}h_2(x,b_n)\mathrm dx \qquad (k\geqslant 2)\\
	\delta_2(n,b_n)&=\sum_{k=n+1}^{\infty}h_2(k,b_n)\leqslant \int_{n}^{\infty}h_1(x,b_n)\mathrm dx\\
	&=b_n (n+1)\exp(-(n+1)^2\pi/b_n)+2\pi n^2\exp(-n^2\pi/b_n)\\
	&+\frac{1}{2}b_n^{3/2}\mathrm{erfc}\left((n+1)\sqrt{\pi/b_n}\right)\\
	&<b_n (n+1)\exp(-(n+1)^2\pi/b_n)+2\pi n^2\exp(-n^2\pi/b_n)\\
	&+\frac{1}{2}b_n^{3/2}\exp\left(-(n+1)^2\pi/b_n\right)\\
	&<b_n (n+1)\exp(-n^2\pi/b_n)+2\pi n^2\exp(-n^2\pi/b_n)\\
	&+\frac{1}{2}b_n^{3/2}\exp\left(-n^2\pi/b_n\right)\\
	&<(2\pi n^2+b_n (n+1)+b_n^{3/2})\exp(-n^2\pi/b_n)\\
	&<(2\pi n^2+2n b_n+b_n^{3/2})\exp(-n^2\pi/b_n)\\
	\endaligned
	\end{equation}
\noindent In the above derivation, we make use of ~\cref{lemmaerfc} as well.
	
	\begin{equation}\label{deltadef}
	\aligned
	\delta(n,b_n):&=\delta_1(b_n)+\delta_2(n,b_n)\\
	&<5\pi b_n^{3/2}\exp(-b_n\pi)+(2\pi n^2+2n b_n+b_n^{3/2})\exp(-n^2\pi/b_n)\\
	\endaligned
	\end{equation}

	We now substitute $b_n=n^{2\beta}, 0<\beta<1$ into the equation above and obtain:
	\begin{equation} 
	\aligned
	\delta(n,n^{2\beta})
	&=5\pi n^{3\beta}\exp(-\pi n^{2\beta})+(2\pi n^2+2n^{1+2\beta}+n^{3\beta})\exp(-\pi n^{2(1-\beta)})\\
	&<5\pi n^{3}\exp(-\pi n^{2\beta})+5\pi n^3\exp(-\pi n^{2(1-\beta)})\\
	\endaligned
	\end{equation}
	
	It is obvious that 
	\begin{equation}
	\lim\limits_{n\to\infty}\delta(n,n^{2\beta})= 0\quad (\because 0<\beta<1).
	\end{equation}
	
	If $0<\beta_0\leqslant\beta\leqslant 1/2$, then 
	\begin{equation} 
	\aligned
	\delta(n,n^{2\beta})
	&\leqslant 5\pi n^{3\beta}\exp(-\pi n^{2\beta})+(2\pi n^2+2n^{1+2\beta}+n^{3\beta})\exp(-\pi n^{2\beta})\\
	&< 5\pi n^{2}\exp(-\pi n^{2\beta})+(2\pi n^2+2n^{2}+(\pi-2)n^{2})\exp(-\pi n^{2\beta})\\
	&= 8\pi n^{2}\exp(-\pi n^{2\beta})\\
	&< 8\pi n^{3}\exp(-\pi n^{2\beta}).\\
	&\leqslant 8\pi n^{3}\exp(-\pi n^{2\beta_0}).\\
	\endaligned
	\end{equation}
	
	If $1/2\leqslant\beta\leqslant 1-\beta_0< 1$, then 
	\begin{equation} 
	\aligned
	\delta(n,n^{2\beta})
	&\leqslant 5\pi n^{3\beta}\exp(-\pi n^{2(1-\beta)})+(2\pi n^2+2n^{1+2\beta}+n^{3\beta})\exp(-\pi n^{2(1-\beta)})\\
	&< 5\pi n^{3}\exp(-\pi n^{2(1-\beta)})+(2\pi n^3+2n^{3}+(\pi-2)n^{3})\exp(-\pi n^{2(1-\beta)})\\
	&= 8\pi n^{3}\exp(-\pi n^{2\beta_0}).\\
	\endaligned
	\end{equation}
	
	Thus
	\begin{equation} 
	\aligned
	\delta(n,n^{2\beta})
	&< 8\pi n^{3}\exp(-\pi n^{2\beta_0})\\
	&:=\tilde{\lambda}(n, \beta),\quad 0<\beta_0\leqslant \beta\leqslant 1-\beta_0<1\\
	\endaligned
	\end{equation}
	
	Since
	\begin{equation} 
	\aligned
	-\partial_n\log(\delta(n,n^{2\beta}))
	&=\frac{2\pi \beta_0 }{n}\left(n^{\beta_0}-\frac{3}{2\pi \beta_0 }\right)
	&=\frac{2\pi \beta_0 }{n}\left(n^{\beta_0}-\nu_0^{\beta_0}\right),\\
	\nu_0&=\left(\frac{3}{2\pi \beta_0 }\right)^{1/\beta_0}\\
	\endaligned
	\end{equation}
	$\delta(n,n^{2\beta})$ is positive and monotonically decreasing for $n\geqslant \lceil\nu_0\rceil$.
	
	Finally we can formulate the conclusion: Let $\nu_1$ be the solution to equation $\tilde{\lambda}(\nu_1, \beta)=\epsilon$. 
	For every $0<\epsilon <1$,  there exists a natural number $N=\max\{\lceil\nu_0\rceil\,,\lceil\nu_1\rceil\}$ such that for all $z\in S_{1/2}$ and for all $n\geqslant N$, we have  $\left|\Xi(z)-E(n,n^{2\beta},z)\right|<\epsilon$.Thus $E(n,n^{2\beta},z)$ converges to $\Xi(z)$ uniformly in $S_{1/2}$.
	
	If $\beta=1/2$, then
	
	\begin{equation} 
	\aligned
	\delta(n,n)
	&=5\pi n^{3/2}e^{-\pi n}+(2\pi n^2+2n^{2}+n^{3/2})e^{-\pi n}\\
	&<5\pi n^{2}e^{-\pi n}+(2\pi n^2+2 n^{2}+(\pi-2)n^2)e^{-\pi n}\\
	&=8\pi n^{2}e^{-\pi n}\\
	&<72(\pi n)^{6}e^{-\pi n}.\\
	\endaligned
	\end{equation}
	
	We now define $\lambda(n)$ as
	\begin{equation}\label{lambdandef}
	\aligned
	\lambda(n)&=72(\pi n)^{6}\exp(-\pi n)\\
	\endaligned
	\end{equation}
	
	\begin{equation} 
	\aligned
	-\partial_{n}\log\lambda(n)
	=(1/n)(n\pi-6)>0,\quad (\because n\geqslant 2)\\
	\endaligned
	\end{equation}
	Thus $\lambda(n)$ is a positive and monotonically decreasing function for $n\geqslant 2$.
	
	Denote $W_{-1}(x)$ the lower branch, $W_{-1}(x)\leqslant -1$, of the Lambert $W$ function.  $W_{-1}(x)$ decreases from $W_{-1}(-1/e)=-1$ to $W_{-1}(0^-)=-\infty$. We can explicitly solve $\lambda(\nu_0)=\epsilon$ for $\nu_0$ in terms of $W_{-1}(x)$ and obtain
	
	\begin{equation} 
	\aligned
	\nu_0=-(6/\pi)W_{-1}\left(-(1/6)(\epsilon/72)^{1/6}\right):=\nu_0(\epsilon),\quad 0<\epsilon < 1.
	\endaligned
	\end{equation}

	Numerical results show that if $\epsilon=10^{-2}$, then $N=\lceil\nu_0(\epsilon)\rceil=10$; if $\epsilon=10^{-10}$, then $N=\lceil\nu_0(\epsilon)\rceil=17$; $\epsilon=10^{-100}$, then $N=\lceil\nu_0(\epsilon)\rceil=86$.
	
\end{proof}

\begin{remark}\label{Eninftyz}
If we begin with setting  $b_n=\infty$ in \eqref{Enbnzdef}, then we obtain $\delta(n,\infty):=\sup_{z\in S_{1/2}}|E(n,\infty,z)-\Xi(z)|=\infty$. Thus we $E(n,\infty,z)$ of \eqref{Enbnzdef} does not uniformly converges to $\Xi(z)$ in $S_{1/2}$. 
	
	Setting $b_n=\infty$ in ~\eqref{Enbnzdef} and completing the integration, we obtain
	\begin{equation}\label{Enbnzdef2}
	\aligned
	E(n,\infty,z)&=2\int_{0}^{\infty}\Phi_n(t)\cos(zt)\mathrm{d}t\\
	&=\sum_{1\leqslant k\leqslant n} \left(\frac{\gamma(\frac{9}{4}+\frac{i}{2}z, \pi k^2)}{(\pi k^2)^{\frac{1}{4}+\frac{i}{2}z}}+\frac{\gamma(\frac{9}{4}-\frac{i}{2}z, \pi k^2)}{(\pi k^2)^{\frac{1}{4}-\frac{i}{2}z}}\right)\\
	&-\frac{3}{2}\sum_{1\leqslant k\leqslant n} \left(\frac{\gamma(\frac{5}{4}+\frac{i}{2}z, \pi k^2)}{(\pi k^2)^{\frac{1}{4}+\frac{i}{2}z}}+\frac{\gamma(\frac{5}{4}-\frac{i}{2}z, \pi k^2)}{2(\pi k^2)^{\frac{1}{4}-\frac{i}{2}z}}\right)\\
	\endaligned  
	\end{equation}
	
	Comparing ~\eqref{Enbnzdef2} against ~\eqref{XizDef4}, we find that $\lim\limits_{n\to\infty}E(n,\infty,z)=\Xi(z)$, but the convergence is not uniform.

\end{remark}

\begin{remark}\label{Einftymz}
	 If we begin with setting  $n=\infty, b=m, m\in\N$ in \eqref{Enbnzdef}, then we obtain $\delta(\infty,m):=\sup_{z\in S_{1/2}}|E(\infty,m,z)-\Xi(z)|<5\pi m^{3/2}\exp(-m\pi)$. It is obvious that $\lim\limits_{m\to\infty}\delta(\infty,m)=0$. Thus  $E(\infty,m,z)$ of \eqref{Enbnzdef} does uniformly converges to $\Xi(z)$ in $S_{1/2}$. However we will not continue along this direction because we only truncate the Fourier transform range from $(-\infty,\infty)$ to $(-\frac{1}{2}\log m, -\frac{1}{2}\log m)$ and there is still an infinite sum in the definition of the kernel $\Phi(t)$ (cf. ~\eqref{Phitdef2}). 
\end{remark}

\begin{remark}
From ~\eqref{Xizdef} and ~\eqref{Phitdef2}, we obtain
\begin{equation}\label{XizdefAll}
\aligned
\Xi(z)&=\frac{1}{2}\int_{-\infty}^{\infty}\exp(izt)\left(\Phi(t)+\Phi(-t)\right)\mathrm{ d}t\\
&=\frac{1}{2}\int_{-\infty}^{\infty}\exp(izt)\left(\sum_{1\leqslant k\leqslant \infty}\left(\phi_k(t)+\phi_k(-t)\right)\right).
\endaligned
\end{equation}	
where $\phi_k(t)$ is defined in ~\eqref{Phitdef2} and the invariance of the kernel with respect to $t\to -t$ is shown explicitly.\\
\\
\indent  There is an infinite summation and an infinite integration in ~\eqref{XizdefAll}.  Summarizing the findings in ~\cref{EnbnzUniformConvergence}, ~\cref{Eninftyz}, and ~\cref{Einftymz}, it is quite interesting to see that \\
(A) truncation of the Fourier transform range alone from $(-\infty, \infty)$ to $(-\frac{1}{2}\log m, \frac{1}{2}\log m)$ in ~\eqref{XizdefAll} \textbf{does} generate a family of entire functions that uniformly converges to $\Xi(z)$ in the critical strip $S_{1/2}$;\\
(B) truncation of the summation range alone from $[1, \infty)$ to $[1,n]$ in ~\eqref{XizdefAll} \textbf{does not} generate a family of entire functions that uniformly converges to $\Xi(z)$ in the critical strip $S_{1/2}$;\\
(C) truncation of the Fourier transform range from $(-\infty, \infty)$ to $(-\frac{1}{2}\log n, \frac{1}{2}\log n)$ and summation range from $[1, \infty)$ to $[1,n]$ in a concerted fashion in ~\eqref{XizdefAll} \textbf{does} generate  a family of entire functions that uniformly converges to $\Xi(z)$ in the critical strip $S_{1/2}$;
\end{remark}

\begin{lemma}[~\cite{CDS2003}]\label{lemmaerfc}
Let $\mathrm{erfc}(x)$ be the complementary error function, defined as:
\begin{equation} 
\mathrm{erfc}(x)=\frac{2}{\sqrt{\pi}}\int_{x}^{\infty}\exp(-t^2)\mathrm{d}t.
\end{equation}

then 
\begin{equation}
\mathrm{erfc}(x)<\exp(-x^2),x>0.
\end{equation}
\end{lemma}

For convenience in the subsequent sections, we will now set $\beta=1/2$ (thus $b_n=n$).
\begin{theorem} \label{FzUniformConvergence}
Let $n\in\N_0+2$, $\mathrm{sinc}(x)=(1/x)\sin x$,and

\begin{equation}\label{cnjdef}
\aligned
\omega_{n}&=(1/2)\log n\\
\varphi_{n,j}&=(j+1/4)\log n=2(j+1/4)\omega_n\\
S_{n,j}&=\sum_{k=1}^{n}k^{j}\\
c_{n,j}&=\log n\frac{(2j+1)\pi^j}{\Gamma(j)} S_{n,2j}=\omega_ n\frac{2(2j+1)\pi^j}{\Gamma(j)} S_{n,2j}\\
\endaligned
\end{equation}
Let

\begin{equation}\label{Fnzdef}
\aligned
F(n,z):=E(n,n,z)
&=\int_{-\omega_n}^{\omega_n}\Phi_n(t)\exp(izt)\mathrm{d}t,\\
\endaligned  
\end{equation}

Then\\ 
\noindent (1) $F(n,z)$ converges to $\Xi(z)$ uniformly in the critical strip $S_{1/2}$.\\
(2)
\begin{equation}\label{Fzgammadef}
\aligned
F(n,z)=
+&\sum_{k=1}^n \left(\frac{\gamma(\frac{9}{4}+\frac{i}{2}z, k^2\pi n)-\gamma(\frac{9}{4}+\frac{i}{2}z, k^2\pi /n)}{(k^2\pi)^{\frac{1}{4}+\frac{i}{2}z}}\right)\\
-\frac{3}{2}&\sum_{k=1}^n \left(\frac{\gamma(\frac{5}{4}+\frac{i}{2}z, k^2\pi n)-\gamma(\frac{5}{4}+\frac{i}{2}z, k^2\pi /n)}{(k^2\pi)^{\frac{1}{4}+\frac{i}{2}z}}\right)\\
+&\sum_{k=1}^n \left(\frac{\gamma(\frac{9}{4}-\frac{i}{2}z, k^2\pi n)-\gamma(\frac{9}{4}-\frac{i}{2}z, k^2\pi /n)}{(k^2\pi)^{\frac{1}{4}-\frac{i}{2}z}}\right)\\
-\frac{3}{2}&\sum_{k=1}^n \left(\frac{\gamma(\frac{5}{4}-\frac{i}{2}z, k^2\pi n)-\gamma(\frac{5}{4}-\frac{i}{2}z, k^2\pi /n)}{(k^2\pi)^{\frac{1}{4}-\frac{i}{2}z}}\right)\\
\endaligned
\end{equation}
(3)
\begin{equation}\label{Fzsincdef}
\aligned
F(n,z)=\sum_{j=1}^{\infty} (-1)^{j}c_{n,j}\left(\mathrm{sinc}(\omega_n z-i\varphi_{n,j})+\mathrm{sinc}(\omega_n z+i\varphi_{n,j})\right).\\
\endaligned
\end{equation}
(4)
\begin{equation}\label{Fzfourierdef}
\aligned
F(n,z)&=\int_{-\omega_n}^{\omega_n}\Phi_{2,n}(t)\exp(izt)\mathrm{d}t,\\
\endaligned
\end{equation}
\noindent where 
\begin{equation}\label{Phi2ntdef}
\aligned
\Phi_{2,n}(t)&=\frac{1}{\omega_n}\sum_{j=1}^{\infty} (-1)^j c_{n,j}\cosh(2t(j+1/4)),\\
\endaligned
\end{equation}

\begin{equation}\label{Phi2neqPhin}
\aligned
\Phi_{2,n}(t)&=\Phi_n(t).\\
\endaligned
\end{equation}
(5)
\begin{equation}\label{Fnzdef4}
\aligned
F(n,z)&=2\int_1^{n}x^{1/4}\Psi_n(x)\cos((z/2)\log x)\frac{\mathrm{ d}x}{x},\\
\endaligned
\end{equation}
With $s=iz+1/2$
\begin{equation}
\aligned
\tilde{F}(n,s)&:=F(n,z)
=\int_{1/n}^{n}x^{s/2}\Psi_n(x)\frac{\mathrm{ d}x}{x}\\
&=\int_{1}^{n}\Psi_n(x)(x^{s/2}+x^{(1-s)/2})\frac{\mathrm{ d}x}{x}\\
\endaligned
\end{equation}
\noindent where
\begin{equation}\label{Psintdef4}
\aligned
x^{1/4}\Psi_n(x)
&=x^{1/4}\sum_{k=1}^{n}\frac{1}{2}\left(\psi_k(x)+\psi_k(1/x)\right)=(1/x)^{1/4}\Psi_n(1/x)\\
\psi_k(x)&=\left(2(\pi k^2)^2 x^2-3\pi k^2x\right)\exp\left(-\pi k^2 x\right).
\endaligned
\end{equation}
\end{theorem}
\begin{proof}
Claim (1) is obvious because this is special case $\beta=1/2$ of ~\cref{FzUniformConvergence}.
Substituting $\Phi_n(t)$ of \eqref{Phindef} into \eqref{Enbnzdef} and carrying out the Fourier transform on compact support $(-\omega_n,\omega_n)$, we prove Claim (2).\\
 
Claim (3): The lower incomplete gamma function admits a series expansion:
\begin{equation}\label{gammaexpansiondef}
\aligned
\gamma(s,x):&=\int_0^{x}e^{-t}t^{s-1} dt=x^s\sum_{j=0}^{\infty}\frac{(-x)^j}{j!(j+s)}.
\endaligned
\end{equation}
\noindent Let 
\begin{equation}\label{fkdef}
\aligned
f_k(z):&=\frac{\gamma(\frac{9}{4}+\frac{i}{2}z, k^2\pi n)}{(k^2\pi)^{\frac{1}{4}+\frac{i}{2}z}}
-\frac{\gamma(\frac{9}{4}+\frac{i}{2}z, k^2\pi /n)}{(k^2\pi)^{\frac{1}{4}+\frac{i}{2}z}}\\
&=(k^2\pi)^2 n^{\frac{9}{4}+\frac{i}{2}z}\frac{\gamma(\frac{9}{4}+\frac{i}{2}z, k^2\pi n)}{(k^2\pi n)^{\frac{9}{4}+\frac{i}{2}z}}
-(k^2\pi)^2 n^{-\frac{9}{4}-\frac{i}{2}z}\frac{\gamma(\frac{9}{4}+\frac{i}{2}z, k^2\pi /n)}{(k^2\pi/n)^{\frac{9}{4}+\frac{i}{2}z}}\\
\endaligned
\end{equation}

\begin{equation}\label{gkdef}
\aligned
g_k(z)&:=-\frac{3}{2}\left(\frac{\gamma(\frac{5}{4}+\frac{i}{2}z, k^2\pi /n)}{(k^2\pi)^{\frac{1}{4}+\frac{i}{2}z}}
-\frac{\gamma(\frac{9}{4}+\frac{i}{2}z, k^2\pi n)}{(k^2\pi)^{\frac{1}{4}+\frac{i}{2}z}}\right)\\
&=-\frac{3}{2}(k^2\pi) n^{\frac{5}{4}+\frac{i}{2}z}\frac{\gamma(\frac{5}{4}+\frac{i}{2}z, k^2\pi n)}{(k^2\pi n)^{\frac{5}{4}+\frac{i}{2}z}}
+\frac{3}{2}(k^2\pi) n^{-\frac{5}{4}-\frac{i}{2}z}\frac{\gamma(\frac{5}{4}+\frac{i}{2}z, k^2\pi /n)}{(k^2\pi/n)^{\frac{5}{4}+\frac{i}{2}z}}\\
\endaligned
\end{equation}

\noindent Substitution of the series expansion of $\gamma(s, x)$ into \eqref{fkdef} and \eqref{gkdef} leads to
\begin{equation}
\aligned
f_k(z)=+&\sum_{j=0}^{\infty}\frac{(-1)^{j}(k^2\pi)^{j+2}}{j!(j+\frac{9}{4}+\frac{i}{2}z)}\left(n^{j+\frac{9}{4}+\frac{i}{2}z}-n^{-(j+\frac{9}{4}+\frac{i}{2}z)}\right)\\
\endaligned
\end{equation}

\begin{equation}
\aligned
g_k(z)=
-\frac{3}{2}&\sum_{j=0}^{\infty}\frac{(-1)^{j}(k^2\pi)^{j+1}}{j!(j+\frac{5}{4}+\frac{i}{2}z)}\left(n^{j+\frac{5}{4}+\frac{i}{2}z}-n^{-(j+\frac{5}{4}+\frac{i}{2}z)}\right)\\
\endaligned
\end{equation}

Further simplification of the result leads to
\begin{equation}
\aligned
f_k(z)
&=\sum_{j=2}^{\infty}\frac{(-1)^{j}(k^2\pi)^{j}}{(j-2)!(j+\frac{1}{4}+\frac{i}{2}z)}\left(n^{j+\frac{1}{4}+\frac{i}{2}z}-n^{-(j+\frac{1}{4}+\frac{i}{2}z)}\right)\\
\endaligned
\end{equation}

\begin{equation}
\aligned
g_k(z)
&=+\frac{3}{2}\sum_{j=1}^{\infty}\frac{(-1)^{j}(k^2\pi)^{j}}{(j-1)!(j+\frac{1}{4}+\frac{i}{2}z)}\left(n^{j+\frac{1}{4}+\frac{i}{2}z}-n^{-(j+\frac{1}{4}+\frac{i}{2}z)}\right)\\
\endaligned
\end{equation}

Since

\begin{equation}
\aligned
\frac{1}{(j-2)!}+\frac{3}{2}\frac{1}{(j-1)!}=\frac{2(j-1)+3}{2(j-1)!}=\frac{2j+1}{2(j-1)!},\quad j\geqslant 2\\
\endaligned
\end{equation}
and

\begin{equation}
\aligned
\left(\frac{2j+1}{2(j-1)!}\right)_{j=1}=\frac{3}{2}\\
\endaligned
\end{equation}

We obtain
\begin{equation}
\aligned
f_{k}(z)+g_{k}(z)=&\sum_{j=1}^{\infty}\frac{(-1)^{j}(2j+1)(k^2\pi)^{j}}{(j-1)!}\frac{(n^{j+\frac{1}{4}+\frac{i}{2}z}-n^{-(j+\frac{1}{4}+\frac{i}{2}z)})}{2\left(j+\frac{1}{4}+\frac{i}{2}z\right)}\\
\endaligned
\end{equation}

Writing $n^{\pm u}=\cosh (\log u)\pm \sinh (\log u)$ and simplification of the result using $\mathrm {sinhc}(i z):=\frac{\sinh (iz)}{iz}=\frac{\sin z}{z}=:\mathrm {sinc}(z)$ leads to
\begin{equation}
\aligned
f_{k}(z)+g_{k}(z)&=\log n \sum_{j=1}^{\infty}\frac{(-1)^{j}(2j+1)(k^2\pi)^{j}}{\Gamma(j)}\mathrm {sinc}(((z/2)-i(j+1/4))\log n)\\
&=\log n\sum_{j=1}^{\infty}\frac{(-1)^{j}(2j+1)(k^2\pi)^{j}}{\Gamma(j)}\mathrm {sinc}(\omega_n z-i\varphi_{n,j})\\
\endaligned
\end{equation}

Thus
\begin{equation}
\aligned
&\sum_{k=1}^n (f_{k}(z)+g_{k}(z))\\
=&\sum_{j=1}^{\infty}(-1)^{j}\log n\left(\frac{(2j+1)\pi^j }{\Gamma(j)}\right)\left(\sum_{k=1}^n k^{2j}\right)\mathrm {sinc}(\omega_n z-i\varphi_{n,j})\\
=&\sum_{j=1}^{\infty}(-1)^{j}c_{n,j}\mathrm {sinc}(\omega_n z-i\varphi_{n,j})\\
\endaligned
\end{equation}

Finally we have:
\begin{equation}\label{Fzfourierdef2}
\aligned
F(n,z)=&\sum_{k=1}^{\infty} (f_{k}(z)+g_{k}(z)+f_{k}(-z)+g_{k}(-z))\\
=&\sum_{j=1}^{\infty}(-1)^{j}c_{n,j}\left((\mathrm {sinc}(\omega_n z-i\varphi_{n,j})+\mathrm {sinc}(\omega_n z+i\varphi_{n,j})\right)\\
\endaligned
\end{equation}
Claim (4):

Since
\begin{equation}
\aligned
\mathrm{sinc}(z-i\varphi)+\mathrm{sinc}(z+i\varphi)
&=\int_{-1}^{1}\cosh(\varphi t)\exp(izt)\mathrm{d}t,\\
\endaligned
\end{equation}

We have
\begin{equation}\label{sincpsinc}
\aligned
&\mathrm {sinc}(\omega_n z-i\varphi_{n,j})+\mathrm {sinc}(\omega_n z+i\varphi_{n,j})\\
&=\int_{-1}^{1}\cosh(\varphi_{n,j}t)\exp(i\omega_n z t)\mathrm{d}t\\
&=\frac{1}{\omega_n}\int_{-\omega_n}^{\omega_n}\cosh(2t'(j+1/4))\exp(it'z)\mathrm{d}t',\quad (t=t'/\omega_n)\\
\endaligned
\end{equation}

Substituting \eqref{sincpsinc} into \eqref{Fzfourierdef2} and comparing the result  against \eqref{Fzfourierdef} leads to

\begin{equation}
\aligned
\Phi_{2,n}(t)&=\frac{1}{\omega_n}\left(\sum_{j=1}^{\infty}(-1)^jc_{n,j}\cosh(2t(j+1/4))\right)\\
&=\sum_{k=1}^n\sum_{j=1}^{\infty}(-1)^j\frac{(2j+1)(\pi k^2)^j}{\Gamma(j)}2\cosh(2t(j+1/4))\\
&=\sum_{k=1}^n\sum_{j=1}^{\infty}(-1)^j\frac{(2j+1)(\pi k^2)^j}{\Gamma(j)}\exp(2t(j+1/4))\\
&+\sum_{k=1}^n\sum_{j=1}^{\infty}(-1)^j\frac{(2j+1)(\pi k^2)^j}{\Gamma(j)}\exp(-2t(j+1/4))\\
&=e^{t/2}\sum_{k=1}^n\sum_{j=1}^{\infty}\frac{(2j+1)}{\Gamma(j)}(-\pi k^2 e^{2t})^j\\
&+e^{-t/2}\sum_{k=1}^n\sum_{j=1}^{\infty}\frac{(2j+1)}{\Gamma(j)}(-\pi k^2 e^{-2t})^j\\
&=\sum_{k=1}^n \left(2(\pi k^2)^2e^{9t/2}-3\pi k^2e^{5t/2}\right)\exp(-\pi k^2 e^{2t})\\
&+\sum_{k=1}^n \left(2(\pi k^2)^2e^{-9t/2}-3\pi k^2e^{-5t/2}\right)\exp(-\pi k^2 e^{-2t})\\
&=\sum_{k=1}^n (1/2)\left(\phi_k(t)+\phi_k(-t)\right)\\
&=\Phi_n(t)\\
\endaligned
\end{equation}
Claim (5):  From ~\eqref{Fzfourierdef} and ~\eqref{Phi2neqPhin}, we have
\begin{equation}
\aligned
F(n,z)&=\int_{-\omega_n}^{\omega_n}\Phi_{n}(t)\exp(izt)\mathrm{d}t,\\
\endaligned
\end{equation}

Substituting $2t=\log x, 2\mathrm{d}t=\frac{\mathrm{d}x}{x},\Phi_{n}(t)=2x^{1/4}\Psi(x)$ into above expression leads to 
\begin{equation}
\aligned
F(n,z)&=2\int_1^{n}x^{1/4}\Psi_n(x)\cos((z/2)\log x)\frac{\mathrm{ d}x}{x},\\
&=\int_{1/n}^{n}x^{1/4+iz/2}\Psi_n(x)\frac{\mathrm{ d}x}{x}.\\
\endaligned
\end{equation}
Setting $s=iz+1/2$ we obtain
\begin{equation}
\aligned
\tilde{F}(n,s)&:=F(n,z)
=\int_{1/n}^{n}x^{s/2}\Psi_n(x)\frac{\mathrm{ d}x}{x}\\
&=\int_{1}^{n}\Psi_n(x)(x^{s/2}+x^{(1-s)/2})\frac{\mathrm{ d}x}{x}\\
\endaligned
\end{equation}
\end{proof}

\begin{remark}
Using \eqref{Xizdef} and \eqref{Fzgammadef} we obtain numerical results: $\Xi(0)=0.497121$, $F(2,0)=0.443590$, $F(3,0)=0.493224$, $F(4,0)=0.496878$, $F(5,0)=0.497107$. Thus $\{F(n,0)\}_{n=2}^{\infty}$ converges to $\Xi(0)$ quite quickly.
\end{remark}
\newpage
\begin{remark}
What is the geometric meaning of the half-length of support, $\omega_n=(1/2)\log n$, of the Fourier transform in \eqref{Fnzdef}? 

Here are the plots of $\Phi_8(y\omega_8)$ vs. $y$ and $\Phi(y\omega_8)$ vs. $y$. Here $\omega_8=(1/2)\log 8$. We can clearly see that $\Phi_8(y\omega_8)<0, 2<y<4$. 

\begin{figure}[H]
\centering
\includegraphics[scale=0.48,keepaspectratio]{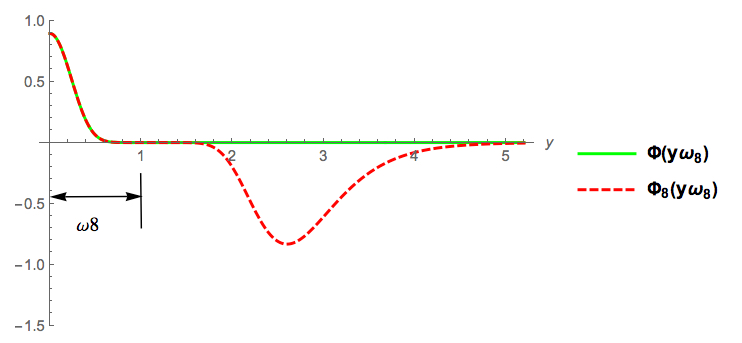}
\caption{Plots of $\Phi_8(y\omega_8)$ vs. $y$, $\Phi(y\omega_8)$ vs. $y$. Here $\omega_8=(1/2)\log 8$. }\label{figure3}
\end{figure}

Thus by truncating the Fourier transform range from $(-\infty,\infty)$ to \\$(-\omega_n, \omega_n)$, we exclude the contribution from the negative part of the kernel $\Phi_n(t)$ and only keep the contribution from the dominant positive part of the kernel $\Phi_n(t)$.

Let $\tau_n$ be the smallest positive zero of $\Phi_n(t)$. Numerical results show that\\ $\omega_{n+1}<\tau_n<\omega_{n+2}, n\in [2,12]$.  Further numerical calculation leads us to propose the following

\begin{conjecture}\label{Phitzeroconj}
For each $n\in\N, n\geqslant 2$, the function $\Phi_n(t), t\in[0,\infty)$ has only one zero, $\tau_n$, in the interval $(\omega_{n+1},\omega_{n+2})\subset [0,\infty)$. And $\Phi_n(0\leqslant t<\tau_n)>0$, $\Phi_n( \tau_n<t<\infty)<0$.
\end{conjecture}

All the reasoning in this remark based on numerical analysis is thus ad hoc in nature. We do not have a proof for  ~\cref{Phitzeroconj}. But our subsequent analysis does not depend on the proof/disproof of ~\cref{Phitzeroconj}.

\end{remark}

\begin{remark}
If we expand $\exp(-\pi k^2 e^{\pm 2t})$ of $\Phi_n(t)$ in \eqref{Phindef} as a Taylor series in terms of $\pi k^2 e^{\pm 2t}$,then we can directly obtain $\Phi_{2,n}(t)$ as in\eqref{Phi2ntdef}. We found this shortcut pretty late in our endeavor. Our current presentation follows closely to our original thought process. 
\end{remark}

In the last part of this section we will introduce an interesting and related recent work by Van Malderen \cite{VM2016} on alternating $\xi_a(s)$ ($\xi_a(s)$ function is related to alternating zeta function $\eta(s)$ just like Riemann $\xi(s)$ function is related to Riemann zeta function $\zeta(s)$). We summarize his results in the following theorem and his reasoning (scattered in his paper) in the proof.

\begin{theorem}[~\cite{VM2016}]Let
	\begin{equation} \label{etadef}
	\aligned
	\eta(s):&=(1-2^{1-s})\zeta(s)\\
	\eta(s)&=\sum_{n=1}^{\infty}\frac{(-1)^{n+1}}{n^s}\quad \Re s>0
	\endaligned
	\end{equation}
	\begin{equation}
	\aligned
	\phi(x)&:=\sum_{n=1}^{\infty}(-1)^{n+1}\exp(-\pi n^2 x)\quad x>0,\\
		\varphi(x)&:=\phi(x/4)-\phi(x),\\
	\endaligned
	\end{equation}
	\begin{equation}\label{xiasdef}
	\aligned
	\xi_a(s):=(2^{s}-1)\pi^{-s/2}\Gamma(s/2)\eta(s)
	\endaligned
	\end{equation}
	Then\\
	\noindent (1)
	\begin{equation}
	\varphi(x+2i)=\varphi(x),
	\end{equation}
	\noindent (2)
	\begin{equation}\label{varphiselfinverse}
	\varphi(1/x)=x^{1/2}\varphi(x),
	\end{equation}
	\noindent (3)
	\begin{equation}
	\aligned
	\varphi(x)&=\sum_{k=0}^{\infty}\exp(-\pi (4k+1)^2 x/4)
	+\sum_{k=0}^{\infty}\exp(-\pi (4k+3)^2 x/4)\\
	&-2\sum_{k=0}^{\infty}\exp(-\pi (4k+2)^2 x/4)\\
	\endaligned
	\end{equation}
	\noindent (4) $\xi_a(s)$ is an entire function and it can be expressed as a Mellin transform
	\begin{equation}
	\aligned
	\xi_a(s)&=\int_0^{\infty}x^{s/2}\varphi(x)\frac{\mathrm{d}x}{x}\qquad \Re(s)>0\\
	&=\int_1^{\infty}\varphi(x)\left(x^{s/2}+x^{(1-s)/2}\right)\frac{\mathrm{d}x}{x}\qquad \Re(s)>0\\
	\endaligned
	\end{equation}

\noindent (5)
\begin{equation}
\aligned
-\frac{\xi_a(s)}{(2^s-1)(2^{1-s}-1)}=\pi^{-s/2}\Gamma(s/2)\zeta(s)=-\frac{\xi(s)}{\frac{1}{2}s(1-s)}
\endaligned
\end{equation}

\noindent (6)
	\begin{equation}
\aligned
\xi_a(s)=-&\sum_{k=0}^{\infty}\left(\frac{\gamma\left(\frac{s}{2},\frac{\pi}{4}(4k+1)^2\right)}{\left(\frac{\pi}{4}(4k+1)\right)^s}+\frac{\gamma\left(\frac{1-s}{2},\frac{\pi}{4}(4k+1)^2\right)}{\left(\frac{\pi}{4}(4k+1)\right)^{1-s}}\right)\\
+2&\sum_{k=0}^{\infty}\left(\frac{\gamma\left(\frac{s}{2},\frac{\pi}{4}(4k+2)^2\right)}{\left(\frac{\pi}{4}(4k+2)\right)^s}+\frac{\gamma\left(\frac{1-s}{2},\frac{\pi}{4}(4k+2)^2\right)}{\left(\frac{\pi}{4}(4k+2)\right)^{1-s}}\right)\\
-&\sum_{k=0}^{\infty}\left(\frac{\gamma\left(\frac{s}{2},\frac{\pi}{4}(4k+3)^2\right)}{\left(\frac{\pi}{4}(4k+3)\right)^s}+\frac{\gamma\left(\frac{1-s}{2},\frac{\pi}{4}(4k+3)^2\right)}{\left(\frac{\pi}{4}(4k+3)\right)^{1-s}}\right)\\
\endaligned
\end{equation}	

\end{theorem}
\begin{proof}
	Claims (1), (2), (3): Note that 
\begin{equation}
\aligned
\phi(x)
&=\sum_{n=1}^{\infty}(-1)^{n+1}\exp(-\pi n^2 x)\\
&=\frac{1}{2}\left(1-\theta_4(0,ix)\right)
\endaligned
\end{equation}
Similarly
\begin{equation}
\aligned
\phi(x/4)
&=\sum_{n=1}^{\infty}(-1)^{n+1}\exp(-\pi n^2 x/4)\\
&=\sum_{k=1}^{\infty}(-1)^{(2k)+1}\exp(-\pi (2k)^2 x/4)\qquad n=2k\\
&+\sum_{k=0}^{\infty}(-1)^{(2k+1)+1}\exp(-\pi (2k+1)^2 x/4)\quad n=2k+1\\
&=\frac{1}{2}\left(1+\theta_2(0,ix)-\theta_3(0,ix)\right)\\
\endaligned
\end{equation}
Where $\theta_l(0,z),l=2,3,4$ are Jacobi theta functions defined
\begin{equation}
\aligned
\theta_4(0,z)
&=\sum_{n=-\infty}^{\infty}(-1)^{n}\exp(i\pi n^2 z)\quad \Im (z)>0\\
\theta_3(0,z)
&=\sum_{n=-\infty}^{\infty}\exp(i\pi n^2 z)\quad \Im (z)>0\\
\theta_2(0,z)
&=\sum_{n=-\infty}^{\infty}\exp(i\pi \left(n+1/2\right)^2 z)\quad \Im (z)>0\\
\endaligned
\end{equation}
Because $\theta_l(0,z),l=2,3,4$ are modular forms of weight $1/2$, satisfying the following two relations

\begin{equation}
\aligned
\theta_l(0,z+2)&=\theta_l(0,z)\quad l=2,3,4 \\
\theta_l(0,-1/z)&=(-iz)^{1/2}\theta_l(0,z)\quad l=2,3,4
\endaligned
\end{equation}
Thus
\begin{equation}
\aligned
\varphi(x)&=\phi(x/4)-\phi(x)\\
&=\frac{1}{2}\left(\theta_2(0,ix)-\theta_3(0,ix)+\theta_4(0,ix)\right)\\
\endaligned
\end{equation}

\begin{equation}
\aligned
\varphi(x+2i)&=\varphi(x)\\
\endaligned
\end{equation}
So we prove Claim (1).

And
\begin{equation}
\aligned
2\phi(1/x)
&=1-\theta_4(0,-1/ix)\\
&=1-x^{1/2}\theta_4(0,ix)\\
&=1-x^{1/2}\left(1-2\phi(x)\right)\\
\endaligned
\end{equation}
\begin{equation}
\aligned
2\phi(1/(4x))
&=1+\theta_2(0,-1/ix)-\theta_3(0,-1/ix)\\
&=1+x^{1/2}\left(\theta_2(0,ix)-\theta_3(0,ix)\right)\\
&=1+x^{1/2}\left(2\phi(x/4)-1\right)\\
\endaligned
\end{equation}
Therefore
\begin{equation}
\aligned
2\varphi(1/x)&=2\left(\phi(1/(4x))-\phi(1/x)\right)\\
&=2x^{1/2}\left(\phi(x/4)-\phi(x)\right)\\
&=2x^{1/2}\varphi(x)
\endaligned
\end{equation}
So we prove Claim (2).

\begin{equation}
\aligned
\phi(x/4)
&=\sum_{n=1}^{\infty}(-1)^{n+1}\exp(-\pi n^2 x/4)\\
&=\sum_{k=0}^{\infty}\exp(-\pi (4k+1)^2 x/4)\quad n=4k+1\\
&-\sum_{k=0}^{\infty}\exp(-\pi (4k+2)^2 x/4)\quad n=4k+2\\
&+\sum_{k=0}^{\infty}\exp(-\pi (4k+3)^2 x/4)\quad n=4k+3\\
&-\sum_{k=0}^{\infty}\exp(-\pi (4k+4)^2 x/4)\qquad n=4k+4\\
&=\sum_{k=0}^{\infty}\exp(-\pi (4k+1)^2 x/4)
+\sum_{k=0}^{\infty}\exp(-\pi (4k+3)^2 x/4)\\
&-\sum_{k=0}^{\infty}\exp(-\pi (4k+2)^2 x/4)
-\sum_{k=0}^{\infty}\exp(-\pi (2k+2)^2 x)\\
\endaligned
\end{equation}

\begin{equation}
\aligned
\phi(x)
&=\sum_{n=1}^{\infty}(-1)^{n+1}\exp(-\pi n^2 x)\\
&=\sum_{k=0}^{\infty}\exp(-\pi (2k+1)^2 x)\quad n=2k+1\\
&-\sum_{k=0}^{\infty}\exp(-\pi (2k+2)^2 x)\quad n=2k+2\\
\endaligned
\end{equation}
Thus
\begin{equation}\label{varphidef}
\aligned
\varphi(x)
&=\sum_{k=0}^{\infty}\exp(-\pi (4k+1)^2 x/4)
+\sum_{k=0}^{\infty}\exp(-\pi (4k+3)^2 x/4)\\
&-2\sum_{k=0}^{\infty}\exp(-\pi (4k+2)^2 x/4)\\
\endaligned
\end{equation}
So we prove Claim (3).\\	
	
Claim(4):	Consider the integral
	\begin{equation}
	\aligned
	\pi ^{-s/2}\Gamma(s/2)n^{-s}=\int_0^{\infty}x^{s/2}\exp(-\pi n^2 x)\frac{\mathrm{d}x}{x}\qquad \Re(s)>0
	\endaligned
	\end{equation}
	
	Then
	\begin{equation}\label{etadef2}
	\aligned
	\pi^{-s/2}\Gamma(s/2)\eta(s)=\int_0^{\infty}x^{s/2}\phi( x)\frac{\mathrm{d}x}{x}\qquad \Re(s)>0
	\endaligned
	\end{equation}
	Multiplying ~\eqref{etadef2} by $2^s$:

	\begin{equation}
	\aligned
	2^{s}\pi^{-s/2}\Gamma(s/2)\eta(s)
	&=\int_0^{\infty}(4x)^{s/2}\phi(x)\frac{\mathrm{d}x}{x}\\
	&=\int_0^{\infty}x^{s/2}\phi(x/4)\frac{\mathrm{d}x}{x}\qquad \Re(s)>0\\
	\endaligned
	\end{equation}
	
	Thus
	\begin{equation}
	\aligned
	(1-2^{s})\pi^{-s/2}\Gamma(s/2)\eta(s)
	&=\int_0^{\infty}x^{s/2}\left(\phi(x)-\phi(x/4)\right)\frac{\mathrm{d}x}{x}\qquad \Re(s)>0\\
	\endaligned
	\end{equation}
	We now define 
	
	\begin{equation}
	\aligned
	\xi_a(s)&:={\color{red}{-}}(1-2^{s})\pi^{-s/2}\Gamma(s/2)\eta(s)\\
	\endaligned
	\end{equation}
	
	\begin{equation}
	\aligned
	\varphi(x)&:={\color{red}{-}}\left(\phi(x)-\phi(x/4)\right)
	\endaligned
	\end{equation}
\begin{remark}
	Our definition of $\varphi(x),\xi_a(s)$ differs by a minus sign from those in ~\cite{VM2016}.  The reason is to make $\varphi(x)$  positive for $x>0$  (c.f. ~\cref{varphiPositive}).
\end{remark}	
	Then
	\begin{equation}
	\aligned
	\xi_a(s)
	&=\int_0^{\infty}x^{s/2}\varphi(x)\frac{\mathrm{d}x}{x}\qquad \Re(s)>0.\\
	\endaligned
	\end{equation}

Claim (5): Using ~\eqref{xisdef}, ~\eqref{etadef}, and ~\eqref{xiasdef}, we obtain
\begin{equation}
\aligned
-\frac{\xi_a(s)}{(2^s-1)(2^{1-s}-1)}=\pi^{-s/2}\Gamma(s/2)\zeta(s)=-\frac{\xi(s)}{\frac{1}{2}s(1-s)}
\endaligned
\end{equation}

\begin{remark}
It is well known that Riemann ~\cite{R1859} used the factor $s$ to compensate the simple pole of $\Gamma(s/2)$ at $s=0$ and the factor $(1-s)$ to compensate the simple pole of $\zeta(s)$ at $s=1$. Van Malderen ~\cite{VM2016} used the factor $(2^s-1)$ to compensate the simple pole of $\Gamma(s/2)$ at $s=0$ and the factor $(2^{1-s}-1)$ to compensate the simple pole of $\zeta(s)$ at $s=1$. Of course that $\xi_a(s)$ has extra zeros at $s\in\left(\frac{2\pi i}{\log 2}\mathbb{Z}\right)\cup \left(1-\frac{2\pi i}{\log 2}\mathbb{Z}\right)$,i.e., on the boundaries of the critical strip $\{0<\Re(s)<1\}$.
\end{remark}

Claim (6): Now we have
\begin{equation}\label{xiasdef2}
\aligned
\xi_a(s)
&=\int_1^{\infty}x^{s/2}\varphi(x)\frac{\mathrm{d}x}{x}
+\int_0^{1}x^{s/2}\varphi(x)\frac{\mathrm{d}x}{x}\\
&=\int_1^{\infty}x^{s/2}\varphi(x)\frac{\mathrm{d}x}{x}
+\int_1^{\infty}y^{(1-s)/2}\varphi(y)\frac{\mathrm{d}y}{y}\quad y=1/x\\
&=\int_1^{\infty}\varphi(x)\left(x^{s/2}+x^{(1-s)/2}\right)\frac{\mathrm{d}x}{x}\qquad \Re(s)>0\\
&=\xi_a(1-s)
\endaligned
\end{equation}
The wording of H.M. Edwards used for the case of $\xi(s)$ [2, p.16] is also applicable to ~\eqref{xiasdef2}: Because $\varphi(x)$ decreases more rapidly than any power of $x$ as $x\to\infty$, the integral in ~\eqref{xiasdef2} converges in the entire $s$ plane.

Substituting ~\eqref{varphidef}	 into ~\eqref{xiasdef3} and completing the integration leads to
	\begin{equation}
	\aligned
	\xi_a(s)=+&\sum_{k=0}^{\infty}\left(\frac{\Gamma\left(\frac{s}{2},\frac{\pi}{4}(4k+1)^2\right)}{\left(\frac{\pi}{4}(4k+1)\right)^s}+\frac{\Gamma\left(\frac{1-s}{2},\frac{\pi}{4}(4k+1)^2\right)}{\left(\frac{\pi}{4}(4k+1)\right)^{1-s}}\right)\\
	-2&\sum_{k=0}^{\infty}\left(\frac{\Gamma\left(\frac{s}{2},\frac{\pi}{4}(4k+2)^2\right)}{\left(\frac{\pi}{4}(4k+2)\right)^s}+\frac{\Gamma\left(\frac{1-s}{2},\frac{\pi}{4}(4k+2)^2\right)}{\left(\frac{\pi}{4}(4k+2)\right)^{1-s}}\right)\\
	+&\sum_{k=0}^{\infty}\left(\frac{\Gamma\left(\frac{s}{2},\frac{\pi}{4}(4k+3)^2\right)}{\left(\frac{\pi}{4}(4k+3)\right)^s}+\frac{\Gamma\left(\frac{1-s}{2},\frac{\pi}{4}(4k+3)^2\right)}{\left(\frac{\pi}{4}(4k+3)\right)^{1-s}}\right)\\
	\endaligned
	\end{equation}
	
	Using $\Gamma(s,x)=\Gamma(s)-\gamma(s,x)$ and $\xi_a(s)=\xi_a(1-s)$, we obtain
	
	\begin{equation}\label{xiasdef3}
	\aligned
	\xi_a(s)=-&\sum_{k=0}^{\infty}\left(\frac{\gamma\left(\frac{s}{2},\frac{\pi}{4}(4k+1)^2\right)}{\left(\frac{\pi}{4}(4k+1)\right)^s}+\frac{\gamma\left(\frac{1-s}{2},\frac{\pi}{4}(4k+1)^2\right)}{\left(\frac{\pi}{4}(4k+1)\right)^{1-s}}\right)\\
	+2&\sum_{k=0}^{\infty}\left(\frac{\gamma\left(\frac{s}{2},\frac{\pi}{4}(4k+2)^2\right)}{\left(\frac{\pi}{4}(4k+2)\right)^s}+\frac{\gamma\left(\frac{1-s}{2},\frac{\pi}{4}(4k+2)^2\right)}{\left(\frac{\pi}{4}(4k+2)\right)^{1-s}}\right)\\
	-&\sum_{k=0}^{\infty}\left(\frac{\gamma\left(\frac{s}{2},\frac{\pi}{4}(4k+3)^2\right)}{\left(\frac{\pi}{4}(4k+3)\right)^s}+\frac{\gamma\left(\frac{1-s}{2},\frac{\pi}{4}(4k+3)^2\right)}{\left(\frac{\pi}{4}(4k+3)\right)^{1-s}}\right)\\
	\endaligned
	\end{equation}

\end{proof}
\begin{remark}
	We now copy a similar formula ~\eqref{xisdef2} for $\xi(s)$ here for comparison
	\begin{equation}\label{xisdef4}
	\aligned
	\xi(s)=&\sum_{1\leqslant k\leqslant \infty} \left(\frac{\gamma(2+\frac{s}{2}, \pi k^2)}{(\pi k^2)^{s/2}}+\frac{\gamma(2+\frac{1-s}{2}, \pi k^2)}{(\pi k^2)^{(1-s)/2}}\right)\\
	-\frac{3}{2}&\sum_{1\leqslant k\leqslant \infty} \left(\frac{\gamma(1+\frac{s}{2}, \pi k^2)}{(\pi k^2)^{s/2}}+\frac{\gamma(1+\frac{1-s}{2}, \pi k^2)}{(\pi k^2)^{(1-s)/2}}\right)\\
	\endaligned
	\end{equation}
	We notice the great similarity between ~\eqref{xisdef3} and ~\eqref{xiasdef3}. We do not use ~\eqref{xisdef4} in the subsequent analysis because the truncated series of  ~\eqref{xisdef4} does not converges to $\xi(s)$ uniformly in $S_{1/2}$ (cf. ~\cref{Eninftyz}). It will be interesting to see whether the truncated series of  ~\eqref{xiasdef3} converges to $\xi_a(s)$ uniformly in $S_{1/2}$ or not.
	
\end{remark}
\begin{lemma}\label{varphiPositive} Let $x>0$ and
	\begin{equation}
	\aligned
\varphi(x)
	&=\sum_{k=0}^{\infty}\exp(- (4k+1)^2 \pi x/4)
	+\sum_{k=0}^{\infty}\exp(-(4k+3)^2 \pi x/4)\\
	&-2\sum_{k=0}^{\infty}\exp(- (4k+2)^2 \pi x/4)\\
	\endaligned
	\end{equation}
	Then $\varphi(x)>0$.
	\begin{proof}
		\begin{equation}
		\aligned
		\varphi(x)&=\sum_{k=0}^{\infty}\exp(-(4k+2)^2 \pi x/4)f(k,x)\\
		f(k,x):&=\exp((8k+3)\pi x/4)-2+\exp(-(8k+5)\pi x/4)\\
		\endaligned
		\end{equation}
If $x\geqslant 1$, then
\begin{equation}
\aligned
f(k,x)
&>1+(8k+3)\pi x/4-2+0\\
&\geqslant 3\pi x/4-1\\
&\geqslant 3\pi/4-1>0.\\
\endaligned
\end{equation}
		If $0<x<1$, then using ~\eqref{varphiselfinverse} we obtain 
\begin{equation}
\aligned
\varphi(x)&=\varphi(1/y)\qquad y=1/x>1\\
&=y^{1/2}\varphi(y)>0\qquad \because y>1\\
\endaligned
\end{equation}
	\end{proof}
\end{lemma}

\section{\textbf{A family of functions $G(m,n,z)$ that converges to $F(n,z)$ uniformly in $S_{1/2}$} }

The expressions for $F(n,z)$ in \eqref{Fzsincdef} and \eqref{Fzfourierdef} contain one finite sum for $k$ hidden in the definition of $c_{n,j}$ [cf. ~\eqref{cnjdef}] and still contain an infinite sum for $j$. In this section, in order to simply it further, we would like  to truncate the infinite sum and only keep the first $m$ terms ($m\in \N$).

We first introduce two lemmas.

\begin{lemma}[~\cite{R1955}]\label{GammaUpperLowerBound} 
	
	An upper bound and a lower bound for $\Gamma(n+1)$, valid for all positive integers $n$, are given by
	\begin{equation}
	\aligned
	\sqrt{2\pi n}\left(\frac{n}{e}\right)^n\exp(1/(12n+1))<\Gamma(n+1)<\sqrt{2\pi n}\left(\frac{n}{e}\right)^n\exp(1/(12n))
	\endaligned
	\end{equation}
	\\
\end{lemma}

\begin{lemma}[\cite{I2007}]\label{HupperLowerBound}
	An lower/upper bound for the sum of $j$th power of the first $n$ integers, $S_{n,j}$, is given by:
	\begin{equation}
	\aligned
	n^j+\frac{(n-1)^{j+1}}{j+1}<S_{n,j}:=\sum_{k=1}^n k^j<n^j+\frac{n^{j+1}}{j+1}\\
	\endaligned
	\end{equation}
\end{lemma}

\begin{proof}
	For fixed $j\in \N$, since $k^{j}$ is monotonically increasing for $k\in [1,n]$, thus
	\begin{equation}
	\aligned
	S_{n-1,j}&=\sum_{k=1}^{n-1} k^j
	<\int_1^{n} x^j \mathrm{d}x=\frac{n^{j+1}-1}{j+1}<\frac{n^{j+1}}{j+1}\\
	S_{n-1,j}&=\sum_{k=1}^{n-1} k^j
	>\int_0^{n-1} x^j \mathrm{d}x=\frac{(n-1)^{j+1}}{j+1}\\
	n^j+\frac{(n-1)^{j+1}}{j+1}<S_{n,j}&=n^j+S_{n-1,j}<n^j+\frac{n^{j+1}}{j+1}\\
	\endaligned
	\end{equation}
	
\end{proof}

\begin{theorem} \label{GmnzUniformConvergence}
Let $n\in\N_0+2, m\in\N$, and

\begin{equation}\label{c_{n,j}}
\aligned
\omega_n&=(1/2)\log n\\
\varphi_{n,j}&=(j+1/4)\log n\\
S_{n,j}&=\sum_{k=1}^{n}k^{j}\\
c_{n,j}&=\log n\frac{(2j+1)\pi^j}{\Gamma(j)} S_{n,2j}\\
\endaligned
\end{equation}

Let

\begin{equation}\label{Gmnzdef}
\aligned
G(m,n,z)&=\sum_{j=1}^m (-1)^{j}c_{n,j}\left(\mathrm{sinc}(\omega_nz-i\varphi_{n,j})+\mathrm{sinc}(\omega_nz+i\varphi_{n,j})\right)\\
\endaligned
\end{equation}

Then $G(m,n,z)$ converges to $F(n,z)$ uniformly in $S_{1/2}$.
\end{theorem}
\begin{proof}
Let $z=x+iy$,
\begin{equation}
\aligned
&2|\sin(\omega_nz-i\varphi_{n,j})|\\
&2|\sin((z/2-i(j+1/4))\log n)|\\
&\leqslant|\exp(+i (x/2) \log n-(j+1/4-y/2)\log n)|\\
&+|\exp(-i (x/2) \log n+(j+1/4-y/2)\log n)|\\
&\leqslant \exp(-(j+1/4-y/2)\log n)\\
&+\exp(+(j+1/4-y/2)\log n)\\
&\leqslant 2\exp((j+3/4)\log n)\\
&=2n^{j+3/4}\\
&<2n^{j+1}\\
\endaligned
\end{equation}

\begin{equation}
\aligned
|\omega_nz-i\varphi_{n,j}|
&=|(z/2-i(j+1/4))\log n|\\
&=(\log n)(x^2/4+(j+1/4-y/2)^2)^{1/2}\\
&> (\log n)(0+(3/4)^2)^{1/2}\\
&=(3/4)\log n.
\endaligned
\end{equation}

Using the upper bound for $S_{n,j}$ shown in ~\cref{HupperLowerBound}, we have
\begin{equation}
\aligned
c_{n,j}&<\log n\frac{(2j+1)\pi^j}{\Gamma(j)} n^{2j}\left(1+\frac{n}{2j+1}\right)\\
&<\log n\frac{\pi^j n^{2j}(2j+1+n)}{\Gamma(j)}\\
\endaligned
\end{equation}

Therefore
\begin{equation}
\aligned
&\sup_{z\in S_{1/2}}|G(m,n,z)-F(n,z)|\\
&\leqslant \sum_{j=m+1}^{\infty} c_{n,j}\left(|\mathrm{sinc}(\omega_nz-i\varphi_{n,j})|
+|\mathrm{sinc}(\omega_nz+i\varphi_{n,j})|\right)\\
&<\log n\sum_{j=m+1}^{\infty}\frac{\pi^j n^{2j}(2j+1+n)}{\Gamma(j)}\left(\frac{n^{j+1}}{(3/4)\log n}+\frac{n^{j+1}}{(3/4)\log n}\right)\\
&=\sum_{j=m+1}^{\infty}\frac{8\pi^j n^{3j+1}(2j+1+n)}{3\Gamma(j)}\\
&=\frac{8n}{3}\sum_{j=m+1}^{\infty}\frac{(\pi n^{3})^j}{\Gamma(j)}(2j+1+n)\\
\endaligned
\end{equation}

Assuming, 
\begin{equation}\label{mgt2pien3}
\aligned
m>2+\pi e n^3,
\endaligned
\end{equation}

so for $j\geqslant m+1$

\begin{equation}
\aligned
\frac{9}{4}(j-1)-(2j+1+n)
&=\frac{1}{4}(j-4n-13)\\
&\geqslant\frac{1}{4}(\pi e n^3-4n-10)\\
&>\frac{1}{4}(8 n^3-4n-12)\quad \because\pi e\approx 8.53973>8\\
&=2n^3-n-3\\
&=2n^2(n-2)+(4n+3)(n-1)\\
&>0\qquad \because n\geqslant 2
\endaligned
\end{equation}

\begin{equation}
\aligned
\sup_{z\in S_{1/2}}|G(m,n,z)-F(n,z)|
&<\frac{8n}{3}\sum_{j=m+1}^{\infty}\frac{(\pi n^{3})^j}{\Gamma(j)}\frac{9}{4}(j-1)\\
&=6n\sum_{j=m+1}^{\infty}\frac{(\pi n^{3})^j(j-1)}{\Gamma(j)}\\
&=6n\sum_{j=m+1}^{\infty}\frac{(\pi n^{3})^j}{\Gamma(j-1)}\\
\endaligned
\end{equation}

Using the lower bound for $\Gamma(j-1)$ shown in ~\cref{GammaUpperLowerBound} we obtain:

\begin{equation}
\aligned
&\sup_{z\in S_{1/2}}|G(m,n,z)-F(n,z)|\\
&<6n\sum_{j=m+1}^{\infty}(\pi n^{3})^{j}\frac{1}{\sqrt{2\pi (j-2)}}\left(\frac{e}{j-2}\right)^n\exp(-1/(12(j-2)+1))\\
&<6n\sum_{j=m+1}^{\infty}(\pi n^{3})^{j}\frac{1}{2}\left(\frac{e}{j-2}\right)^{j-2}\quad \because j\geqslant m+1>3+\pi e n^3>3,\\
&=3\pi^2 n^{7}\sum_{j=m+1}^{\infty}\left(\frac{\pi e n^{3}}{j-2}\right)^{j-2}\\
&=3\pi^2 n^{7}\sum_{j=m}^{\infty}\left(\frac{\pi e n^{3}}{j-1}\right)^{j-1}\\
\endaligned
\end{equation}

\begin{equation}
\aligned
&\sup_{z\in S_{1/2}}|G(m,n,z)-F(n,z)|\\
&<3\pi^2 n^{7}\sum_{j=m}^{\infty}\left(\frac{\pi e n^{3}}{m-1}\right)^{j-1},\qquad \because j\geqslant m\\
&=3\pi^2 n^{7}\left(\frac{\pi e n^{3}}{m-1}\right)^{m-1}\sum_{j=0}^{\infty}\left(\frac{\pi e n^{3}}{m-1}\right)^{j}\\
&=3\pi^2 n^{7}\left(\frac{\pi e n^{3}}{m-1}\right)^{m-1}\left(1-\frac{\pi e n^{3}}{m-1}\right)^{-1}\\
&=3\pi^2 n^{7}\left(\frac{\pi e n^{3}}{m-1}\right)^{m-1}\left(\frac{m-1}{m-1-\pi e n^{3}}\right)\\
&<3\pi^2 n^{7}\left(\frac{\pi e n^{3}}{m-1}\right)^{m-1}\left(m-1\right),\qquad \because m>2+\pi en^3\\
&<3\pi^2 n^{7}(\pi e n^{3})\left(\frac{\pi e n^{3}}{m-1}\right)^{m-2}\\
&<3\pi^2 n^{7}(\pi e n^{3})\left(\frac{\pi e n^{3}}{m-2}\right)^{m-2}\\
&=3e\pi^3 n^{10}\left(\frac{\pi e n^{3}}{m-2}\right)^{m-2}\\
\endaligned
\end{equation}

We now define
\begin{equation}\label{rhomndef}
\aligned
\rho(m,n):=3e\pi^3 n^{10}\left(\frac{\pi e n^3}{m-2}\right)^{m-2}\\
\endaligned
\end{equation}

Because $ m> 2+\pi e n^3$ (cf. \eqref{mgt2pien3}), therefore
\begin{equation} 
\aligned
-\partial_{m}\ln \rho(m,n)
&=\log \left(\frac{m-2}{\pi n^3}\right)\\
&>\log \left(\frac{\pi e n^3}{\pi n^3}\right)\\
&=1>0,\\
&\therefore \partial_{m}\log \rho(m,n)<0.
\endaligned
\end{equation}
Thus $\rho(m,n)$ is a positive and monotonically decreasing function for $m>2+\pi e n^3$.

\noindent Let $W_0(x)$ be the upper branch, $W_0(x)\geqslant-1$, of the Lambert W function.  $W_0(x)$ increase from $W_0(-1/e)=-1$ through $W_0(0)=0$ to $W_0(\infty)=\infty$. We can explicitly solve $\rho(\mu_0,n)=\epsilon$ for $\mu_0$ in terms of $\epsilon,n$.  Here are the detailed steps:
\begin{equation} 
\aligned
\rho(\mu_0,n)&=3e\pi^3 n^{10}\left(\frac{\pi e n^3}{\mu_0-2}\right)^{\mu_0-2}=\epsilon\\
\endaligned
\end{equation}

\begin{equation} 
\aligned
\left(\frac{\pi e n^3}{\mu_0-2}\right)^{\mu_0-2}&=\frac{\epsilon}{3e\pi^3 n^{10}}\\
\endaligned
\end{equation}

\begin{equation}\label{omegadef}
\aligned
\left(\frac{\pi e n^3}{\mu_0-2}\right)^{(\mu_0-2)/(\pi e n^3)}&=\left(\frac{\epsilon}{3e\pi^3 n^{10}}\right)^{1/(\pi e n^3)} =:\exp(-\eta(\epsilon,n))\\
\endaligned
\end{equation}
where
\begin{equation}\label{omegaendef}
\aligned
\eta(\epsilon,n)&=\frac{1}{\pi e n^3}\log\left(\frac{3e\pi^3 n^{10} }{\epsilon}\right)\\
\endaligned
\end{equation}
From \eqref{omegadef}, we obtain
\begin{equation}\label{sigmaxdef} 
\aligned
\frac{\mu_0-2}{\pi e n^3}&=\frac{\eta(\epsilon,n)}{W_{0}(\eta(\epsilon,n))}:=\sigma(\eta(\epsilon,n))\\
\endaligned
\end{equation}

\begin{equation}\label{mu0endef} 
\aligned
\mu_0:&=\mu_0(\epsilon,n)=2+(\pi e n^3)\sigma(\eta(\epsilon,n))\\
\endaligned
\end{equation}

\noindent Because for $0< x< \infty $, $\partial_x W_0(x)=\frac{W_0(x)}{x(1+W_0(x))}>0$, thus $W_0(x)$ is a positive and monotonically increasing function of $x$. Since $W_0(xe^x)=x>0$, therefore $0<W_0(x)<x$ and $\sigma(x)=\frac{x}{W_0(x)}>1$. So we have

\begin{equation} 
\aligned
\mu_0&>2+\pi e n^3\\
\endaligned
\end{equation}

Thus when $m\geqslant \lceil\mu_0\rceil$, the assumption $m>2+\pi e n^3$ that we make in \eqref{mgt2pien3} is automatically satisfied.

Finally we can formulate the conclusion: Let $\mu_0$ be the solution to equation $\rho(\mu_0,n)=\epsilon$ as shown in \eqref{rhomndef}.
For any $0<\epsilon < 1$,  there exists a natural number $M=\lceil\mu_0(\epsilon,n)\rceil$ such that for all $z\in S_{1/2}$ and for all $m\geqslant M$, we have  $\left|G(m,n,z)-F(n,z)\right|\leqslant\epsilon$.  Thus $G(m,n,z)$ converges to $F(n,z)$ uniformly in $S_{1/2}$.
\end{proof}

\begin{lemma}[~\cite{WkLambertW}]\label{LambertW0}

\noindent (1) The Taylor series of Lambert-W function $W_0(x)$ around $0$ can be found using Lagrange inversion theorem and is given by

\begin{equation} 
\aligned
W_0(x)=\sum_{n=1}^{\infty}\frac{(-n)^{n-1}}{n!}x^n=x-x^2+\frac{3}{2}x^3-\frac{8}{3}x^3+\frac{125}{24}x^3-\cdots.
\endaligned
\end{equation}
The radius of convergence is $e^{-1}$,as may be seen by the ratio test. 

(2) For $0<x<e^{-1}$, we have

\begin{equation} 
\aligned
1<\sigma(x):=\frac{x}{W_0(x)}<\frac{e}{e-1}\approx 1.58198
\endaligned
\end{equation}

\end{lemma}
\begin{proof}
The proof of Claim (1) can be found in ~\cite{WkLambertW}.

To prove Claim (2), we notice that the Taylor series is alternating in sign so for $0<x<e^{-1}$, we have

\begin{equation} 
\aligned
x&>W_0(x)>x-x^2=x(1-x)>x(1-e^{-1})\\
1&<\sigma(x)=\frac{x}{W_0(x)}<\frac{1}{1-x}<\frac{e}{e-1}\\
\endaligned
\end{equation}
\end{proof}

We also need upper bounds for various $n$ dependent functions defined below.

\begin{lemma} \label{mu0UpperBound}
Let $n \in \N_0+2$ and $\lambda(n)$ as defined in \eqref{lambdandef}, $\eta(\epsilon,n)$ as defined in \eqref{omegaendef}, $\sigma(x)$ as defined in \eqref{sigmaxdef}, $\mu_0(\epsilon,n)$ as defined in \eqref{mu0endef}, $\rho(m,n)$ as defined in \eqref{rhomndef},
let 
\begin{equation} 
\aligned
\tilde{\eta}(n)&:=\eta(\lambda(n),n),\\
\tilde{\sigma}(n)&:=\sigma(\eta(\lambda(n),n))=\frac{\tilde{\eta}(n)}{W_0(\tilde{\eta}(n))},\\
\tilde{\mu_0}(n)&:=\mu_0(\lambda(n),n)=2+(\pi e n^3) \tilde{\sigma}(n),\\
m(l,n)&:=2+l n^3,l\geqslant 9,l\in\N,\\
\tilde{\rho}(l,n)&:=\rho(m(lo,n),n),
\endaligned
\end{equation}
Then
\begin{equation} 
\aligned
\tilde{\eta}(n)&<\tilde{\eta}(2)\approx 0.0504045,\\
\tilde{\sigma}(n)&<\tilde{\sigma}(2)\approx 1.04921,\\
\tilde{\mu}_0(n)&<2+9n^3,\\
\partial_n\tilde{\rho}(l,n)&<0.\\
\endaligned
\end{equation}

\end{lemma}
\begin{proof}

\begin{equation} 
\aligned
-\partial_n \tilde{\eta}(n)
&=-\partial_n \left(\frac{1}{\pi e n^3}\log\left(\frac{3e\pi^3 n^{10}}{722(\pi n)^{6}e^{-\pi n}}\right)\right)\\
&=-\partial_n \left(\frac{1}{\pi e n^3}\log\left(\frac{e n^{4}}{24\pi^3e^{-\pi n}}\right)\right)\\
&=\frac{1}{\pi e n^4}
\left(2\pi n+12\log n-3\log(24\pi^2)-1\right)\\
&\geqslant\frac{1}{\pi e n^4}
\left(4\pi+12\log 2-3\log(24\pi^2)-1\right)\quad (\because n\geqslant 2)\\
&\approx \frac{1}{\pi e n^4}0.0474063>0,\\
\endaligned
\end{equation}

\begin{equation} 
\aligned
-\partial_n \tilde{\sigma}(n)&=\frac{1}{1+W_{0}(\tilde{\eta}(n))}\left(-\partial_n \tilde{\eta}(n)\right)>0.\\
\endaligned
\end{equation}

So $\tilde{\eta}(n),\tilde{\sigma}(n)$ are positive and monotonic decreasing functions for $n\geqslant 2$.

Thus $\tilde{\eta}(n)<\tilde{\eta}(2)\approx 0.0504045$, $\tilde{\sigma}(n)<\tilde{\sigma}(2)\approx 1.04921$ .

\begin{equation} 
\aligned
\tilde{\mu_0}(n)&=2+\pi e\tilde{\sigma}(n) n^3 \\
&<2+\pi e \tilde{\sigma}(2) n^3 \\
&<2+9 n^3,\quad \left(\because \pi e \tilde{\sigma}(2)\approx 8.95999<8.96<9\right)\\
\endaligned
\end{equation}

\begin{equation} 
\aligned
-n\partial_n\log \tilde{\rho}(l,n)
&=3ln^3\log\left(\frac{l}{\pi e}\right)-10\\
&\geqslant 27  n^3\log\left(\frac{9}{\pi e}\right)-10\quad \because l\geqslant 9\\
&\geqslant 27\cdot 2^3\log\left(\frac{9}{\pi e}\right)-10\quad \because n\geqslant 2\\
&\approx 1.33885>0\\
\endaligned
\end{equation}

therefore
\begin{equation} 
\aligned
\partial_n\tilde{\rho}(l,n)<0.
\endaligned
\end{equation}

\end{proof}

\begin{theorem} \label{GqnnUniformConvergence}
Let $n\in\N_0+2$, and $\Xi(z)$ as defined in ~\cref{Xizdef}, $F(n,z)$ as defined in \eqref{Fzsincdef} of ~\cref{FzUniformConvergence}, $G(m,n,z)$ as defined in \eqref{Gmnzdef} of ~\cref{GmnzUniformConvergence}. Let
\begin{equation}
\aligned
m(l,n)=2+l n^3,l\in\N_0+9\\
\endaligned
\end{equation}

\begin{equation}\label{Hnzdef}
\aligned
H(l,n,z):=G(m(l,n),n,z).\\
\endaligned
\end{equation}

Then $H(l,n,z)$ converges to $\Xi(z)$ uniformly in $S_{1/2}$.
\end{theorem}

\begin{proof}
For given $0<2\epsilon<1$,let $\nu_1,\nu_2$ be the unique and positive solutions to the following two equations

\begin{equation}
\aligned
\lambda(\nu_1)&=\epsilon\\
\tilde{\rho}(\nu_2)=\rho(q(\nu_2),\nu_2)&=\epsilon\\
\endaligned
\end{equation}

Let $N=\max(\lceil\nu_1\rceil,\lceil\nu_2\rceil)$. When $n\geqslant N$, because 
\begin{equation}
\aligned
\partial_{\nu}\lambda(\nu)&<0,\\
\endaligned
\end{equation}

We have
\begin{equation}\label{mu0ineq1}
\aligned
\lambda(n)&\leqslant\epsilon\\
\mu_0(\lambda(n),n)&\geqslant \mu_0(\epsilon,n)\\
\endaligned
\end{equation}

From ~\cref{mu0UpperBound} we know that

\begin{equation}\label{mu0ineq2}
\aligned
m(l,n)=2+l n^3\geqslant 2+9n^3>\mu_0(\lambda(n),n).\\
\endaligned
\end{equation}

\noindent Combination of \eqref{mu0ineq1} with \eqref{mu0ineq2} leads to
\begin{equation}
\aligned
m(l,n)>\mu_0(\epsilon,n).\\
\endaligned
\end{equation}

From ~\cref{mu0UpperBound} we also know that for $\mu>2+\pi en^3$
\begin{equation}
\aligned
\partial_{\mu}\rho(\mu,n)&<0,\\
\endaligned
\end{equation}

\noindent Thus we have
\begin{equation}
\aligned
\rho(m(l,n),n)&<\rho(\mu_0(\epsilon,n),n)=\epsilon.\\
\endaligned
\end{equation}

\noindent Since
\begin{equation}
\aligned
|F(n,z)-\Xi(z)|
&\leqslant\lambda(n)\leqslant\epsilon\\
\endaligned
\end{equation}
\noindent and
\begin{equation}
\aligned
|H(l,n,z)-F(n,z)|
&\leqslant\rho(m(l,n),n)\leqslant \epsilon,\\
\endaligned
\end{equation}

\noindent we finally obtain
\begin{equation}
\aligned
|H(l,n,z)-\Xi(z)|
&\leqslant |H(l,n,z)-F(n,z)|+|F(n,z)-\Xi(z)|\\
&\leqslant\epsilon+\epsilon=2\epsilon.\\
\endaligned
\end{equation}

\noindent Thus $H(l,n,z)$ converges to $\Xi(z)$ uniformly in $S_{1/2}$.

Numerical results show that if $\epsilon=2\times10^{-2}$, then $n=\lceil\nu_0(\epsilon)\rceil=10$,$m(9,n)=2+9n^3=9,002$; if $\epsilon=2\times 10^{-10}$, then $n=\lceil\nu_0(\epsilon)\rceil=17$, $m(9,n)=2+9n^3=44,219$; $\epsilon=2\times 10^{-100}$, then $n=\lceil\nu_0(\epsilon)\rceil= 86$, $m(9,n)=2+9n^3=5,724,506$.

\end{proof}

\begin{theorem}[Hurwitz's theorem]\label{Herwitz}
	Let $\{f_n(z)\}_{n=0}^{\infty}$ be a sequence of functions, analytic on a region $R(z)\subset\C$, which converges uniformly to a function $f(z),f(z)\not=0$, on a compact set, $D(z)\subset  R(z)$. If $z_0\in D(z)$ is a limit point of zeros of the sequence $\{f_n(z)\}_{n=0}^{\infty}$, then $f(z_0)=0$.
	
	Conversely, if $z_0\in D(z)$,$f(z_0)=0$, possibly of multiplicity greater than one, and $B(z)\subset D(z)$ is a neighborhood of $z_0$ which contains no other zeros of $f(z)$ inside of it or on its boundary, then there exists $N\in\N$ such that for $n\geqslant N$, $f_n(z)$ and $f(z)$ have the same number of zeros inside $B(z)$ (counting multiplicities).
\end{theorem}

\begin{corollary}[of Hurwitz's theorem \cite{C1978}, VII.2.5-6]\label{corollary2Hurwitz}
	if $f(z)$ and $\{f_n(z)\}$ are analytic functions on a domain $D(z)$, $\{f_n(z)\}$ converges to $f(z)$ uniformly on compact subsets of $D(z)$, and all but finitely many $f_n(z)$ have no zeros in $D(z)$, then either $f(z)$ is identically zero or $f(z)$ has no zeros in $D(z)$.
\end{corollary}

\begin{theorem} \label{GqnnzUniformConvergence}
Let $n\in\N_0+2, l\in\N_0+9$, $\Xi(z)$ as defined in ~\eqref{Xizdef}, $H(l,n,z)$ as defined in \eqref{Hnzdef} of ~\cref{GqnnUniformConvergence}.

If there exists a sufficiently large and positive integer $N$ such that for all  $n\geqslant N$ all the zeros of $H(l,n,z)$ in $S_{1/2}$ are real,

then all the zeros of $\Xi(z)$ in $S_{1/2}$ are real. 
\end{theorem}

\begin{proof}

Both $H(l,n,z)$ and $\Xi(z)$ are entire functions so they are analytical in $S_{1/2}$. In ~\cref{GqnnzUniformConvergence} we prove that $H(l,n,z)$ converges to $\Xi(z)$ uniformly in $S_{1/2}$.
\\
\\ 
\indent Since for all $n\geqslant N$ all the zeros of $H(l,n,z)$ in $S_{1/2}$ are real, we have that $\{H(l,n,z)\}_{n\geqslant N}$, have no zeros in $B(z):=\left(S_{1/2}\setminus \R\right)$. Since $\Xi(z)$ is not identically zero, we deduce by applying ~\cref{corollary2Hurwitz} of Hurwitz's Theorem that $\Xi(z)$ has no zeros in $B(z)$.   
\\
\\
\indent Since $\Xi(z)$ is not identically zero and all the zeros of $\Xi(z)$ are in $S_{1/2}$, therefore all the zeros of $\Xi(z)$ in $S_{1/2}=B(z)\cup\R$ must be in $\R$. Thus they are all real.  
\end{proof}

\begin{remark}
Why we prefer to use ~\cref{corollary2Hurwitz} of Hurwitz's Theorem instead of of Hurwitz's Theorem itself (~\cref{Herwitz} ) in this proof is explained in ~\cref{zeroCounting}.
\end{remark}

Our goal in the next sections becomes to prove that all the zeros of $\{H(l,n,z)\}_{n\geqslant N}$ in $S_{1/2}$ are real.
\\

For future convenience we now set $l=14$ and study the zeros of  $W(n,z):=H(14,n,z/\omega_n)=H(14,n,2z/\log n)$. 

\begin{lemma} \label{HWdef}
Let $n\in\N_0+2$, $H(l,n,z)$ as defined in \eqref{Hnzdef} of ~\cref{GqnnUniformConvergence}. Let

\begin{equation}\label{WtoH}
\aligned
W(n,z):&=H(14,n,2z/\log n)\\
\endaligned
\end{equation}

If there exists a sufficiently large and positive integer $N\geqslant 8$ such that all the zeros of $\{W(n,z)\}_{n\geqslant N}$ in $S_{1/2}$ are real,

then all the zeros of $\{H(14,n,z)\}_{n\geqslant N}$ in $S_{1/2}$ are real. 
\end{lemma}

\begin{proof}
Because of ~\eqref{WtoH}, all the zeros of $\{W(n,z)\}_{n\geqslant N}$ in $S_{1/2}=\{z:|\Im(z)|\leqslant 1/2\}$ are real is equivalent to that all the zeros of $\{H(14,n,z)\}_{n\geqslant N}$ in $S_{(1/4)\log n}=\{z:|\Im(z)|\leqslant (1/4)\log n\}$ are real.

Since $(1/2)\log n\geqslant (1/2)\log N\geqslant (1/2)\log 8\approx 1.03972>1$, we have $S_{1/2}\subset S_{(1/4)\log n}$. The claim then follows.

Using \eqref{Hnzdef} we obtain

\begin{equation}\label{Wnzdef0}
\aligned
W(n,z)&=\sum_{j=1}^{2+14n^3} (-1)^{j}c_{n,j}\left(\mathrm{sinc}(z-i\varphi_{n,j})+\mathrm{sinc}(z+i\varphi_{n,j})\right)\\
\endaligned
\end{equation}
where
\begin{equation}
\aligned
\varphi_{n,j}&=(j+1/4)\log n\\
c_{n,j}&=\log n\frac{(2j+1)\pi^j}{\Gamma(j)} S_{n,2j}\\
S_{n,j}&=\sum_{k=1}^{n}k^{j}\\
\endaligned
\end{equation}

Comparing ~\eqref{Wnzdef0} against \eqref{Hnzdef}, we find out that $\omega_n=\frac{1}{2}\log n$ is no longer multiplying $z$ in ~\eqref{Wnzdef0}.  Thus the function $W(n,z)$ is little bit simpler than the function $H(14,n,z)$. This is the reason we apply this scaling change. 
\end{proof}

\section{\textbf{$W(n,z)=U(n,z)-V(n,z)$}}

We will first introduce a theorem by P\'{o}lya.
\begin{theorem}[~\cite{P1918}]\label{PolyaTheorem}
If $f(t)$ is positive, continuous, and increasing function for $t\in [0,1)$, then \\
(1)the function
\begin{equation}\label{Ufdef}
U(f;z):=\int_{0}^{1}f(t)\cos(zt)\mathrm{d}t
\end{equation}
\noindent has only real and simple zeros.\\
(2) Each interval $\left((2k-1)\pi/2, (2k+1)\pi/2\right), k \in\N$ contains exactly one zero.
\end{theorem}

\begin{proposition} \label{w2gtw1prop}
Let $a=1,2; x\in\R; n\in 2\N+1; m=7n^3$ (i.e., $n,m$ are odd), and

\begin{equation}
\aligned
S_{n,j}&=\sum_{k=1}^{n}k^{j}\\
\varphi_{n,j}&=(j+1/4)\log n,\\
c_{n,j}&=  \log n\frac{(2j+1)\pi^j}{\Gamma(j)} S_{n,2j}.\\
\endaligned
\end{equation}

\begin{equation}
\aligned
w_{a}(n,x)&:=u_{a}(n,x)/v_{a}(n,x),\\
u_{a}(n,x)&:=\sum_{0\leqslant j\leqslant m}c_{n,2j+a}\frac{\varphi_{n,2j+a}\sinh \left(\varphi_{n,2j+a}\right)}{x^2+\varphi_{n,2j+a}^2},\\
v_{a}(n,x)&:=\sum_{0\leqslant j\leqslant m}c_{n,2j+a}\frac{\cosh \left(\varphi_{n,2j+a}\right)}{x^2+\varphi_{n,2j+a}^2},\\
\endaligned
\end{equation}

Then there exists a sufficiently large and positive integer $N$ such that

\begin{equation}\label{Xinmz2}
\aligned
u_{2}(n,x)>u_{1}(n,x),\\
v_{2}(n,x)>v_{1}(n,x),\\
w_{2}(n,x)>w_{1}(n,x),\\
\endaligned
\end{equation}
hold for all $n\geqslant N$ and all $x\in\R$.

\end{proposition}
\begin{proof}
see section 9 ~\cref{w3gtw2gtw1}.
\end{proof}

We now introduce a theorem by Zhou.
\begin{theorem}[~\cite{Z2016}]\label{RootsInterlacing}
\text{ }\newline
Suppose $c < d$.

(A) $w_1(x), w_2(x)$ are smooth, positive functions for $x \in [c, d]$;

(B) $h(x)$ is a smooth, nonnegative and monotonically decreasing function
for $x \in [c, d]$;

(C) $\lim\limits_{x\to c^+}h(x) = +\infty, h(d) = 0$;

(D) $0 < w_1(x) < w_2(x) < +\infty$ for $x \in [c, d]$;

(E) there is exactly one $x_1 \in (c, d)$ such that $w_1(x_1) = h(x_1)$;

(F) there is exactly one $x_2 \in (c, d)$ such that $w_2(x_2) = h(x_2)$.

Then $c < x_2 < x_1 < d$.
\end{theorem}

\begin{proof}: We first prove that for any $t\in [x_1, d]$, $w_1(t) \geqslant h(t)$.

Assume there was $x\in [x_1, d]$ such that $w_1(x) < h(x)$. Since $w_1(d) > 0 =
h(d)$, there would be $x' \to \in (x, d)$ such that $w_1(x') = h(x')$. This contradicts (E).

From this assertion, we have $w_2(x) > w_1(x) \geqslant h(x)$ for any $x \in [x_1, d]$.
Since $w_2(x_2) = h(x_2)$, $x_2 < x_1$ must hold.
\end{proof}

\begin{theorem} \label{xknltykn}
Let $a=1,2; n\in2\N+1; m=7 n^3$, $\mathrm{sinc}(z)=(1/z)\sin z$,and 

\begin{equation}
\aligned
\varphi_{n,j}&=(j+1/4)\log n,\\
S_{n,j}&=\sum_{k=1}^{n}k^{j},\\
c_{n,j}&=  \log n\frac{(2j+1)\pi^j}{2\Gamma(j)} S_{n,2j}.\\
\endaligned
\end{equation}

Let
\begin{equation}\label{Wnzdef}
\aligned
W(n,z):&=\sum_{j=1}^{2m+2} (-1)^{j}c_{n,j}
\left(\mathrm{sinc}( z-i\varphi_{n,j})+\mathrm{sinc}( z+i\varphi_{n,j})\right)\\
\endaligned
\end{equation}
\noindent then

\noindent (1)
\begin{equation}
\aligned
W(n,z)
&=U(n,z)-V(n,z)\\
\endaligned
\end{equation}

where
\begin{equation}\label{UdefVdef}
\aligned
U(n,z)&=2\int_{0}^{1}f_{n,2}(t)\cos(zt)\mathrm{d}t,\\
V(n,z)&=2\int_{0}^{1}f_{n,1}(t)\cos(zt)\mathrm{d}t,\\
\endaligned
\end{equation}
\noindent and 
\begin{equation}
\aligned
f_{n,a}(t)
&=\sum_{j=0}^{m} c_{n,2j+a}\cosh(\varphi_{n,2j+a}t),\\
\endaligned
\end{equation}

\noindent(2) $U(n,z)$ and $V(n,z)$ have only real zeros.

\noindent(3A) Each interval $I_k:=\left((k-1/2)\pi, (k+1/2)\pi \right), k \in\N$ contains exactly one zero,$x_{k}(n)$, for $U(n,x)$ and exactly one zero, $y_{k}(n)$, for $V(n,x)$. i.e., $x_{k}(n)\in I_k$ and $y_{k}(n)\in I_k$. 

\noindent (3B) $-x_{k}(n),k\in\N$ are also zeros of  $U(n,x)$. $-y_{k}(n),k\in\N$ are also zeros of  $V(n,x)$.
\\
\\
\begin{remark}
 The claim (3A) and (3B) account for all the zeros of $U(n,x)$ and $V(n,x)$.
\end{remark}

\noindent(4A) The real zeros of $U(n,x)$ are determined by
\begin{equation}
x\tan x + w_{2}(n,x)=0.
\end{equation}

\noindent(4B) The real zeros of $V(n,x)$ are determined by
\begin{equation}
x\tan x + w_{1}(n,x)=0.
\end{equation}

where
\begin{equation}\label{w1w2w3def}
\aligned
w_{a}(n,x)&:=u_{a}(n,x)/v_{a}(n,x),\\
u_{a}(n,x)&:=\sum_{j=0}^{m_n}c_{n,2j+a}\frac{\varphi_{n,2j+a}\sinh \left(\varphi_{n,2j+a}\right)}{x^2+\varphi_{n,2j+a}^2},\\
v_{a}(n,x)&:=\sum_{j=0}^{m_n}c_{n,2j+a}\frac{\cosh \left(\varphi_{n,2j+a}\right)}{ x^2+\varphi_{n,2j+a}^2},\\
\endaligned
\end{equation}

\noindent(5) The intervals $I_k$  can be shortened to $J_k$ such that

\begin{equation}
\aligned
x_{k}(n), y_{k}(n)&\in J_k:=\left((k-1/2)\pi, k\pi \right)\subset I_k, &k \in\N\\
\endaligned
\end{equation}

\noindent There exists a sufficiently large and positive integer $N$ such that

\noindent (6A) the positive zeros of $U(n,x)$ are left-interlacing with the positive zeros of $V(n,x)$. i.e,
\begin{equation}
\aligned
(k-1/2)\pi&<x_{k}(n)<y_{k}(n)< k\pi, &k \in\N,\quad & n\geqslant N,\\
\endaligned
\end{equation}

\noindent (6B) the negative zeros of $V(n,x)$ are right-interlacing with the negative zeros of $U(n,x)$. i.e,
\begin{equation}
\aligned
-k\pi&<-y_{n,k}<-x_{n,k}< -(k-1/2)\pi, &k \in\N,\quad & n\geqslant N.
\endaligned
\end{equation}

\end{theorem}

\begin{proof}

Claim(1): \\
For $\varphi_{n,j}=(j+1/4)\log n>0$,we have
\begin{equation}
\aligned
\mathrm{sinc}(z-i\varphi_{n,j})+\mathrm{sinc}(z+i\varphi_{n,j})
&=\int_{-1}^{1}\cosh(\varphi_{n,j}t)\exp(izt)\mathrm{d}t\\
&=2\int_{0}^{1}\cosh(\varphi_{n,j}t)\cos(zt)\mathrm{d}t\\
\endaligned
\end{equation}

Substitution of this into \eqref{Wnzdef} and separation of the terms even in $j$ from the terms odd in $j$ and comparison of the results against \eqref{UdefVdef} leads to Claim (1).

Claim (2) and Claim (3A):
\\
To apply P\'{o}lya ~\cref{PolyaTheorem}, we need to prove that $f_{n,a}(t)$ are positive and increasing for $t\in [0,1)$. Thus it is suffice to prove that (A) $f_{n,a}(0)>0$,(B) $\partial_t f_{n,a}(t)>0,t\in (0,1)$.

Since

\begin{equation}
\aligned
f_{n,a}(0)&=\sum_{0\leqslant j\leqslant m} c_{n,2j+a}>0,\\\\
\endaligned
\end{equation}
And
\begin{equation}
\aligned
&\partial_t f_{n,a}(t)\\
&=\sum_{0\leqslant j\leqslant m} c_{n,2j+a}\varphi_{n,2j+a}\sinh(\varphi_{n,2j+a}t)>0,\quad t\in (0,1).\\
\endaligned
\end{equation}
Claim(3B): This is obvious because $U(n,z)$ and $V(n,z)$ are even functions of $z$.\\

Claim (4A) and Claim (4B): To study the real zeros of $U(n,z), V(n,z)$, we can now focus on $U(n,x),V(n,x),x\in\R$.

Since
\begin{equation}
\aligned
&\mathrm{sinc}(x-i\varphi)+\mathrm{sinc}(x+i\varphi)\\
&\frac{\sin(x-i\varphi)}{x-i\varphi}+\frac{\sin(x+i\varphi)}{x+i\varphi}\\
=&\frac{2}{x^2+\varphi^2}\left(\varphi  \sinh \varphi \cos x+\cosh \varphi (x \sin x) \right),\\
\endaligned
\end{equation}
 
We have
\begin{equation}
\aligned
U(n,x)&=\sum_{j=0}^{m}c_{n,2j+2}
\left(\mathrm{sinc}( x-i\varphi_{n,2j+2})+\mathrm{sinc}( x+i\varphi_{n,2j+2})\right)\\
&=2\left(\sum_{j=0}^{m}c_{n,2j+2}\frac{\varphi_{n,2j+2}\sinh \varphi_{n,2j+2}}{ x^2+\varphi_{n,2j+2}^2}\right)\cos x\\
&+2\left(\sum_{j=0}^{m}c_{n,2j+2}\frac{\cosh \varphi_{n,2j+2}}{x^2+\varphi_{n,2j+2}^2}\right)x\sin x\\
&=2u_{2}(n,x)\cos x  +2v_{2}(n,x) x\sin x  \\
&=2\left(u_{2}^2(n,x)+x^2v_{2}^2(n,x)\right)^{1/2}\\
&\times\cos\left(x-\arctan\left(x \frac{v_{2}(n,x)}{u_{2}(n,x)}\right)\right)
\endaligned
\end{equation}

The real zeros for $U(n,x)$ are then given by:
\begin{equation}
\aligned
x&-\arctan\left(x \frac{v_{2}(n,x)}{u_{2}(n,x)}\right)=2k\pi+\pi/2,\quad k\in \Z\\
\endaligned
\end{equation}
\noindent or
\begin{equation}
\aligned
g_{2}(n,x):&=x\tan x+w_{2}(n,x)=0\\
\endaligned
\end{equation}

Similarly we have
\begin{equation}
\aligned
V(n,x)&=\sum_{j=0}^{m}c_{n,2j+1}
\left(\mathrm{sinc}( x-i\varphi_{n,2j+1})+\mathrm{sinc}( x+i\varphi_{n,2j+1})\right)\\
&=2\left(\sum_{j=0}^{m}c_{n,2j+1}\frac{\varphi_{n,2j+1}\sinh \varphi_{n,2j+1}}{ x^2+\varphi_{n,2j+1}^2}\right)\cos x\\
&+2\left(\sum_{j=0}^{m}c_{n,2j+1}\frac{\cosh \varphi_{n,2j+1}}{x^2+\varphi_{n,2j+1}^2}\right)x\sin x\\
&=2u_{1}(n,x)\cos x  +2v_{1}(n,x) x\sin x  \\
&=2\left(u_{1}^2(n,x)+x^2v_{1}^2(n,x)\right)^{1/2}\\
&\times\cos\left(x-\arctan\left(x \frac{v_{1}(n,x)}{u_{1}(n,x)}\right)\right)
\endaligned
\end{equation}

The real zeros for $V_{n}(x)$ are then given by:
\begin{equation}
\aligned
g_{1}(n,x):&=x\tan x+w_{1}(n,x)=0\\
\endaligned
\end{equation}

Claim (5A): We notice that
\begin{equation}
\aligned
g_{a}\left(n,(k-1/2)\pi+0^+\right)=-\infty<0, a=1,2;k\in\N\\
\endaligned
\end{equation}

\begin{equation}
\aligned
g_{a}\left(n,k\pi\right)=w_{a}\left(n,k\pi\right)>0,a=1,2;\quad k\in\N\\
\endaligned
\end{equation}
so there is one zero for $U(n,x)$ and one zero for $V(n,x)$ in $J_k$. Combining this result with Claim (3A), we conclude that each interval $J_k, k \in\N$ contains exactly one zero for $U(n,x)$ and one zero for $V(n,x)$. Claim (5B) follows because $U(n,x)$ and $U(n,x)$ are even functions of $x$.

Claim (6A):

Let $c=(k-1/2)\pi,d=k\pi$ and 

\begin{equation}\label{hxdef}
\aligned
h(x)&:=-x\tan x
\endaligned
\end{equation}

From the definition \eqref{w1w2w3def} and \eqref{hxdef}, we know that

\noindent (a) $w_1(n,x),w_2(n,x)$ are smooth, positive functions of $x$;

\noindent (b) $h(x)$ is a smooth, nonnegative and monotonically decreasing (thus single-valued) function for $x\in\left((k-\frac{1}{2})\pi,k\pi\right)$;

\noindent (c) $h(c+0^+)=+\infty,h(d)=0$.

From \cref{w2gtw1prop}, which is proved in \cref{w3gtw2gtw1}, we know that

\noindent (d) $0<w_1(n,x)<w_2(n,x)<\infty,  n\geqslant N$,

From Claim (3A), we know that 

\noindent (e) the curve $h(x)$ intersects curve $w_1(n,x)$ only once at $x=y_k(n)\in (c,d)$.  

\noindent (f) the curve $h(x)$ also intersects curve $w_2(n,x)$ only once at $x= x_k(n)\in (c,d)$

Using Zhou's ~\cref{RootsInterlacing}, we conclude that

\begin{equation}
(k-1/2)\pi=c<x_k(n)<y_k(n)<d=k\pi,\quad n\geqslant N.
\end{equation}

Claim (6B) follows because $U(n,x)$ and $V(n,x)$ are even functions of $x$.

\end{proof}

\begin{lemma}
	Let $a=1,2; x\in\R; n\in 2\N+1; m=7n^3$ (i.e., $n,m$ are odd), and
	
	\begin{equation}
	\aligned
	S_{n,j}&=\sum_{k=1}^{n}k^{j}\\
	\varphi_{n,j}&=(j+1/4)\log n,\\
	c_{n,j}&=  \log n\frac{(2j+1)\pi^j}{\Gamma(j)} S_{n,2j}.\\
	\endaligned
	\end{equation}
	
	\begin{equation}\label{uan0def}
	\aligned
	u_{a}(n,0):&=\sum_{0\leqslant j\leqslant m}c_{n,2j+a}\frac{\sinh \left(\varphi_{n,2j+a}\right)}{\varphi_{n,2j+a}},\\
	&=\sum_{0\leqslant j\leqslant 7n^3}\frac{(2(2j+a)+1)\pi^{2j+a}}{\Gamma(2j+a)} 
	\frac{\sinh \left((2j+a+1/4)\log n\right)}{(2j+a+1/4)}S_{n,2j+a},\\
	\endaligned
	\end{equation}
	
\begin{equation}\label{van0def}
\aligned
v_{a}(n,0):&=\sum_{0\leqslant j\leqslant m}c_{n,2j+a}\frac{\cosh \left(\varphi_{n,2j+a}\right)}{\varphi_{n,2j+a}^2},\\
&=\sum_{0\leqslant j\leqslant 7n^3}\frac{(2(2j+a)+1)\pi^{2j+a}}{\Gamma(2j+a)} \frac{\cosh \left((2j+a+1/4)\log n\right)}{(2j+a+1/4)^2\log n}S_{n,2j+a},\\
\endaligned
\end{equation}	
	Then 
	\begin{equation}
	\aligned
	u_{a}(n+1,0)>u_{a}(n,0),\\
	v_{a}(n+1,0)>v_{a}(n,0).\\
	\endaligned
	\end{equation}

\end{lemma}
\begin{proof}
	Let $j\in\N$ and
	\begin{equation}
	\aligned
	g(n,j):&=\frac{\cosh \left((j+1/4)\log n\right)}{\log n},\\
	\endaligned
	\end{equation}
We now inspect each term in the summation in ~\eqref{uan0def} and in ~\eqref{van0def}.
	Since for each $j\in\N$, $\sinh \left((j+1/4)\log n\right)$ and $S_{n,j}$ are positive and monotonically increasing function for $n\geqslant 3$, it is then suffice to prove that for each $j\in\N$, $g(n,j)$ is positive and monotonically increasing function for $n\geqslant 3$.\\
	
	The derivative $\partial_n g(n,j)$ is given by
	\begin{equation}
	\aligned
	&\frac{4n\log n}{\cosh((j+1/4)\log n)}\partial_n g(n,j)\\
	&=-4+(4j+1)\log n \tanh((j+1/4)\log n))\\
	&\geqslant-4+5\log n \tanh((5/4)\log n))\quad \because j\geqslant 1\\
	&\geqslant-4+\tanh(5/4)\quad \because n\geqslant 3\\
	&\approx 0.241418
	\endaligned
	\end{equation}
	
	Thus $g(n,j)$ is monotonically increasing function for $n\geqslant 3$. The Claim then follows.            
\end{proof}

\newpage
\section{\textbf{Hadamard's factorization for $U(n,z)$ and $V(n,z)$}}

\begin{theorem}[Paley-Wiener Theorem \cite{B1954}, p.103]  If $K(t)$ is continuous on the interval $[-r,r],0<r<\infty$,

\begin{equation}
\aligned
f(z)=\int_{-r}^{r}K(t)e^{izt}\mathrm{d}t,
\endaligned
\end{equation}

\noindent then $f(z)$ is entire, and of order one and type less than or equal to $r$.
\end{theorem}

\begin{theorem}[Hadamard's Factorization Theorem \cite{M1965}, p.289] If $f(z)$ is an entire function
of finite order $\rho$, then
\begin{equation}
\aligned
f(z)=\exp(g(z))z^{m_1} \prod_{k=1}^{\omega}\left(1-\frac{z}{z_k}\right)\exp\left(\sum_{n=1}^N\frac{z^n}{nz^n_k}\right),
\endaligned
\end{equation}
where $g(z)$ is a polynomial of degree less than or equal to $\rho$, $m_1 \in\N_0$, and $0\leqslant N_1\leqslant\lfloor\rho \rfloor$ (with the convention that the exponential factors in the product are not present when $N_1 = 0$).
The integer $q := \max\{\deg(g);N_1\}$ is called the genus of $f(z)$ and it
may be shown that either $q = \lfloor\rho \rfloor$ or $q = \lfloor\rho \rfloor-1$.

\end{theorem}

\begin{definition}. A real entire function $f(z)$
is in the Laguerre-P\'{o}lya class, written 
$f(z)\in \mathcal{LP}$, if

\begin{equation}
\aligned
f(z)=cz^{m_1}\exp(-az^2+bz)
\prod_{k=1}^{\omega}\left(1+\frac{z}{z_k}\right)\exp(-z/z_k)  
\endaligned
\end{equation}

where $b,c,z_k\in\R$, $m_1 \in \N_0, a\geqslant0, 0\leqslant \omega\leqslant\infty$, and $\sum_{k=1}^{\infty}z_k^{-2}<\infty$
\end{definition}
\begin{remark}
The significance of the Laguerre-P\'{o}lya class in the theory of entire function stems from the fact that functions in this class, and only these, are the uniform limits, on compact subsets of the complex plane, of polynomials with only real zeros (see \cite{CC2004} and ~\cite{L1980}, ch.8).
\end{remark}

\begin{definition}. A real entire function $f(z)$
is in the Laguerre-P\'{o}lya class $\mathcal{LP^+}$, written $f(z)\in \mathcal{LP^+}$, if

\begin{equation}
\aligned
f(z)=cz^{m_1}\exp(bz)
\prod_{k=1}^{\omega}\left(1+\frac{z}{z_k}\right)  
\endaligned
\end{equation}

where $c,z_k>0,k\in\N$, $m_1 \in \N_0, b\geqslant 0, 0\leqslant \omega\leqslant\infty$, and $\sum_{k=1}^{\infty}z_k^{-1}<\infty$
\end{definition}

\begin{definition}
A function $f(z) \in\mathcal{LP}$ is in $\mathcal{LP}(-\infty,0]$ if all of its zeros are
in the interval $(-\infty,0]$.
\end{definition}

The relations $\mathcal{LP}^+\subset\mathcal{LP}(-\infty,0]\subset\mathcal{LP}$ follow directly from the definitions above.


\begin{theorem}[Hermite-Kakeya Theorem \cite{RS2002}, p.197]\label{Hermite-Kakeya}
Given two real-valued polynomials, $f$ and $g$, then $f(x)+r g(x)$ has only real zeros for every $r\in\R$ $\mathbf{if}$ and $\mathbf{only}$ $\mathbf{if}$ $f$ and $g$ have real interlacing zeros.
\end{theorem}

\begin{theorem} 
Let $n\in 2\N+1; m=7n^3$, 

\begin{equation}\label{a_{n,j}}
\aligned
\varphi_{n,j}&=(j+1/4)\log n,\\
S_{n,j}&=\sum_{k=1}^{n}k^{j},\\
c_{n,j}&=  \log n\frac{(2j+1)\pi^j}{2\Gamma(j)} S_{n,2j}.\\
\endaligned
\end{equation}
let
\noindent 
\begin{equation}
\aligned
f_{n,a}(t)
&=\sum_{j=0}^{m} c_{n,2j+a}\cosh(\varphi_{n,2j}t),\quad a=1,2\\
\endaligned
\end{equation} 
and
\begin{equation}
\aligned
U(n,z)
&=\int_{-1}^{1}f_{n,2}(t)\exp(izt)\mathrm{d}t,\quad \\
V(n,z)
&=\int_{-1}^{1}f_{n,1}(t)\exp(izt)\mathrm{d}t,\quad \\
W(n,z)&=U(n,z)-V(n,z)
\endaligned
\end{equation}

Let  $\pm x_{k}(n),k \in\N$ be the zeros of $U(n,z)$ and $\pm y_{k}(n),k \in\N$ be the zeros of $V(n,z)$. Let 

\begin{equation}\label{f(z)}
\aligned
(k-1/2)\pi&<x_{k}(n)<k\pi,\quad k \in\N\\
(k-1/2)\pi&<y_{k}(n)<k\pi,\quad k \in\N\\
\endaligned
\end{equation}

Then

\noindent (1) $U(n,z),V(n,z),W(n,z)$ are entire functions of order 1.

\noindent (2)

\begin{equation}\label{UVdef}
\aligned
U(n,z)&=U(n,0)\prod_{k=1}^{\infty}\left(1-\frac{z^2}{x^2_{k}(n)}\right)\in\mathcal{LP}\\
V(n,z)&=V(n,0)\prod_{k=1}^{\infty}\left(1-\frac{z^2}{y^2_{k}(n)}\right)\in\mathcal{LP}\\
\endaligned
\end{equation}

\end{theorem}

\begin{proof}\noindent claim (1):
This is obvious because from Paley-Wiener Theorem we know that $U(n,z),V(n,z),W(n,z)$ are entire functions of order 1. 

claim (2):  Since
\begin{equation}
\aligned
\sum_{k=1}^{\infty}\left(\frac{1}{x^2_{k}(n)}+\frac{1}{(-x_{k}(n))^2}\right)
&<\frac{2}{\pi^2}\sum_{k=1}^{\infty}\frac{1}{(k-\frac{1}{2})^2}=\frac{1}{2}\\
\endaligned
\end{equation}

\noindent Applying Hadamard's factorization theorem (with order $\rho=1$ and $N_1=1$), we can express $U(n,z)$ as an infinite product

\begin{equation}\label{UdefH}
\aligned
U(n,z)&=c z^{m_1} e^{a_1z+b} \left(\prod_{k=1}^{\infty}\left(1-\frac{z}{x_{k}(n)}\right)\exp(z/x_{k}(n))\right)\\
&\times\left(\prod_{k=1}^{\infty}\left(1+\frac{z}{x_{k}(n)}\right)\exp(-z/x_{k}(n))\right)\\
&=c z^{m_1} e^{a_1z+b} \prod_{k=1}^{\infty}\left(1-\frac{z^2}{x^2_{k}(n)}\right)\\
\endaligned
\end{equation}
Since $U(n,z)$ is an even function, the constant $a_1$ above must be set to zero. Since  $ce^b=U(n,0)=2u_2(n,0)>0$ (cf. \eqref{w1w2w3def}) the nonnegative integer $m_1$ must be set to zero. Then ~\eqref{UdefH} becomes

\begin{equation}
\aligned
U(n,z)&=U(n,0)\prod_{k=1}^{\infty}\left(1-\frac{z^2}{x^2_{k}(n)}\right)\in\mathcal{LP},\\
\endaligned
\end{equation}

Similarly

\begin{equation}
\aligned
V(n,z)&=V(n,0)\prod_{k=1}^{\infty}\left(1-\frac{z^2}{y^2_{k}(n)}\right)\in\mathcal{LP}.\\
\endaligned
\end{equation}

\end{proof}

\begin{theorem}[~\cite{Z20171F1}]\label{gp2g}

Let $(k-1/2)\pi<x_{k}<k\pi,k\in\N$ and $p\in\N$ (in this paper $p$ does not specifically refer to prime )
\begin{equation}\label{f(z)}
\aligned
g(z)&=\prod_{k=1}^{\infty}\left(1-\frac{z^2}{x^2_{k}}\right)\\
g_p(z)&=\prod_{k=1}^{2p}\left(1-\frac{z^2}{x^2_{k}}\right),\quad (\text{a 4p-th order polynomial of } z)\\
\endaligned
\end{equation}
Then 

\noindent $g_{p}(z)$ converges to $g(z)$ uniformly in the compact disk $D(\pi r):=\{z=x+iy:|z|<\pi r\}$.
\end{theorem}

\begin{proof}

\begin{equation}\label{gpmg}
\aligned
\sup_{z\in D(\pi r)}\left|g_p(z)-g(z)\right|
&=\left(\prod_{k=1}^{2p}\left|1-\frac{z^2}{x^2_{k}}\right|\right)
\left|1-\prod_{k=2p+1}^{\infty}\left(1-\frac{z^2}{x^2_{k}}\right)\right|\\
&\leqslant\left(\prod_{k=1}^{2p}\left(1+\frac{\pi^2r^2}{x^2_{k}}\right)\right)
\left(1-\prod_{k=2p+1}^{\infty}\left(1-\frac{\pi^2r^2}{x^2_{k}}\right)\right)\\
&<\left(\prod_{k=1}^{2p}\left(1+\frac{r^2}{(k-1/2)^2}\right)\right)
\left(1-\prod_{k=2p+1}^{\infty}\left(1-\frac{r^2}{(k-1/2)^2}\right)\right)\\
&<\left(\prod_{k=1}^{2p}\left(1+\frac{r^2}{(k-1/2)^2}\right)\right)
\left(1-\prod_{k=1}^{\infty}\left(1-\frac{r^2}{(2p+k-1/2)^2}\right)\right)\\
&=:I_p\left(1-J_p\right)\\
\endaligned
\end{equation}

In the first part of \eqref{gpmg},
\begin{equation}\label{f(z)}
\aligned
I_p&=\prod_{k=1}^{2p}\left(1+\frac{r^2}{(k-1/2)^2}\right)\\
&<\prod_{k=1}^{\infty}\left(1+\frac{r^2}{(k-1/2)^2}\right)\\
&=\cosh (\pi r)\\
\endaligned
\end{equation}

In the second part of \eqref{gpmg},
\begin{equation}\label{f(z)}
\aligned
J_p&=\prod_{k=1}^{\infty}\left(1-\frac{r^2}{(2p+k-1/2)^2}\right)\\
&=\exp\left(\sum_{k=1}^{\infty}\log\left(1-\frac{r^2}{(2p+k-1/2)^2}\right)\right)\\
&\geqslant\exp\left(-r^2\sum_{k=1}^{\infty}\frac{1}{(2p+k-1/2)^2}\right)\\
&=\exp\left(-r^2\psi^{(1)}(2p+1/2)\right)\\
\endaligned
\end{equation}

where $\psi^{(1)}(x)=\frac{\mathrm{d}^2}{\mathrm{d}x^2}\left(\log \Gamma(x)\right)$ is a polygamma function. It  may be represented as
\begin{equation}
\aligned
\psi^{(1)}(x)=\int_0^1\frac{t^{x-1}}{1-t}(-\log t)\mathrm{d}t>0,\quad x>0\\
\endaligned
\end{equation}

Since
\begin{equation}
\aligned
0<\psi^{(1)}(2p+1/2)<\frac{1}{2p},\quad \text{ as }p\to\infty\\
\endaligned
\end{equation}
Therefore
\begin{equation}
\aligned
J_p&>\exp\left(-r^2/(2p)\right)>1-\frac{r^2}{2p},\quad \text{ as }p\to\infty\\
\endaligned
\end{equation}

Consequently
\begin{equation}
\aligned
\sup_{z\in D(\pi r)}|g_p(z)-g(z)|&<I_p(1-J_p)\\
&<\frac{r^2\cosh(\pi r)}{2p}
<\frac{r^2\cosh(\pi r)}{p}\quad \text{ as }p\to\infty\\
\endaligned
\end{equation}

Let $K=K(\epsilon, r)=r^2\cosh(\pi r)/\epsilon$.
Thus for any $0<\epsilon<1$,  there exist a natural number $\lceil K(\epsilon,r) \rceil$ such that for all $z\in D(\pi r)$ and for all $p\geqslant \lceil K(\epsilon,r) \rceil$ we have $|g_p(z)-g(z)|< \epsilon$.  i.e., $g_p(z)$ convergences to $g(z)$ uniformly in the compact disk $D(\pi r)$.

\end{proof}

\begin{theorem} \label{RnzWnz}

Let $n\in 2\N+1$ and let $N\geqslant 8$ be a sufficiently large and positive integer. Let

\begin{equation}
\aligned
p(n)&=n\lceil(U^2(n,0)+V^2(n,0))^{1/2}\rceil \\
\endaligned
\end{equation}

\begin{equation}
\aligned
U(n+1,0)&>U(n,0); \qquad n \in 2\N+1\\
V(n+1,0)&>V(n,0); \qquad n \in 2\N+1\\
\endaligned
\end{equation}

\begin{equation}
\aligned
U(n,0)&>V(n,0); \qquad n \geqslant N\\
\endaligned
\end{equation}

\begin{equation}
\aligned
(k-1/2)\pi&<x_{k}(n)<y_{k}(n)<k\pi,k\in\N,\quad n \geqslant N\\
\endaligned
\end{equation}
 
\begin{equation}
\aligned
U(n,z)&=U(n,0)\prod_{k=1}^{\infty}\left(1-\frac{z^2}{x^2_{k}(n)}\right)\\
V(n,z)&=V(n,0)\prod_{k=1}^{\infty}\left(1-\frac{z^2}{y^2_{k}(n)}\right)\\
W(n,z)&=U(n,z)-V(n,z)\\
P(n,z)&=U(n,0)\prod_{k=1}^{2p(n)}\left(1-\frac{z^2}{x^2_{k}(n)}\right),\quad &(\text{a 4p(n)-th order polynomial of } z)\\
Q(n,z)&=V(n,0)\prod_{k=1}^{2p(n)}\left(1-\frac{z^2}{y^2_{k}(n)}\right),\quad &(\text{a 4p(n)-th order polynomial of } z)\\
R(n,z)&=P(n,z)-Q(n,z),\quad &(\text{a 4p(n)-th order polynomial of } z)\\
\endaligned
\end{equation}
Then 

\noindent (1) $P(n,z)-U(n,z)$ converges to zero uniformly in the compact disk $D(\pi r):=\{|z|<\pi r\}$.

\noindent (2) $Q(n,z)-V(n,z)$ converges to zero uniformly in $D(\pi r)$.

\noindent (3) $R(n,z)-W(n,z)$ converges to zero uniformly in $D(\pi r)$;

\noindent (4) all the zeros of $\{R(n,z)\}_{n\geqslant N}$ in $\C$ are real.

\noindent (5) all the zeros of $\{W(n,z)\}_{n\geqslant N}$ in $\C$ are real.

\end{theorem}

\begin{proof}

\noindent Claims (1),(2), and (3):
Because $U(n,0),V(n,0)$ are monotonically increasing for $n\geqslant 3$, $p(n)$ is also monotonically increasing for $n\geqslant 3$.
Set $p$ to $p(n)$ in the proof of ~\cref {gp2g} above, we obtain
\begin{equation}
\aligned
|P(n,z)-U(n,z)|&<\frac{U(n,0)r^2\cosh(\pi r)}{2p(n)}\\
&\leqslant\frac{U(n,0)r^2\cosh(\pi r)}{2n U(n,0)}\\
&=\frac{r^2\cosh(\pi r)}{2n},\quad \text{ as } n\to\infty\\
\endaligned
\end{equation}

\begin{equation}
\aligned
|Q(n,z)-V(z)|&<\frac{V(n,0) r^2\cosh(\pi n)}{2p(n)}\\
&\leqslant\frac{V(n,0) r^2\cosh(\pi r)}{2n V(n,0)}\\
&=\frac{r^2\cosh(\pi r)}{2n},\quad \text{ as } n\to\infty\\
\endaligned
\end{equation}

\begin{equation}
\aligned
|R(n,z)-W(n,z)|&\leqslant|P(n,z)-U(n,z)|+|Q(n,z)-V(n,z)|\\
&<\frac{r^2\cosh(\pi r)}{n},\quad \text{ as } n\to\infty\\
\endaligned
\end{equation}

Let $K=K(\epsilon, r)=r^2\cosh(\pi r)/\epsilon$.
Thus for any $0<\epsilon<1$, there exist a natural number $\lceil K(\epsilon,r) \rceil$ such that for all $z\in D(\pi r)$ and for all $n\geqslant \lceil K(\epsilon,r) \rceil$ we have $|P(n,z)-U(n,z)|< \epsilon$, $|Q(n,z)-V(n,z)|< \epsilon$, and $|R(n,z)-W(n,z)|< \epsilon$.  i.e., $P(n,z)-U(n,z)$ convergence to $0$ uniformly in $D(\pi r)$; $Q(n,z)-V(n,z)$ convergence to $0$ uniformly in $D(\pi r)$; $R(n,z)-W(n,z)$ convergence to $0$ uniformly in $D(\pi r)$.

\noindent Claim (4):

\begin{remark}
Since the positive zeros of $P(n,z)$ are strictly left-interlacing with the positive zeros of $Q(n,z)$ and the negative zeros of $P(n,z)$ are strictly right-interlacing with the negative zeros of $Q(n,z)$, we can not directly use Hermite-Kakeya theorem ~\cref{Hermite-Kakeya} to prove claim (4) that all the zeros of $R(n,z)$ are real.
\end{remark}

\begin{remark}
	Since all the zeros of $P(n,z^{1/2})$ and $Q(n,z^{1/2})$ are positive and the zeros of $P(n,z^{1/2})$ are strictly left-interlacing with those of $Q(n,z)$, we could directly use Hermite-Kakeya theorem ~\cref{Hermite-Kakeya} to prove that all the zeros of  $R(n,z^{1/2})$ are real. But we still can not prove that they are positive. Thus we can not directly prove Claim (4) that all the zeros of $R(n,z)$ are real.
\end{remark}
So we will use a different method.  
\newpage
Typical profiles of polynomials $P(n,x)$ and $Q(n,x)$ for $0\leqslant x < y_1(n)$ is displayed in the Figure 4 below.
\begin{figure}[H]
	\centering
	\includegraphics[scale=0.45,keepaspectratio]{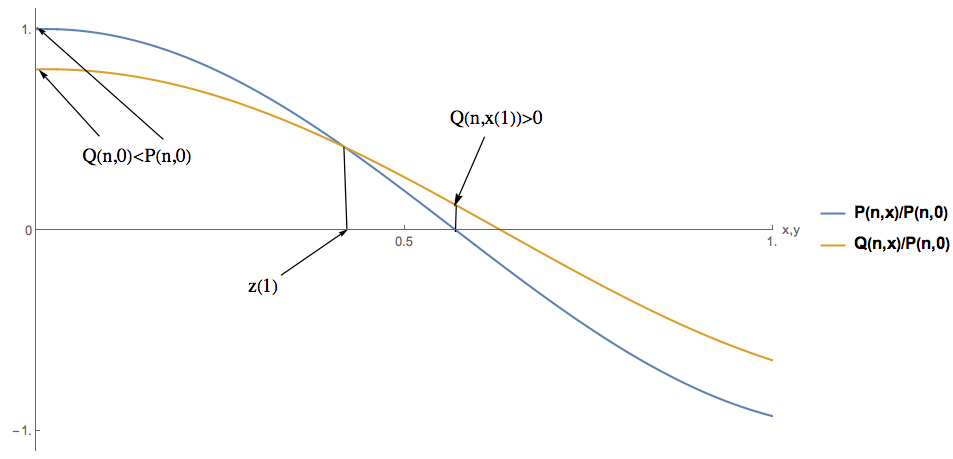}
	\caption{Plots of $\frac{P(n,x)}{P(n,0)}$ vs. $x$, $\frac{Q(n,x)}{P(n,0)}$ vs. $x$. In the figure we use the notation $x(1)=x_{1}(n), z(1)=z_{1}(n)$ etc.}\label{figure4}
\end{figure}

Because
\begin{equation}
\aligned
R(n,0)&=P(n,0)-Q(n,0)=U(n,0)-V(n,0)>0,&n\geqslant N\\
R(n,x_{1}(n))&=P(n,x_{1}(n))-Q(n,x_{1}(n))\\
&=-Q(n,x_{1}(n))<0,&n\geqslant N\\
\endaligned
\end{equation}

Thus there exist one real zero, $0<z_{1}(n)\in \R$, of $R(n,z)$ in this interval, i.e., 
\begin{equation}\label{rootz1}
0<z_{1}(n)<x_{1}(n).
\end{equation}

\newpage
Typical profiles of polynomials $P(n,x)$ and $Q(n,x)$ for $z_{2k-1}< x < x_{2k+1}(n)$ is displayed in the Figure 5 below.
\begin{figure}[H]
	\centering
	\includegraphics[scale=0.45,keepaspectratio]{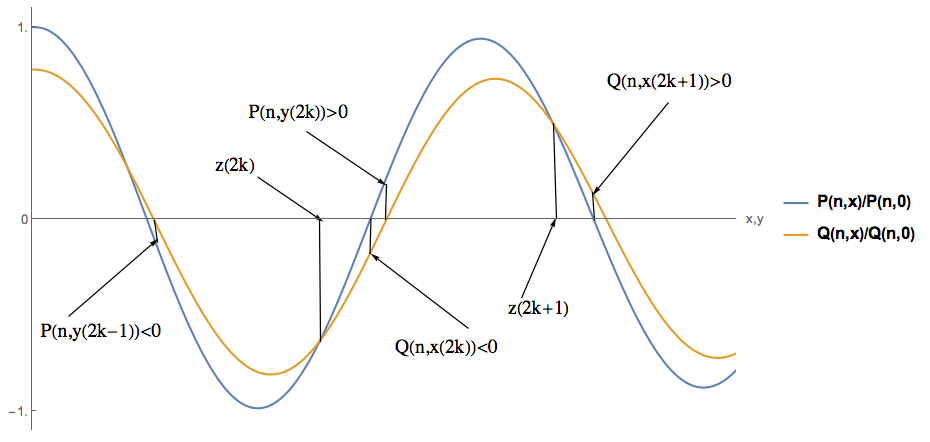}
	\caption{Plots of $\frac{P(n,x)}{P(n,0)}$ vs. $x$, $\frac{Q(n,x)}{P(n,0)}$ vs. $x$. In the figure we use the notation $x(2k+1)=x_{2k+1}(n), y(2k)=y_{2k}(n), z(2k)=z_{2k}(n)$ etc.}\label{figure5}
\end{figure}

We notice that
\begin{equation}
\aligned
R(n,y_{2k-1}(n))&=P(n,y_{2k-1}(n))-Q(n,y_{2k-1}(n))\\
&=P(n,y_{2k-1}(n))<0;&\quad k \in [1,p(n)],n\geqslant N\\
R(n,x_{2k}(n))&=P(n,x_{2k}(n))-Q(n,x_{2k}(n))\\
&=-Q(n,x_{2k}(n))>0;&\quad k \in [1,p(n)],n\geqslant N\\
\endaligned
\end{equation}

Thus there exist one real zero, $0<z_{2k}(n)\in \R$, of $R(n,z)$ in this interval, i.e., 
\begin{equation}\label{rootz2k}
y_{2k-1}(n)<z_{2k}(n)<x_{2k}(n),k\in [1,p]
\end{equation}

Similarly since

\begin{equation}
\aligned
R(n,y_{2k}(n))&=P(n,y_{2k}(n))-Q(n,y_{2k}(n))\\
&=P(n,y_{2k}(n))>0;&\quad k\in [1,p(n)],n\geqslant N\\
R(n,x_{2k+1}(n))&=P(n,x_{2k+1}(n))-Q(n,x_{2k+1}(n))\\
&=-Q(n,x_{2k+1}(n))<0;&\quad k \in [1,p(n)-1],n\geqslant N\\
\endaligned
\end{equation}

Thus there exist exactly one real zero, $0<z_{2k+1}(n)\in \R$, of $R(n,z)$ in this interval, i.e.,
\begin{equation}\label{rootz2kp1}
y_{2k}(n)<z_{2k+1}(n)<x_{2j+1}(n),k\in [1,p(n)-1]
\end{equation}

We notice that the intervals bounding the roots $z_1(n), z_{2k-1}(n), z_{2k}(n)$ in ~\eqref{rootz1}, ~\eqref{rootz2k}, and ~\eqref{rootz2kp1} are disjointed, thus we located $2p(n)$ simple and positive zeros for $R(n,z)$.

Since $R(n,z)$ is a $(4p(n))$-th order even polynomial of $z$, and we located all $4p(n)$ zeros, $\pm z_k(n),k\in[1,2p(n)]$, on the real axis. Thus all the zeros of $R(n,z)$ in $\C$ are real.

Claim (5): $\{R(n,z)\}_{n\geqslant N}$ and $W(n,z)$ are entire functions and  $R(n,z)-W(n,z)$ uniformly converges to zero in $D(\pi r)$. Since all the zeros of $\{R(n,z)\}_{n\geqslant N}$ are real, we deduce that $\{R(n,z)\}_{n\geqslant N}$ have no zeros in $\C\setminus \R$. Since $W(n,z)$ is not identically zero, we conclude, after applying \cref{corollary2Hurwitz} of Hurwitz's theorem, that $W(n,z)$ has no zeros in $\C\setminus \R$. Thus all the zeros of $W(n,z)$ are real.
\end{proof}

\begin{remark}\label{zeroCounting}
Now we can count how many zeros of $W(n,x)$ in the interval $0< x<T$.
From ~\cref{xknltykn} we know that the positive zeros of $U(n,x)$, $x_{k}(n)$, are left-interlacing with the positive zeros of $V(n,x)$, $y_{k}(n)$. i.e,
	\begin{equation}
	\aligned
	(k-1/2)\pi&<x_{k}(n)<y_{k}(n)< k\pi, &k \in\N,\quad & n\geqslant N,\\
	\endaligned
\end{equation}
From ~\cref{RnzWnz} we know that the positive zeros of $R(n,x)$ are in the range
\begin{equation}\label{zrootULbounds}
\aligned
0&<z_{1}(n)<y_{1}(n)<\pi,\\
y_{2j-1}(n)&<z_{2j}(n)<x_{2j}(n)<y_{2j}(n),\\
y_{2j}(n)&<z_{2j+1}(n)<x_{2j+1}(n)<y_{2j+}(n)<(2j+1)\pi,\qquad j\in\N,\\
\endaligned
\end{equation}
Denote $\tilde{z}_j(n),j\in\N$ the positive zeros of $W(n,x)$. 
Since $R(n,x)-W(n,z)$ uniformly converges to zero in the critical section $S_{1/2}$ and the regions in ~\eqref{zrootULbounds} that contain the positive zeros of $R(n,x)$ are disjoint.  Thus we deduce that the number of positive zeros of $W(n,x)$ in the interval $0< x<T$ is given by $T/\pi$. And the number of positive zeros of $H(14,n,x)=W(n,(x/2)\log n)$ in the interval $0< x<T$ is given by $\frac{T}{2\pi}\log n$. This number clearly goes to infinity as $n\to\infty$. But $H(14,n,z)$ converges to $\Xi(z)$ uniformly in $S_{1/2}$ and the numbers of zeros of $\Xi(z)$ in $\{|\Im(z)|<1/2,0<\Re(z)=x<T\}$ is $\frac{T}{2\pi}\log (T/e)(1+o(1))$. One possible way to resolve this dilemma of number of zeros discrepancy could be that there exists a large scale accumulation of these positive zeros of $H(14,n,x)$ as $n\to\infty$. Let $\mu(n,T)=(\frac{T}{2\pi}\log n)\left(\frac{T}{2\pi}\log (T/e)\right)^{-1}=\log n/\log (T/e)$. On average a single positive zero of $\Xi(x)$ is the accumulation of about $\lfloor\mu(n,T)\rfloor$ positive zeros of $H(14,n,x)$.  

Detailed analysis of the this possible accumulation process seems extremely difficult to us and is beyond our current capability.  This is the main reason why in ~\cref{GqnnzUniformConvergence} we prefer to use ~\cref{corollary2Hurwitz} of Hurwitz theorem in complex analysis instead of Hurwitz theorem itself (~\cref{Herwitz}). As shown in ~\cref{GqnnzUniformConvergence} that after applying ~\cref{corollary2Hurwitz} of Hurwitz theorem we prove that $\left(S_{1/2}\setminus\R\right)$ is a zero-free region for $H(14,n,z)$ and for $\Xi(z)$. This is because the zero-free region $\left(S_{1/2}\setminus\R\right)$ is unaffected by the possible accumulation of the real zeros (which remain real). Thus we are able to bypass analyzing the number of zeros discrepancy dilemma in this paper.
\end{remark}

\section{\textbf{Formulas for $\alpha^{(\pm)}(x,y)$ and $\beta^{(\pm)}(x,y)$}}

In this section we prove a theorem that will be used in the proof of \cref{w3gtw2gtw1}.

\begin{theorem} \label{thmAlphaBetaInAB}

Let $a=1,2; q=1/4; m\in\N;x\in\R
;y>0$. Let
\begin{equation}\label{alphapmdef}
\aligned
\alpha^{(\pm)}(x,y)
&:=(1/2)(\alpha_{2}(x,y)\pm \alpha_{1}(x,y)),\\
\alpha_{a}(x,y)
&:=\sum_{0\leqslant j \leqslant m}\frac{(2(2j+a)+1)y^{2j+a}}{\Gamma(2j+a)}
\frac{(2j+a+q)}{x^2+(2j+a+q)^2}\\
\endaligned
\end{equation}

\begin{equation}\label{betapmdef}
\aligned
\beta^{(\pm)}_{a}(x,y)
&:=(1/2)(\beta_{2}(x,y)\pm \beta_{1}(x,y)),\\
\beta_{a}(x,y)
&:=\sum_{0\leqslant j \leqslant m}\frac{(2(2j+a)+1)y^{2j+a}}{\Gamma(2j+a)}
\frac{1}{x^2+(2j+a+q)^2}\\
\endaligned
\end{equation}
Then
\begin{equation}\label{Apdef1}
\aligned
\alpha^{(+)}(x,y)&=\Re A^{(+)}_1(x,y)+\Re A^{(+)}_2(x,y)\\
A^{(+)}_1(x,y)
&=y e^{y}
+y(q-ix)\frac{\gamma(1+q+ix,-y)}{(-y)^{1+q+ix}}\\
A^{(+)}_2(x,y)
&=- \frac{y^{m+2}}{\Gamma(m+2)}{_1}F_1(1;m+2;y)\\
&-\frac{y^{m+2}(q-ix)}{\Gamma(m+2)(m+2+q+ix)}\\
&\times
{_2}F_2\begin{pmatrix}{\begin{matrix} 1, & m+2+q+ix \\  m+2, & m+3+q+ix  \\  \end{matrix}} & ;y \end{pmatrix}\\
\endaligned
\end{equation}

\begin{equation}\label{Bpdef1}
\aligned
\beta^{(+)}(x,y)&=\Re B^{(+)}_1(x,y)+\Re B^{(+)}_2(x,y)\\
B^{(+)}_1(x,y)
&=-\frac{y(q-ix)}{(+ix)}\frac{\gamma(1+q+ix,-y)}{(-y)^{1+q+ix}}\\
B^{(+)}_2(x,y)
&=\frac{y^{m+2}(q-ix)}{(+ix)\Gamma(m+2)(m+2+q+ix)}\\
&\times
{_2}F_2\begin{pmatrix}{\begin{matrix} 1, & m+2+q+ix \\  m+2, & m+3+q+ix  \\  \end{matrix}} & ;y \end{pmatrix}\\
\endaligned
\end{equation}

\begin{equation}\label{Amdef1}
\aligned
\alpha^{(-)}(x,y)&=\Re A^{(-)}_1(x,y)+\Re A^{(-)}_2(x,y)\\
A^{(-)}_1(x,y)
&=-y e^{-y}
-y(q-ix)\frac{\gamma(1+q+ix,y)}{y^{1+q+ix}}\\
A^{(-)}_2(x,y)
&= \frac{(-1)^{m}y^{m+2}}{\Gamma(m+2)}{_1}F_1(1;m+2;-y)\\
&-\frac{(-1)^{m}y^{m+2}(q-ix)}{\Gamma(m+2)(m+2+q+ix)}\\
&\times
{_2}F_2\begin{pmatrix}{\begin{matrix} 1, & m+2+q+ix \\  m+2, & m+3+q+ix  \\  \end{matrix}} & ;-y \end{pmatrix}\\
\endaligned
\end{equation}

\begin{equation}\label{Bmdef1}
\aligned
\beta^{(-)}(x,y)&=\Re B^{(-)}_1(x,y)+\Re B^{(-)}_2(x,y)\\
B^{(-)}_1(x,y)
&=\frac{y(q-ix)}{(+ix)}\frac{\gamma(1+q+ix,y)}{y^{1+q+ix}}\\
B^{(-)}_2(x,y)
&=\frac{(-1)^{m}y^{m+2}(q-ix)}{(+ix)\Gamma(m+2)(m+2+q+ix)}\\
&\times
{_2}F_2\begin{pmatrix}{\begin{matrix} 1, & m+2+q+ix \\  m+2, & m+3+q+ix  \\  \end{matrix}} & ;-y \end{pmatrix}\\
\endaligned
\end{equation}

\end{theorem}

\begin{proof}

The functions $\alpha_a(x,y),\beta_a(x,y)$ ($a=1,2$) in \eqref{alphapmdef}, \eqref{betapmdef} can be rewritten as
\begin{equation}
\aligned
\alpha_{a}(x,y)
&:=\sum_{0\leqslant j \leqslant m}\left(2+\frac{2c_1}{2j+a+q+ix}+\frac{2c_2}{2j+a+q-ix}\right)\frac{y^{2j+a}}{\Gamma(2j+a)}\\
\beta_{a}(x,y)
&:=\sum_{0\leqslant j \leqslant m}\left(\frac{2d_1}{2j+a+q+ix}+\frac{2d_2}{2j+a+q-ix}\right)\frac{y^{2j+a}}{\Gamma(2j+a)}\\
\endaligned
\end{equation}
\noindent and where

\begin{equation}\label{c1d1def}
\aligned
c_1&=\frac{1}{4}(1-2q-2ix)=\frac{1}{2}(q-ix)\qquad \because q=1/4\\
     &=-\frac{1}{4}(2(1+q+ix)-3)\\
c_2&=\frac{1}{4}(1-2q+2ix)=\frac{1}{2}(q+ix)\\
     &=-\frac{1}{4}(2(1+q-ix)-3)\\
d_1&=-\frac{1}{4ix}(1-2q-2ix)=\frac{c_1}{(+ix)}\\
d_2&=+\frac{1}{4ix}(1-2q+2ix)=\frac{c_2}{(-ix)}\\
\endaligned
\end{equation}

Because ~\cite{LeV2016} $y^{2j}=\frac{1}{2}(y^{2j}+(-y)^{2j})$,$y^{2j+1}=\frac{1}{2}(y^{2j+1}-(-y)^{2j+1})$, we obtain

\begin{equation}
\aligned
\alpha_{1}(x,y)
&=y\sum_{0\leqslant j \leqslant m}\frac{y^{2j}+(-y)^{2j}}{(2j)!}\\
&+c_1 y\sum_{0\leqslant j \leqslant m}\frac{1}{2j+1+q+ix}\frac{y^{2j}+(-y)^{2j}}{(2j)!}\\
&+c_2 y\sum_{0\leqslant j \leqslant m}\frac{1}{2j+1+q-ix}\frac{y^{2j}+(-y)^{2j}}{(2j)!}\\
\alpha_{2}(x,y)
&=y\sum_{0\leqslant j \leqslant m}\frac{y^{2j+1}-(-y)^{2j+1}}{(2j+1)!}\\
&+c_1 y\sum_{0\leqslant j \leqslant m}\frac{1}{2j+2+q+ix}\frac{y^{2j+1}-(-y)^{2j+1}}{(2j+1)!}\\
&+c_2 y\sum_{0\leqslant j \leqslant m}\frac{1}{2j+2+q-ix}\frac{y^{2j+1}-(-y)^{2j+1}}{(2j+1)!}\\
\endaligned
\end{equation}

\begin{equation}
\aligned
\beta_{1}(x,y)
&=d_1 y\sum_{0\leqslant j \leqslant m}\frac{1}{2j+1+q+ix}\frac{y^{2j}+(-y)^{2j}}{(2j)!}\\
&+d_2 y\sum_{0\leqslant j \leqslant m}\frac{1}{2j+1+q-ix}\frac{y^{2j}+(-y)^{2j}}{(2j)!}\\
\beta_{2}(x,y)
&=d_1 y\sum_{0\leqslant j \leqslant m}\frac{1}{2j+2+q+ix}\frac{y^{2j+1}-(-y)^{2j+1}}{(2j+1)!}\\
&+d_2 y\sum_{0\leqslant j \leqslant m}\frac{1}{2j+2+q-ix}\frac{y^{2j+1}-(-y)^{2j+1}}{(2j+1)!}\\
\endaligned
\end{equation}

\noindent Since $y^{2j+1}+(-y)^{2j+1}=0$,$y^{2j}-(-y)^{2j}=0$, we can add odd power terms in the summation of $\alpha_{1}(x,y),\beta_{1}(x,y)$ and add even power terms in the summation of $\alpha_{2}(x,y),\beta_{2}(x,y)$, obtaining

\begin{equation}
\aligned
\alpha_{1}(x,y)
&=y\sum_{0\leqslant j \leqslant m}\frac{y^{j}+(-y)^{j}}{j!}\\
&+c_1 y\sum_{0\leqslant j \leqslant m}\frac{1}{j+1+q+ix}\frac{y^{j}+(-y)^{j}}{j!}\\
&+c_2 y\sum_{0\leqslant j \leqslant m}\frac{1}{j+1+q-ix}\frac{y^{j}+(-y)^{j}}{j!}\\
\alpha_{2}(x,y)
&=y\sum_{0\leqslant j \leqslant m}\frac{y^{j}-(-y)^{j}}{j!}\\
&+c_1 y\sum_{0\leqslant j \leqslant m}\frac{1}{j+1+q+ix}\frac{y^{j}-(-y)^{j}}{j!}\\
&+c_2 y\sum_{0\leqslant j \leqslant m}\frac{1}{j+1+q-ix}\frac{y^{j}-(-y)^{j}}{j!}\\
\endaligned
\end{equation}

\begin{equation}
\aligned
\beta_{1}(x,y)
&=d_1 y\sum_{0\leqslant j \leqslant m}\frac{1}{j+1+q+ix}\frac{y^{j}+(-y)^{j}}{j!}\\
&+d_2 y\sum_{0\leqslant j \leqslant m}\frac{1}{j+1+q-ix}\frac{y^{j}+(-y)^{j}}{j!}\\
\beta_{2}(x,y)
&=d_1 y\sum_{0\leqslant j \leqslant m}\frac{1}{j+1+q+ix}\frac{y^{j}-(-y)^{j}}{j!}\\
&+d_2 y\sum_{0\leqslant j \leqslant m}\frac{1}{j+1+q-ix}\frac{y^{j}-(-y)^{j}}{j!}\\
\endaligned
\end{equation}

\begin{equation}
\aligned
\alpha^{(+)}(x,y)
:&=\frac{1}{2}\left(\alpha_{2}(x,y)+\alpha_{1}(x,y)\right)\\
&=y\sum_{0\leqslant j \leqslant m}\frac{y^{j}}{j!}
+c_1 y\sum_{0\leqslant j \leqslant m}\frac{1}{j+1+q+ix}\frac{y^{j}}{j!}\\
&+c_2 y\sum_{0\leqslant j \leqslant m}\frac{1}{j+1+q-ix}\frac{y^{j}}{j!}\\
\alpha^{(-)}(x,y)
:=&\frac{1}{2}\left(\alpha_{2}(x,y)-\alpha_{1}(x,y)\right)\\
=&-y\sum_{0\leqslant j \leqslant m}\frac{(-y)^{j}}{j!}
-c_1 y\sum_{0\leqslant j \leqslant m}\frac{1}{j+1+q+ix}\frac{(-y)^{j}}{j!}\\
&-c_2 y\sum_{0\leqslant j \leqslant m}\frac{1}{j+1+q-ix}\frac{(-y)^{j}}{j!}\\
\endaligned
\end{equation}

\begin{equation}
\aligned
\beta^{(+)}(x,y)
:&=\frac{1}{2}\left(\beta_{2}(x,y)+\beta_{1}(x,y)\right)\\
&=d_1 y\sum_{0\leqslant j \leqslant m}\frac{1}{j+1+q+ix}\frac{y^{j}}{j!}\\
&+d_2 y\sum_{0\leqslant j \leqslant m}\frac{1}{j+1+q-ix}\frac{y^{j}}{j!}\\
\beta^{(-)}(x,y)
:=&\frac{1}{2}\left(\beta_{2}(x,y)-\beta_{1}(x,y)\right)\\
=&-d_1 y\sum_{0\leqslant j \leqslant m}\frac{1}{j+1+q+ix}\frac{(-y)^{j}}{j!}\\
&-d_2 y\sum_{0\leqslant j \leqslant m}\frac{1}{j+1+q-ix}\frac{(-y)^{j}}{j!}\\
\endaligned
\end{equation}

We notice that 

\begin{equation}
\aligned
\alpha^{(+)}(x,-y)=\alpha^{(-)}(x,y),\\
\beta^{(+)}(x,-y)=\beta^{(-)}(x,y).\\
\endaligned
\end{equation}

After the completion of summation for $j=0,\cdots,m$, we obtain

\begin{equation}
\aligned
\alpha^{(+)}(x,y)
&=y e^{y}\left(1-\frac{\gamma(m+1,y)}{\Gamma(m+1)}\right)\\
&-\frac{y^{m+2}(q-ix)}{2\Gamma(m+2)(m+2+q+ix)}
{_2}F_2\begin{pmatrix}{\begin{matrix} 1, & m+2+q+ix \\  m+2, & m+3+q+ix  \\  \end{matrix}} & ;y \end{pmatrix}\\
&-\frac{y^{m+2}(q+ix)}{2\Gamma(m+2)(m+2+q-ix)}
{_2}F_2\begin{pmatrix}{\begin{matrix} 1, & m+2+q-ix \\  m+2, & m+3+q-ix  \\  \end{matrix}} & ;y \end{pmatrix}\\
&+\frac{y(q-ix)}{2}\frac{\gamma(1+q+ix,-y)}{(-y)^{1+q+ix}}\\
&+\frac{y(q+ix)}{2}\frac{\gamma(1+q-ix,-y)}{(-y)^{1+q-ix}}\\
\endaligned
\end{equation}

\begin{equation}\label{betapdefN}
\aligned
\beta^{(+)}(x,y)
&=\frac{y^{m+2}(q-ix)}{2(+ix)\Gamma(m+2)(m+2+q+ix)}
{_2}F_2\begin{pmatrix}{\begin{matrix} 1, & m+2+q+ix \\  m+2, & m+3+q+ix  \\  \end{matrix}} & ;y \end{pmatrix}\\
&+\frac{y^{m+2}(q+ix)}{2(-ix)\Gamma(m+2)(m+2+q-ix)}
{_2}F_2\begin{pmatrix}{\begin{matrix} 1, & m+2+q-ix \\  m+2, & m+3+q-ix  \\  \end{matrix}} & ;y \end{pmatrix}\\
&-\frac{y(q-ix)}{2(+ix)}\frac{\gamma(1+q+ix,-y)}{(-y)^{1+q+ix}}\\
&-\frac{y(q+ix)}{2(-ix)}\frac{\gamma(1+q-ix,-y)}{(-y)^{1+q-ix}}\\
\endaligned
\end{equation}

\begin{equation}
\aligned
\alpha^{(-)}(x,y)
&=-y e^{-y}
\left(1-\frac{\gamma(m+1,-y)}{\Gamma(m+1)}\right)\\
&-\frac{(-1)^{m}y^{m+2}(q-ix)}{2\Gamma(m+2)(m+2+q+ix)}
{_2}F_2\begin{pmatrix}{\begin{matrix} 1, & m+2+q+ix \\  m+2, & m+3+q+ix  \\  \end{matrix}} & ;-y \end{pmatrix}\\
&-\frac{(-1)^{m} y^{m+2}(q+ix)}{2\Gamma(m+2)(m+2+q-ix)}
{_2}F_2\begin{pmatrix}{\begin{matrix} 1, & m+2+q-ix \\  m+2, & m+3+q-ix  \\  \end{matrix}} & ;-y \end{pmatrix}\\
&-\frac{y(q-ix)}{2}\frac{\gamma(1+q+ix,y)}{y^{1+q+ix}}\\
&-\frac{y(q+ix)}{2}\frac{\gamma(1+q-ix,y)}{y^{1+q-ix}}\\
\endaligned
\end{equation}

\begin{equation}\label{betamdefN}
\aligned
\beta^{(-)}(x,y)
&=\frac{(-1)^{m}y^{m+2}(q-ix)}{2(+ix)\Gamma(m+2)(m+2+q+ix)}
{_2}F_2\begin{pmatrix}{\begin{matrix} 1, & m+2+q+ix \\  m+2, & m+3+q+ix  \\  \end{matrix}} & ;-y \end{pmatrix}\\
&+\frac{(-1)^{m}y^{m+2}(q+ix)}{2(-ix)\Gamma(m+2)(m+2+q-ix)}
{_2}F_2\begin{pmatrix}{\begin{matrix} 1, & m+2+q-ix \\  m+2, & m+3+q-ix  \\  \end{matrix}} & ;-y \end{pmatrix}\\
&+\frac{y(q-ix)}{2(+ix)}\frac{\gamma(1+q+ix,y)}{y^{1+q+ix}}\\
&+\frac{y(q+ix)}{2(-ix)}\frac{\gamma(1+q-ix,y)}{y^{1+q-ix}}\\
\endaligned
\end{equation}

Using 
\begin{equation}
\gamma(s,x)=\frac{x^{s}e^{-x}}{s}{_1}F_1(1;1+s;x)
\end{equation}
with $s=m+1,x=\pm y$, we obtain
\begin{equation}\label{alphapdefN}
\aligned
\alpha^{(+)}(x,y)
&=y e^{y}\\
&+\frac{y(q-ix)}{2}\frac{\gamma(1+q+ix,-y)}{(-y)^{1+q+ix}}\\
&+\frac{y(q+ix)}{2}\frac{\gamma(1+q-ix,-y)}{(-y)^{1+q-ix}}\\
&- \frac{y^{m+2}}{\Gamma(m+2)}{_1}F_1(1;m+2;y)\\
&-\frac{y^{m+2}(q-ix)}{2\Gamma(m+2)(m+2+q+ix)}
{_2}F_2\begin{pmatrix}{\begin{matrix} 1, & m+2+q+ix \\  m+2, & m+3+q+ix  \\  \end{matrix}} & ;y \end{pmatrix}\\
&-\frac{y^{m+2}(q+ix)}{2\Gamma(m+2)(m+2+q-ix)}
{_2}F_2\begin{pmatrix}{\begin{matrix} 1, & m+2+q-ix \\  m+2, & m+3+q-ix  \\  \end{matrix}} & ;y \end{pmatrix}\\
\endaligned
\end{equation}

\begin{equation}\label{alphamdefN}
\aligned
\alpha^{(-)}(x,y)
&=-y e^{-y}
+ \frac{(-1)^{m}y^{m+2}}{\Gamma(m+2)}{_1}F_1(1;m+2;-y)\\
&-\frac{(-1)^{m}y^{m+2}(q-ix)}{2\Gamma(m+2)(m+2+q+ix)}
{_2}F_2\begin{pmatrix}{\begin{matrix} 1, & m+2+q+ix \\  m+2, & m+3+q+ix  \\  \end{matrix}} & ;-y \end{pmatrix}\\
&-\frac{(-1)^{m} y^{m+2}(q+ix)}{2\Gamma(m+2)(m+2+q-ix)}
{_2}F_2\begin{pmatrix}{\begin{matrix} 1, & m+2+q-ix \\  m+2, & m+3+q-ix  \\  \end{matrix}} & ;-y \end{pmatrix}\\
&-\frac{y(q-ix)}{2}\frac{\gamma(1+q+ix,y)}{y^{1+q+ix}}\\
&-\frac{y(q+ix)}{2}\frac{\gamma(1+q-ix,y)}{y^{1+q-ix}}\\
\endaligned
\end{equation}


\noindent For convenience we rewrite \eqref{alphapdefN}, \eqref{betapdefN}, \eqref{alphamdefN},  and \eqref{betamdefN} as summation of two components that are suitable for further treatment.

\begin{equation}\label{Apdef1}
\aligned
\alpha^{(+)}(x,y)&=\Re A^{(+)}_1(x,y)+\Re A^{(+)}_2(x,y)\\
A^{(+)}_1(x,y)
&=y e^{y}
+y(q-ix)\frac{\gamma(1+q+ix,-y)}{(-y)^{1+q+ix}}\\
A^{(+)}_2(x,y)
&=- \frac{y^{m+2}}{\Gamma(m+2)}{_1}F_1(1;m+2;y)\\
&-\frac{y^{m+2}(q-ix)}{\Gamma(m+2)(m+2+q+ix)}\\
&\times
{_2}F_2\begin{pmatrix}{\begin{matrix} 1, & m+2+q+ix \\  m+2, & m+3+q+ix  \\  \end{matrix}} & ;y \end{pmatrix}\\
\endaligned
\end{equation}

\begin{equation}\label{Bpdef1}
\aligned
\beta^{(+)}(x,y)&=\Re B^{(+)}_1(x,y)+\Re B^{(+)}_2(x,y)\\
B^{(+)}_1(x,y)
&=-\frac{y(q-ix)}{(+ix)}\frac{\gamma(1+q+ix,-y)}{(-y)^{1+q+ix}}\\
B^{(+)}_2(x,y)
&=\frac{y^{m+2}(q-ix)}{(+ix)\Gamma(m+2)(m+2+q+ix)}\\
&\times
{_2}F_2\begin{pmatrix}{\begin{matrix} 1, & m+2+q+ix \\  m+2, & m+3+q+ix  \\  \end{matrix}} & ;y \end{pmatrix}\\
\endaligned
\end{equation}

\begin{equation}\label{Amdef1}
\aligned
\alpha^{(-)}(x,y)&=\Re A^{(-)}_1(x,y)+\Re A^{(-)}_2(x,y)\\
A^{(-)}_1(x,y)
&=-y e^{-y}
-y(q-ix)\frac{\gamma(1+q+ix,y)}{y^{1+q+ix}}\\
A^{(-)}_2(x,y)
&= \frac{(-1)^{m}y^{m+2}}{\Gamma(m+2)}{_1}F_1(1;m+2;-y)\\
&-\frac{(-1)^{m}y^{m+2}(q-ix)}{\Gamma(m+2)(m+2+q+ix)}\\
&\times
{_2}F_2\begin{pmatrix}{\begin{matrix} 1, & m+2+q+ix \\  m+2, & m+3+q+ix  \\  \end{matrix}} & ;-y \end{pmatrix}\\
\endaligned
\end{equation}

\begin{equation}\label{Bmdef1}
\aligned
\beta^{(-)}(x,y)&=\Re B^{(-)}_1(x,y)+\Re B^{(-)}_2(x,y)\\
B^{(-)}_1(x,y)
&=\frac{y(q-ix)}{(+ix)}\frac{\gamma(1+q+ix,y)}{y^{1+q+ix}}\\
B^{(-)}_2(x,y)
&=\frac{(-1)^{m}y^{m+2}(q-ix)}{(+ix)\Gamma(m+2)(m+2+q+ix)}\\
&\times
{_2}F_2\begin{pmatrix}{\begin{matrix} 1, & m+2+q+ix \\  m+2, & m+3+q+ix  \\  \end{matrix}} & ;-y \end{pmatrix}\\
\endaligned
\end{equation}

We also notice that 

\begin{equation}
\aligned
A^{(+)}_l(x,-y)&=A^{(-)}_l(x,y),\quad l=1,2\\
B^{(+)}_l(x,-y)&=B^{(-)}_l(x,y),\quad l=1,2.\\
\endaligned
\end{equation}
\end{proof}

\section{\textbf{Bounds on $\Re A^{(\pm)}(x,y),\Re B^{(\pm)}(x,y)$}}
In this section we prove several lemmas and theorems that will be used in the eventual proof of \cref{w3gtw2gtw1}.
\begin{definition}
The generalized hyper geometric function $_1F_1(a;b;z)$ is defined 
\begin{equation}
\aligned
{_1F_1}(a;b;z)
:&=\sum_{k=0}^{\infty}\frac{(a)_k}{(b)_k}\frac{z^k}{k!}\\
 &=\sum_{k=0}^{\infty}\frac{\Gamma(a+k)}{\Gamma(a)}
\frac{\Gamma(b)}{\Gamma(b+k)}\frac{z^k}{k!}\\
\endaligned
\end{equation}

The generalized hyper geometric function $_2F_2(a_1,a_2;b_1,b_2;z)$ is defined 
\begin{equation}
\aligned
{_2F_2}(a_1,a_2;b_1,b_2;z)
:&=\sum_{k=0}^{\infty}\frac{(a_1)_k(a_2)_k}{(b_1)_k(b_2)_k}\frac{z^k}{k!}\\
 &=\sum_{k=0}^{\infty}\frac{\Gamma(a_1+k)\Gamma(a_2+k)}{\Gamma(a_1)\Gamma(a_2)}
\frac{\Gamma(b_1)\Gamma(b_2)}{\Gamma(b_1+k)\Gamma(b_2+k)}\frac{z^k}{k!}\\
\endaligned
\end{equation}
\end{definition}

\begin{lemma}\label{GammaUpperLowerBound2}
 Let $n\in\N$, then
	\begin{equation}
	\aligned
	\frac{(\pi n^3)^{7n^3+2}}{\Gamma(7n^3+2)}
	&<\frac{\pi^{3/2} e}{7\sqrt{14}} \left(n^{3/2}\exp\left(7n^3\log(\pi e/7)\right)\right)\\
	\frac{(\pi n^3)^{7n^3+2}}{\Gamma(7n^3+2)}
	&>\left(\frac{65}{84e}\right)\frac{\pi^{3/2} e}{7\sqrt{14}} \left(n^{3/2}\exp\left(7n^3\log(\pi e/7)\right)\right)\\
	\endaligned
	\end{equation}
	where $\frac{\pi^{3/2} e}{7\sqrt{14}} \approx 0.577906$, $\left(\frac{65}{84e}\right)\frac{\pi^{3/2} e}{7\sqrt{14}}\approx 0.164512$, and $7\log(\pi e/7)\approx 1.39174<\pi$.

\end{lemma}
\begin{proof}
	Using the upper bound for $\Gamma(m+2)$ in \cref{GammaUpperLowerBound}, we have
	\begin{equation}
	\aligned
	\frac{(\pi n^3)^{7n^3+2}}{\Gamma(7n^3+2)}
	&<\frac{(\pi n^3)^{7n^3+2}}{\sqrt{2\pi(7n^3+1)}}
	\left(\frac{e}{7n^3+1}\right)^{7n^3+1}\\
	&<\frac{(\pi n^3)}{\sqrt{2\pi(7n^3+1)}}
	\left(\frac{\pi e n^3}{7n^3+1}\right)^{7n^3+1}\\
	&<\frac{(\pi n^3)}{\sqrt{2\pi(7n^3)}}
	\left(\frac{\pi e n^3}{7n^3}\right)^{7n^3+1}\\
	&=\frac{(\pi n^3)}{\sqrt{2\pi(7n^3)}}
	\left(\frac{\pi e}{7}\right)\left(\frac{\pi e}{7}\right)^{7n^3}\\
	&<\frac{\pi^{3/2} e}{7\sqrt{14}}
	\left(n^{3/2}\exp\left(7n^3\log(\pi e/7)\right)\right)\\
	&\approx 0.577906 \left(n^{3/2}\exp\left(1.39174n^3)\right)\right)\\
	&=O\left(n^{3/2}\exp\left(7n^3\log(\pi e/7)\right)\right)\\
	\endaligned
	\end{equation}
	
	where $\frac{\pi^{3/2} e}{7\sqrt{14}} \approx 0.577906$ and $0<7\log(\pi e/7)\approx 1.39174<\pi$
	
	Using the lower bound for $\Gamma(7n^3+2)$ in \cref{GammaUpperLowerBound}, we have
	
	\begin{equation}
	\aligned
	\frac{(\pi n^3)^{7n^3+2}}{\Gamma(7n^3+2)}
	&>\frac{(\pi n^3)^{7n^3+2}}{\sqrt{2\pi(7n^3+1)}}
	\left(\frac{e}{7n^3+1}\right)^{7n^3+1}\left(1-\frac{1}{12 (7n^3+1)}\right)\\
	&=\frac{(\pi n^3)}{\sqrt{2\pi(7n^3+1)}}\left(\frac{\pi e n^3}{7n^3+1}\right)^{7n^3+1}\left(1-\frac{1}{12 (7n^3+1)}\right)\\
	&>\frac{(\pi n^3)}{\sqrt{2\pi\cdot 7n^3}}\left(1+\frac{1}{7n^3}\right)^{-1/2}\left(\frac{\pi e n^3}{7n^3}\right)^{7n^3+1}\left(1+\frac{1}{7n^3}\right)^{-7n^3-1}\\
	&\times\left(1-\frac{1}{12\cdot 7n^3}\right)\\
	&>\left(\frac{(\pi n^3)}{\sqrt{2\pi\cdot 7n^3}}\left(\frac{\pi e }{7}\right)\left(\frac{\pi e }{7}\right)^{7n^3}\right)\left(1+\frac{1}{7n^3}\right)^{-7n^3}\left(1-\frac{1}{7n^3}\right)\\
	&\times \left(1-\frac{1}{2\cdot 7n^3}\right)\left(1-\frac{1}{12\cdot 7n^3}\right)\\
	&>\frac{\pi^{3/2} e}{7\sqrt{14}}
	\left(n^{3/2}\exp\left(7n^3\log(\pi e/7)\right)\right)\\
	&\times \frac{1}{e}
	\left(1-\frac{3}{2\cdot 7n^3}\right)\left(1-\frac{1}{12\cdot 7n^3}\right)\\
	&>\frac{\pi^{3/2} e}{7\sqrt{14}}
	\left(n^{3/2}\exp\left(7n^3\log(\pi e/7)\right)\right)\\
	&\times \frac{1}{e} \left(1-\frac{19}{12\cdot 7n^3}\right)\\
	&>\left(\frac{65}{84e}\right)\frac{\pi^{3/2} e}{7\sqrt{14}}
	\left(n^{3/2}\exp\left(7n^3\log(\pi e/7)\right)\right) \quad \because n\geqslant 1\\
	&\approx 0.284669\times 0.577906 \left(n^{3/2}\exp\left(1.39174n^3)\right)\right)\\
	&=O\left(n^{3/2}\exp\left(7n^3\log(\pi e/7)\right)\right)\\
	\endaligned
	\end{equation}
	Where $\frac{65}{84e} \approx 0.284669$.

\end{proof}

\begin{lemma}\label{sec6F11g}
Let $b>0;x\in\R ;y>0;\epsilon_1(x,y)\in (-1,1),\epsilon_2(x,y)\in (-1,1)$
\begin{equation}
\aligned
f(x,y):
&=\frac{e^{-y}}{b+ix}{_1}F_1(1;1+b+ix; y)\\
&=\frac{1}{b+ix}{_1}F_1(b+ix;1+b+ix; -y)\\
\endaligned
\end{equation}
then

\begin{equation}
\aligned
f(x,y)&=\frac{\epsilon_1(x,y)}{b}+\frac{x\epsilon_2(x,y)}{b^2}
\endaligned
\end{equation}

\end{lemma}

\begin{proof}
It is suffice to prove that

\begin{equation}
\aligned
|\Re f(x,y)|&< \frac{1}{b}\\
|\Im f(x,y)|&\leqslant \frac{|x|}{b^2}\\
\endaligned
\end{equation}
 
Using Euler's integral transform formula and ${_0}F_0(;; -y)=e^{-y}$, we have
\begin{equation}
\aligned
f(x,y)&=\frac{e^{-y}}{b+ix}{_1}F_1(1;1+b+ix; y)\\
&=\frac{1}{b+ix}{_1}F_1(b+ix;1+b+ix; -y)\\
&=\left(\frac{1}{b+ix}\right)\frac{\Gamma(1+b+ix)}{\Gamma(b+ix)\Gamma(1+b+ix-(b+ix))}\\
&\times \int_{0}^1 t^{b+ix-1}(1-t)^{1+b+ix-(b+ix)-1}{_0}F_0(;; -t y)\mathrm{d}t\\
&=\int_{0}^1 t^{b-1+ix}e^{-t y}\mathrm{d}t\\
\endaligned
\end{equation}

Therefore

\begin{equation}
\aligned
|\Re f(x,y)|
&\leqslant\int_{0}^1 |\cos(x\log t)|t^{b-1}e^{-t y}\mathrm{d}t\\
&< \int_{0}^1 t^{b-1}e^{-t y}\mathrm{d}t<\int_{0}^1 t^{b-1}\mathrm{d}t=\frac{1}{b}
\endaligned
\end{equation}

\begin{equation}
\aligned
|\Im f(x,y)|
&\leqslant\int_{0}^1 |\sin(x\log t)|t^{b-1}e^{-t y}\mathrm{d}t\\
&=\int_{0}^1 |x\log t|\left|\frac{\sin(x\log t)}{x\log t}\right|t^{b-1}e^{-t y}\mathrm{d}t\\
&< \int_{0}^1 |x\log t|t^{b-1}e^{-t y}{d}t\\
&\leqslant |x|\int_{0}^1 (-\log t)t^{b-1}\mathrm{d}t\\
&=\frac{|x|}{b^2}
\endaligned
\end{equation}

Finally we have
\begin{equation}
\aligned
f(x,y)=\frac{\epsilon_1(x)}{b}+\frac{ix\epsilon_2(x)}{b^2}
\endaligned
\end{equation}
where $\epsilon_1(x,y)\in(-1,1)$ and $\epsilon_2(x,y)\in(-1,1)$.

\end{proof}

Now we introduce a theorem (important to us) by Prof. Temme and Prof. Zhou.
\begin{proposition}[~\cite{TZ2017}]\label{TemmeZhou}
	
	Let $x \in \mathbb{R} $; $b>0$; and  $y>0$. Let the function $g(x,y)$ be defined in terms  of the incomplete gamma function or Kummer function as
	\begin{equation}
	\aligned
	g(x,y):&=\frac{ye^{-y}\gamma(a,-y)}{(-y)^a}\\
	&=\frac{y}{a}{_1}F_1(1;1+a;-y)\\
	&=y \int_0^1 e^{-yt}(1-t)^{a-1}\,dt, \quad a=1+b+ix,
	\endaligned
	\end{equation}
	Then there exists a positive constant $M$ and  complex functions $\epsilon(x,y)\in(-1,1)$ such that
	
\begin{equation}\label{gdefsec7A}
g(x,y)=\frac{y}{y+b+ix}+\epsilon(x,y) M\frac{1}{y}\\
\end{equation}
	holds for all $y>0$ and all $x\in\R$.\\
\end{proposition}
\begin{proof}
	See \cref{TemmeZhouapp}.
\end{proof}

\begin{remark}
	In an earlier draft of this paper, we presented our original proof ~\cref{TemmeZhou}. This ~\cref{TemmeZhou} is then heavily used in the subsequent analysis.  An expert pointed out a critical error, which was also spotted by Prof. Zhou, in our original proof of ~\cref{TemmeZhou}. We then asked help from Prof. Temme, an expert in, among others, the asymptotic expansion theory of the incomplete gamma function $\gamma(a,y)$.  Prof. Temme and Prof. Zhou eventually worked out a proof, which is presented in \cref{TemmeZhouapp} of Appendix A. So the full credit of proof for \cref{TemmeZhou} goes to them. If there is any error in the presentation of this proof, it is due to us.\\
	
	In later version draft of this paper, we realized that the error bound for $g(x,y)$ in ~\eqref{gdefsec7A}, $O\left(\frac{1}{y}\right)$ , is not sharp enough for us.   So ~\cref{TemmeZhou} is eventually not used in the current version of this paper.  In ~\cref{TemmeZhou2} and ~\cref{TemmeZhou3} below we, follow the main idea in the proof of ~\cref{TemmeZhou} by Prof. Temme and Prof. Zhou,  carried out one more integration by parts, completed tedious analysis and finally obtained sharp error bound for  $g(x,y)$ in ~\eqref{gdefsec7C}, $O\left(\frac{ 1}{(y+b)^2+x^2}\right)$ and sharp error bound for $\Im g(x,y)$ in ~\eqref{gdefsec7D}, $O\left(\frac{x}{(y+b)^2+x^2}\right)$. These bounds just fit to our needs. So the major credit for the proof of ~\cref{TemmeZhou2} (presented in ~\cref{TemmeZhouapp2}) and ~\cref{TemmeZhou3} (presented in ~\cref{TemmeZhouapp3}) still goes to Prof. Temme and Prof. Zhou. Of course, if there is any error in current presentation, it is due to us.

\end{remark}

\begin{proposition} \label{TemmeZhou2}

Let $x \in \mathbb{R} $;  $1<b<2$; $y>0$. Let the function $g(x,y)$ be defined in terms  of the incomplete gamma function or Kummer function as
\begin{equation}
\aligned
g(x,y):&=\frac{ye^{-y}\gamma(a,-y)}{(-y)^a}\\
&=\frac{y}{a}{_1}F_1(1;1+a;-y)\\
&=y \int_0^1 e^{-yt}(1-t)^{a-1}\,\mathrm{d}t, \quad a=1+b+ix,
\endaligned
\end{equation}

\begin{equation}
\aligned
g_0(x,y)&=\frac{y}{y+b+ix}-\frac{ixy}{(y+b+i x)^3}
\endaligned
\end{equation}

Then there exists a positive constant $M$ and a complex function $\epsilon(x,y),|\epsilon(x,y)|<1$ such that

\begin{equation}\label{gdefsec7C}
\aligned
g(x,y)&=g_0(x,y)+\frac{\epsilon(x,y) M}{(y+b)^2+x^2}\\
\endaligned
\end{equation}
holds for all $x\in\R$ and all $y>0$.
\end{proposition}

\begin{proof}
	See \cref{TemmeZhouapp2}.
\end{proof}

\begin{proposition} \label{TemmeZhou3}
	
	Let $x \in \mathbb{R} $;  $b=5/4$; $y>0$. Let the function $g(x,y)$ be defined in terms  of the incomplete gamma function or Kummer function as
	\begin{equation}
	\aligned
	g(x,y):&=\frac{ye^{-y}\gamma(a,-y)}{(-y)^a}\\
	&=\frac{y}{a}{_1}F_1(1;1+a;-y)\\
	&=y \int_0^1 e^{-yt}(1-t)^{a-1}\,dt, \quad a=1+b+ix,
	\endaligned
	\end{equation}
	
	\begin{equation}
	\aligned
	g_0(x,y)&=\frac{y}{y+b+ix}-\frac{ixy}{(y+b+i x)^3}
	\endaligned
	\end{equation}

	Then there exists a positive constant $M$ and a real function $\epsilon(x,y)\in(-1,1)$ such that
	
	\begin{equation}\label{gdefsec7D}
	\aligned
	\Im g(x,y)&=\Im g_0(x,y)+\epsilon(x,y)\frac{ xM}{(y+b)^2+x^2}\\
	\endaligned
	\end{equation}
	holds for all $x\in\R$ and all $y>0$.
\end{proposition}
\begin{proof}
	See \cref{TemmeZhouapp3}.
\end{proof}

We now present another theorem (important to us) by Prof. Zhou.
\begin{proposition}[~\cite{Z20171F1}]\label{Zhou1F1B}
	
	Let $0<2y< m$. Then there exists a positive constant $M$ and a real function $\epsilon(m,\mu,y) \in (-1,1)$ such that
	
	\begin{equation}
	{_1}F_1(1;m+2;\mu y)=\left(1-\frac{\mu y}{m}\right)^{-1}+\frac{\epsilon(m,\mu,y) M}{m}
	\end{equation}
	holds for all $m>2y>0$.
	
\end{proposition}
\begin{proof}
	
See \cref{Zhou1F1BApp}.
	
\end{proof}

\begin{proposition} \label{Zhou2F2C}
	
	Let $0<2y<m;0<q<1;\mu=\pm;x\in\R$. Then there exists positive number $M=5$ and real functions $\epsilon_1(m,\mu,x,y),\epsilon_2(m,\mu,x,y) \in (-1,1)$ such that
	
\begin{equation}
\aligned
&	{_2}F_2\begin{pmatrix}{\begin{matrix} 1, & m+2+q+ix \\  m+2, & m+3+q+ix  \\  \end{matrix}} & ;\mu y \end{pmatrix}\\
&={_1}F_1(1,m+2;\mu y)+\frac{(y\epsilon_1+ix\epsilon_2) M}{(m+2+q)^2+x^2}
\endaligned
\end{equation}
	holds for all $m>2y>0$ and all $x\in\R$.
	
\end{proposition}

\begin{proof}
	
See \cref{Zhou2F2CApp}.
\end{proof}
\begin{remark}
	This is just a slight extension of ~\cref{Zhou1F1B}.  So the major credit for the proof of ~\cref{Zhou2F2C} (presented in ~\cref{Zhou2F2CApp})  goes to Prof. Zhou. If there is any error in current presentation, it is due to us.	
\end{remark}

\subsection[A(+1),B(+1)]{Bounds of $A^{(+)}_1(x,y)$ and $B^{(+)}_1(x,y)$}
\begin{theorem} \label{thmAp1Bound}
	Let $q=1/4; n\in2\N_0+9$. Let
	\begin{equation}\label{Ap1def}
	\aligned
	A^{(+)}_1(x,y)=y e^{y}
	+y(q-ix)\frac{\gamma(1+q+ix,-y)}{(-y)^{1+q+ix}}\\
	\endaligned
	\end{equation}

\noindent Then there exists a constant $M_6$ and a real function $\epsilon_6(n,x)\in (-1,1)$ such that 
\begin{equation}
\aligned
&\sum_{1\leqslant k\leqslant n}^{n}\left(\Re A^{(+)}_1(x,\pi n k^2)-n^{-1/2}\Re A^{(+)}_1(x,\pi k^2/n)\right)\\
&=\exp(\pi n^3)\left(\frac{(\pi n^3+2+q)^3}{(\pi n^3+2+q)^2+x^2}\right)\left(1+\frac{\epsilon_6(n,x)M_6}{\pi n^3}\right)\\
\endaligned
\end{equation}
holds for $n\in2\N_0+9$ and all $x\geqslant 0$.
	
\end{theorem}

\begin{proof}
The incomplete gamma function $\gamma(s,x)$ is related to Kummer function ${_1}F_1(1;1+s;x)$ via
\begin{equation}
\gamma(s,x)=\frac{x^{s}e^{-x}}{s}{_1}F_1(1;1+s;x)
\end{equation}

Setting $x=-y,s=1+q+ix$ and substituting the result into ~\eqref{Ap1def} leads to

\begin{equation}\label{Ap1Bdef}
\aligned
A^{(+)}_{1}(x,y)
:&=ye^y-e^{y}(q-ix)\frac{y}{(1+q+ix)}{_1}F_1(1 ;2+q+ix;-y)\\
\endaligned
\end{equation}

Set $b=1+q=5/4$ in \cref{TemmeZhou2}, we know that there exists a positive constant $M_1$ and real functions $\epsilon_{1A}=\epsilon_{1A}(x,y)\in(-1,1)$ and  $\epsilon_{1B}=\epsilon_{1B}(x,y)\in (-1,1)$ such that for all $x\geqslant 0$ and all $y>0$ we have
\begin{equation}
\aligned
&\frac{y}{(1+q+ix)}{_1}F_1(1 ;2+q+ ix;-y)\\
&=g(x,y)\\
&=\frac{y}{y+1+q+ix}\left(1-\frac{ix}{(y+1+q+ix)^2}\right)+\frac{(\epsilon_{1A}+i \epsilon_{1B})M_1}{(y+1+q)^2+x^2}.\\
\endaligned
\end{equation}

Substituting this into \eqref{Ap1Bdef} leads to
	
\begin{equation}\label{Ap1Bdef2}
	\aligned
	&\frac{1}{e^y}A^{(+)}_{1}(x,y)\\
	&=y-\frac{y(q-ix)}{(y+1+q+ix)}\left(1-\frac{ix}{(y+1+q+ix)^2}\right)\\
	&-(q-ix)\frac{(\epsilon_{1A}+i\epsilon_{1B})M_1}{(y+1+q)^2+x^2}.\\
	\endaligned
\end{equation}

To simplify the notation we set $\overline{y}=y+1+q$. Since $ y>0$, we have $\overline{y}>1+q$. Then the real part of ~\eqref{Ap1Bdef2} becomes

\begin{equation}
	\aligned
	&\frac{1}{e^y}\Re A^{(+)}_{1}(x,y)\\
	&=
	\frac{\overline{y}^3}{\overline{y}^2+x^2}+\frac{2\overline{y}^2}{\overline{y}^2+x^2}-\left[\frac{\overline{y}(3+3q+q^2)}{\overline{y}^2+x^2}\right]-\left[\frac{q(1+q)}{\overline{y}^2+x^2}\right]\\
	&-\frac{7\overline{y}^4}{(\overline{y}^2+x^2)^2}+\left[\frac{(7+2q)\overline{y}^3}{(\overline{y}^2+x^2)^2}\right]+\left[\frac{5q(1+q)\overline{y}^2}{(\overline{y}^2+x^2)^2}\right]\\
	&+\frac{4\overline{y}^6}{(\overline{y}^2+x^2)^3}-\left[\frac{4\overline{y}^5}{(\overline{y}^2+x^2)^3}\right]-\left[\frac{4q(1+q)\overline{y}^4}{(\overline{y}^2+x^2)^3}\right]\\
	&+\left[\frac{(q\epsilon_{1A}+\epsilon_{1B})M_1}{\overline{y}^2+x^2}\right]\\
	\endaligned
\end{equation}

Since all he terms associated with square brackets $[]$ are of oder $O\left(\frac{\overline{y}}{\overline{y}^2+x^2}\right)$, we can combine them and define

\begin{equation}
\aligned
\frac{R_1(x,y)}{y^2+x^2}
:=&-\left[\frac{\overline{y}(3+3q+q^2)}{\overline{y}^2+x^2}\right]-\left[\frac{q(1+q)}{\overline{y}^2+x^2}\right]\\
&+\left[\frac{(7+2q)\overline{y}^3}{(\overline{y}^2+x^2)^2}\right]+\left[\frac{5q(1+q)\overline{y}^2}{(\overline{y}^2+x^2)^2}\right]\\
&-\left[\frac{4\overline{y}^5}{(\overline{y}^2+x^2)^3}\right]-\left[\frac{4q(1+q)\overline{y}^4}{(\overline{y}^2+x^2)^3}\right]
+\left[\frac{(q\epsilon_{1A}+\epsilon_{1B})M_1}{\overline{y}^2+x^2}\right]\\
\endaligned
\end{equation}

Because $\overline{y}>1+q$,
\begin{equation}
\aligned
|R_1(x,y)|
&\leqslant
(3+3q+q^2)+\frac{q(1+q)}{\overline{y}(\overline{y}^2+x^2)}\\
&+\frac{(7+2q)\overline{y}^2}{(\overline{y}^2+x^2)}+\frac{5q(1+q)\overline{y}}{(\overline{y}^2+x^2)}\\
&+\frac{4\overline{y}^4}{(\overline{y}^2+x^2)^2}+\frac{4q(1+q)\overline{y}^3}{(\overline{y}^2+x^2)^2}
+\frac{(1+q)M_1}{\overline{y}}\\
&\leqslant(3+3q+q^2)+1\\
&+(7+2q)+5q\\
&+4+4q+M_1\\
&=15+14q+q^2+M_1=:M_2
\endaligned
\end{equation}

Thus with $\epsilon_{2}=\epsilon_{2}(x,y)\in(-1,1)$ we can write

\begin{equation}
\aligned
\frac{1}{e^y}\Re A^{(+)}_{1}(x,y)
&=
\frac{\overline{y}^3}{\overline{y}^2+x^2}+\left[\frac{2\overline{y}^2}{\overline{y}^2+x^2}\right]\\
&-\left[\frac{7\overline{y}^4}{(\overline{y}^2+x^2)^2}\right]
+\left[\frac{4\overline{y}^6}{(\overline{y}^2+x^2)^3}\right]
+\left[\frac{\overline{y}\epsilon_{2}M_2}{\overline{y}^2+x^2}\right]\\
\endaligned
\end{equation}

Since all he terms associated with square brackets $[]$ are of oder $O\left(\frac{\overline{y}^2}{\overline{y}^2+x^2}\right)$, we can combine them and define

\begin{equation}
\aligned
\frac{\overline{y}^2R_2(x,\overline{y})}{\overline{y}^2+x^2}\
&=\left[\frac{2\overline{y}^2}{\overline{y}^2+x^2}\right]
-\left[\frac{7\overline{y}^4}{(\overline{y}^2+x^2)^2}\right]\\
&+\left[\frac{4\overline{y}^6}{(\overline{y}^2+x^2)^3}\right]
+\left[\frac{\overline{y}\epsilon_{2}M_2}{\overline{y}^2+x^2}\right]\\
\endaligned
\end{equation}

\begin{equation}
\aligned
R_2(x,\overline{y})
&\leqslant 2
+\frac{7\overline{y}^2}{(\overline{y}^2+x^2)}
+\frac{4\overline{y}^4}{(\overline{y}^2+x^2)^2}
+\frac{M_2}{\overline{y}}\\
&\leqslant 2+7+4+M_2\qquad \because \overline{y}>1+q\\
&=13+M_2=:M_3.
\endaligned
\end{equation}

Let $\epsilon_3(x,y)\in(-1,1)$, we finally obtain that
\begin{equation}\label{Ap1Finaldef}
\aligned
\Re A^{(+)}_1(x,y)
&=e^{y}\left(\frac{\overline{y}^3}{\overline{y}^2+x^2}\right)
\left(1+\frac{\epsilon_4M_4}{\overline{y}}\right)\\
&=e^{y}\left(\frac{(y+2+q)^3}{(y+2+q)^2+x^2}\right)
\left(1+\frac{\epsilon_3(x,y)M_3}{(y+2+q)}\right)\\
\endaligned
\end{equation}
Holds for all $x\geqslant 0$ and all $y>0$.

The most significant term in $\Re A^{(+)}_1(x,\pi n k^2), k\in[1,n]$ and $\Re A^{(+)}_1(x,\pi k^2/n), k\in[1,n]$ is $\Re A^{(+)}_1(x,\pi n^3)$.

\begin{equation}
\aligned
\Re A^{(+)}_1(x,\pi n^3)&=h^{(+)}_1(n,x)\left(1+\frac{\epsilon_4M_4}{(\pi n^3+2+q)}\right)\\
h^{(+)}_1(n,x):&=\exp(\pi n^3)\frac{(\pi n^3+2+q)^3}{(\pi n^3+2+q)^2+x^2}
\endaligned
\end{equation}
We can use $h^{(+)}_1(n,x)$ to bound all the other $2n+1$ terms.
	
For all $k\in [1,n-1]$ we have
	\begin{equation}
	\aligned
&\frac{\left|\Re A^{(+)}_1(x,\pi n k^2)\right|}{h^{(+)}_1(n,x)}\\
	&<\frac{\exp(\pi n k^2)}{\exp(\pi n^3)}\left(1+\frac{M_4}{\pi n k^2+2+q}\right)\\
	&\times \left(\frac{(\pi n^3+2+q)^2+x^2}{(\pi n k^2+2+q)^2+x^2}\right)\left(\frac{\pi n k^2+2+q}{\pi n^3+2+q}\right)^3\\
	&<\frac{\exp(\pi n (n-1)^2)}{\exp(\pi n^3)}\left(1+\frac{M_4}{\pi n+2+q}\right)\\
	&\times \left(\frac{(\pi n^3+2+q)^2+x^2}{(\pi n+2+q)^2+x^2}\right)\\
	&<\left(1+M_4\right)n^5\exp(-\pi n(2n-1))\\
	&<\left(1+M_4\right)n^5\exp(-\pi n^2)\quad \because n\geqslant 9\\
	&=M_5n^5\exp(-\pi n^2)\qquad M_5:=1+ M_4\\
	\endaligned
	\end{equation}
	And

	\begin{equation}
\aligned
&\frac{\sum_{k=1}^{n-1}\left|\Re A^{(+)}_1(x,\pi n k^2)\right|}{h^{(+)}_1(n,x)}
<M_5n^6\exp(-\pi n^2)\\
\endaligned
\end{equation}

For all $k\in [1,n-1]$ we have
\begin{equation}
\aligned
&\frac{\left|\Re A^{(+)}_1(x,\pi k^2/n)\right|}{h^{(+)}_1(n,x)}\\
&<\frac{\exp(\pi k^2/n)}{\exp(\pi n^3)}\left(1+\frac{M_4}{\pi k^2/n+2+q}\right)\\
&\times \left(\frac{(\pi n^3+2+q)^2+x^2}{(\pi  k^2/n+2+q)^2+x^2}\right)\left(\frac{\pi k^2/n+2+q}{\pi n^3+2+q}\right)^3\\
&<\frac{\exp(\pi n)}{\exp(\pi n^3)}\left(1+\frac{nM_4}{\pi +n(2+q)}\right)\\
&\times \left(\frac{(\pi n^4+n(2+q))^2+n^2x^2}{(\pi+n(2+q))^2+n^2x^2}\right)\\
&<\left(1+M_4\right)n^5\exp(-\pi n (n^2-1))\\
&<M_5n^5\exp(-\pi n^2)\\
\endaligned
\end{equation}
And

\begin{equation}
\aligned
&\frac{\sum_{k=1}^{n}\left|\Re A^{(+)}_1(x,\pi  k^2/n)\right|}{h^{(+)}_1(n,x)}
&=M_5n^6\exp(-\pi n^2)\\
\endaligned
\end{equation}

Let $\epsilon_5:=\epsilon_5(n,x)\in(-1,1)$.Thus
\begin{equation}
\aligned
&\sum_{k=1}^{n}\left(\Re A^{(+)}_1(x,\pi n k^2)-n^{-1/2}\Re A^{(+)}_1(x,\pi k^2/n)\right)\\
&=h^{(+)}_1(n,x)\left(1+\frac{\epsilon_4M_4}{\pi n^3+2+q}+2\epsilon_5 M_5n^6\exp(-\pi n^2)\right)\\
&=h^{(+)}_1(n,x)\left(1+\frac{R_5(n,x)}{\pi n^3}\right)\\
\endaligned
\end{equation}

where
\begin{equation}
\aligned
\frac{R_5(n,x)}{\pi n^3}&:=\frac{\epsilon_4M_4}{\pi n^3+2+q}+2\epsilon_5M_5 n^6\exp(-\pi n^2)
\endaligned
\end{equation}

Since 
\begin{equation}
\aligned
R_5(n,x)&\leqslant\frac{M_4\pi n^3}{\pi n^3+2+q}+2\pi n^3 n^6M_5\exp(-\pi n^2)\\
&<M_4+2\pi M_5 n^{10}\exp(-\pi n^2)\\
&\leqslant M_4+2\pi M_5\left(\frac{5}{\pi e}\right)^5:=M_6\\
&\approx M_4+0.432323 M_5\\
\endaligned
\end{equation}

Let $\epsilon_6:=\epsilon_6(n,x)\in(-1,1)$.Then we have
\begin{equation}
\aligned
&\sum_{k=1}^{n}\left(\Re A^{(+)}_1(x,\pi n k^2)-n^{-1/2}\Re A^{(+)}_1(x,\pi k^2/n)\right)\\
&=h^{(+)}_1(n,x)\left(1+\frac{\epsilon_6M_6}{\pi n^3}\right)\\
\endaligned
\end{equation}

\end{proof}

\begin{theorem} \label{thmBp1Bound}
	 Let $q=1/4; n\in2\N_0+9$.
	Let
\begin{equation}\label{Bp1def}
\aligned
B^{(+)}_1(x,y)
=&-\frac{y(q-ix)}{ix}\frac{\gamma(1+q+ix,-y)}{(-y)^{1+q+ix}}\\
\endaligned
\end{equation}

\noindent Then there exists a constant $M_6$ and a real function $\epsilon_6(n,x)\in (-1,1)$ such that 
\begin{equation}
\aligned
&\sum_{k=1}^{n}\left(\Re B^{(+)}_1(x,\pi n k^2)+n^{-1/2}\Re B^{(+)}_1(x,\pi k^2/n)\right)\\
&=\exp(\pi n^3)\left(\frac{(\pi n^3+2+q)^2}{(\pi n^3+2+q)^2+x^2}\right)\left(1+\frac{\epsilon_6(n,x)M_6}{\pi n^3}\right)\\
\endaligned
\end{equation}
holds for $n\in2\N_0+9$ and all $x\geqslant 0$.
	
\end{theorem}

\begin{proof}
The incomplete $\gamma(s,x)$ function is related to the Kummer via
\begin{equation}
\gamma(s,x)=\frac{x^{s}e^{-x}}{s}{_1}F_1(1;1+s;x)
\end{equation}

Setting $x=-y,s=1+q+ix$ and substituting the result into ~\eqref{Bp1def} leads to

\begin{equation}\label{Bp1Bdef}
\aligned
B^{(+)}_{1}(x,y)
&=-\frac{e^{y}(q-ix)}{ix}\left(\frac{y}{(1+q+ix)}{_1}F_1(1 ;2+q+ix;-y)\right)\\
\endaligned
\end{equation}

	Set $b=1+q=5/4$ in \cref{TemmeZhou3}, we know that there exists a positive constant $M_1=522$ and real functions $\epsilon_1=\epsilon_1
	(x,y)$ and $\epsilon_2=\epsilon_2(x,y),|\epsilon_1|<1,|\epsilon_2|<1$ such that for all $x\in\R$ and all $y>0$ we have
	\begin{equation}
	\aligned
	&\frac{y}{(1+q+ix)}{_1}F_1(1 ;2+q+ ix;-y)\\
	&=g(x,y)\\
	&=\frac{y}{y+1+q+ix}\left(1-\frac{ix}{(y+1+q+ix)^2}\right)\\
	&+\frac{M_1}{(y+1+q)^2+x^2}(\epsilon_1(x,y)+i{\color{red}{x}}\epsilon_2(x,y))\\
	\endaligned
	\end{equation}

	Substituting this into \eqref{Bp1Bdef} leads to
	
	\begin{equation}\label{Bp1Bdef2}
	\aligned
	&\frac{1}{e^y}B^{(+)}_{1}(x,y)\\
	=&-\frac{(q-ix)}{i{\color{red}{x}}}\frac{y}{(y+1+q+ix)}\left(1-\frac{ix}{(y+1+q+ix)^2}\right)\\
	&-\frac{(q-ix)}{i{\color{red}{x}}}\frac{M_1}{(y+1+q)^2+x^2}(\epsilon_{1A}+i{\color{red}{x}}\epsilon_{1B}).\\
	\endaligned
	\end{equation}

To simplify the notation we set $\overline{y}=y+1+q$. Since $ y>0$, we have $\overline{y}>1+q$. Then the real part of ~\eqref{Bp1Bdef2} becomes

\begin{equation}
\aligned
&\frac{1}{e^y}\Re B^{(+)}_{1}(x,y)\\
&=\frac{\overline{y}^2}{\overline{y}^2+x^2}+\left[\frac{(1+q)^2}{\overline{y}^2+x^2}	\right]\\
&-\frac{5\overline{y}^3}{(\overline{y}^2+x^2)^4}
+\left[\frac{(5+2q)\overline{y}^2+3q(1+q)\overline{y}}{(\overline{y}^2+x^2)^4}\right]\\
&+\frac{4\overline{y}^5}{(\overline{y}^2+x^2)^6}
-\left[\frac{4\overline{y}^3(\overline{y}+q(1+q))}{(\overline{y}^2+x^2)^6}\right]\\
&+\left[\frac{(\epsilon_{1B}-q\epsilon_{2B})M_1}{\overline{y}^2+x^2}\right]\\
\endaligned
\end{equation}	
	Since the terms associated with square brackets $[]$ are of oder $O\left((\overline{y}^2+x^2)^{-1}\right)$, we can combine them and define
	
	\begin{equation}
	\aligned
	\frac{R_2(x,\overline{y})}{\overline{y}^2+x^2}	
:&=\left[\frac{(1+q)^2}{\overline{y}^2+x^2}	\right]
+\left[\frac{(5+2q)\overline{y}^2+3q(1+q)\overline{y}}{(\overline{y}^2+x^2)^2}\right]\\
&-\left[\frac{4\overline{y}^3(\overline{y}+q(1+q))}{(\overline{y}^2+x^2)^3}\right]
+\left[\frac{(\epsilon_{1B}-q\epsilon_{2B})M_1}{\overline{y}^2+x^2}\right]\\
	\endaligned
	\end{equation}

Since
\begin{equation}
\aligned
R_2(x,\overline{y})
&\leqslant \frac{(1+q)^2}{\overline{y}^2+x^2}	
+\frac{(5+2q)\overline{y}^2+3q(1+q)\overline{y}}{(\overline{y}^2+x^2)}\\
&+\frac{4\overline{y}^3(\overline{y}+q(1+q))}{(\overline{y}^2+x^2)^2}
+\frac{(1+q)M_1}{\overline{y}^2+x^2}\\
&\leqslant 1+((5+2q)+3q)+(4+4q)+M_1\quad \because \overline{y}>1+q\\
&=10+9q+M_1=:M_2
\endaligned
\end{equation}

Thus we conclude that there exists a real function
$\epsilon_2=\epsilon_2(x,y)\in(-1,1)$ such that
\begin{equation}
\aligned
&\frac{1}{e^y}\Re B^{(+)}_{1}(x,y)\\
&=\frac{\overline{y}^2}{\overline{y}^2+x^2}
-\left[\frac{5\overline{y}^3}{(\overline{y}^2+x^2)^4}\right]
+\left[\frac{4\overline{y}^5}{(\overline{y}^2+x^2)^6}\right]
+\left[\frac{\epsilon_2M_2}{\overline{y}^2+x^2}\right]\\
\endaligned
\end{equation}	
	
	Since the terms associated with square brackets $[]$ are of oder $O\left(\frac{\overline{y}}{\overline{y}^2+x^2}\right)$, we can combine them and define
\begin{equation}
\aligned
\frac{\overline{y}R_3(x,\overline{y})}{\overline{y}^2+x^2}
:=&
-\left[\frac{5\overline{y}^3}{(\overline{y}^2+x^2)^4}\right]
+\left[\frac{4\overline{y}^5}{(\overline{y}^2+x^2)^6}\right]
+\left[\frac{\epsilon_2M_2}{\overline{y}^2+x^2}\right]\\
\endaligned
\end{equation}	

\begin{equation}
\aligned
R_3(x,\overline{y})
&\leqslant
\frac{5\overline{y}^2}{(\overline{y}^2+x^2)^2}
+\frac{4\overline{y}^4}{(\overline{y}^2+x^2)^4}
+\frac{M_2}{\overline{y}}\\
&\leqslant
5+4+M_2\qquad \because \overline{y}>1+q\\
&=:M_3
\endaligned
\end{equation}	

Let $\epsilon_3=\epsilon_3(x,y)\in(-1,1)$.
\begin{equation}
\aligned
\Re B^{(+)}_{1}(x,y)
&=e^y\frac{\overline{y}^2}{\overline{y}^2+x^2}
\left(1+\frac{\epsilon_3M_3}{\overline{y}}\right)\\
&=e^y\frac{(y+1+q)^2}{(y+1+q)^2+x^2}
\left(1+\frac{\epsilon_3M_3}{(y+1+q)}\right)\\
\endaligned
\end{equation}	
	
The most significant term in $\Re B^{(+)}_1(x,\pi n k^2), k\in[1,n]$ and $\Re B^{(+)}_1(x,\pi k^2/n), k\in[1,n]$ is $\Re B^{(+)}_1(x,\pi n^3)$.

\begin{equation}
\aligned
\Re B^{(+)}_1(x,\pi n^3)&=h^{(+)}_{2}(n,x)\left(1+\frac{\epsilon_4M_4}{(\pi n^3+2+q)}\right)\\
h^{(+)}_{2}(n,x):&=\exp(\pi n^3)\frac{(\pi n^3+2+q)^2}{(\pi n^3+2+q)^2+x^2}
\endaligned
\end{equation}
We can use $h^{(+)}_2(n,x)$ to bound all the other $2n+1$ terms.

For all $k\in [1,n-1]$ we have
\begin{equation}
\aligned
&\frac{\left|\Re B^{(+)}_1(x,\pi n k^2)\right|}{h^{(+)}_2(n,x)}\\
&<\frac{\exp(\pi n k^2)}{\exp(\pi n^3)}\left(1+\frac{M_4}{\pi n k^2+2+q}\right)\\
&\times \left(\frac{(\pi n^3+2+q)^2+x^2}{(\pi n k^2+2+q)^2+x^2}\right)\left(\frac{\pi n k^2+2+q}{\pi n^3+2+q}\right)^2\\
&<\frac{\exp(\pi n (n-1)^2)}{\exp(\pi n^3)}\left(1+\frac{M_4}{\pi n+2+q}\right)\\
&\times \left(\frac{(\pi n^3+2+q)^2+x^2}{(\pi n+2+q)^2+x^2}\right)\\
&<\left(1+M_4\right)n^5\exp(-\pi n(2n-1))\\
&<\left(1+M_4\right)n^5\exp(-\pi n^2)\quad \because n\geqslant 9\\
&=M_5n^5\exp(-\pi n^2)\qquad M_5:=1+ M_4\\
\endaligned
\end{equation}
And

\begin{equation}
\aligned
&\frac{\sum_{k=1}^{n-1}\left|\Re B^{(+)}_1(x,\pi n k^2)\right|}{h^{(+)}_2(n,x)}
<M_5n^6\exp(-\pi n^2)\\
\endaligned
\end{equation}

For all $k\in [1,n-1]$ we have
\begin{equation}
\aligned
&\frac{\left|\Re B^{(+)}_1(x,\pi k^2/n)\right|}{h^{(+)}_2(n,x)}\\
&<\frac{\exp(\pi k^2/n)}{\exp(\pi n^3)}\left(1+\frac{M_4}{\pi k^2/n+2+q}\right)\\
&\times \left(\frac{(\pi n^3+2+q)^2+x^2}{(\pi  k^2/n+2+q)^2+x^2}\right)\left(\frac{\pi k^2/n+2+q}{\pi n^3+2+q}\right)^2\\
&<\frac{\exp(\pi n)}{\exp(\pi n^3)}\left(1+\frac{nM_4}{\pi +n(2+q)}\right)\\
&\times \left(\frac{(\pi n^4+n(2+q))^2+n^2x^2}{(\pi+n(2+q))^2+n^2x^2}\right)\\
&<\left(1+M_4\right)n^5\exp(-\pi n (n^2-1))\\
&<M_5n^5\exp(-\pi n^2)\\
\endaligned
\end{equation}
And

\begin{equation}
\aligned
&\frac{\sum_{k=1}^{n}\left|\Re B^{(+)}_1(x,\pi  k^2/n)\right|}{h^{(+)}_2(n,x)}
&=M_5n^6\exp(-\pi n^2)\\
\endaligned
\end{equation}

Let $\epsilon_5:=\epsilon_5(n,x)\in(-1,1)$.Thus
\begin{equation}
\aligned
&\sum_{k=1}^{n}\left(\Re B^{(+)}_1(x,\pi n k^2)+n^{-1/2}\Re B^{(+)}_1(x,\pi k^2/n)\right)\\
&=h^{(+)}_2(n,x)\left(1+\frac{\epsilon_4M_4}{\pi n^3+2+q}+2\epsilon_5 M_5n^6\exp(-\pi n^2)\right)\\
&=h^{(+)}_2(n,x)\left(1+\frac{R_5(n,x)}{\pi n^3}\right)\\
\endaligned
\end{equation}

where
\begin{equation}
\aligned
\frac{R_5(n,x)}{\pi n^3}&:=\frac{\epsilon_4M_4}{\pi n^3+2+q}+2\epsilon_5M_5 n^6\exp(-\pi n^2)
\endaligned
\end{equation}

Since 
\begin{equation}
\aligned
R_5(n,x)&\leqslant\frac{M_4\pi n^3}{\pi n^3+2+q}+2\pi n^3 n^6M_5\exp(-\pi n^2)\\
&<M_4+2\pi M_5 n^{10}\exp(-\pi n^2)\\
&\leqslant M_4+2\pi M_5\left(\frac{5}{\pi e}\right)^5:=M_6\\
&\approx M_4+0.432323 M_5\\
\endaligned
\end{equation}

Let $\epsilon_6:=\epsilon_6(n,x)\in(-1,1)$.Then we have
\begin{equation}
\aligned
&\sum_{k=1}^{n}\left(\Re B^{(+)}_1(x,\pi n k^2)+n^{-1/2}\Re B^{(+)}_1(x,\pi k^2/n)\right)\\
&=h^{(+)}_2(n,x)\left(1+\frac{\epsilon_6M_6}{\pi n^3}\right)\\
\endaligned
\end{equation}

\end{proof}

\subsection[A(+2),B(+2)]{Bounds of $A^{(+)}_2(x,y)$ and $B^{(+)}_2(x,y)$}
\begin{theorem} \label{thmAp2Bound}
	 Let $q=1/4; n\in2\N_0+9; x\geqslant 0; m=7n^3$; $M=931; N=465/\pi\approx 148.014; \epsilon(n,x)\in (-1,1)$. Let
\begin{equation}\label{Ap4Ap5def}
\aligned
A^{(+)}_2(x,y)
=&-\frac{y^{m+2}}{\Gamma(m+2)}{_1}F_1(1;m+2;y )\\
&-\frac{y^{m+2}(q-ix)}{\Gamma(m+2)(m+2+q+ix)}
{_2}F_2\begin{pmatrix}{\begin{matrix} 1, & m+2+q+ix \\  m+2, & m+3+q+ix  \\  \end{matrix}} & ;y \end{pmatrix}\\
\endaligned
\end{equation}

Then

\begin{equation}
\aligned
&\sum_{1\leqslant k\leqslant n}\left(\Re A^{(+)}_2(x,\pi n k^2)-n^{-1/2}\Re A^{(+)}_2(x,\pi k^2/n)\right)\\
&={\color{red}{-}}\frac{(\pi n ^3)^{7n^3+2}}{\Gamma(7n^3+2)}\frac{(7n^3+2+q)(7n^3+2+2q)}{\left((7n^3+2+q)^2+x^2\right)} \left(\frac{7n^3}{7n^3-\pi n^3}\right)\left(1+\frac{\epsilon(n,x) M}{7 n^3}\right)\\
\endaligned
\end{equation}

\end{theorem}

\begin{proof}

Set $\mu=+$ in \cref{Zhou2F2C} we know that there exists a positive constant $M_3$ and real functions $\epsilon_{3A}=\epsilon_{3A}(m,x,y)\in (-1,1),\epsilon_{3B} =\epsilon_{3B}(m,x,y)\in (-1,1)$ such that

\begin{equation}\label{F221m2qm2m3q}
\aligned
&{_2}F_2\begin{pmatrix}{\begin{matrix} 1, & m+2+q+ ix \\  m+2, & m+3+q+ ix  \\  \end{matrix}} & ;y \end{pmatrix}\\
&={_1}F_1(1;m+2;y)+\frac{ (y\epsilon_{3A}+ix\epsilon_{3B})M_3}{(m+2+q)^2+x^2}
\endaligned
\end{equation}
holds for all $x\in\R$ and $m>2y>0$.

Thus
\begin{equation}\label{Ap4Ap5def}
\aligned
&-\frac{\Gamma(m+2)}{y^{m+2}} A^{(+)}_2(x,y)\\
&={_1}F_1(1;m+2;y)\\
&+\frac{(q-ix)}{2(m+2+q+ix)}
\left({_1}F_1(1,m+2;y)+\frac{(y\epsilon_{3A}+ix\epsilon_{3B}) M_3}{(m+2+q)^2+x^2}\right)\\
\endaligned
\end{equation}

\begin{equation}\label{Ap4Ap5def}
\aligned
&-\frac{\Gamma(m+2)}{y^{m+2}} \Re A^{(+)}_2(x,y)\\
&=\frac{(m+2+q)(m+2+2q)}{\left((m+2+q)^2+x^2\right)}{_1}F_1(1;m+2;y)\\
&+\frac{M_3}{(m+2+q)^2+x^2}\left[\frac{y\epsilon_{3A}(q(m+2+q)-x^2)}{(m+2+q)^2+x^2}+\frac{x\epsilon_{3B} (m+2+2q)}{(m+2+q)^2+x^2}\right]\\
\endaligned
\end{equation}
We now define the quantity in square bracket [] as $R_1(m,x,y)$

\begin{equation}
\aligned
R_1(m,x,y):=\frac{y\epsilon_{3A}(q(m+2+q)-x^2)}{(m+2+q)^2+x^2}+\frac{x\epsilon_{3B} (m+2+2q)}{(m+2+q)^2+x^2}
\endaligned
\end{equation}

Since 
\begin{equation}
\aligned
|R_1(m,x,y)|&\leqslant \frac{yq(m+2+q)+x^2}{(m+2+q)^2+x^2}+\frac{2x (m+2+q)}{(m+2+q)^2+x^2}\\
&\leqslant y+1
\endaligned
\end{equation}
We have
\begin{equation}
\aligned
&-\frac{\Gamma(m+2)}{y^{m+2}}\Re A^{(+)}_2(x,y)\\
&=\frac{(m+2+q)(m+2+2q)}{\left((m+2+q)^2+x^2\right)}{_1}F_1(1;m+2;y)+\frac{2\epsilon_3 (y+1) M_3}{\left((m+2+q)^2+x^2\right)}
\endaligned
\end{equation}
where $\epsilon_3=\epsilon_3(m,x,y)\in (-1,1)$

Set $\mu=+$ in \cref{Zhou1F1B}, we know that there exists a positive constant $M_2=16+\frac{216}{(e\log 2)^3}\approx 48.2919$ and real function $\epsilon_2=\epsilon_2(m,y)\in (-1,1)$ such that

\begin{equation}
\aligned
{_1}F_1(1;m+2;y)
&=\frac{m}{m-y}+\frac{\epsilon_2 M_2}{m},
\endaligned
\end{equation}
holds for all $m>2y>0$.

Substituting it into (xx) leads to
\begin{equation}
\aligned
&-\frac{\Gamma(m+2)}{y^{m+2}}\Re A^{(+)}_2(x,y)\\
&=\frac{(m+2+q)(m+2+2q)}{\left((m+2+q)^2+x^2\right)}\left(\frac{m}{m-y}+\frac{\epsilon_2 M_2}{m}\right)
+\frac{2\epsilon_3 (y+1) M_3}{\left((m+2+q)^2+x^2\right)}\\
&=\frac{(m+2+q)(m+2+2q)}{\left((m+2+q)^2+x^2\right)}\left(\frac{m}{m-y}\right)\\
&\times\left(1+\frac{\epsilon_2M_2}{m}\left(\frac{m-y}{m}\right)+\frac{2\epsilon_3 (y+1) M_3}{(m+2+q)(m+2+2q)}\left(\frac{m-y}{m}\right)\right)\\
&=\frac{(m+2+q)(m+2+2q)}{\left((m+2+q)^2+x^2\right)}\left(\frac{m}{m-y}\right)\\
&\times\left(1+\frac{R_2(m,x,y)}{m}\right)\\
\endaligned
\end{equation}
where
\begin{equation}
\aligned
\frac{R_2(m,x,y)}{m}&:=\frac{\epsilon_2M_2}{m}\left(\frac{m-y}{m}\right)+\frac{2\epsilon_3 (y+1) M_3}{(m+2+q)(m+2+2q)}\left(\frac{m-y}{m}\right)
\endaligned
\end{equation}

Since
\begin{equation}
\aligned
|R_2(m,x,y)|&\leqslant M_2\left(\frac{m-y}{m}\right)+\frac{2M_3(y+1) (m-y)}{(m+2+q)(m+2+2q)}\\
&< M_2+M_3\qquad \because m> 2y\\
&=16+\frac{216}{(e\log 2)^3}+17+\frac{216}{(e\log 2)^3}\\
&=33+\frac{432}{(e\log 2)^3}\\
&>33+432\qquad \because e>2,\log 2>1/2\\
&=465=:M_4
\endaligned
\end{equation}

Thus
\begin{equation}\label{Ap3Ap4y}
\aligned
&-\frac{\Gamma(m+2)}{y^{m+2}}\Re A^{(+)}_2(x,y)\\
&=\frac{(m+2+q)(m+2+2q)}{\left((m+2+q)^2+x^2\right)}\left(\frac{m}{m-y}\right)\left(1+\frac{\epsilon_{4}M_4}{m}\right)\\
\endaligned
\end{equation}
or
\begin{equation}\label{Ap3Ap4y}
\aligned
\Re A^{(+)}_2(x,y)
&=-\frac{y^{m+2}}{\Gamma(m+2)}\frac{(m+2+q)(m+2+2q)}{\left((m+2+q)^2+x^2\right)}\left(\frac{m}{m-y}\right)\left(1+\frac{\epsilon_{4}M_4}{m}\right)\\
\endaligned
\end{equation}

Where $\epsilon_{4}=\epsilon_4(m,x,y)\in (-1,1)$.

Note that $m=7n^3$. Among the $2n$ terms $|\Re A^{(+)}_2(x,\pi n k^2)|,k\in[1,n]$ and $|\Re A^{(+)}_2(x,\pi k^2/n)|,k\in[1,n]$, the largest one is $|\Re A^{(+)}_2(x,\pi n^3)|$.

\begin{equation}\label{Ap3Ap4y}
\aligned
\Re A^{(+)}_2(x,\pi n^3)
&=-h^{(+)}_3(n,x)\left(1+\frac{\epsilon_{4}M_4}{7n^3}\right)\\
\endaligned
\end{equation}
\begin{equation}\label{Ap3Ap4y}
\aligned
h^{(+)}_{3}(n,x)
&:=\frac{(\pi n^3)^{7n^3+2}}{\Gamma(7n^3+2)}\frac{(7n^3+2+q)(7n^3+2+2q)}{\left((7n^3+2+q)^2+x^2\right)}\left(\frac{7n^3}{7n^3-\pi n^3}\right).\\
\endaligned
\end{equation}
Thus we can use $h^{(+)}_{3}(n,x)$ to bound the other $2n-1$ terms.

For all $k\in [1,n-1]$ we have
\begin{equation}
\aligned
&\frac{\left|\Re A^{(+)}_2(x,\pi n k^2)\right|}{h^{(+)}_{3}(n,x)}\\
&<\frac{(\pi n k^2)^{7n^3+2}}{(\pi n^3)^{7n^3+2}}\frac{(7n^3-\pi n^3)}{(7n^3-\pi n k^2)}\left(1+\frac{M_4}{7 n^3}\right)\\
&<\left(\frac{n-1}{n}\right)^{14n^3+4}\frac{(7n^3-\pi n^3)}{(7n^3-\pi n^3)}\left(1+M_4\right)\\
&<\left(\left(1-\frac{1}{n}\right)^n\right)^{14n^2}\left(1+M_4\right)\\
&<(1+M_4)\exp(-14n^2)\\
\endaligned
\end{equation}
And

\begin{equation}
\aligned
\frac{\left|\sum_{k=1}^{n-1}\Re A^{(+)}_2(x,\pi n k^2)\right|}{h^{(+)}_{3}(n,x)}
&\leqslant
\frac{\sum_{k=1}^{n-1}\left|\Re A^{(+)}_2(x,\pi n k^2)\right|}{h^{(+)}_{3}(n,x)}\\
&<n\left(1+M_4\right)\exp(-14n^2)\\
\endaligned
\end{equation}

For all $k\in [1,n]$ we also have
\begin{equation}
\aligned
&\frac{\left|\Re A^{(+)}_2(x,\pi k^2/n)\right|}{h^{(+)}_{3}(n,x)}\\
&<\frac{(\pi k^2/n)^{7n^3+2}}{(\pi n^3)^{7n^3+2}}\frac{(7n^3-\pi n^3)}{(7n^3-\pi k^2/n)}\left(1+\frac{M_4}{7 n^3}\right)\\
&<n^{-14n^3-2}\frac{(7n^3-\pi n^3)}{(7n^3-\pi n^3)}\left(1+M_4\right)\\
&=n^{-14n^3-2}\left(1+M_4\right)\\
&<n^{-14n^3}\left(1+M_4\right)\\
&=\left(1+M_4\right)\exp(-(14n^3)\log n)\\
&<\left(1+M_4\right)\exp(-14n^2)\qquad \because \text{ for } n\geqslant 9,\log n>1.
\endaligned
\end{equation}
And

\begin{equation}
\aligned
\frac{\left|\sum_{k=1}^{n-1}\Re A^{(+)}_2(x,\pi k^2/n)\right|}{h^{(+)}_{3}(n,x)}
&\leqslant
\frac{\sum_{k=1}^{n-1}\left|\Re A^{(+)}_2(x,\pi k^2/n)\right|}{h^{(+)}_{3}(n,x)}\\
&<n\left(1+M_4\right)\exp(-14n^2).\\
\endaligned
\end{equation}

Thus we can write
\begin{equation}
\aligned
&\sum_{1\leqslant k\leqslant n}\left(\Re A^{(+)}_2(x,\pi n k^2)-n^{-1/2}\Re A^{(+)}_2(x,\pi k^2/n)\right)\\
&=-h^{(+)}_{3}(n,x)\\
&\times\left(1+\frac{\epsilon_4M_4}{7n^3}+n\left(1+M_4\right)\left(\epsilon_{5A}(n,x)+\epsilon_{5B}(n,x)\right)\exp(-14n^2)\right)\\
\\
\endaligned
\end{equation}
Where $\epsilon_{5A}(n,x)\in(-1,1),\epsilon_{5B}(n,x)\in(-1,1)$.

We can combine the 3 error terms and define
\begin{equation}
\aligned
\frac{R_5(n,x)}{7 n^3}:&=\frac{\epsilon_{4}M_4}{7 n^3}+n(1+M_4)\left(\epsilon_{5A}(n,x)+\epsilon_{5B}(n,x)\right)\exp(-14n^2)\\
\endaligned
\end{equation}

Since 
\begin{equation}
\aligned
R_5(n,x)&\leqslant M_4+7n^3\cdot 2n(1+M_4)\exp(-14n^2)\\
&< M_4+14(1+M_4) n^4\exp(-14n^2)\\
&\leqslant M_4+(1+M_4)\frac{14}{49e^2}\\
&< M_4+1+M_4\qquad \because e>2\\
&= 2M_4+1=:M\\
&=2\cdot 465+1=931
\endaligned
\end{equation}

We have
\begin{equation}
\aligned
&\sum_{1\leqslant k\leqslant n}\left(\Re A^{(+)}_2(x,\pi n k^2)-n^{-1/2}\Re A^{(+)}_2(x,\pi k^2/n)\right)\\
&=-h^{(+)}_3(n,x)\left(1+\frac{\epsilon M}{7 n^3}\right)\\
&={\color{red}{-}}\frac{(\pi n ^3)^{7 n ^3+2}}{\Gamma(7 n ^3+2)}\frac{(7 n ^3+2+q)(7 n ^3+2+2q)}{\left((7 n ^3+2+q)^2+x^2\right)} \left(\frac{7 n ^3}{7 n ^3-\pi n ^3}\right)\left(1+\frac{\epsilon M}{7 n^3}\right)\\
\endaligned
\end{equation}
where $\epsilon=\epsilon(n,x)\in (-1,1)$.

\end{proof}

\begin{theorem} \label{thmBp2Bound}
	 Let $q=1/4; n\in2\N_0+9; x\geqslant 0; m=7n^3$; $M=931; N=465/\pi\approx 148.014; \epsilon(n,x)\in (-1,1)$. Let
	\begin{equation}\label{Bpdef2}
	\aligned
	B^{(+)}_2(x,y)
	&=\frac{y^{m+2}(q-ix)}{ix\Gamma(m+2)(m+2+q+ix)}
	{_2}F_2\begin{pmatrix}{\begin{matrix} 1, & m+2+q+ix \\  m+2, & m+3+q+ix  \\  \end{matrix}} & ;y \end{pmatrix}\\
	\endaligned
	\end{equation}

Then
\begin{equation}
\aligned
&\sum_{1\leqslant k\leqslant n}\left(\Re B^{(+)}_2(x,\pi n k^2)+n^{-1/2}\Re B^{(+)}_2(x,\pi k^2/n)\right)\\
&={\color{red}{-}}\frac{(\pi n ^3)^{7 n ^3+2}}{\Gamma(7 n ^3+2)}\frac{(7 n ^3+2+2q)}{\left((7 n ^3+2+q)^2+x^2\right)} \left(\frac{7 n ^3}{7 n ^3-\pi n ^3}\right)\left(1+\frac{\epsilon(n,x) M}{7 n^3}\right).\\
\endaligned
\end{equation}

\end{theorem}

\begin{proof}

Set $\mu=+$ in \cref{Zhou2F2C} we know that there exists a positive constant $M_1=17+\frac{216}{(e\log 2)^3}\approx 49.2919$ and real functions $\epsilon_{1A}=\epsilon_{1A}(m,x,y)\in (-1,1),\epsilon_{1B} =\epsilon_{1B}(m,x,y)\in (-1,1)$ such that

\begin{equation}\label{F221m2qm2m3q}
\aligned
&{_2}F_2\begin{pmatrix}{\begin{matrix} 1, & m+2+q+ ix \\  m+2, & m+3+q+ ix  \\  \end{matrix}} & ;y \end{pmatrix}\\
&={_1}F_1(1;m+2;y)+\frac{ (y\epsilon_{1A}+ix\epsilon_{1B})M_1}{(m+2+q)^2+x^2}
\endaligned
\end{equation}
holds for all $x\in\R$ and $m>2y>0$.

Thus
\begin{equation}
\aligned
&-\frac{\Gamma(m+2)}{y^{m+2}}B^{(+)}_2(x,y)\\
=&-\frac{(q-ix)}{i{\color{red}{x}}(m+2+q+ix)}
\left({_1}F_1(1,m+2;y)+\frac{(y\epsilon_{1A}+i{\color{red}{x}}\epsilon_{1B}) M_1}{(m+2+q)^2+x^2}\right)\\
\endaligned
\end{equation}

\begin{equation}
\aligned
&-\frac{\Gamma(m+2)}{y^{m+2}}\Re B^{(+)}_2(x,y)\\
&=\frac{(m+2+2q)}{\left((m+2+q)^2+x^2\right)}{_1}F_1(1;m+2;y)
+\frac{(m+2+2q)M_1}{\left((m+2+q)^2+x^2\right)}\\
&\times \left[\frac{y\epsilon_{1A}}{((m+2+q)^2+x^2)}
+\frac{(x^2-q(m+2+q)}{\left((m+2+q)^2+x^2\right)}\frac{\epsilon_{1B}}{(m+2+2q)}\right]\\
&=\frac{(m+2+2q)}{\left((m+2+q)^2+x^2\right)}{_1}F_1(1;m+2;y)
+\frac{(m+2+2q)M_1}{\left((m+2+q)^2+x^2\right)}\left(\frac{R_1(m,x,y)}{m}\right)
\endaligned
\end{equation}

Where

\begin{equation}
\aligned
\frac{R_1(m,x,y)}{m}&:=\frac{y\epsilon_{1A}}{((m+2+q)^2+x^2)}
+\frac{(x^2-q(m+2+q)}{\left((m+2+q)^2+x^2\right)}\frac{\epsilon_{1B}}{(m+2+2q)}\\
\endaligned
\end{equation}

Since 
\begin{equation}
\aligned
|R_1(m,x,y)|&\leqslant\frac{ym}{((m+2+q)^2+x^2)}
+\frac{(x^2+q(m+2+q)}{\left((m+2+q)^2+x^2\right)}\frac{m}{(m+2+2q)}\\
&< \frac{1}{2}+1<2\qquad \because m>2y
\endaligned
\end{equation}

Consequently
\begin{equation}
\aligned
&-\frac{\Gamma(m+2)}{y^{m+2}}\Re B^{(+)}_2(x,y)\\
&=\frac{(m+2+2q)}{\left((m+2+q)^2+x^2\right)}\left({_1}F_1(1;m+2;y)+\frac{2\epsilon_{1C}M_1}{m}\right)
\endaligned
\end{equation}
where $\epsilon_{1C}=\epsilon_{1C}(m,x,y)\in(-1,1)$.

Set $\mu=+$ in \cref{Zhou1F1B}, we know that there exists a positive constant $M_2=16+\frac{216}{(e\log 2)^3}\approx 48.2919$ and real function $\epsilon_2=\epsilon_2(m,y)\in (-1,1)$ such that

\begin{equation}
\aligned
{_1}F_1(1;m+2;y)
&=\left(1-\frac{y}{m}\right)^{-1}+\frac{\epsilon_2 M_2}{m}\\
&=\frac{m}{m-y}+\frac{\epsilon_2 M_2}{m},
\endaligned
\end{equation}
holds for all $m>2y>0$.

\begin{equation}\label{Bp3y}
\aligned
&-\frac{\Gamma(m+2)}{y^{m+2}}\Re B^{(+)}_2(x,y)\\
&=\frac{(m+2+2q)}{\left((m+2+q)^2+x^2\right)}\left(\frac{m}{m-y}+\frac{\epsilon_2 M_2}{m}+\frac{2\epsilon_{1C} M_1}{}\right)\\
&=\frac{(m+2+2q)}{\left((m+2+q)^2+x^2\right)}\left(\frac{m}{m-y}\right)\left(1+\frac{\epsilon_3 M_3}{m}\right)\\
\endaligned
\end{equation}
where $\epsilon_3=\epsilon(m,x,y)\in (-1,1)$ and
\begin{equation}
\aligned
 M_3&=2M_1+M_2\\
 &=2\cdot \left(17+\frac{216}{(e\log 2)^3}\right)+\left(16+\frac{216}{(e\log 2)^3}\right)\\
 &=50+\frac{864}{(e\log 2)^3}=:M_4\\
& \approx 293.374
 \endaligned
 \end{equation}
 
 \begin{equation}\label{Bp3y}
 \aligned
 \Re B^{(+)}_2(x,y)
 =-\frac{y^{m+2}}{\Gamma(m+2)}\frac{(m+2+2q)}{\left((m+2+q)^2+x^2\right)}\left(\frac{m}{m-y}\right)\left(1+\frac{\epsilon_4 M_4}{m}\right)\\
 \endaligned
 \end{equation}
 
 Note $m=7n^3$. Among the $2n$ terms $|\Re B^{(+)}_2(x,\pi n k^2)|,k\in[1,n]$ and $|\Re B^{(+)}_2(x,\pi k^2/n)|,k\in[1,n]$, the largest one is $|\Re B^{(+)}_2(x,\pi n^3)|$.
 
 \begin{equation}\label{Ap3Ap4y}
 \aligned
 \Re B^{(+)}_2(x,\pi n^3)
 &=-h^{(+)}_{4}(n,x)\left(1+\frac{\epsilon_{4}M_4}{7n^3}\right)\\
 \endaligned
 \end{equation}
 \begin{equation}\label{Ap3Ap4y}
 \aligned
 h^{(+)}_{4}(n,x)
 &:=\frac{(\pi n^3)^{7n^3+2}}{\Gamma(7n^3+2)}\frac{(7n^3+2+2q)}{\left((7n^3+2+q)^2+x^2\right)}\left(\frac{7n^3}{7n^3-\pi n^3}\right).
 \endaligned
 \end{equation}
 Thus we can use $h^{(+)}_{4}(n,x)$ to bound the other $2n-1$ terms.
 
 For all $k\in [1,n-1]$ we have
 \begin{equation}
 \aligned
 &\frac{\left|\Re B^{(+)}_2(x,\pi n k^2)\right|}{h^{(+)}_{4}(n,x)}\\
 &<\frac{(\pi n k^2)^{7n^3+2}}{(\pi n^3)^{7n^3+2}}\frac{(7n^3-\pi n^3)}{(7n^3-\pi n k^2)}\left(1+\frac{M_4}{7 n^3}\right)\\
 &<\left(\frac{n-1}{n}\right)^{14n^3+4}\frac{(7n^3-\pi n^3)}{(7n^3-\pi n ^3)}\left(1+\frac{M_4}{7 n^3 }\right)\\
 &=\left(1-\frac{1}{n}\right)^{14n^3+4}\left(1+\frac{M_4}{7n^3 }\right)\\
 &<\left(\left(1-\frac{1}{n}\right)^n\right)^{14n^2}\left(1+\frac{M_4}{7 n^3 }\right)\\
 &<\exp(-14n^2)\left(1+\frac{M_4}{7 n^3 }\right)\\
 &<\exp(-14n^2)\left(1+\frac{M_4}{7 n^3 }\right)\quad \because 7-\pi >\pi\\
 &<(1+M_4)\exp(-14n^2)\\
 \endaligned
 \end{equation}
 And
 
 \begin{equation}
 \aligned
 \frac{\left|\sum_{k=1}^{n-1}\Re B^{(+)}_2(x,\pi n k^2)\right|}{h^{(+)}_{4}(n,x)}
 &\leqslant
 \frac{\sum_{k=1}^{n-1}\left|\Re B^{(+)}_2(x,\pi n k^2)\right|}{h^{(+)}_{4}(n,x)}\\
 &<n\left(1+M_4\right)\exp(-14n^2)\\
 \endaligned
 \end{equation}
 
 For all $k\in [1,n]$ we also have
 \begin{equation}
 \aligned
 &\frac{\left|\Re B^{(+)}_2(x,\pi k^2/n)\right|}{h^{(+)}_{4}(n,x)}\\
 &<\frac{(\pi k^2/n)^{7n^3+2}}{(\pi n^3)^{7n^3+2}}\frac{(7n^3-\pi n^3)}{(7n^3-\pi k^2/n)}\left(1+\frac{M_4}{7 n^3}\right)\\
 &<n^{-14n^3-2}\frac{(7n^3-\pi n^3)}{(7n^3-\pi n^3)}\left(1+\frac{M_4}{7 n^3 }\right)\\
 &=n^{-14n^3-2}\left(1+\frac{M_4}{7n^3 }\right)\\
 &<n^{-14n^3}\left(1+M_4\right)\\
 &=\left(1+M_4\right)\exp(-(14n^3)\log n)\\
 &<\left(1+M_4\right)\exp(-14n^2)\qquad \because \text{ for } n\geqslant 9,\log n>1
 \endaligned
 \end{equation}
 And
 
 \begin{equation}
 \aligned
 \frac{\left|\sum_{k=1}^{n-1}\Re B^{(+)}_2(x,\pi k^2/n)\right|}{h^{(+)}_{4}(n,x)}
 &\leqslant
 \frac{\sum_{k=1}^{n-1}\left|\Re B^{(+)}_2(x,\pi k^2/n)\right|}{h^{(+)}_{4}(n,x)}\\
 &<n\left(1+M_4\right)\exp(-14n^2).\\
 \endaligned
 \end{equation}

 Thus we can write
 \begin{equation}
 \aligned
 &\sum_{1\leqslant k\leqslant n}\left(\Re B^{(+)}_2(x,\pi n k^2)+n^{-1/2}\Re B^{(+)}_3(x,\pi k^2/n)\right)\\
 &=-h^{(+)}_{4}(n,x)\\
 &\times\left(1+\frac{\epsilon_4M_4}{7n^3}+n\left(1+M_4\right)\left(\epsilon_{5A}(n,x)+\epsilon_{5B}(n,x)\right)\exp(-14n^2)\right)\\
 \\
 \endaligned
 \end{equation}
 Where $\epsilon_{5A}(n,x)\in(-1,1),\epsilon_{5B}(n,x)\in(-1,1)$.
 
 We can combine the 3 error terms and define
 \begin{equation}
 \aligned
 \frac{R_5(n,x)}{7 n^3}:&=\frac{\epsilon_{4}M_4}{7 n^3}+n(1+M_4)\left(\epsilon_{5A}(n,x)+n\epsilon_{5B}(n,x)\right)\exp(-14n^2)\\
 \endaligned
 \end{equation}
 
 Since 
 \begin{equation}
 \aligned
 R_5(n,x)&\leqslant M_4+7n^3\cdot 2n(1+M_4)\exp(-14n^2)\\
 &< M_4+14(1+M_4) n^4\exp(-14n^2)\\
 &\leqslant M_4+(1+M_4)\frac{14}{49e^2}\\
 &< M_4+1+M_4\qquad \because e>2\\
 &= 2M_4+1=:M\\
 &=2\cdot 465+1=931
 \endaligned
 \end{equation}

 We have
 \begin{equation}
 \aligned
 &\sum_{1\leqslant k\leqslant n}\left(\Re B^{(+)}_2(x,\pi n k^2)+n^{-1/2}\Re B^{(+)}_2(x,\pi k^2/n)\right)\\
 &=-h^{(+)}_4(n,x)\left(1+\frac{\epsilon M}{7 n^3}\right)\\
 &={\color{red}{-}}\frac{(\pi n ^3)^{7 n ^3+2}}{\Gamma(7 n ^3+2)}\frac{(7 n ^3+2+2q)}{\left((7 n ^3+2+q)^2+x^2\right)} \left(\frac{7 n ^3}{7 n ^3-\pi n ^3}\right)\left(1+\frac{\epsilon M}{7 n^3}\right).\\
 \endaligned
 \end{equation}

\end{proof}

\subsection[A(-2),B(-2)]{Bounds of $A^{(-)}_2(x,y)$ and $B^{(-)}_2(x,y)$}
\begin{theorem} \label{thmAm2Bound}
	 Let $q=1/4; n\in2\N_0+9; x\geqslant 0; m=7n^3$; $M=931; N=465/\pi\approx 148.014; \epsilon(n,x)\in (-1,1)$. Let
	\begin{equation}\label{Ap4Ap5def}
	\aligned
	A^{(-)}_2(x,y)
	=&-\frac{(-y)^{m+2}}{\Gamma(m+2)}{_1}F_1(1;m+2;-y )\\
	&-\frac{(-y)^{m+2}(q-ix)}{\Gamma(m+2)(m+2+q+ix)}
	{_2}F_2\begin{pmatrix}{\begin{matrix} 1, & m+2+q+ix \\  m+2, & m+3+q+ix  \\  \end{matrix}} & ;-y \end{pmatrix}\\
	\endaligned
	\end{equation}

	Then
	
	\begin{equation}
	\aligned
	&\sum_{1\leqslant k\leqslant n}\left(\Re A^{(-)}_2(x,\pi n k^2)-n^{-1/2}\Re A^{(-)}_2(x,\pi k^2/n)\right)\\
	&={\color{red}{+}}\frac{(\pi n ^3)^{7n^3+2}}{\Gamma(7n^3+2)}\frac{(7n^3+2+q)(7n^3+2+2q)}{\left((7n^3+2+q)^2+x^2\right)} \left(\frac{7n^3}{7n^3+\pi n^3}\right)\left(1+\frac{\epsilon(n,x) M}{7 n^3}\right)\\
	\endaligned
	\end{equation}
	
\end{theorem}

\begin{proof}
	Because $m$ is odd, so $(-y)^{m+2}=-y^{m+2}$. 	
	
	Set $\mu=-$ in \cref{Zhou2F2C} we know that there exists a positive constant $M_3=17+\frac{216}{(e\log 2)^3}\approx 49.2919$ and real functions $\epsilon_{3A}=\epsilon_{3A}(m,x,y)\in (-1,1),\epsilon_{3B} =\epsilon_{3B}(m,x,y)\in (-1,1)$ such that

	\begin{equation}
	\aligned
	&{_2}F_2\begin{pmatrix}{\begin{matrix} 1, & m+2+q+ ix \\  m+2, & m+3+q+ ix  \\  \end{matrix}} & ;-y \end{pmatrix}\\
	&={_1}F_1(1;m+2;-y)+\frac{ (y\epsilon_{3A}+ix\epsilon_{3B})M_3}{(m+2+q)^2+x^2}
	\endaligned
	\end{equation}
	holds for all $x\in\R$ and $m>2y>0$.

Thus
	\begin{equation}\label{Ap4Ap5def}
	\aligned
	&{\color{red}{+}}\frac{\Gamma(m+2)}{y^{m+2}} A^{(-)}_2(x,y)\\
	&={_1}F_1(1;m+2;-y)\\
	&+\frac{(q-ix)}{2(m+2+q+ix)}
	\left({_1}F_1(1,m+2;-y)+\frac{(y\epsilon_{3A}+ix\epsilon_{3B}) M_3}{(m+2+q)^2+x^2}\right)\\
	\endaligned
	\end{equation}

	\begin{equation}\label{Ap4Ap5def}
	\aligned
	&{\color{red}{+}}\frac{\Gamma(m+2)}{y^{m+2}} \Re A^{(-)}_2(x,y)\\
	&=\frac{(m+2+q)(m+2+2q)}{\left((m+2+q)^2+x^2\right)}{_1}F_1(1;m+2;y)\\
	&+\frac{M_3}{(m+2+q)^2+x^2}\left[\frac{y\epsilon_{3A}(q(m+2+q)-x^2)}{(m+2+q)^2+x^2}+\frac{x\epsilon_{3B} (m+2+2q)}{(m+2+q)^2+x^2}\right]\\
	\endaligned
	\end{equation}
	We now define the quantity in square bracket [] as $R_1(m,x,y)$
	
	\begin{equation}
	\aligned
	R_1(m,x,y):=\frac{y\epsilon_{3A}(q(m+2+q)-x^2)}{(m+2+q)^2+x^2}+\frac{x\epsilon_{3B} (m+2+2q)}{(m+2+q)^2+x^2}
	\endaligned
	\end{equation}
	
	Since 
	\begin{equation}
	\aligned
	|R_1(m,x,y)|&\leqslant \frac{yq(m+2+q)+x^2}{(m+2+q)^2+x^2}+\frac{2x (m+2+q)}{(m+2+q)^2+x^2}\\
	&\leqslant y+1
	\endaligned
	\end{equation}
	We have
	\begin{equation}
	\aligned
	&{\color{red}{+}}\frac{\Gamma(m+2)}{y^{m+2}}\Re A^{(-)}_2(x,y)\\
	&=\frac{(m+2+q)(m+2+2q)}{\left((m+2+q)^2+x^2\right)}{_1}F_1(1;m+2;y)+\frac{2\epsilon_3 (y+1) M_3}{\left((m+2+q)^2+x^2\right)}
	\endaligned
	\end{equation}
	where $\epsilon_3=\epsilon_3(m,x,y)\in (-1,1)$
	
	Set $\mu=-1$ in \cref{Zhou1F1B}, we know that there exists a positive constant $M_2=16+\frac{216}{(e\log 2)^3}\approx 48.2919$ and real function $\epsilon_2=\epsilon_2(m,y)\in (-1,1)$ such that
	
	\begin{equation}
	\aligned
	{_1}F_1(1;m+2;-y)
	&=\frac{m}{m+y}+\frac{\epsilon_2 M_2}{m},
	\endaligned
	\end{equation}
	holds for all $m>2y>0$.
	
	Substituting it into (xx) leads to
	\begin{equation}
	\aligned
	&{\color{red}{+}}\frac{\Gamma(m+2)}{y^{m+2}}\Re A^{(-)}_2(x,y)\\
	&=\frac{(m+2+q)(m+2+2q)}{\left((m+2+q)^2+x^2\right)}\left(\frac{m}{m{\color{red}{+}}y}+\frac{\epsilon_2 M_2}{m}\right)
	+\frac{2\epsilon_3 (y+1) M_3}{\left((m+2+q)^2+x^2\right)}\\
	&=\frac{(m+2+q)(m+2+2q)}{\left((m+2+q)^2+x^2\right)}\left(\frac{m}{m{\color{red}{+}}y}\right)\\
	&\times\left(1+\frac{\epsilon_2M_2}{m}\left(\frac{m{\color{red}{+}}y}{m}\right)+\frac{2\epsilon_3 (y+1) M_3}{(m+2+q)(m+2+2q)}\left(\frac{m{\color{red}{+}}y}{m}\right)\right)\\
	&=\frac{(m+2+q)(m+2+2q)}{\left((m+2+q)^2+x^2\right)}\left(\frac{m}{m{\color{red}{+}}y}\right)\\
	&\times\left(1+\frac{R_2(m,x,y)}{m}\right)\\
	\endaligned
	\end{equation}
	where
	\begin{equation}
	\aligned
	\frac{R_2(m,x,y)}{m}&:=\frac{\epsilon_2M_2}{m}\left(\frac{m{\color{red}{+}}y}{m}\right)+\frac{2\epsilon_3 (y+1) M_3}{(m+2+q)(m+2+2q)}\left(\frac{m{\color{red}{+}}y}{m}\right)
	\endaligned
	\end{equation}
	
	Since
	\begin{equation}
	\aligned
	|R_2(m,x,y)|&\leqslant M_2\left(\frac{m{\color{red}{+}}y}{m}\right)+\frac{2M_3(y+1) (m{\color{red}{+}}y)}{(m+2+q)(m+2+2q)}\\
	&< 2(M_2+M_3)=:M_4\qquad \because m> 2y\\
	&=2\left(16+\frac{216}{(e\log 2)^3}\right)+2\left(17+\frac{216}{(e\log 2)^3}\right)\\
	\endaligned
	\end{equation}
	
	Thus
	\begin{equation}\label{Ap3Ap4y}
	\aligned
	&{\color{red}{+}}\frac{\Gamma(m+2)}{y^{m+2}}\Re A^{(-)}_2(x,y)\\
	&=\frac{(m+2+q)(m+2+2q)}{\left((m+2+q)^2+x^2\right)}\left(\frac{m}{m{\color{red}{+}}y}\right)\left(1+\frac{\epsilon_{4}M_4}{m}\right)\\
	\endaligned
	\end{equation}
	or
	\begin{equation}\label{Ap3Ap4y}
	\aligned
	\Re A^{(-)}_2(x,y)
	&={\color{red}{+}}\frac{y^{m+2}}{\Gamma(m+2)}\frac{(m+2+q)(m+2+2q)}{\left((m+2+q)^2+x^2\right)}\left(\frac{m}{m{\color{red}{+}}y}\right)\left(1+\frac{\epsilon_{4}M_4}{m}\right)\\
	\endaligned
	\end{equation}
	
	Where $\epsilon_{4}=\epsilon_4(m,x,y)\in (-1,1)$.
	
	Note $m=7n^3$. Among the $2n$ terms $|\Re A^{(-)}_2(x,\pi n k^2)|,k\in[1,n]$ and $|\Re A^{(-)}_2(x,\pi k^2/n)|,k\in[1,n]$, the largest one is $|\Re A^{(-)}_2(x,\pi n^3)|$.
	
	\begin{equation}\label{Ap3Ap4y}
	\aligned
	\Re A^{(-)}_2(x,\pi n^3)
	&={\color{red}{+}}h^{(-)}_{3}(n,x)\left(1+\frac{\epsilon_{4}M_4}{7n^3}\right)\\
	\endaligned
	\end{equation}
	\begin{equation}\label{Ap3Ap4y}
	\aligned
	h^{(-)}_{3}(n,x)
	&:=\frac{(\pi n^3)^{7n^3+2}}{\Gamma(7n^3+2)}\frac{(7n^3+2+q)(7n^3+2+2q)}{\left((7n^3+2+q)^2+x^2\right)}\left(\frac{7n^3}{7n^3{\color{red}{+}}\pi n^3}\right).\\
	\endaligned
	\end{equation}
	Thus we can use $h^{(-)}_{3}(n,x)$ to bound the other $2n-1$ terms.
	
	For all $k\in [1,n-1]$ we have
	\begin{equation}
	\aligned
	&\frac{\left|\Re A^{(-)}_2(x,\pi n k^2)\right|}{h^{(-)}_{3}(n,x)}\\
	&<\frac{(\pi n k^2)^{7n^3+2}}{(\pi n^3)^{7n^3+2}}\frac{(7n^3{\color{red}{+}}\pi n^3)}{(7n^3{\color{red}{+}}\pi n k^2)}\left(1+\frac{M_4}{7 n^3}\right)\\
	&<\left(\frac{n-1}{n}\right)^{14n^3+4}\frac{(7n^3+\pi n^3)}{7n^3}\left(1+\frac{M_4}{7 n^3 }\right)\\
	&=\left(1-\frac{1}{n}\right)^{14n^3+4}\left(\frac{7+\pi}{7}\right)\left(1+\frac{M_4}{7n^3 }\right)\\
	&<2\left(\left(1-\frac{1}{n}\right)^n\right)^{14n^2}\left(1+\frac{M_4}{7 n^3 }\right)\quad \because \pi <7\\
	&<2\exp(-14n^2)\left(1+\frac{M_4}{7 n^3 }\right)\\
	&<2(1+M_4)\exp(-14n^2)\\
	\endaligned
	\end{equation}
	And
	
	\begin{equation}
	\aligned
	\frac{\left|\sum_{k=1}^{n-1}\Re A^{(-)}_2(x,\pi n k^2)\right|}{h^{(-)}_{3}(n,x)}
	&\leqslant
	\frac{\sum_{k=1}^{n-1}\left|\Re A^{(-)}_2(x,\pi n k^2)\right|}{h^{(-)}_{3}(n,x)}\\
	&<2n\left(1+M_4\right)\exp(-14n^2)\\
	\endaligned
	\end{equation}
	
	For all $k\in [1,n]$ we also have
	\begin{equation}
	\aligned
	&\frac{\left|\Re A^{(-)}_2(x,\pi k^2/n)\right|}{h^{(-)}_{3}(n,x)}\\
	&<\frac{(\pi k^2/n)^{7n^3+2}}{(\pi n^3)^{7n^3+2}}\frac{(7n^3+\pi n^3)}{(7n^3+\pi k^2/n)}\left(1+\frac{M_4}{7 n^3}\right)\\
	&<n^{-14n^3-2}\left(\frac{7n^3+\pi n^3}{7n^3}\right)\left(1+\frac{M_4}{7 n^3 }\right)\\
	&<2n^{-14n^3-2}\left(1+\frac{M_4}{7n^3 }\right)\\
	&<2n^{-14n^3}\left(1+M_4\right)\\
	&=2\left(1+M_4\right)\exp(-14n^3\log n)\\
	&<2\left(1+M_4\right)\exp(-14n^2)\qquad \because \text{ for } n\geqslant 9,\log n>1.
	\endaligned
	\end{equation}
	And
	
	\begin{equation}
	\aligned
	\frac{\left|\sum_{k=1}^{n-1}\Re A^{(-)}_2(x,\pi k^2/n)\right|}{h^{(-)}_{3}(n,x)}
	&\leqslant
	\frac{\sum_{k=1}^{n-1}\left|\Re A^{(-)}_2(x,\pi k^2/n)\right|}{h^{(-)}_{3}(n,x)}\\
	&<2n\left(1+M_4\right)\exp(-14n^2).\\
	\endaligned
	\end{equation}

	Thus we can write
	\begin{equation}
	\aligned
	&\sum_{1\leqslant k\leqslant n}\left(\Re A^{(-)}_2(x,\pi n k^2)-n^{-1/2}\Re A^{(-)}_2(x,\pi k^2/n)\right)\\
	&=h^{(-)}_{3}(n,x)
	\left(1+\frac{\epsilon_4M_4}{7n^3}+2n\left(1+M_4\right)\left(\epsilon_{5A}(n,x)+\epsilon_{5B}(n,x)\right)\exp(-14n^2)\right)\\
	\\
	\endaligned
	\end{equation}
	Where $\epsilon_{5A}(n,x)\in(-1,1),\epsilon_{5B}(n,x)\in(-1,1)$.
	
	We can combine the 3 error terms and define
	\begin{equation}
	\aligned
	\frac{R_5(n,x)}{7 n^3}:&=\frac{\epsilon_{4}M_4}{7 n^3}+2n(1+M_4)\left(\epsilon_{5A}(n,x)+\epsilon_{5B}(n,x)\right)\exp(-14n^2)\\
	\endaligned
	\end{equation}
	
	Since 
	\begin{equation}
	\aligned
	R_5(n,x)&\leqslant M_4+7n^3\cdot 4n(1+M_4)\exp(-14n^2)\\
	&< M_4+28(1+M_4) n^4\exp(-14n^2)\\
	&\leqslant M_4+(1+M_4)\frac{28}{49e^2}\\
	&< M_4+1+M_4\qquad \because e>2\\
	&= 2M_4+1\\
	&=2\cdot 465+1=931=:M
	\endaligned
	\end{equation}

	We have
	\begin{equation}
	\aligned
	&\sum_{1\leqslant k\leqslant n}\left(\Re A^{(-)}_2(x,\pi n k^2)-n^{-1/2}\Re A^{(-)}_2(x,\pi k^2/n)\right)\\
	&=h^{(-)}_3(n,x)\left(1+\frac{\epsilon M}{7 n^3}\right)\\
	&={\color{red}{+}}\frac{(\pi n ^3)^{7 n ^3+2}}{\Gamma(7 n ^3+2)}\frac{(7 n ^3+2+q)(7 n ^3+2+2q)}{\left((7 n ^3+2+q)^2+x^2\right)} \left(\frac{7 n ^3}{7 n ^3{\color{red}{+}}\pi n ^3}\right)\left(1+\frac{\epsilon M}{7 n^3}\right)\\
	\endaligned
	\end{equation}
	where $\epsilon=\epsilon(n,x)\in (-1,1)$.

\end{proof}

\begin{theorem} \label{thmBm2Bound}
	 Let $q=1/4; n\in2\N_0+9; x\geqslant 0; m=7n^3$; $M=931; N=465/\pi\approx 148.014; \epsilon(n,x)\in (-1,1)$. Let
	\begin{equation}\label{Bpdef2}
	\aligned
	B^{(-)}_2(x,y)
	&=\frac{(-y)^{m+2}(q-ix)}{ix\Gamma(m+2)(m+2+q+ix)}
	{_2}F_2\begin{pmatrix}{\begin{matrix} 1, & m+2+q+ix \\  m+2, & m+3+q+ix  \\  \end{matrix}} & ;-y \end{pmatrix}\\
	\endaligned
	\end{equation}

	Then
	\begin{equation}
	\aligned
	&\sum_{1\leqslant k\leqslant n}\left(\Re B^{(-)}_2(x,\pi n k^2)+n^{-1/2}\Re B^{(+)}_2(x,\pi k^2/n)\right)\\
	&={\color{red}{+}}\frac{(\pi n ^3)^{7 n ^3+2}}{\Gamma(7 n ^3+2)}\frac{(7 n ^3+2+2q)}{\left((7 n ^3+2+q)^2+x^2\right)} \left(\frac{7 n ^3}{7 n ^3+\pi n ^3}\right)\left(1+\frac{\epsilon(n,x) M}{7 n^3}\right).\\
	\endaligned
	\end{equation}

\end{theorem}

\begin{proof}
	Because $m$ is odd, so $(-y)^{m+2}=-y^{m+2}$. 		
	
	Set $\mu=-$ in \cref{Zhou2F2C} we know that there exists a positive constant $M_1=17+\frac{216}{(e\log 2)^3}\approx 49.2919$ and real functions $\epsilon_{1A}=\epsilon_{1A}(m,x,y)\in (-1,1),\epsilon_{1B} =\epsilon_{1B}(m,x,y)\in (-1,1)$ such that

	\begin{equation}\label{F221m2qm2m3q}
	\aligned
	&{_2}F_2\begin{pmatrix}{\begin{matrix} 1, & m+2+q+ ix \\  m+2, & m+3+q+ ix  \\  \end{matrix}} & ;-y \end{pmatrix}\\
	&={_1}F_1(1;m+2;-y)+\frac{ (y\epsilon_{1A}+ix\epsilon_{1B})M_1}{(m+2+q)^2+x^2}
	\endaligned
	\end{equation}
	holds for all $x\in\R$ and $m>2y>0$.
	
	Thus
	\begin{equation}
	\aligned
	&\frac{\Gamma(m+2)}{y^{m+2}}B^{(-)}_2(x,y)\\
	=&-\frac{(q-ix)}{i{\color{red}{x}}(m+2+q+ix)}
	\left({_1}F_1(1,m+2;-y)+\frac{(y\epsilon_{1A}+i{\color{red}{x}}\epsilon_{1B}) M_1}{(m+2+q)^2+x^2}\right)\\
	\endaligned
	\end{equation}
	
	\begin{equation}
	\aligned
	&\frac{\Gamma(m+2)}{y^{m+2}}\Re B^{(-)}_2(x,y)\\
	&=\frac{(m+2+2q)}{\left((m+2+q)^2+x^2\right)}{_1}F_1(1;m+2;-y)
	+\frac{(m+2+2q)M_1}{\left((m+2+q)^2+x^2\right)}\\
	&\times \left[\frac{y\epsilon_{1A}}{((m+2+q)^2+x^2)}
	+\frac{(x^2-q(m+2+q)}{\left((m+2+q)^2+x^2\right)}\frac{\epsilon_{1B}}{(m+2+2q)}\right]\\
	&=\frac{(m+2+2q)}{\left((m+2+q)^2+x^2\right)}{_1}F_1(1;m+2;-y)
	+\frac{(m+2+2q)M_1}{\left((m+2+q)^2+x^2\right)}\left(\frac{R_1(m,x,y)}{m}\right)
	\endaligned
	\end{equation}
	
	Where
	
	\begin{equation}
	\aligned
	\frac{R_1(m,x,y)}{m}&:=\frac{y\epsilon_{1A}}{((m+2+q)^2+x^2)}
	+\frac{(x^2-q(m+2+q)}{\left((m+2+q)^2+x^2\right)}\frac{\epsilon_{1B}}{(m+2+2q)}\\
	\endaligned
	\end{equation}
	
	Since 
	\begin{equation}
	\aligned
	|R_1(m,x,y)|&\leqslant\frac{ym}{((m+2+q)^2+x^2)}
	+\frac{(x^2+q(m+2+q)}{\left((m+2+q)^2+x^2\right)}\frac{m}{(m+2+2q)}\\
	&< \frac{1}{2}+1<2\qquad \because m>2y
	\endaligned
	\end{equation}
	
	Consequently
	\begin{equation}
	\aligned
	&\frac{\Gamma(m+2)}{y^{m+2}}\Re B^{(-)}_2(x,y)\\
	&=\frac{(m+2+2q)}{\left((m+2+q)^2+x^2\right)}\left({_1}F_1(1;m+2;-y)+\frac{2\epsilon_{1C}M_1}{m}\right)
	\endaligned
	\end{equation}
	where $\epsilon_{1C}=\epsilon_{1C}(m,x,y)\in(-1,1)$.
	
	Set $\mu=-$ in \cref{Zhou1F1B}, we know that there exists a positive constant $M_2=16+\frac{216}{(e\log 2)^3}\approx 48.2919$ and real function $\epsilon_2=\epsilon_2(m,y)\in (-1,1)$ such that
	
	\begin{equation}
	\aligned
	{_1}F_1(1;m+2;-y)
	&=\left(1+\frac{y}{m}\right)^{-1}+\frac{\epsilon_2 M_2}{m}\\
	&=\frac{m}{m+y}+\frac{\epsilon_2 M_2}{m},
	\endaligned
	\end{equation}
	holds for all $m>2y>0$.
	
	\begin{equation}\label{Bp3y}
	\aligned
	&\frac{\Gamma(m+2)}{y^{m+2}}\Re B^{(-)}_2(x,y)\\
	&=\frac{(m+2+2q)}{\left((m+2+q)^2+x^2\right)}\left(\frac{m}{m+y}+\frac{\epsilon_2 M_2}{m}+\frac{2\epsilon_{1C} M_1}{}\right)\\
	&=\frac{(m+2+2q)}{\left((m+2+q)^2+x^2\right)}\left(\frac{m}{m+y}\right)\left(1+\frac{\epsilon_3 M_3}{m}\right)\\
	\endaligned
	\end{equation}
	where $\epsilon_3=\epsilon(m,x,y)\in (-1,1)$ and
	\begin{equation}
	\aligned
	M_3&=2M_1+M_2\\
	&=2\cdot \left(17+\frac{216}{(e\log 2)^3}\right)+\left(16+\frac{216}{(e\log 2)^3}\right)\\
	&=50+\frac{864}{(e\log 2)^3}=:M_4\\
	& \approx 293.374
	\endaligned
	\end{equation}
	
	\begin{equation}\label{Bp3y}
	\aligned
	\Re B^{(-)}_2(x,y)
	=\frac{y^{m+2}}{\Gamma(m+2)}\frac{(m+2+2q)}{\left((m+2+q)^2+x^2\right)}\left(\frac{m}{m+y}\right)\left(1+\frac{\epsilon_4 M_4}{m}\right)\\
	\endaligned
	\end{equation}
	
	Note $m=7n^3$. Among the $2n$ terms $|\Re B^{(-)}_2(x,\pi n k^2)|,k\in[1,n]$ and $|\Re B^{(-)}_2(x,\pi k^2/n)|,k\in[1,n]$, the largest one is $|\Re B^{(-)}_2(x,\pi n^3)|$.
	
	\begin{equation}\label{Ap3Ap4y}
	\aligned
	\Re B^{(-)}_2(x,\pi n^3)
	&=h^{(-)}_{4}(n,x)\left(1+\frac{\epsilon_{4}M_4}{7n^3}\right)\\
	\endaligned
	\end{equation}
	\begin{equation}\label{Ap3Ap4y}
	\aligned
	h^{(-)}_{4}(n,x)
	&:=\frac{(\pi n^3)^{7n^3+2}}{\Gamma(7n^3+2)}\frac{(7n^3+2+2q)}{\left((7n^3+2+q)^2+x^2\right)}\left(\frac{7n^3}{7n^3+\pi n^3}\right).\\
	\endaligned
	\end{equation}
	Thus we can use $h^{(-)}_{4}(n,x)$ to bound the other $2n-1$ terms.
	
	For all $k\in [1,n-1]$ we have
	\begin{equation}
	\aligned
	&\frac{\left|\Re B^{(-)}_2(x,\pi n k^2)\right|}{h^{(-)}_{4}(n,x)}\\
	&<\frac{(\pi n k^2)^{7n^3+2}}{(\pi n^3)^{7n^3+2}}\frac{(7n^3+\pi n^3)}{(7n^3+\pi n k^2)}\left(1+\frac{M_4}{7 n^3}\right)\\
	&<\left(\frac{n-1}{n}\right)^{14n^3+4}\frac{(7n^3+\pi n^3)}{(7n^3)}\left(1+M_4\right)\\
	&<2\left(\left(1-\frac{1}{n}\right)^n\right)^{14n^2}\left(1+M_4\right)\\
	&<(1+M_4)\exp(-14n^2)\\
	\endaligned
	\end{equation}
	And
	
	\begin{equation}
	\aligned
	\frac{\left|\sum_{k=1}^{n-1}\Re B^{(-)}_2(x,\pi n k^2)\right|}{h^{(-)}_{4}(n,x)}
	&\leqslant
	\frac{\sum_{k=1}^{n-1}\left|\Re B^{(-)}_2(x,\pi n k^2)\right|}{h^{(-)}_{4}(n,x)}\\
	&<2n\left(1+M_4\right)\exp(-14n^2)\\
	\endaligned
	\end{equation}
	
	For all $k\in [1,n]$ we also have
	\begin{equation}
	\aligned
	&\frac{\left|\Re B^{(-)}_2(x,\pi k^2/n)\right|}{h^{(-)}_{4}(n,x)}\\
	&<\frac{(\pi k^2/n)^{7n^3+2}}{(\pi n^3)^{7n^3+2}}\frac{(7n^3+\pi n^3)}{(7n^3+\pi k^2/n)}\left(1+M_4\right)\\
	&<n^{-14n^3-2}\frac{(7n^3+\pi n^3)}{(7n^3)}\left(1+M_4\right)\\
	&<2n^{-14n^3}\left(1+M_4\right)\\
	&=2\left(1+M_4\right)\exp(-14n^3\log n)\\
	&<2\left(1+M_4\right)\exp(-14n^2)\qquad \because \text{ for } n\geqslant 9,\log n>1
	\endaligned
	\end{equation}
	And
	
	\begin{equation}
	\aligned
	\frac{\left|\sum_{k=1}^{n-1}\Re B^{(-)}_2(x,\pi k^2/n)\right|}{h^{(-)}_{4}(n,x)}
	&\leqslant
	\frac{\sum_{k=1}^{n-1}\left|\Re B^{(-)}_2(x,\pi k^2/n)\right|}{h^{(-)}_{4}(n,x)}\\
	&<2n\left(1+M_4\right)\exp(-14n^2).\\
	\endaligned
	\end{equation}

	Thus we can write
	\begin{equation}
	\aligned
	&\sum_{1\leqslant k\leqslant n}\left(\Re B^{(-)}_2(x,\pi n k^2)+n^{-1/2}\Re B^{(-)}_2(x,\pi k^2/n)\right)\\
	&=h^{(-)}_{4}(n,x)
	\left(1+\frac{\epsilon_4M_4}{7n^3}+2n\left(1+M_4\right)\left(\epsilon_{5A}(n,x)+\epsilon_{5B}(n,x)\right)\exp(-14n^2)\right)\\
	\\
	\endaligned
	\end{equation}
	Where $\epsilon_{5A}(n,x)\in(-1,1),\epsilon_{5B}(n,x)\in(-1,1)$.
	
	We can combine the 3 error terms and define
	\begin{equation}
	\aligned
	\frac{R_5(n,x)}{7 n^3}:&=\frac{\epsilon_{4}M_4}{7 n^3}+n(1+M_4)\left(\epsilon_{5A}(n,x)+4n\epsilon_{5B}(n,x)\right)\exp(-14n^2)\\
	\endaligned
	\end{equation}
	
	Since 
	\begin{equation}
	\aligned
	R_5(n,x)&\leqslant M_4+7n^3\cdot 4n(1+M_4)\exp(-14n^2)\\
	&< M_4+28(1+M_4) n^4\exp(-14n^2)\\
	&\leqslant M_4+(1+M_4)\frac{29}{49e^2}\\
	&< M_4+1+M_4\\
	&= 2M_4+1=:M\\
	&=2\cdot 465+1=931
	\endaligned
	\end{equation}

	We have
	\begin{equation}
	\aligned
	&\sum_{1\leqslant k\leqslant n}\left(\Re B^{(-)}_2(x,\pi n k^2)+n^{-1/2}\Re B^{(-)}_2(x,\pi k^2/n)\right)\\
	&=h^{(-)}_4(n,x)\left(1+\frac{\epsilon M}{7 n^3}\right)\\
	&={\color{red}{+}}\frac{(\pi n ^3)^{7 n ^3+2}}{\Gamma(7 n ^3+2)}\frac{(7 n ^3+2+2q)}{\left((7 n ^3+2+q)^2+x^2\right)} \left(\frac{7 n ^3}{7 n ^3+\pi n ^3}\right)\left(1+\frac{\epsilon M}{7 n^3}\right).\\
	\endaligned
	\end{equation}

\end{proof}

\subsection[A(-1),B(-1)]{Bounds of $A^{(-)}_1(x,y)$ and $B^{(-)}_1(x,y)$}
\begin{theorem} \label{thmAm1Bound}
	 Let $q=1/4; n\in\N; x\geqslant 0; y>3(1+q)$. Let
	
\begin{equation}\label{Am1def0}
\aligned
A^{(-)}_1(x,y)
=&-y e^{-y}-y(q-ix)\frac{\gamma(1+q+ix,y)}{y^{1+q+ix}}\\
\endaligned
\end{equation}

	Then
	
\begin{equation}
\aligned
\left|\sum_{1\leqslant k\leqslant n}\left(\Re A^{(-)}_1(x,\pi n k^2)-n^{-1/2}\Re A^{(-)}_1(x,\pi k^2/n)\right)\right|
&<\frac{208n(\pi n^3)^4}{(1+q)^2+x^2}
\endaligned
\end{equation}

holds for all $n\in N$ and all $x\geqslant 0$.
	
\end{theorem}

\begin{proof}
From the recursive relation for the incomplete gamma function,
\begin{equation}
\gamma(s+1,y)=s\gamma(s,y)-y^s e^{-y}
\end{equation}

We obtain
\begin{equation}\label{gammaRecur1}
\aligned
\frac{\gamma(s,y)}{y^{s}}
&=\frac{e^{-y}}{s}
+\left(\frac{y}{s}\right)\frac{\gamma(s+1,y)}{y^{s+1}}\\
\endaligned
\end{equation}

The incomplete gamma function $\gamma(s,x)$ is related to Kummer function ${_1}F_1(1;1+s;x)$ via
\begin{equation}\label{KummerRelation1}
\gamma(s,y)=\frac{y^{s}e^{-y}}{s}{_1}F_1(1;1+s;y)
\end{equation}

Using ~\eqref{gammaRecur1} three times and ~\eqref{KummerRelation1} once we can convert ~\eqref{Am1def0} to

\begin{equation}
\aligned
-A^{(-)}_1(x,y)
&=y e^{-y}
+\frac{y(q-ix)e^{-y}}{(1+q+ix)}\\
&+\frac{y^2(q-ix)e^{-y}}{(1+q+ix)(2+q+ix)}\left(1+\frac{y}{(3+q+ix)}\right)\\
&+\frac{y^4(q-ix)}{(1+q+ix)(2+q+ix)(3+q+ix)}\\
&\times\frac{1}{(4+q+ix)}{_1}F_{1}(4+q+ix; 5+q+ix;-y)\\
\endaligned
\end{equation}

\begin{equation}\label{Bm1Bm2def2}
\aligned
-\Re A^{(-)}_1(x,y)
&=\frac{ye^{-y}h_0(x,y)}{2((1+q)^2+x^2)}+\frac{e^{-y}y^2 h_1(x,y)}{2((1+q)^2+x^2)}+\frac{y^4 h_2(x,y)}{2((1+q)^2+x^2)}\\
\endaligned
\end{equation}

Where
\begin{equation}\label{h0def}
\aligned
\frac{h_0(x,y)}{((1+q)^2+x^2)}&:=2
+\frac{(q-ix)}{(1+q+ix)}+\frac{(q+ix)}{(1+q-ix)}\\
\endaligned
\end{equation}

\begin{equation}\label{h1def}
\aligned
\frac{h_1(x,y)}{((1+q)^2+x^2)}
:&=\frac{(q-ix)}{(1+q+ix)(2+q+ix)}\left(1+\frac{y}{(3+q+ix)}\right)\\
&+\frac{(q+ix)}{(1+q-ix)(2+q-ix)}\left(1+\frac{y}{(3+q-ix)}\right)\\
\endaligned
\end{equation}

\begin{equation}\label{h2def}
\aligned
\frac{h_2(x,y)}{((1+q)^2+x^2)}
:&=\frac{(q-ix)}{(1+q+ix)(2+q+ix)(3+q+ix)}f(x,y)\\
&+\frac{(q+ix)}{(1+q-ix)(2+q-ix)(3+q-ix)}f(-x,y)\\
\endaligned
\end{equation}
where
\begin{equation}\label{fxydef}
\aligned
f(x,y):&=\frac{e^{-y}}{(4+q+ix)}{_1}F_1(1 ;5+q+ix;y)\\
&=\frac{1}{(4+q+ix)}{_1}F_1(4+q+ix ;5+q+ix;-y)\\
\endaligned
\end{equation}

\begin{equation}
\aligned
h_0(x,y)=2(1+q)(1+2q)=15/4<4\quad \because q=1/4\\
\endaligned
\end{equation}

Setting $q=1/4$ we obtain
\begin{equation}\label{Am1def}
\aligned
h_1(x,y)
:&=-\frac{15}{2}\left(1-\frac{84}{16x^2+81}\right)\\
&+2y\left(1+\frac{315}{16x^2+81}-\frac{819}{16x^2+169}\right)\\
\endaligned
\end{equation}

\begin{equation}
\aligned
|h_1(x,y)|
&\leqslant \frac{15}{2}\left(1+\frac{84}{16x^2+81}\right)\\
&+2y\left(1+\frac{315}{16x^2+81}+\frac{819}{16x^2+169}\right)\\
&\leqslant \frac{15}{2}\left(1+\frac{84}{81}\right)
+2y\left(1+\frac{315}{81}+\frac{819}{169}\right)\\
&< \frac{15}{2}\left(1+2\right)
+2y\left(1+4+5\right)\\
&<24(1+y)
\endaligned
\end{equation}

Set $b=4+q=17/4$ in \cref{sec6F11g}, we know that there exist  real functions $\epsilon_{1B}(x,y)\in [-1,1]$ and  $\epsilon_{2B}(x,y)\in [-1,1]$ such that for all $x\in\R$ and all $y>0$ we have
\begin{equation}
\aligned
	&\frac{e^{-y}}{(4+q+ix)}{_1}F_1(1 ;5+q+ ix;y)\\
	&=f(x,y)=\frac{\epsilon_1(x,y)}{(4+q)}+\frac{ix\epsilon_2(x,y)}{(4+q)^2}
\endaligned
\end{equation}

Thus
\begin{equation}
\aligned
f(x,y)&=\epsilon_{1}(x,y)+ix\epsilon_{2}(x,y)\\
\endaligned
\end{equation}
Where $\epsilon_{1}(x,y)=\frac{1}{4+q}\epsilon_{1B}(x,y)\in(-1,1)$ and $\epsilon_{2}(x,y)=\frac{1}{(4+q)^2}\epsilon_{2B}(x,y)\in(-1,1)$.

Thus
\begin{equation}\label{h2def}
\aligned
\frac{h_2(x,y)}{((1+q)^2+x^2)}
&=\frac{(q-ix)}{(1+q+ix)(2+q+ix)(3+q+ix)}\left(\epsilon_1(x,y)+ix\epsilon_2(x,y)\right)\\
&+\frac{(q+ix)}{(1+q-ix)(2+q-ix)(3+q-ix)}\left(\epsilon_1(x,y)-ix\epsilon_2(x,y)\right)\\
\endaligned
\end{equation}

Set $q=1/4$ we obtain
	
\begin{equation}
\aligned
h_2(x,y)
&=2\epsilon_{1}\left(1+\frac{315}{16x^2+81}-\frac{819}{16x^2+169}\right)\\
&-14\epsilon_{2}\left(1+\frac{405}{4(16x^2+81)}-\frac{1521}{4(16x^2+169)}\right)\\
\endaligned
\end{equation}

Thus
\begin{equation}
\aligned
|h_2(x,y)|
&\leqslant 2\left(1+\frac{315}{16x^2+81}+\frac{819}{16x^2+169}\right)\\
&+14\left(1+\frac{405}{4(16x^2+81)}+\frac{1521}{4(16x^2+169)}\right)\\
&\leqslant 2\left(1+\frac{315}{81}+\frac{819}{169}\right)
+14\left(1+\frac{405}{4\cdot 81}+\frac{1521}{4\cdot 169}\right)\\
&< 2\left(1+4+5\right)+14(1+2+3)\\
&=104
\endaligned
\end{equation}

\begin{equation}\label{Am1Am2def2}
\aligned
&2((1+q)^2+x^2)\left|\Re A^{(-)}_1(x,y)\right|\\
&\leqslant  ye^{-y}|h_0(x,y)|+y^2e^{-y} |h_1(x,y)|+y^4 |h_2(x,y)|\\
&\leqslant 4y e^{-y}+24(1+y)y^2e^{-y}+104y^4 \\
&= 4y e^{-y}+24y^2e^{-y}+24y^3e^{-y}+104y^4 \\
&\leqslant 2+6+3+104y^4  \qquad \because e>2\\
&<12+104y^4\\
&<104(1+y^4)
\endaligned
\end{equation}

\begin{equation}\label{Am1Am2def2}
\aligned
\left|\Re A^{(-)}_1(x,y)\right|
&<\frac{52(1+y^4)}{((1+q)^2+x^2)}\\
\endaligned
\end{equation}
	
Therefore
\begin{equation}
\aligned
&\left|\sum_{1\leqslant k\leqslant n}\left(\Re A^{(-)}_1(x,\pi n k^2)-n^{-1/2}\Re A^{(-)}_1(x,\pi k^2/n)\right)\right|\\	
&\leqslant 
\sum_{1\leqslant k\leqslant n}\left(\left|\Re A^{(-)}_1(x,\pi n k^2)\right|+\left|\Re A^{(-)}_1(x,\pi k^2/n)\right|\right)\\	
&\leqslant
\frac{52}{((1+q)^2+x^2)}\sum_{1\leqslant k\leqslant n}\left(1+(\pi n k^2)^4+1+(\pi k^2/n)^4\right)\\
&
<\frac{52}{((1+q)^2+x^2)}\sum_{1\leqslant k\leqslant n}\left(4(\pi n k^2)^4\right)\\
&<\frac{208n(\pi n^3)^4}{(1+q)^2+x^2}
\endaligned
\end{equation}

\end{proof}

\begin{theorem} \label{thmBm1Bound}
	 Let $q=1/4; n\in\N; x\geqslant 0; y>3(1+q)$. Let
	
	\begin{equation}\label{Bm1def}
	\aligned
B^{(-)}_1(x,y)&=\frac{y (q-ix)}{ix}\frac{\gamma(1+q+ix,y)}{y^{1+q+ix}}\\
	\endaligned
	\end{equation}

	Then
	
	\begin{equation}
\aligned
\left|\sum_{1\leqslant k\leqslant n}\left(\Re B^{(-)}_1(x,\pi n k^2)+n^{-1/2}\Re B^{(-)}_1(x,\pi k^2/n)\right)\right|
&<\frac{128n(\pi n^3)^4}{(1+q)^2+x^2}
\endaligned
\end{equation}
	
	holds for all $n\in N$ and all $x\geqslant 0$.
	
\end{theorem}

\begin{proof}
From the recursive relation for the incomplete gamma function,
\begin{equation}
\gamma(s+1,y)=s\gamma(s,y)-y^s e^{-y}
\end{equation}

We obtain
\begin{equation}\label{gammaRecur}
\aligned
\frac{\gamma(s,y)}{y^{s}}
&=\frac{e^{-y}}{s}
+\left(\frac{y}{s}\right)\frac{\gamma(s+1,y)}{y^{s+1}}\\
\endaligned
\end{equation}

The incomplete gamma function $\gamma(s,x)$ is related to Kummer function ${_1}F_1(1;1+s;x)$ via
\begin{equation}\label{KummerRelation}
\gamma(s,y)=\frac{y^{s}e^{-y}}{s}{_1}F_1(1;1+s;y)
\end{equation}

Using ~\eqref{gammaRecur} three times and ~\eqref{KummerRelation} once we can convert ~\eqref{Bm1def} to

	\begin{equation}
	\aligned
	&B^{(-)}_1(x,y)\\
	&=\frac{y(q-ix)e^{-y}}{ix(1+q+ix)}\\
	&+\frac{y^2(q-ix)e^{-y}}{ix(1+q+ix)(2+q+ix)}\left(1
	+\frac{y}{(3+q+ix)}\right)\\
	&+\frac{y^4(q-ix)}{ix(1+q+ix)(2+q+ix)(3+q+ix)}\frac{e^{-y}}{(4+q+ix)}{_1}F_1(1 ;5+q+ix;y)\\
	\endaligned
	\end{equation}

\begin{equation}\label{Bm2def2}
\aligned
\Re B^{(-)}_1(x,y)
&=\frac{ye^{-y}h_0(x,y)}{2((1+q)^2+x^2)}+\frac{y^2e^{-y} h_1(x,y)}{2((1+q)^2+x^2)}+\frac{y^4 h_2(x,y)}{2((1+q)^2+x^2)}\\
\endaligned
\end{equation}
	
	Where
	\begin{equation}\label{h0def}
	\aligned
	\frac{h_0(x,y)}{((1+q)^2+x^2)}&:=\frac{(q-ix)}{(+ix)(1+q+ix)}+\frac{(q+ix)}{(-ix)(1+q-ix)}\\
	\endaligned
	\end{equation}
	
	\begin{equation}\label{h1def}
	\aligned
	\frac{h_1(x,y)}{((1+q)^2+x^2)}
	:&=\frac{(q-ix)}{(+ix)(1+q+ix)(2+q+ix)}\left(1+\frac{y}{(3+q+ix)}\right)\\
	&+\frac{(q+ix)}{(-ix)(1+q-ix)(2+q-ix)}\left(1+\frac{y}{(3+q-ix)}\right)\\
	\endaligned
	\end{equation}
	
where	
	\begin{equation}\label{h2def}
	\aligned
	\frac{h_2(x,y)}{((1+q)^2+x^2)}
	:&=\frac{(q-ix)}{(+ix)(1+q+ix)(2+q+ix)(3+q+ix)}f(+x,y)\\
	&+\frac{(q+ix)}{(-ix)(1+q-ix)(2+q-ix)(3+q-ix)}f(-x,y)\\
	\endaligned
	\end{equation}
	
	\begin{equation}\label{fxydef}
	\aligned
	f(x,y):&=\frac{e^{-y}}{(4+q+ix)}{_1}F_1(1 ;5+q+ix;y)\\
	&=\frac{1}{(4+q+ix)}{_1}F_1(4+q+ix ;5+q+ix;-y)\\
	\endaligned
	\end{equation}
	
	\begin{equation}
	\aligned
	h_0(x,y)=-2(1+2q)=-3\\
	\endaligned
	\end{equation}
	
	Setting $q=1/4$ we obtain
	\begin{equation}
	\aligned
	h_1(x,y)
	&=2\left(1-\frac{140}{16x^2+81}\right)
	+8y\left(-\frac{35}{16x^2+81}+\frac{63}{16x^2+169}\right)\\
	\endaligned
	\end{equation}
	
	\begin{equation}
	\aligned
	|h_1(x,y)|
&\leqslant 2\left(1+\frac{140}{16x^2+81}\right)
+8y\left(\frac{35}{16x^2+81}+\frac{63}{16x^2+169}\right)\\
&\leqslant 2\left(1+\frac{140}{81}\right)
+8y\left(\frac{35}{81}+\frac{63}{169}\right)\\
	&<2(1+2)+8y(1+1)\\
	&=6+16y
	\endaligned
	\end{equation}

	Set $b=4+q=17/4$ in \cref{sec6F11g}, we know that there exist  real functions $\epsilon_{1B}(x,y)\in [-1,1]$ and  $\epsilon_{2B}(x,y)\in [-1,1]$ such that for all $x\in\R$ and all $y>0$ we have
	\begin{equation}
	\aligned
	&\frac{e^{-y}}{(4+q+ix)}{_1}F_1(1 ;5+q+ ix;y)\\
	&=f(x,y)=\frac{\epsilon_{1B}(x,y)}{(4+q)}+\frac{ix\epsilon_{2B}(x,y)}{(4+q)^2}
	\endaligned
	\end{equation}
	
	Thus
	\begin{equation}
	\aligned
	f(x,y)&=\epsilon_{1}(x,y)+ix\epsilon_{2}(x,y)
	\endaligned
	\end{equation}
	Where $\epsilon_{1}(x,y)=\frac{1}{4+q}\epsilon_{1B}(x,y)\in(-1,1)$ and $\epsilon_{2}(x,y)=\frac{1}{(4+q)^2}\epsilon_{2B}(x,y)\in(-1,1)$.
	
	Thus
	\begin{equation}\label{h2def}
	\aligned
	\frac{h_2(x,y)}{((1+q)^2+x^2)}
	&=\frac{(q-ix)}{ix(1+q+ix)(2+q+ix)(3+q+ix)}\left(\epsilon_1(x,y)+ix\epsilon_2(x,y)\right)\\
	&-\frac{(q+ix)}{ix(1+q-ix)(2+q-ix)(3+q-ix)}\left(\epsilon_1(x,y)-ix\epsilon_2(x,y)\right)\\
	\endaligned
	\end{equation}
	
	Set $q=1/4$ we obtain
	
	\begin{equation}
	\aligned
	h_2(x,y)
	&=8\epsilon_{1B}\left(-\frac{35}{16x^2+81}+\frac{63}{16x^2+169}\right)\\
	&+2\epsilon_{2B}\left(1+\frac{315}{16x^2+81}-\frac{819}{16x^2+169}\right)\\
	\endaligned
	\end{equation}
	
	Thus
	\begin{equation}
	\aligned
h_2(x,y)&\leqslant 
8\left(\frac{35}{16x^2+81}+\frac{63}{16x^2+169}\right)\\
&+2\left(1+\frac{315}{16x^2+81}+\frac{819}{16x^2+169}\right)\\
&\leqslant 
8\left(\frac{35}{81}+\frac{63}{169}\right)
+2\left(1+\frac{315}{81}+\frac{819}{169}\right)\\
&<8(1+1)+2(1+4+5)\\
&=36
	\endaligned
	\end{equation}
	
	\begin{equation}
	\aligned
	&2((1+q)^2+x^2)\left|B^{(-)}_1(x,y)\right|\\
	&\leqslant  ye^{-y}|h_0(x,y)|+y^2e^{-y} |h_1(x,y)|+y^4 |h_2(x,y)|\\
	&\leqslant 3 e^{-y}+(6+16y)y^2e^{-y}+36y^4 \\
	&= 3 e^{-y}+6y^2e^{-y}+16y^3e^{-y}+36y^4 \\
	&\leqslant 3 +6\cdot 4e^{-2}+16\cdot 27e^{-3}+36y^4 \\
	&< 3+6+ 54+36y^4\qquad \because e>2\\
	&<64+36y^4.\\
	&<64(1+y^4).\\
	\endaligned
	\end{equation}
	
\begin{equation}
\aligned
	\left|\Re B^{(-)}_1(x,y)\right|
	&<\frac{32(1+y^4)}{((1+q)^2+x^2)}.\\
	\endaligned
\end{equation}
	
	Therefore
	\begin{equation}
	\aligned
	&\left|\sum_{1\leqslant k\leqslant n}\left(\Re B^{(-)}_1(x,\pi n k^2)+n^{-1/2}\Re B^{(-)}_1(x,\pi k^2/n)\right)\right|\\	
	&\leqslant 
	\sum_{1\leqslant k\leqslant n}\left(\left|\Re B^{(-)}_1(x,\pi n k^2)\right|+\left|\Re B^{(-)}_1(x,\pi k^2/n)\right|\right)\\	
	&\leqslant
	\frac{32}{((1+q)^2+x^2)}\sum_{1\leqslant k\leqslant n}\left(1+(\pi n k^2)^4+1+(\pi k^2/n)^4\right)\\
	&\leqslant
	\frac{32}{((1+q)^2+x^2)}\sum_{1\leqslant k\leqslant n}4(\pi n k^2)^4\\
	&<\frac{128n(\pi n^3)^4}{(1+q)^2+x^2}
	\endaligned
	\end{equation}
	
\end{proof}

\begin{theorem} \label{FpmGpm}
	
	Let $a=1,2; q=1/4; n\in2\N_0+9;m=7n^3;x\in\R
	;y>0$. Let
	\begin{equation}\label{alphapmdef}
	\aligned
	F^{(\pm)}(n,x)
	&:=\sum_{1\leqslant k \leqslant n}\left(\alpha^{(\pm)}(x,\pi n k^2)-n^{-1/2}\alpha^{(\pm)}(x,\pi k^2/n)\right)\\
	\endaligned
	\end{equation}
	\begin{equation}\label{alphapmdef}
\aligned
G^{(\pm)}(n,x)
&:=\sum_{1\leqslant k \leqslant n}\left(\beta^{(\pm)}(x,\pi n k^2)+n^{-1/2}\beta^{(\pm)}(x,\pi k^2/n)\right)\\
\endaligned
\end{equation}	
	\begin{equation}\label{betapmdef}
	\aligned
	\beta^{(\pm)}_{a}(x,y)
	&:=(1/2)(\beta_{2}(x,y)\pm \beta_{1}(x,y)),\\
	\beta_{a}(x,y)
	&:=\sum_{0\leqslant j \leqslant m}\frac{(2(2j+a)+1)y^{2j+a}}{\Gamma(2j+a)}
	\frac{1}{x^2+(2j+a+q)^2}\\
	\endaligned
	\end{equation}

\noindent Then there exists a constant $M_6$ and a real function $\epsilon_6(n,x)\in (-1,1)$ such that for all $n\in2\N_0+9$ and all $x\geqslant 0$ we have
\begin{equation}
\aligned
F^{(+)}(n,x)
&=\exp(\pi n^3)\left(\frac{(\pi n^3+2+q)^3}{(\pi n^3+2+q)^2+x^2}\right)\left(1+\frac{\epsilon_{8A}(n,x)M_{8A}}{n^3}\right)\\
G^{(+)}(n,x)
&=\exp(\pi n^3)\left(\frac{(\pi n^3+2+q)^2}{(\pi n^3+2+q)^2+x^2}\right)\left(1+\frac{\epsilon_{8B}(n,x)M_{8B}}{n^3}\right)\\
\endaligned
\end{equation}

	\begin{equation}
\aligned
F^{(-)}(n,x)&=\frac{(\pi n ^3)^{7n^3+2}}{\Gamma(7n^3+2)}\frac{(7n^3+2+q)(7n^3+2+2q)}{\left((7n^3+2+q)^2+x^2\right)}\\ &\times\left(\frac{7n^3}{7n^3+\pi n^3}\right)\left(1+\frac{\epsilon_{3A}(n,x) M_{3A}}{n^3}\right)\\
G^{(-)}(n,x)&=\frac{(\pi n ^3)^{7n^3+2}}{\Gamma(7n^3+2)}\frac{(7n^3+2+2q)}{\left((7n^3+2+q)^2+x^2\right)}\\ &\times\left(\frac{7n^3}{7n^3+\pi n^3}\right)\left(1+\frac{\epsilon_{3B}(n,x) M_{3B}}{n^3}\right)\\
\endaligned
\end{equation}

\end{theorem}

\begin{proof}
\noindent (A) Using \cref{thmAlphaBetaInAB} we obtain
\begin{equation}
\aligned
F^{(+)}(n,x)&=\sum_{1\leqslant l \leqslant 2,}\sum_{1\leqslant k\leqslant n}\left(\Re A^{(+)}_l(x,\pi n k^2)-n^{-1/2}\Re A^{(+)}_l(x,\pi k^2/n)\right).\\
\endaligned
\end{equation}
Further using \cref{thmAp1Bound}, and \cref{thmAp2Bound}, we know that there exists constants $M_{6A},M_{7A}$ and a real function $\epsilon_{6A}(n,x)\in (-1,1),\epsilon_{7A}(n,x)\in (-1,1)$ such that for $n\in2\N_0+9$ and all $x\geqslant 0$ we have
\begin{equation}
\aligned
F^{(+)}(n,x)
&=\exp(\pi n^3)\left(\frac{(\pi n^3+2+q)^3}{(\pi n^3+2+q)^2+x^2}\right)\left(1+\frac{\epsilon_{6A}(n,x)M_{6A}}{\pi n^3}\right)\\
&-\frac{(\pi n ^3)^{7 n^3+2}}{\Gamma(7 n^3+2)}\frac{(7 n^3+2+q)(7 n^3+2+2q)}{\left((7 n^3+2+q)^2+x^2\right)} \\
&\times\left(\frac{7 n^3}{7 n^3-\pi n^3}\right)\left(1+\frac{\epsilon_{7A}(n,x) M_{7A}}{7 n^3}\right)\\
&=\exp(\pi n^3)\left(\frac{(\pi n^3+2+q)^3}{(\pi n^3+2+q)^2+x^2}\right)\left(1+\frac{R_{8A}(n,x)}{n^3}\right)\\
\endaligned
\end{equation}

where

\begin{equation}
\aligned
\frac{R_{8A}(n,x)}{n^3}
:&=\frac{\epsilon_{6A}(n,x)M_{6A}}{\pi n^3}\\
&-\left(\exp(\pi n^3)\left(\frac{(\pi  n^3+2+q)^3}{(\pi n^3+2+q)^2+x^2}\right)\right)^{-1}\\
&\times\frac{(\pi n ^3)^{7 n^3+2}}{\Gamma(7 n^3+2)}\frac{(7 n^3+2+q)(7 n^3+2+2q)}{\left((7 n^3+2+q)^2+x^2\right)} \\
&\times \left(\frac{7 n^3}{7 n^3-\pi n^3}\right)\left(1+\frac{\epsilon_{7A}(n,x) M_{7A}}{7 n^3}\right)\\
\endaligned
\end{equation}
Since
\begin{equation}
\aligned
R_{8A}(n,x)
&\leqslant M_{6A}\\
&+\frac{n^3(7 n^3+2+q)(7 n^3+2+2q)}{(\pi  n^3+2+q)^3}\\
&\times \frac{(\pi n^3+2+q)^2+x^2}{(7 n^3+2+q)^2+x^2} \\
&\times \left(\frac{7}{7-\pi}\right)\left(1+M_{7A}\right)\\
&\times \exp(-\pi n^3)\frac{(\pi n ^3)^{7 n^3+2}}{\Gamma(7 n^3+2)}\\
&\leqslant M_{6A}
+\frac{7^2}{3^3}\left(\frac{7}{3}\right)\left(1+M_{7A}\right)\quad \because 3<\pi<4\\
&\times \exp(-\pi n^3)\frac{(\pi n ^3)^{7 n^3+2}}{\Gamma(7 n^3+2)}\\
\endaligned
\end{equation}

Substituting the upper bound of $\frac{(\pi n ^3)^{7 n^3+2}}{\Gamma(7 n^3+2)}$ from \cref{GammaUpperLowerBound} we obtain

\begin{equation}
\aligned
R_{8A}(n,x)
&\leqslant M_{6A}
+\frac{7^3}{3^4}(1+M_{7A})
\exp(-\pi n^3)\frac{\pi^{3/2} e}{7\sqrt{14}} \left(n^{3/2}\exp\left(7n^3\log(\pi e/7)\right)\right)\\
&=M_{6A}
+\frac{7^3}{3^4}(1+M_{7A})
\frac{\pi^{3/2} e}{7\sqrt{14}} \left(n^{3/2}\exp\left(-n^3\left(\pi-7\log(\pi e/7)\right)\right)\right)\\
&\leqslant M_{6A}
+\frac{7^3}{3^4}(1+M_{7A})
\frac{\pi^{3/2} e}{7\sqrt{14}} \left(n^{3}\exp\left(-n^3\left(\pi-7\log(\pi e/7)\right)\right)\right)\\
&\leqslant M_{6A}
+\frac{7^3}{3^4}(1+M_{7A})
\frac{\pi^{3/2} e}{7\sqrt{14}} \left(\frac{1}{e\left(\pi-7\log(\pi e/7)\right)}\right)\\
&= M_{6A}
+\frac{7^2}{3^4}(1+M_{7A})
\frac{\pi^{3/2}}{\sqrt{14}} \left(\frac{1}{\left(\pi-7\log(\pi e/7)\right)}\right)\\
&=:M_{8A}
\endaligned
\end{equation}

Thus we have
\begin{equation}
\aligned
F^{(+)}(n,x)
&=\exp(\pi n^3)\left(\frac{(\pi n^3+2+q)^3}{(\pi n^3+2+q)^2+x^2}\right)\left(1+\frac{\epsilon_{8A}(n,x)M_{8A}}{n^3}\right)\\
\endaligned
\end{equation}
where $\epsilon_{8A}(n,x)\in(-1,1)$.

\noindent (B) Similarly
Using \cref{thmAlphaBetaInAB} we obtain
\begin{equation}
\aligned
G^{(+)}(n,x)&=\sum_{1\leqslant l \leqslant 2,}\sum_{1\leqslant k\leqslant n}\left(\Re B^{(+)}_l(x,\pi n k^2)+n^{-1/2}\Re B^{(+)}_l(x,\pi k^2/n)\right).\\
\endaligned
\end{equation}
Further using \cref{thmBp1Bound}, and \cref{thmBp2Bound}, we know that there exists constants $M_{6B},M_{7B}$ and a real functions $\epsilon_{6B}(n,x)\in (-1,1),\epsilon_{7B}(n,x)\in (-1,1)$ such that for $n\in2\N_0+9$ and all $x\geqslant 0$ we have
\begin{equation}
\aligned
G^{(+)}(n,x)
&=\exp(\pi n^3)\left(\frac{(\pi n^3+2+q)^2}{(\pi n^3+2+q)^2+x^2}\right)\left(1+\frac{\epsilon_{6B}(n,x)M_{6B}}{\pi n^3}\right)\\
&-\frac{(\pi n ^3)^{7 n^3+2}}{\Gamma(7 n^3+2)}\frac{(7 n^3+2+2q)}{\left((7 n^3+2+q)^2+x^2\right)} \\
&\times\left(\frac{7 n^3}{7 n^3-\pi n^3}\right)\left(1+\frac{\epsilon_{7B}(n,x) M_{7B}}{7 n^3}\right)\\
&=\exp(\pi n^3)\left(\frac{(\pi n^3+2+q)^3}{(\pi n^3+2+q)^2+x^2}\right)\left(1+\frac{R_{8B}(n,x)}{n^3}\right)\\
\endaligned
\end{equation}

where

\begin{equation}
\aligned
\frac{R_{8B}(n,x)}{n^3}
:&=\frac{\epsilon_{6B}(n,x)M_{6B}}{\pi n^3}\\
&-\left(\exp(\pi n^3)\left(\frac{(\pi  n^3+2+q)^2}{(\pi n^3+2+q)^2+x^2}\right)\right)^{-1}\\
&\times\frac{(\pi n ^3)^{7 n^3+2}}{\Gamma(7 n^3+2)}\frac{(7 n^3+2+2q)}{\left((7 n^3+2+q)^2+x^2\right)} \\
&\times \left(\frac{7 n^3}{7 n^3-\pi n^3}\right)\left(1+\frac{\epsilon_{7B}(n,x) M_{7B}}{7 n^3}\right)\\
\endaligned
\end{equation}
Since
\begin{equation}
\aligned
R_{8B}(n,x)
&\leqslant M_{6B}\\
&+\frac{n^3(7 n^3+2+2q)}{(\pi  n^3+2+q)^2}\\
&\times \frac{(\pi n^3+2+q)^2+x^2}{(7 n^3+2+q)^2+x^2} \\
&\times \left(\frac{7}{7-\pi}\right)\left(1+M_{7B}\right)\\
&\times \exp(-\pi n^3)\frac{(\pi n ^3)^{7 n^3+2}}{\Gamma(7 n^3+2)}\\
&\leqslant M_{6B}
+\frac{7}{3^2}\left(\frac{7}{3}\right)\left(1+M_{7B}\right)\quad \because 3<\pi<4\\
&\times \exp(-\pi n^3)\frac{(\pi n ^3)^{7 n^3+2}}{\Gamma(7 n^3+2)}\\
\endaligned
\end{equation}

Substituting the upper bound of $\frac{(\pi n ^3)^{7 n^3+2}}{\Gamma(7 n^3+2)}$ from \cref{GammaUpperLowerBound} we obtain

\begin{equation}
\aligned
R_{8B}(n,x)
&\leqslant M_{6B}
+\frac{7^2}{3^3}(1+M_{7B})
\exp(-\pi n^3)\frac{\pi^{3/2} e}{7\sqrt{14}} \left(n^{3/2}\exp\left(7n^3\log(\pi e/7)\right)\right)\\
&=M_{6B}
+\frac{7^2}{3^3}(1+M_{7B})
\frac{\pi^{3/2} e}{7\sqrt{14}} \left(n^{3/2}\exp\left(-n^3\left(\pi-7\log(\pi e/7)\right)\right)\right)\\
&\leqslant M_{6B}
+\frac{7^2}{3^3}(1+M_{7B})
\frac{\pi^{3/2} e}{7\sqrt{14}} \left(n^{3}\exp\left(-n^3\left(\pi-7\log(\pi e/7)\right)\right)\right)\\
&\leqslant M_{6B}
+\frac{7^2}{3^3}(1+M_{7B})
\frac{\pi^{3/2} e}{7\sqrt{14}} \left(\frac{1}{e\left(\pi-7\log(\pi e/7)\right)}\right)\\
&= M_{6B}
+\frac{7}{3^3}(1+M_{7B})
\frac{\pi^{3/2}}{\sqrt{14}} \left(\frac{1}{\left(\pi-7\log(\pi e/7)\right)}\right)\\
&=:M_{8B}
\endaligned
\end{equation}

Thus we have
\begin{equation}
\aligned
G^{(+)}(n,x)
&=\exp(\pi n^3)\left(\frac{(\pi n^3+2+q)^2}{(\pi n^3+2+q)^2+x^2}\right)\left(1+\frac{\epsilon_{8B}(n,x)M_{8B}}{n^3}\right)\\
\endaligned
\end{equation}
where $\epsilon_{8B}(n,x)\in(-1,1)$.\\

\noindent (C) Using \cref{thmAlphaBetaInAB} we obtain
\begin{equation}
\aligned
F^{(-)}(n,x)&=\sum_{1\leqslant l \leqslant 2,}\sum_{1\leqslant k\leqslant n}\left(\Re A^{(-)}_l(x,\pi n k^2)-n^{-1/2}\Re A^{(-)}_l(x,\pi k^2/n)\right).\\
\endaligned
\end{equation}
Further using \cref{thmAm1Bound}, and \cref{thmAm2Bound}, we know that there exists constants $M_{1A},M_{2A}$ and a real function $\epsilon_{1A}(n,x)\in (-1,1),\epsilon_{2A}(n,x)\in (-1,1)$ such that for $n\in2\N_0+9$ and all $x\geqslant 0$ we have
\begin{equation}
\aligned
F^{(-)}(n,x)
&=\epsilon_{1A}\frac{208n(\pi n^3)^4}{(1+q)^2+x^2}\\
&+\frac{(\pi n ^3)^{7 n^3+2}}{\Gamma(7 n^3+2)}\frac{(7 n^3+2+q)(7 n^3+2+2q)}{\left((7 n^3+2+q)^2+x^2\right)} \left(\frac{7 n^3}{7 n^3+\pi n^3}\right)\\
&\times\left(1+\frac{\epsilon_{2A}(n,x) M_{2A}}{7 n^3}\right)\\
&=\frac{(\pi n ^3)^{7 n^3+2}}{\Gamma(7 n^3+2)}\frac{(7 n^3+2+q)(7 n^3+2+2q)}{\left((7 n^3+2+q)^2+x^2\right)} \left(\frac{7 n^3}{7 n^3+\pi n^3}\right)\\
&\times\left(1+\frac{R_{3A}(n,x)}{n^3}\right)\\
\endaligned
\end{equation}

where

\begin{equation}
\aligned
\frac{R_{3A}(n,x)}{n^3}
:&=\frac{\epsilon_{2A}(n,x) M_{2A}}{7 n^3}
+\epsilon_{1A}\frac{208n(\pi n^3)^4}{(1+q)^2+x^2}\\
&\times\left(\frac{(\pi n ^3)^{7 n^3+2}}{\Gamma(7 n^3+2)}\right)^{-1}\frac{\left((7 n^3+2+q)^2+x^2\right)}{(7 n^3+2+q)(7 n^3+2+2q)} \left(\frac{7 n^3+\pi n^3}{7 n^3}\right)\\
\endaligned
\end{equation}
Since
\begin{equation}
\aligned
|R_{3A}(n,x)|
&\leqslant M_{2A}
+\left(\frac{(\pi n ^3)^{7 n^3+2}}{\Gamma(7 n^3+2)}\right)^{-1}\\
&\times\left(\frac{208n(\pi n^3)^4}{(1+q)^2+x^2}\right)\frac{n^3\left((7 n^3+2+q)^2+x^2\right)}{(7 n^3+2+q)(7 n^3+2+2q)} \left(\frac{7 n^3+\pi n^3}{7 n^3}\right)\\
&= M_{2A}
+\left(\frac{(\pi n ^3)^{7 n^3+2}}{\Gamma(7 n^3+2)}\right)^{-1}\\
&\times\left(\frac{208n(\pi n^3)^4n^3}{(7 n^3+2+q)(7 n^3+2+2q)}\right)\frac{\left((7 n^3+2+q)^2+x^2\right)}{((1+q)^2+x^2)} \left(\frac{7+\pi}{7}\right)\\
&< M_{2A}
+\left(\frac{(\pi n ^3)^{7 n^3+2}}{\Gamma(7 n^3+2)}\right)^{-1}\\
&\times\left(\frac{208n(\pi n^3)^4n^3}{(7 n^3)^2}\right)\frac{\left((7 n^3+2+q)^2\right)}{((1+q)^2)} \left(\frac{7+\pi}{7}\right)\\
&< M_{2A}
+\left(\frac{(\pi n ^3)^{7 n^3+2}}{\Gamma(7 n^3+2)}\right)^{-1}
\left(208n(\pi n^3)^2n^3\right)\left((7 n^3+2)^2\right) 2\quad \because \pi<7\\
&< M_{2A}
+\left(\frac{(\pi n ^3)^{7 n^3+2}}{\Gamma(7 n^3+2)}\right)^{-1}
208\pi^2\cdot 128 n^{16}\\
\endaligned
\end{equation}

Substituting the lower bound of $\frac{(\pi n ^3)^{7 n^3+2}}{\Gamma(7 n^3+2)}$ from \cref{GammaUpperLowerBound} we obtain

\begin{equation}
\aligned
|R_{3A}(n,x)|
&< M_{2A}
+208\pi^2\cdot 128 n^{16}\\
&\times\left(\left(\frac{65}{84e}\right)\frac{\pi^{3/2} e}{7\sqrt{14}} \right)^{-1}\left(n^{-3/2}\exp\left(-7n^3\log(\pi e/7)\right)\right)\\
&< M_{2A}
+208\pi^2\cdot 128\\
&\times\left(\left(\frac{84e}{65}\right)\frac{7\sqrt{14}}{\pi^{3/2} e} \right)\left(n^{15}\exp\left(-7n^3\log(\pi e/7)\right)\right)\\
&< M_{2A}
+208\pi^2\cdot 128\left(\frac{84e}{65}\right)\frac{7\sqrt{14}}{\pi^{3/2} e} \left(\frac{5}{7e\log(\pi e/7)}\right)^5\\
&=:M_{3A}
\endaligned
\end{equation}

Thus we have
\begin{equation}
\aligned
F^{(-)}(n,x)
&=\frac{(\pi n ^3)^{7 n^3+2}}{\Gamma(7 n^3+2)}\frac{(7 n^3+2+q)(7 n^3+2+2q)}{\left((7 n^3+2+q)^2+x^2\right)} \left(\frac{7 n^3}{7 n^3+\pi n^3}\right)\\
&\times\left(1+\frac{\epsilon_{3A}(n,x) M_{3A}}{ n^3}\right)\\
\endaligned
\end{equation}
where $\epsilon_{3A}(n,x)\in(-1,1)$.

\noindent (D) Using \cref{thmAlphaBetaInAB} we obtain
\begin{equation}
\aligned
G^{(-)}(n,x)&=\sum_{1\leqslant l \leqslant 2,}\sum_{1\leqslant k\leqslant n}\left(\Re B^{(-)}_l(x,\pi n k^2)+n^{-1/2}\Re B^{(-)}_l(x,\pi k^2/n)\right).\\
\endaligned
\end{equation}
Further using \cref{thmBm1Bound}, and \cref{thmBm2Bound}, we know that there exists constants $M_{1B},M_{2B}$ and a real function $\epsilon_{1B}(n,x)\in (-1,1),\epsilon_{2B}(n,x)\in (-1,1)$ such that for $n\in2\N_0+9$ and all $x\geqslant 0$ we have
\begin{equation}
\aligned
G^{(-)}(n,x)
&=\epsilon_{1B}\frac{128n(\pi n^3)^4}{(1+q)^2+x^2}\\
&+\frac{(\pi n ^3)^{7 n^3+2}}{\Gamma(7 n^3+2)}\frac{(7 n^3+2+2q)}{\left((7 n^3+2+q)^2+x^2\right)} \left(\frac{7 n^3}{7 n^3+\pi n^3}\right)\\
&\times\left(1+\frac{\epsilon_{2B}(n,x) M_{2B}}{7 n^3}\right)\\
&=\frac{(\pi n ^3)^{7 n^3+2}}{\Gamma(7 n^3+2)}\frac{(7 n^3+2+2q)}{\left((7 n^3+2+q)^2+x^2\right)} \left(\frac{7 n^3}{7 n^3+\pi n^3}\right)\\
&\times\left(1+\frac{R_{3B}(n,x)}{n^3}\right)\\
\endaligned
\end{equation}

where

\begin{equation}
\aligned
\frac{R_{3B}(n,x)}{n^3}
:&=\frac{\epsilon_{2B}(n,x) M_{2B}}{7 n^3}
+\epsilon_{1B}\frac{128n(\pi n^3)^4}{(1+q)^2+x^2}\\
&\times\left(\frac{(\pi n ^3)^{7 n^3+2}}{\Gamma(7 n^3+2)}\right)^{-1}\frac{\left((7 n^3+2+q)^2+x^2\right)}{(7 n^3+2+2q)} \left(\frac{7 n^3+\pi n^3}{7 n^3}\right)\\
\endaligned
\end{equation}
Since
\begin{equation}
\aligned
|R_{3B}(n,x)|
&\leqslant M_{2B}
+\left(\frac{(\pi n ^3)^{7 n^3+2}}{\Gamma(7 n^3+2)}\right)^{-1}\\
&\times\left(\frac{128n(\pi n^3)^4}{(1+q)^2+x^2}\right)\frac{n^3\left((7 n^3+2+q)^2+x^2\right)}{(7 n^3+2+2q)} \left(\frac{7 n^3+\pi n^3}{7 n^3}\right)\\
&= M_{2B}
+\left(\frac{(\pi n ^3)^{7 n^3+2}}{\Gamma(7 n^3+2)}\right)^{-1}\\
&\times\left(\frac{128n(\pi n^3)^4n^3}{(7 n^3+2+2q)}\right)\frac{\left((7 n^3+2+q)^2+x^2\right)}{((1+q)^2+x^2)} \left(\frac{7+\pi}{7}\right)\\
&< M_{2B}
+\left(\frac{(\pi n ^3)^{7 n^3+2}}{\Gamma(7 n^3+2)}\right)^{-1}\\
&\times\left(\frac{128n(\pi n^3)^4n^3}{7 n^3}\right)\frac{\left((7 n^3+2+q)^2\right)}{((1+q)^2)} \left(\frac{7+\pi}{7}\right)\\
&< M_{2B}
+\left(\frac{(\pi n ^3)^{7 n^3+2}}{\Gamma(7 n^3+2)}\right)^{-1}
\left(128n(\pi n^3)^3n^3\right)\left((7 n^3+2)^2\right) 2\quad \because \pi<7\\
&< M_{2B}
+\left(\frac{(\pi n ^3)^{7 n^3+2}}{\Gamma(7 n^3+2)}\right)^{-1}
128\pi^3\cdot 128 n^{19}\\
\endaligned
\end{equation}

Substituting the lower bound of $\frac{(\pi n ^3)^{7 n^3+2}}{\Gamma(7 n^3+2)}$ from \cref{GammaUpperLowerBound} we obtain

\begin{equation}
\aligned
|R_{3B}(n,x)|
&< M_{2B}
+128\pi^3\cdot 128 n^{19}\\
&\times\left(\left(\frac{65}{84e}\right)\frac{\pi^{3/2} e}{7\sqrt{14}} \right)^{-1}\left(n^{-3/2}\exp\left(-7n^3\log(\pi e/7)\right)\right)\\
&< M_{2B}
+128\pi^3\cdot 128\\
&\times\left(\left(\frac{84e}{65}\right)\frac{7\sqrt{14}}{\pi^{3/2} e} \right)\left(n^{18}\exp\left(-7n^3\log(\pi e/7)\right)\right)\\
&< M_{2B}
+128\pi^3\cdot 128\left(\frac{84e}{65}\right)\frac{7\sqrt{14}}{\pi^{3/2} e} \left(\frac{6}{7e\log(\pi e/7)}\right)^6\\
&=:M_{3B}
\endaligned
\end{equation}

Thus we have
\begin{equation}
\aligned
G^{(-)}(n,x)
&=\frac{(\pi n ^3)^{7 n^3+2}}{\Gamma(7 n^3+2)}\frac{(7 n^3+2+2q)}{\left((7 n^3+2+q)^2+x^2\right)} \left(\frac{7 n^3}{7 n^3+\pi n^3}\right)\\
&\times\left(1+\frac{\epsilon_{3B}(n,x) M_{3B}}{ n^3}\right)\\
\endaligned
\end{equation}
where $\epsilon_{3B}(n,x)\in(-1,1)$.
\end{proof}

\section{\textbf{$u_2(x)>u_1(x)$, $v_2(x)>v_1(x)$, $w_2(x)>w_1(x)$}}
\begin{theorem}[(new result)=~\cref{w2gtw1prop}]\label{w3gtw2gtw1}
	
	Let $n\in 2\N_0+9; m=7n^3;x\geqslant 0$.  Let
	\begin{equation}\label{bjcjdef5}
	\aligned
	\varphi_{n,j}&=(j+1/4)\log n,\\
	S_{n,j}&=\sum_{k=1}^{n}k^{j},\\
	c_{n,j}&=  \log n\frac{(2j+1)\pi^j}{\Gamma(j)} S_{n,2j}.\\
	\endaligned
	\end{equation}
	
	\begin{equation}\label{uvwdef5}
	\aligned
	w_{a}(n,x)&:=u_{a}(n,x)/v_{a}(n,x),\quad a=1,2\\
	u_{a}(n,x)&:=\sum_{0\leqslant j \leqslant m}c_{n,2j+a}\frac{\varphi_{n,2j+a}\sinh \left(\varphi_{n,2j+a}\right)}{x^2+\varphi_{n,2j+a}^2},&\quad a=1,2\\
	v_{a}(n,x)&:=\sum_{0\leqslant j \leqslant m}c_{n,2j+a}\frac{\cosh \left(\varphi_{n,2j+a}\right)}{x^2+\varphi_{n,2j+a}^2},&\quad a=1,2\\
	\endaligned
	\end{equation}
	
	\noindent then exists a sufficiently large natural number $N$ such that

	\noindent(A1)
	\begin{equation}
	\aligned
	u_2(n,x)>u_1(n,x),\quad n> N,\\
	\endaligned
	\end{equation}
	
	\noindent(B1)
	\begin{equation}
	\aligned
	v_2(n,x)>v_1(n,x),\quad n> N,\\
	\endaligned
	\end{equation}
	
	\noindent(C1)
\begin{equation}
\aligned
w_2(n,x)>w_1(n,x),\quad n> N.\\
\endaligned
\end{equation}
	
\end{theorem}

\begin{proof}
	
	Let 
	\begin{equation}
	\aligned
	u^{(\pm)}(n,x)&=(1/2)(u_{2}(n,x)\pm u_{1}(n,x))\\
	v^{(\pm)}(n,x)&=(1/2)(v_{2}(n,x)\pm v_{1}(n,x))\\
	\endaligned
	\end{equation}

Then
\begin{equation}
\aligned
&w_{2}(n,x)-w_1(n,x)\\
&=\frac{u_{2}(n,x)v_{1}(n,x)-u_{1}(n,x)v_{2}(n,x)}{v_{1}(n,x)v_{2}(n,x)}\\
&=\frac{2}{v_{1}(n,x)v_{2}(n,x)}\left(u^{(-)}(n,x)v^{(+)}(n,x)-u^{(+)}(n,x)v^{(-)}(n,x)\right)\\
\endaligned
\end{equation}
	
	Because $v_{1}(n,x)>0,v_{2}(n,x)>0$, claims (A1), (B1), and (C1) are then equivalent to\\
\noindent (A2)
	\begin{equation}
	\aligned
	u^{(-)}(n,x)>0,\quad n> N.
	\endaligned
	\end{equation}
\noindent (B2)	
	\begin{equation}
	\aligned
	v^{(-)}(n,x)>0,\quad n> N.
	\endaligned
	\end{equation}
	
\noindent (C2)
	\begin{equation}
	\aligned
	u^{(-)}(n,x)v^{(+)}(n,x)-u^{(+)}(n,x)v^{(-)}(n,x)>0,\quad n> N.
	\endaligned
	\end{equation}

	Substitution of \eqref{bjcjdef5} into \eqref{uvwdef5} and making use of 
	
	\begin{equation}
	\aligned
	2\cosh((2j+a+1/4)\log n)&=n^{(2j+a+1/4)}+n^{-(2j+a+1/4)}\\
	2\sinh((2j+a+1/4)\log n)&=n^{(2j+a+1/4)}-n^{-(2j+a+1/4)}\\
	\endaligned
	\end{equation}
	leads to
	\begin{equation}\label{udef}
	\aligned
	F_a(n,x):&=2n^{1/4}(\log n)u_{a}(n,x\log n)\\
	&=\sum_{0\leqslant j \leqslant m}\sum_{1\leqslant k \leqslant n}
	\frac{(2(2j+a)+1)(\pi n k^2) ^{2j+a}}{\Gamma(2j+a)}
	\frac{(2j+a+1/4) }{x^2+(2j+a+1/4)^2}\\
	&-n^{-1/2}\sum_{0\leqslant j \leqslant m}\sum_{1\leqslant k \leqslant n}
	\frac{(2(2j+a)+1)(\pi k^2/n) ^{2j+a}}{\Gamma(2j+a)}
	\frac{(2j+a+1/4) }{x^2+(2j+a+1/4)^2}\\
	&=\sum_{1\leqslant k \leqslant n}\alpha_a(x,\pi n k^2)-n^{-1/2}\sum_{1\leqslant k \leqslant n}\alpha_a(x,\pi k^2/n)
	\endaligned
	\end{equation}

	\begin{equation}\label{vdef}
	\aligned
G_a(n,x)
	:&=2n^{1/4}(\log n)^2v_{a}(n,x\log n)\\
	&=\sum_{0\leqslant j \leqslant m}\sum_{1\leqslant k \leqslant n}
	\frac{(2(2j+a)+1)(\pi n k^2)^{2j+a}}{\Gamma(2j+a)}
	\frac{1}{x^2+(2j+a+1/4)^2}\\
	&+n^{-1/2}\sum_{0\leqslant j \leqslant m}\sum_{1\leqslant k \leqslant n}
	\frac{(2(2j+a)+1)(\pi k^2/n)^{2j+a}}{\Gamma(2j+a)}
	\frac{1}{x^2+(2j+a+1/4)^2}\\
	&=\sum_{1\leqslant k \leqslant n}\beta_a(x,\pi n k^2)+n^{-1/2}\sum_{1\leqslant k \leqslant n}\beta_a(x,\pi k^2/n)
	\endaligned
	\end{equation}

	Where
	
	\begin{equation}
	\aligned
	\alpha_{a}(x,y)
	&=\sum_{0\leqslant j \leqslant m}\frac{(2(2j+a)+1)y^{2j+a}}{\Gamma(2j+a)}
	\frac{(2j+a+1/4)}{x^2+(2j+a+1/4)^2}\\
	\beta_{a}(x,y)
	&=\sum_{0\leqslant j \leqslant m}\frac{(2(2j+a)+1)y^{2j+a}}{\Gamma(2j+a)}
	\frac{1}{x^2+(2j+a+1/4)^2}\\
	\endaligned
	\end{equation}

Let
	\begin{equation}\label{betamkappa0}
\aligned
F^{(\pm)}(n,x):&=(1/2)(F_{2}(n,x)\pm F_{1}(n,x))\\
J^{(\pm)}(n,x):&=(1/2)(J_{2}(n,x)\pm J_{1}(n,x))\\
\endaligned
\end{equation}

The claims (A2), (B2), and (C2) are then equivalent to\\
\noindent (A3)
\begin{equation}\label{u0minusv0minus}
\aligned
F^{(-)}(n,x)>0,\quad n> N\\
\endaligned
\end{equation}
\noindent (B3)
\begin{equation}\label{u0minusv0minus}
\aligned
G^{(-)}(n,x)>0,\quad n> N\\
\endaligned
\end{equation}

\noindent (C3)
\begin{equation}\label{vdef}
\aligned
F^{(-)}(n,x)G^{(+)}(n,x)-F^{(+)}(n,x)G^{(-)}(n,x)
&>0,\quad n> N.\\
\endaligned
\end{equation}

	In  \cref{FpmGpm} we proved that
	\begin{equation}\label{FminusDef}
	\aligned
F^{(-)}(n,x)
	&=\frac{(\pi n ^3)^{7n^3+2}}{\Gamma(7n^3+2)}\frac{(7n^3+2+q)(7n^3+2+2q)}{\left((7n^3+2+q)^2+x^2\right)} \left(\frac{7n^3}{7n^3+\pi n^3}\right)\\
	&\times\left(1+\frac{\epsilon_1(n,x) M_1}{n^3}\right)\\
	\endaligned
	\end{equation}

	\begin{equation}\label{GminusDef}
	\aligned
G^{(-)}
	&=\frac{(\pi n ^3)^{7n^3+2}}{\Gamma(7n^3+2)}\frac{(7n^3+2+2q)}{\left((7n^3+2+q)^2+x^2\right)} \left(\frac{7n^3}{7n^3+\pi n^3}\right)\\
	&\times\left(1+\frac{\epsilon_2(n,x) M_2}{n^3}\right)\\
	\endaligned
	\end{equation}

\begin{equation}
\aligned
F^{(+)}(n,x)
&=\exp(\pi n^3)\left(\frac{(\pi n^3+2+q)^3}{(\pi n^3+2+q)^2+x^2}\right)\left(1+\frac{\epsilon_3(n,x)M_3}{n^3}\right)\\
\endaligned
\end{equation}

\begin{equation}
\aligned
G^{(+)}(n,x)
&=\exp(\pi n^3)\left(\frac{(\pi n^3+2+q)^2}{(\pi n^3+2+q)^2+x^2}\right)\left(1+\frac{\epsilon_4(n,x)M_4}{n^3}\right)\\
\endaligned
\end{equation}

Thus we have
	\begin{equation}
	\aligned
F^{(-)}(n,x)G^{(+)}(n,x)-F^{(+)}(n,x)G^{(-)}(n,x)
	&=f_0(n,x)g_0(n,x)\\
	\endaligned
	\end{equation}

where
	\begin{equation}
	\aligned
f_0(n,x)
	&=\frac{(\pi n ^3)^{7n^3+2}}{\Gamma(7n^3+2)}\frac{7n^3(7n^3+2+2q)}{\left((7n^3+2+q)^2+x^2\right)}\\
	&\times\exp(\pi n^3)\left(\frac{(\pi n^3+2+q)^2}{(\pi n^3+2+q)^2+x^2}\right)>0\\
	\endaligned
	\end{equation}
	\begin{equation}
	\aligned
g_0(n,x)
	&=\frac{(7n^3+2+q)}{(7n^3+\pi n^3)}\left(1+\frac{\epsilon_1(n,x) M_1}{ n^3}\right)\left(1+\frac{\epsilon_4(n,x) M_4}{n^3}\right)\\
	&-\frac{(\pi n^3+2+q)}{7n^3+\pi n^3}\left(1+\frac{\epsilon_2(n,x) M_2}{ n^3}\right)\left(1+\frac{\epsilon_3(n,x) M_3}{n^3}\right)\\
	\endaligned
	\end{equation}

	\begin{equation}
	\aligned
g_0(n,x)
	&=\frac{(7n^3-\pi n^3)}{(7n^3+\pi n^3)}\left(1+\frac{R_5(n,x)}{n^3}\right)\\
	\endaligned
	\end{equation}
where 

	\begin{equation}
	\aligned
&\frac{(7n^3-\pi n^3)}{(7n^3+\pi n^3)}\frac{R_5(n,x)}{n^3}\\
	:&=\frac{(7n^3+2+q)}{(7n^3+\pi n^3)}\left(\frac{\epsilon_1(n,x) M_1}{n^3}+\frac{\epsilon_4(n,x) M_4}{n^3}+\frac{\epsilon_1(n,x) M_1}{n^3}\frac{\epsilon_4(n,x) M_4}{n^3}\right)\\
	&-\frac{(\pi n^3+2+q)}{7n^3+\pi n^3}\left(\frac{\epsilon_2(n,x) M_2}{n^3}+\frac{\epsilon_3(n,x) M_3}{n^3}+\frac{\epsilon_2(n,x) M_2}{n^3}\frac{\epsilon_3(n,x) M_3}{n^3}\right)\\
	\endaligned
	\end{equation}
Since
	\begin{equation}
	\aligned
|R_5(n,x)|
	&\leqslant \frac{(7n^3+2+q)}{(7n^3-\pi n^3)}\left(M_1+ M_4+\frac{ M_1M_4}{n^3}\right)\\
	&+\frac{(\pi n^3+2+q)}{(7n^3-\pi n^3)}\left( M_2+M_3+\frac{ M_2M_3}{7n^3}\right)\\
	&<\frac{8}{7-\pi}\left(M_1+M_2+M_3+ M_4+M_1M_4+ M_2M_3\right)\\
	&=:M_5
	\endaligned
	\end{equation}
We can thus write

	\begin{equation}\label{Delta0Def}
	\aligned
g_0(n,x)
	&=\frac{(7n^3-\pi n^3)}{(7n^3+\pi n^3)}\left(1+\frac{\epsilon_5(n,x)M_5}{n^3}\right)\\
	\endaligned
	\end{equation}
where $\epsilon_5(n,x)\in(-1,1)$.

From ~\eqref{FminusDef}, ~\eqref{GminusDef} and ~\eqref{Delta0Def}, we deduce that there exists a natural number $N=\lceil(M_5)^{1/3}\rceil$,
 such that for all $n> N$ and all $x\geqslant 0$, the following inequalities hold
	\begin{equation}
	\aligned
	F^{(-)}(n,x)&>0,\\
	G^{(-)}(n,x)&>0,\\
	F^{(-)}(n,x)G^{(+)}(n,x)-G^{(-)}(n,x)F^{(+)}(n,x)&>0.
	\endaligned
	\end{equation}
Thus we proved claims (A3), (B3), and (C3).
\end{proof}

\section{\textbf{Acknowledgment}}

We really appreciate the continuing support and help from Prof. Zixiang ZHOU (math. dept. of Fudan University, Shanghai). Without his help the current paper could only be completed with much lower quality.  We also like to thank Prof. Haseo Ki (math. dept. of Yonsei University, Seoul) for reading an earlier draft of this paper and for pointing out an crucial error in the original proof of \cref{TemmeZhou}.  We like to thank Prof. Nico Temme (Centrum Wiskunde and Informatica (CWI), Amsterdam) as well, together with Prof. Zixiang ZHOU, for providing us with the proof of \cref{TemmeZhou}. We appreciate the continuing support and help from Prof. Jie QING (math dept. of University of California at Santa Cruz, California), Prof. Xinyi YUAN (math dept. of University of California at Berkeley, California). We appreciate the help from Dr. Liming YANG, Dr. Lefang ZHONG, and Dr. Shaowei Zhang. We appreciate the support and help from the (anonymous) experts at open forums: math.stackexchange.com, tex.stackexchange.com.
\section{\textbf{Appendix A }}
In this appendix we present Temme and Zhou's proof for \cref{TemmeZhou} and two more theorems that build on \cref{TemmeZhou}. For convenience we copy \cref{TemmeZhou} as

\begin{theorem}[=~\cref{TemmeZhou} ~\cite{TZ2017}]\label{TemmeZhouapp}

Let $x \in \mathbb{R} $; $b>0$; and  $y>0$. Let the function $g(x,y)$ be defined in terms  of the incomplete gamma function or Kummer function as
\begin{equation}
\aligned
g(x,y):&=\frac{ye^{-y}\gamma(a,-y)}{(-y)^a}\\
&=\frac{y}{a}{_1}F_1(1;1+a;-y)\\
&=y \int_0^1 e^{-yt}(1-t)^{a-1}\,\mathrm{d}t, \quad a=1+b+ix,
\endaligned
\end{equation}
Then there exists a positive constant $M$ and a complex function $\epsilon(x,y),|\epsilon(x,y)|<1$ such that

\begin{equation}\label{eq:13.1.2}
g(x,y)=\frac{y}{y+b+ix}+\frac{\epsilon(x,y) M}{y}
\end{equation}
holds for all $x\in\R$ and all $y>0$.
\end{theorem}
\begin{proof}

To show this, we use integration by parts starting with
\begin{equation}
g(x,y)=y \int_0^1  e^{-yP(t)}\,dt, \quad P(t)= t+\lambda\ln(1-t),\quad \lambda=-\frac{1}{y}(b+ix),
\end{equation}
which gives, because the integrated term vanishes at $t=1$ when $b>0$, 
\begin{equation}
\aligned
g(x,y)&=- \int_0^1 \frac{1}{P^\prime(t)}  \mathrm{d}e^{-yP(t)}\\
&=\left(\frac{e^{-yP(t)}}{P^\prime(t)} \right)_{t=0}- \int_0^1 \frac{1}{P^\prime(t)}  \mathrm{d}e^{-yP(t)}\\
&=\frac{y}{y+b+ix}+\int_0^1 e^{-yt}f(t,x,y) \mathrm{d}t,\\
\endaligned
\end{equation}
where
\begin{equation}
\aligned
f(t,x,y)
&=(1-t)^{b+ix}\frac{d}{dt}\left(\frac{1}{P^\prime(t)}\right)\\
&=(1-t)^{b+ix}\frac{\lambda}{(1-t-\lambda)^2}\\
&={\color{red}{-}}(1-t)^{b+ix}\frac{y(b+ix)}{(b+(1-t)y+ix)^2}\\
\endaligned
\end{equation}

We now define
\begin{equation}
\aligned
F(x,y)&:=\left|\int_0^1e^{-yt}f(t,x,y)\,\mathrm{d} t\right|\\
\endaligned
\end{equation}

It is then suffice to prove that $\frac{(y+b)^2+x^2}{\sqrt{x^2+b^2}}F(x,y)$ is bounded by $M$.

\begin{equation}
\aligned
\frac{1}{\sqrt{b^2+x^2}}F(x,y)
&\leqslant\int_0^1e^{-yt}
   \frac{y}{(b+y-yt)^2+x^2}\,\mathrm{d} t\\
&=\int_0^y\frac{e^{-u}}{(b+y-u)^2+x^2}\mathrm{d} u\qquad  u=yt\\
&=\int_0^{y/2}\frac{e^{-u}}{(b+y-u)^2+x^2}\mathrm{d} u+\int_{y/2}^y\frac{e^{-u}}{(b+y-u)^2+x^2}\,\mathrm{d}
   u\\
&\leqslant \int_0^{y/2}\frac{e^{-u}}{(b+\frac y2)^2+x^2}\mathrm{d}u
   +e^{-y/2}\int_0^{y/2}\frac{e^{-v}}{b^2+x^2}\mathrm{d}v\qquad v=u-y/2\\
&<\frac{4}{(y+b)^2+x^2}+\frac{e^{-y/2}}{b^2+x^2}\\
\endaligned
\end{equation}

Thus
\begin{equation}
\aligned
yF(x,y)
&<\frac{4y\sqrt{x^2+b^2}}{(y+b)^2+x^2}+ye^{-y/2}\frac{\sqrt{x^2+b^2}}{(b^2+x^2)}\\
&\leqslant 4+\frac{\cdot 2}{eb}\\
&=:M
\endaligned
\end{equation}

Thus
$yF(x,y)$ is bounded by $M$ for all $x\in\R$ and all $y>0$.

This shows the validity of \eqref{eq:13.1.2}.

\end{proof}

\begin{theorem}[=~\cref{TemmeZhou2}]\label{TemmeZhouapp2}

Let $x \in \mathbb{R} $;  $1<b<2$; $y>0$. Let the function $g(x,y)$ be defined in terms  of the incomplete gamma function or Kummer function as
\begin{equation}
\aligned
g(x,y):&=\frac{ye^{-y}\gamma(a,-y)}{(-y)^a}\\
&=\frac{y}{a}{_1}F_1(1;1+a;-y)\\
&=y \int_0^1 e^{-yt}(1-t)^{a-1}\,\mathrm{d}t, \quad a=1+b+ix,
\endaligned
\end{equation}

\begin{equation}
\aligned
g_0(x,y)&=\frac{y}{y+b+ix}-\frac{ixy}{(y+b+i x)^3}
\endaligned
\end{equation}

Then there exists a positive constant $M$ and a complex function $\epsilon(x,y),|\epsilon(x,y)|<1$ such that

\begin{equation}\label{eq:AppB.2}
\aligned
g(x,y)&=g_0(x,y)+\frac{\epsilon(x,y) M}{(y+b)^2+x^2}\\
\endaligned
\end{equation}
holds for all $x\in\R$ and all $y>0$.
\end{theorem}
\begin{proof}

To show this, we use integration by parts starting with
\begin{equation}
g(x,y)=y \int_0^1  e^{-yp(t)}\,dt, \quad P(t)= t+\lambda\ln(1-t),\quad \lambda=-\frac{1}{y}(b+ix),
\end{equation}
which gives, because the integrated term vanishes at $t=1$ when $b>0$, 
\begin{equation}
\aligned
g(x,y)&=- \int_0^1 \frac{1}{P^\prime(t)}  \,\mathrm{d}e^{-yP(t)}\\
		&=\left(\frac{e^{-yP(t)}}{P^\prime(t)}\right)_{t=0}+\int_0^1 e^{-yP(t)}\mathrm{d}\left(\frac{1}{P^\prime(t)}\right)\\
		&=\frac{y}{y+b+ix}+\int_0^1 e^{-yt}f(t,x,y)\,\mathrm{d}t,\\
\endaligned
\end{equation}
where
\begin{equation}
\aligned
f(t,x,y)
&=(1-t)^{b+ix}\frac{d}{dt}\left(P^\prime(t)\right)^{-1}\\
&=(1-t)^{b+ix}\frac{\lambda}{(1-t-\lambda)^2}\\
\endaligned
\end{equation}

We use integration by parts again starting with
\begin{equation}
\aligned
h(x,y)&:=\int_0^1 e^{-yt}f(t,x,y)\,\mathrm{d}t
=- \frac{1}{y}\int_0^1  \frac{1}{Q^{\prime}(t)}\mathrm{d}e^{-yQ(t)}\\
Q(t)&= t+\lambda\ln(1-t)+\mu(t),\\
\mu(t)&=-\frac{1}{y}\log\left(\frac{\lambda}{(1-t-\lambda)^2}\right)
\endaligned
\end{equation}
which gives, because the integrated term vanishes at $t=1$ when $b>0$, 
\begin{equation}
\aligned
h(x,y)&= \frac{1}{y}\int_0^1 e^{-yQ(t)}\mathrm{d}\left(\frac{1}{Q^\prime(t)}\right) +\frac{1}{y}\left(\frac{e^{-yQ(t)} }{Q^\prime(t)}\right)_{t=0} \\
&=\left(\int_0^1 e^{-yt}j(t,x,y)\,\mathrm{d}t\right)
-\frac{y}{(y+b+ix)}\frac{(b+ix)}{((y+b+i x)^2-2y)}\\
&=\left(\int_0^1 e^{-yt}j(t,x,y)\,\mathrm{d}t\right)
-\frac{y(b+ix)}{(y+b+i x)^3}\left(1-\frac{2y}{(y+b+ix)^2}\right)^{-1}\\
&=\left(\int_0^1 e^{-yt}j(t,x,y)\,\mathrm{d}t\right)-\frac{y(b+ix)}{(y+b+i x)^3}\left(\left(1-\frac{2y}{(y+b+ix)^2}\right)^{-1}-1\right)\\
&-\frac{iyx}{(y+b+i x)^3}-\frac{yb}{(y+b+i x)^3}\\
\endaligned
\end{equation}
where
\begin{equation}
\aligned
j(t,x,y)
&=\frac{1}{y}e^{-yQ(t)}\frac{\mathrm{d}}{\mathrm{d}t}\left(\frac{1}{Q^\prime(t)}\right)\\
&=\frac{1}{y}(1-t)^{b+ix}\frac{\lambda}{(1-t-\lambda)^2}\frac{\mathrm{d}}{\mathrm{d}t}\left(\frac{1}{Q^\prime(t)}\right)\\
&=(1-t)^{b+ix}\frac{y(b+ix)}{(b+ix+(1-t)y)^2}\\
&\times \frac{(b+ix)^2(b+ix+2(1-t)y)+(b+ix-2)(1-t)^2y^2}{\left((b+ix)^2+2(b+ix-1)(1-t)y+(1-t)^2y^2\right)^2}
\endaligned
\end{equation}
Thus we can write 

\begin{equation}
\aligned
g(x,y)&=\frac{y}{y+b+ix}-\frac{iyx}{(y+b+i x)^3}\\
&+G(x,y)+L(x,y)+H(x,y)\\
\endaligned
\end{equation}
Where
\begin{equation}
\aligned
G(x,y)&:=\int_0^1e^{-yt}j(t,x,y)\,\mathrm{d} t\\
\endaligned
\end{equation}

\begin{equation}
\aligned
L(x,y)&=\frac{y(b+ix)}{(y+b+i x)^3}\left(1-\left(1-\frac{2y}{(y+b+ix)^2}\right)^{-1}\right)\\
\endaligned
\end{equation}

\begin{equation}
\aligned
H(x,y)&=-\frac{yb}{(y+b+i x)^3}\\
\endaligned
\end{equation}

It is then suffice to prove that $((y+b)^2+x^2)\left(\left|G(x,y)\right|+\left|L(x,y)\right|+\left|H(x,y)\right|\right)$ is bounded by $M_2$.

\begin{equation}
\aligned
((y+b)^2+x^2)\left|H(x,y)\right|\
&\leqslant \frac{yb((y+b)^2+x^2)}{((y+b)^2+x^2)^{3/2}}\\
&\leqslant \frac{yb}{((y+b)^2+x^2)^{1/2}}<b\\
&:=M_{2A}\\
\endaligned
\end{equation}

$((y+b)^2+x^2)\left|H(x,y)\right|$ is bounded for all $x\in\R$ and all $y>0$.

\begin{equation}
\aligned
&\left|G(x,y)\right|\\
&\leqslant\int_0^y\frac{e^{-u}(b^2+x^2)^{1/2}}{((b+y-u)^2+x^2)}\\
&\qquad\times\frac{(b^2+x^2)(x^2+(b+2y-2u)^2)^{1/2}+((b-2)^2+x^2)^{1/2}(y-u)^2}{(b^2-x^2+(y-u)^2+2(b-1)(y-u))^2+4(b+y-u)^2x^2}\,\mathrm{d} u\qquad  u=yt\\
&<\int_0^y\frac{e^{-u}(b^2+x^2)^{1/2}}{((b+y-u)^2+x^2)}\\
&\qquad\times\frac{(b^2+x^2)(x^2+(b+2y)^2)^{1/2}+((b-2)^2+x^2)^{1/2}y^2}{(b^2-x^2+(y-u)^2+2(b-1)(y-u))^2+4(b+y-u)^2x^2}\,\mathrm{d} u\\
&=K(x,y)\int_0^y\frac{e^{-u}}{D_1(x,y-u)D_2(x,y-u)}\,\mathrm{d} u\\
\endaligned
\end{equation}

where
\begin{equation}
\aligned
K(x,y):&=(b^2+x^2)^{1/2}
\left((b^2+x^2)(x^2+(b+2y)^2)^{1/2}+y^2((b-2)^2+x^2)^{1/2}\right)\\
&\leqslant
\frac{1}{2}(b^2+x^2)(2x^2+b^2+(b+2y)^2)+\frac{1}{2}y^2(2x^2+b^2+(b-2)^2)\\
&\leqslant
\frac{1}{2}(b^2+x^2)(2x^2+b^2+(b+2y)^2)+\frac{1}{2}y^2(2x^2+2b^2)\quad \because 1<b<2\\
&=
(b^2+x^2)(x^2+b^2+2y^2+2by)+y^2(x^2+b^2)\\
&=(b^2+x^2)(x^2+b^2+3y^2+2by)\\
&=(b^2+x^2)(x^2+2b^2+4y^2)\\
&\leqslant 4(b^2+x^2)\left(y^2+b^2+x^2\right)=:K_{max}(x,y)
\endaligned
\end{equation}

\begin{equation}
\aligned
D_1(x,y-u):&=((b+y-u)^2+x^2)\\
&\geqslant (x^2+b^2+(y-u)^2)=:D_{1,min}(x,y-u)\\
D_2(x,y-u):&=(b^2-x^2+(y-u)^2+2(b-1)(y-u))^2+4(b+y-u)^2x^2\\
&=x^4+2x^2((b+y-u)^2+2(y-u))\\
&+(b^2+2(b-1)(y-u)+(y-u)^2)^2\\
&\geqslant x^4+2x^2(b^2+(y-u)^2)+(b^2+(y-u)^2)^2\quad \because y-u\geqslant 0, b>1\\
&=(x^2+b^2+(y-u)^2)^2=:D_{2,min}(x,y-u)
\endaligned
\end{equation}

\begin{equation}
\aligned
&\int_0^y \frac{e^{-u}}{D_{1,min}(x,y-u)D_{2,min}(x,y-u)}\mathrm{d}u\\
&=\int_0^{y/2}\frac{e^{-u}}{\left(x^2+b^2+(y-u)^2\right)^3}\,\mathrm{d} u
+\int_{y/2}^y\frac{e^{-u}}{\left(x^2+b^2+(y-u)^2\right)^3}\,\mathrm{d} u\\
&\leqslant\int_0^{y/2}\frac{e^{-u}}{\left(x^2+b^2+\frac{y^2}{4}\right)^3}\,\mathrm{d} u
+\int_{y/2}^y\frac{e^{-u}}{(x^2+b^2)^3}\,\mathrm{d} u\\
&<\frac{64}{\left(x^2+b^2+y^2\right)^3}\int_0^{y/2}e^{-u}\,\mathrm{d} u
+\frac{e^{-y/2}}{(x^2+b^2)^3}\int_0^{y/2}e^{-v}\,\mathrm{d} u\\
&<\frac{64}{\left(x^2+b^2+y^2\right)^3}
+\frac{e^{-y/2}}{(x^2+b^2)^3}\\
\endaligned
\end{equation}

\begin{equation}
\aligned
\frac{1}{4}\left|G(x,y)\right|&<\frac{ 64(b^2+x^2)(x^2+b^2+y^2)}{\left(x^2+b^2+y^2\right)^3}
+\frac{e^{-y/2}(b^2+x^2)(x^2+b^2+y^2)}{(b^2+x^2)^3}\\
&<\frac{ 64}{\left(x^2+b^2+y^2\right)}
+\frac{e^{-y/2}(x^2+b^2+y^2)}{(b^2+x^2)^2}\\
\endaligned
\end{equation}

\begin{equation}
\aligned
&\frac{1}{4}((y+b)^2+x^2)\left|G(x,y)\right|\\
&<\frac{ 64((y+b)^2+x^2)}{\left(x^2+b^2+y^2\right)}
+\frac{e^{-y/2}(x^2+b^2+y^2)((y+b)^2+x^2)}{(b^2+x^2)^2}\\
&<128
+\frac{2e^{-y/2}(x^2+b^2+y^2)^2}{(b^2+x^2)^2}\\
&=128+2e^{-y/2}
+\frac{4y^2e^{-y/2}}{(b^2+x^2)}+\frac{2y^4e^{-y/2}}{(b^2+x^2)^2}\\
&<128+2+4b^{-2}y^2e^{-y/2}+2b^{-4}y^4e^{-y/2}\\
&\leqslant 130+4\cdot 4^2 (eb)^{-2}+2\cdot 8^4 (eb)^{-4}\\
&< 130+16 b^{-2}+512 b^{-4}\quad \because e>2\\
\endaligned
\end{equation}

\begin{equation}
\aligned
((y+b)^2+x^2)\left|G(x,y)\right|
&=4\cdot (130+16 b^{-2}+512 b^{-4})\\
&=:M_{2B}
\endaligned
\end{equation}

Thus
$(y^2+b^2+x^2)\left|G(x,y)\right|$ is bounded for all $x\in\R$ and all $y>0$.

\begin{equation}
\aligned
\left|L(x,y)\right|&=\left|\frac{y(b+ix)}{(y+b+i x)^3}\left(1-\left(1-\frac{2y}{(y+b+ix)^2}\right)^{-1}\right)\right|\\
&=\left|\frac{y(b+ix)}{(y+b+i x)^3}\frac{2y}{((y+b+ix)^2-2y)}\right|\\
&\leqslant \frac{(b^2+x^2)^{1/2}}{((y+b)^2+ x^2)^{3/2}}\frac{2y^2}{(((y+b)^2-x^2-2y)^2+4x^2(y+b)^2)^{1/2}}\\
\endaligned
\end{equation}
Therefore
\begin{equation}
\aligned
(((y+b)^2+x^2)\left|L(x,y)\right|&\leqslant \sqrt{L_1(x,y)}\\
L_1(x,y)&:= \frac{y^2((y+b)^2+x^2)^2}{((y+b)^2+ x^2)^{3}}\\
&\times\frac{4y^2(b^2+x^2)}{(((y+b)^2-x^2-2y)^2+4x^2(y+b)^2)}\\
&<\frac{4y^2(b^2+x^2)}{(((y+b)^2-x^2-2y)^2+4x^2(y+b)^2)}\\
\endaligned
\end{equation}
Note
\begin{equation}
\aligned
&((y+b)^2-x^2-2y)^2+4x^2(y+b)^2\\
&=((y+b)^2-2y)^2+x^4-2x^2((y+b)^2-2y)+4x^2(y+b)^2\\
&=((y+b)^2-2y)^2+x^4+2x^2(y+b)^2+4x^2y\\
&=(y^2+b^2+2(b-1)y)^2+x^4+2x^2(y+b)^2+4x^2y\\
&>(y^2+b^2)^2+x^4+2x^2(y^2+b^2)\qquad \because b>1\\
&=(y^2+b^2+x^2)^2
\endaligned
\end{equation}

Thus
\begin{equation}
\aligned
L_1(x,y)&< \frac{4y^2(b^2+x^2)}{(y^2+b^2+x^2)^2}< 4\\
\endaligned
\end{equation}

and
\begin{equation}
\aligned
((y+b)^2+x^2)\left|L(x,y)\right|&<2=:M_{2C}
\endaligned
\end{equation}

Finally we have
\begin{equation}
\aligned
(y^2+b^2+x^2)(\left|H(x,y)\right|+\left|G(x,y)\right|+\left|L(x,y)\right|)
&<M_{2A}+M_{2B}+M_{2C}=:M_2
\endaligned
\end{equation}

Thus
$(y^2+b^2+x^2)(|H(x,y)|+|G(x,y)|+|L(x,y)|)$ is bounded by $M$ for all $x\in\R$ and all $y>0$.

This shows the validity of \eqref{eq:AppB.2}.

\end{proof}

\begin{theorem}[=~\cref{TemmeZhou3}]\label{TemmeZhouapp3}
	
	Let $x \in \mathbb{R} $;  $b=5/4$; $y>0$. Let the function $g(x,y)$ be defined in terms  of the incomplete gamma function or Kummer function as
	\begin{equation}
		\aligned
		g(x,y):&=\frac{ye^{-y}\gamma(a,-y)}{(-y)^a}\\
		&=\frac{y}{a}{_1}F_1(1;1+a;-y)\\
		&=y \int_0^1 e^{-yt}(1-t)^{a-1}\,dt, \quad a=1+b+ix,
		\endaligned
	\end{equation}
	
\begin{equation}
	\aligned
	g_0(x,y)&=\frac{y}{y+b+ix}\left(1-\frac{ix}{(y+b+i x)^2}\right)
	\endaligned
	\end{equation}

	Then there exists a positive constant $M_3$ and a real function $\epsilon_2(x,y)\in[-1,1]$ such that
	
	\begin{equation}\label{eq:13.3.2}
		\aligned
		\Im g(x,y)&=\Im g_0(x,y)+\frac{ x\epsilon_3(x,y)M_3}{(y+b)^2+x^2}\\
		\endaligned
	\end{equation}
	holds for all $x\in\R$ and all $y>0$.
\end{theorem}
\begin{proof}
	
	To show this, we use integration by parts starting with
	\begin{equation}
		g(x,y)=y \int_0^1  e^{-yp(t)}\,dt, \quad P(t)= t+\lambda\ln(1-t),\quad \lambda=-\frac{1}{y}(b+ix),
	\end{equation}
	which gives, because the integrated term vanishes at $t=1$ when $b>0$, 
	\begin{equation}\label{gxydef2}
		\aligned
		g(x,y)&=- \int_0^1 \frac{1}{P^\prime(t)}  \,de^{-yP(t)}\\
		&=\left(\frac{e^{-yP(t)}}{P^\prime(t)}\right)_{t=0}+\int_0^1 e^{-yP(t)}\mathrm{d}\left(\frac{1}{P^\prime(t)}\right)\\
		&=\frac{y}{y+b+ix}+\int_0^1 e^{-yt}f(t,x,y)\mathrm{d}t,\\
		\endaligned
	\end{equation}
	where
	\begin{equation}
		\aligned
		f(t,x,y)
		&=(1-t)^{b+ix}\frac{d}{dt}\left(P^\prime(t)\right)^{-1}\\
		&=(1-t)^{b+ix}\frac{\lambda}{(1-t-\lambda)^2}\\
		\endaligned
	\end{equation}
	
	We use integration by parts again starting with
	\begin{equation}
		\aligned
		h(x,y)&:=\int_0^1 e^{-yt}f(t,x,y)\,\mathrm{d}t
		=- \frac{1}{y}\int_0^1  \frac{1}{Q^{\prime}(t)}\mathrm{d}e^{-yQ(t)}\\
		Q(t)&= t+\lambda\ln(1-t)+\mu(t),\\
		\mu(t)&=-\frac{1}{y}\log\left(\frac{\lambda}{(1-t-\lambda)^2}\right)
		\endaligned
	\end{equation}
which gives, because the integrated term vanishes at $t=1$ when $b>0$, 
\begin{equation}
\aligned
h(x,y)&= \frac{1}{y}\int_0^1 e^{-yQ(t)}\mathrm{d}\left(\frac{1}{Q^\prime(t)}\right) +\frac{1}{y}\left(\frac{e^{-yQ(t)} }{Q^\prime(t)}\right)_{t=0} \\
&=\left(\int_0^1 e^{-yt}j(t,x,y)\,\mathrm{d}t\right)
-\frac{y}{(y+b+ix)}\frac{(b+ix)}{((y+b+i x)^2-2y)}\\
&=\left(\int_0^1 e^{-yt}j(t,x,y)\,\mathrm{d}t\right)
-\frac{y(b+ix)}{(y+b+i x)^3}\left(1-\frac{2y}{(y+b+ix)^2}\right)^{-1}\\
&=\left(\int_0^1 e^{-yt}j(t,x,y)\,\mathrm{d}t\right)-\frac{y(b+ix)}{(y+b+i x)^3}\left(\left(1-\frac{2y}{(y+b+ix)^2}\right)^{-1}-1\right)\\
&-\frac{iyx}{(y+b+i x)^3}-\frac{yb}{(y+b+i x)^3}\\
\endaligned
\end{equation}
where
\begin{equation}
\aligned
j(t,x,y)
&=\frac{1}{y}e^{-yQ(t)}\frac{\mathrm{d}}{\mathrm{d}t}\left(\frac{1}{Q^\prime(t)}\right)\\
&=\frac{1}{y}(1-t)^{b+ix}\frac{\lambda}{(1-t-\lambda)^2}\frac{\mathrm{d}}{\mathrm{d}t}\left(\frac{1}{Q^\prime(t)}\right)\\
&=(1-t)^{b+ix}\frac{y(b+ix)}{(b+ix+(1-t)y)^2}\\
&\times \frac{(b+ix)^2(b+ix+2(1-t)y)+(b+ix-2)(1-t)^2y^2}{\left((b+ix)^2+2(b+ix-1)(1-t)y+(1-t)^2y^2\right)^2}
\endaligned
\end{equation}
Thus we can write 

\begin{equation}
\aligned
g(x,y)&=\frac{y}{y+b+ix}-\frac{iyx}{(y+b+i x)^3}\\
&+G(x,y)+L(x,y)+H(x,y)\\
\endaligned
\end{equation}
Where
\begin{equation}
\aligned
G(x,y)&:=\int_0^1e^{-yt}j(t,x,y)\,\mathrm{d} t\\
\endaligned
\end{equation}

\begin{equation}
\aligned
L(x,y)&=\frac{y(b+ix)}{(y+b+i x)^3}\left(1-\left(1-\frac{2y}{(y+b+ix)^2}\right)^{-1}\right)\\
\endaligned
\end{equation}

\begin{equation}
\aligned
H(x,y)&=-\frac{yb}{(y+b+i x)^3}\\
\endaligned
\end{equation}

	It is then suffice to prove that $((y+b)^2+x^2)(|\Im G(x,y)|+|\Im L(x,y)|+|\Im H(x,y)|)$ is bounded by $|x|M_3$.

When $x=0$, $\Im G(x,y)=0,\Im L(x,y)=0,\Im H(x,y)=0$, so we only need to consider the case when $x!=0$. 

\begin{equation}
\aligned
\frac{1}{(y+i x)^3}&=\frac{(y-i x)^3}{(y^2+x^2)^3}
=\frac{y(y^2-3x^2)-ix(3y^2-x^2)}{(y^2+x^2)^3}\\ 
\Im\left(\frac{1}{(y+b+i x)^3}\right)
&=-\frac{x(3(y+b)^2-x^2)}{((y+b)^2+x^2)^3}\\ 
\endaligned
\end{equation}

\begin{equation}
\aligned
\left|\Im H(x,y)\right|\
&=\frac{|x|yb(3(y+b)^2-x^2)}{((y+b)^2+x^2)^3}\\ 
\endaligned
\end{equation}

\begin{equation}
\aligned
|x|^{-1}((y+b)^2+x^2)\left|\Im H(x,y)\right|\
&\leqslant yb\frac{(3(y+b)^2+x^2)((y+b)^2+x^2)}{((y+b)^2+x^2)^{3}}\\
&\leqslant \frac{3yb}{(y+b)^2+x^2}\\
&\leqslant \frac{3yb}{4yb}=\frac{3}{4}=:M_{3A}\\
\endaligned
\end{equation}

$((y+b)^2+x^2)\left|\Im H(x,y)\right|$ is bounded by $|x|M_{3A}$ for all $x\in\R$ and all $y>0$.

\begin{equation}
\aligned
L(x,y)&=\frac{y(b+ix)}{(y+b+i x)^3}\left(1-\left(1-\frac{2y}{(y+b+ix)^2}\right)^{-1}\right)\\
&=\frac{y(b+ix)}{(y+b+i x)^3}\frac{2y}{(2y-(y+b+ix)^2)}\\
&= \frac{(b+ix)((y+b)((y+b)^2-3x^2)+ix(x^2-3(y+b)^2) )}{((y+b)^2+ x^2)^{3}}\\
&\times\frac{2y^2((2y-(y+b)^2+x^2)+2ix(y+b))}{(2y-(y+b)^2+x^2)^2+4x^2(y+b)^2}\\
\endaligned
\end{equation}

\begin{equation}
\aligned
\frac{1}{x}\Im L(x,y)&= \frac{2y^2a_6}{b_6c_6}\\
a_6:&=x^4(b-5(y+b))+x^2(2by-6y(y+b)+10y(y+b)^2)\\
&+(y+b)^4(5b-(y+b))+2y(y+b)^2((y+b)-3b)\\
b_6:&=((y+b)^2+ x^2)^{3}\\
c_6:&=(2y-(y+b)^2+x^2)^2+4x^2(y+b)^2
\endaligned
\end{equation}

Note
\begin{equation}
\aligned
b_6&>(y^2+b^2+ x^2)^{3}
\endaligned
\end{equation}
\begin{equation}
\aligned
c_6&=((y+b)^2-2y)^2+x^4-2x^2((y+b)^2-2y)+4x^2(y+b)^2\\
&=((y+b)^2-2y)^2+x^4+2x^2(y+b)^2+4x^2y\\
&=(y^2+b^2+2(b-1)y)^2+x^4+2x^2(y+b)^2\\
&>(y^2+b^2)^2+x^4+2x^2(y^2+b^2)\qquad \because b>1\\
&=(y^2+b^2+x^2)^2
\endaligned
\end{equation}

\begin{equation}
\aligned
|a_6|&\leqslant x^4(b+5(y+b))+x^2(2by+6y(y+b)+10y(y+b)^2)\\
&+(y+b)^4(5b+(y+b))+2y(y+b)^2((y+b)+3b)\\
&\leqslant 6x^4(y+b)+18x^2(y+b)^3
+6(y+b)^5+8(y+b)^4\quad \because b>1\\
&< (y+b)\left(6x^4+18x^2(y+b)^2
+14(y+b)^4\right)\quad \because b>1\\
&< (y+b)\left(6x^4+18\cdot 2x^2(y^2+b^2)
+14\cdot 4(y^2+b^2)^2\right)\\
&< 56(y+b)\left(y^2+b^2+x^2\right)^2
\endaligned
\end{equation}

Since $b=5/4$, we have $y^2+b^2-y-b=(y-\frac{1}{4})^2+\frac{25}{16}-\frac{5}{4}-\frac{1}{4}=(y-\frac{1}{4})^2+\frac{1}{16}>0$.
Therefore

\begin{equation}
\aligned
|a_6|
&< 56(y^2+b^2)\left(y^2+b^2+x^2\right)^2\\
&\leqslant 56\left(y^2+b^2+x^2\right)^3\\
\endaligned
\end{equation}

Thus
\begin{equation}
\aligned
\frac{1}{|x|}|\Im L(x,y)|
&\leqslant \frac{112y^2\left(y^2+b^2+x^2\right)^3}{(y^2+b^2+x^2)^3(y^2+b^2+x^2)^2}\\
&< \frac{112}{y^2+b^2+x^2}
\endaligned
\end{equation}

\begin{equation}
\aligned
\frac{1}{|x|}((y+b)^2+x^2)|\Im L(x,y)|
&\leqslant 
\frac{112((y+b)^2+x^2)}{(y^2+b^2+x^2)}\\
&\leqslant 
\frac{224(y^2+b^2+x^2)}{(y^2+b^2+x^2)}\\
&=224=:M_{3D}
\endaligned
\end{equation}
$((y+b)^2+x^2)\left|\Im L(x,y)\right|$ is bounded by $|x|M_{3D}$ for all $x\in\R$ and all $y>0$.
\begin{equation}
\aligned
j(1-v,x,y/v)&=v^{b+ix}\frac{(b+ix)}{(b+ix+y)^2}\\
&\times \frac{(b+ix)^2(b+ix+2y)+(b+ix-2)y^2}{\left((b+ix)^2+2(b+ix-1)y+y^2\right)^2}
\endaligned
\end{equation}

\begin{equation}
\aligned
v^{ix}&=\cos(x\log v)+i\sin(x\log v))\\
&=\cos(x\log v)+ix(\log v )\frac{\sin(x\log v)}{x \log v}\\
&=\cos(x\log v)+ix(\log v )\text{sinc}(x\log v)\\
\endaligned
\end{equation}

\begin{equation}
\aligned
\frac{(b+ix)}{(y+b+ix)^2}
&=\frac{\left[b(y+b)^2+x^2(2y+b)\right]+ix\left[y^2-b^2-x^2\right]}{((b+y)^2+x^2)^2}\\
&=:\frac{a_1(x,y)+ix b_1(x,y)}{c_1(x,y)}\\
\endaligned
\end{equation}
where
\begin{equation}
\aligned
a_1(x,y)&:=b(y+a)^2+x^2(2y+b)\\
b_1(x,y)&:=y^2-b^2-x^2\\
c_1(x,y)&:=((y+b)^2+x^2)^2\\
\endaligned
\end{equation}

\begin{equation}
\aligned
&(b+ix)^2(b+ix+2y)+(b+ix-2)y^2\\
&=[b(y+b)^2-2y^2-x^2(2y+3b)]+ix [(y+b)(y+3b)-x^2]\\
&=:a_2(x,y)+ixb_2(x,y)
\endaligned
\end{equation}
where
\begin{equation}
\aligned
a_2(x,y)&:=b(y+b)^2-2y^2-x^2(2y+3b)\\
b_2(x,y)&:=(y+b)(y+3b)-x^2\\
\endaligned
\end{equation}

\begin{equation}
\aligned
&\frac{1}{\left((b+ix)^2+2(b+ix-1)y+y^2\right)^2}\\
&=\frac{1}{\left((b+y)^2-x^2-2y+2ix(b+y)\right)^2}\\
&=\frac{\left((y+b)^2-x^2-2y-2ix(y+b)\right)^2}{\left(((y+b)^2-x^2-2y)^2+4x^2(y+b)^2\right)^2}\\
&=\frac{\left[((y+b)^2-x^2-2y)^2-4x^2(b+y)^2\right]-ix\left[4(y+b)((y+b)^2-x^2-2y)\right]}{\left[((y+b)^2-x^2-2y)^2+4x^2(y+b)^2\right]^2}\\
&=:\frac{a_3(x,y)+ixb_3(x,y)}{c_3^2(x,y)}\\
\endaligned
\end{equation}
where
\begin{equation}
\aligned
a_3(x,y)&:=((y+b)^2-x^2-2y)^2-4x^2(y+b)^2\\
b_3(x,y)&:=-4(y+b)((b+y)^2-x^2-2y)\\
c_3(x,y)&:=((y+b)^2-x^2-2y)^2+4x^2(y+b)^2\\
\endaligned
\end{equation}

Thus
\begin{equation}
\aligned
j(1-v,x,y/v)&=v^b (\cos(x\log v)+ix(\log v)\text{sinc}(x\log v))\frac{a_1(x,y)+ixb_1(x,y)}{c_1(x,y)}\\
&\times (a_2(x,y)+ixb_2(x,y))\frac{a_3(x,y)+ixb_3(x,y)}{c_3^2(x,y)}\\
\endaligned
\end{equation}

\begin{equation}
\aligned
&\Im j(1-v,x,y/v)\\
&=x v^b\cos(x\log v)\frac{a_4(x,y)}{c_1(x,y)c_3^2(x,y)}\\
&-x v^b(\log v)\text{sinc}(x\log v)\frac{b_4(x,y)}{c_1(x,y)c_3^2(x,y)}\\
\endaligned
\end{equation}

where
\begin{equation}
\aligned
a_4(x,y)
&:=-x^2b_1(x,y)b_2(x,y)b_3(x,y)+b_1(x,y)a_2(x,y)a_3(x,y)\\
&+a_1(x,y)a_2(x,y)b_3(x,y)+a_1(x,y)b_2(x,y)a_3(x,y)
\endaligned
\end{equation}

\begin{equation}
\aligned
b_4(x,y)
&:=-a_1(x,y)a_2(x,y)a_3(x,y)+x^2a_1(x,y)b_2(x,y)b_3(x,y)\\
&+x^2b_1(x,y)b_2(x,y)a_3(x,y)+x^2b_1(x,y)a_2(x,y)b_3(x,y)
\endaligned
\end{equation}

Thus

\begin{equation}
\aligned
&\frac{1}{x}\Im j(t,x,y)\\
&=(1-t)^b\cos(x\log (1-t))\frac{a_4(x,(1-t)y)}{c_1(x,(1-t)y)c_3^2(x,(1-t)y)}\\
&-(\log (1-t)))(1-t)^b\text{sinc}(x\log (1-t))\frac{b_4(x,(1-t)y)}{c_1(x,(1-t)y)c_3^2(x,(1-t)y)}\\
\endaligned
\end{equation}

Because $b>1,0\leqslant t\leqslant 1$, so we have $|(1-t)^b|\leqslant 1$; $|\cos(x\log (1-t))|\leqslant 1$; $|\text{sinc}(x\log (1-t))|\leqslant 1$; $|(1-t)^b\log (1-t)|\leqslant 1$,

Thus
\begin{equation}
\aligned
\frac{1}{|x|}|\Im G(x,y)|&\leqslant\frac{1}{|x|}\int_0^1e^{-yt}|\Im j(t,x,y)|\,y\mathrm{d} t\\
&\leqslant\int_0^1e^{-yt}\frac{|a_4(x,(1-t)y)|}{c_1(x,(1-t)y)c_3^2(x,(1-t)y)}\,y\mathrm{d} t\\
&+\int_0^1e^{-yt}\frac{|b_4(x,(1-t)y)|}{c_1(x,(1-t)y)c_3^2(x,(1-t)y)}\,y\mathrm{d} t\\
&\leqslant\int_0^ye^{-u}\frac{|a_4(x,y-u)|}{c_1(x,y-u)c_3^2(x,y-u)}\,\mathrm{d} u\qquad u=yt\\
&+\int_0^ye^{-u}\frac{|b_4(x,y-u)|}{c_1(x,y-u)c_3^2(x,y-u)}\,\mathrm{d} u\qquad u=yt\\
\endaligned
\end{equation}

	\begin{equation}
		\aligned
		0<c_1(x,y-u)&=((b+y-u)^2+x^2)^2\\
		&\geqslant (x^2+b^2+(y-u)^2)^2=:c_{1,min}(x,y-u)\\
		0<c_3(x,y-u)&=(b^2-x^2+(y-u)^2+2(b-1)(y-u))^2+4(b+y-u)^2x^2\\
		&=x^4+2x^2((b+y-u)^2+2(y-u))\\
		&+(b^2+2(b-1)(y-u)+(y-u)^2)^2\\
		&\geqslant x^4+2x^2(b^2+(y-u)^2)+(b^2+(y-u)^2)^2\quad \because y>u\geqslant 0, b>1\\
		&=(x^2+b^2+(y-u)^2)^2=:c_{3,min}(x,y-u)
		\endaligned
	\end{equation}
	
	Therefore
\begin{equation}
	\aligned
	&\int_0^y\frac{e^{-u}}{c_1(x,y-u)c_3^2(x,y-u)}\,\mathrm{d} u\\
	&\leqslant \int_0^y\frac{e^{-u}}{c_{1,min}(x,y-u)c_{3,min}^2(x,y-u)}\,\mathrm{d} u\\
	&=\int_0^{y/2}\frac{e^{-u}}{(x^2+b^2+(y-u)^2)^6}\,\mathrm{d} u\\
	&+\int_{y/2}^y\frac{e^{-u}}{(x^2+b^2+(y-u)^2)^6}\,\mathrm{d} u\\
	&\leqslant\int_0^{y/2}\frac{e^{-u}}{(x^2+b^2+(y/2)^2)^6}\,\mathrm{d} u\\
	&+e^{-y/2}\int_{0}^{y/2}\frac{e^{-v}}{(x^2+b^2)^6}\,\mathrm{d} v,\quad v=u-y/2\\
	&\leqslant \frac{1}{(x^2+b^2+(y/2)^2)^6}+\frac{e^{-y/2}}{(x^2+b^2)^6}\\
	&\leqslant \frac{2^{12}}{(y^2+b^2+x^2)^6}+\frac{e^{-y/2}}{(x^2+b^2)^6}\\
	\endaligned
\end{equation}

\begin{equation}
\aligned
a_1(x,y)&:=b(y+b)^2+x^2(2y+b)\\
a_2(x,y)&:=b(y+b)^2-2y^2-x^2(2y+3b)\\
a_3(x,y)&:=((y+b)^2-x^2-2y)^2-4x^2(y+b)^2\\
\endaligned
\end{equation}

\begin{equation}
\aligned
|a_1(x,y-u)|
&\leqslant b(b+y)^2+x^2(2y+b)\\
&< 2(y+b)(y^2+b^2+x^2)\\
\endaligned
\end{equation}

\begin{equation}
\aligned
|a_2(x,y-u)|
&\leqslant b(y+b)^2+2y^2+x^2(2y+3b)\\
&=3(y+b)(y^2+b^2+x^2) \\
&-\left(3y^3+2(b-1)y^2+(b^2+x^2)y+2b^3\right)\\
&\leqslant 3(y+b)(y^2+b^2+x^2) \quad \because b>1\\
\endaligned
\end{equation}

\begin{equation}
\aligned
|a_3(x,y-u)|
&\leqslant ((y+b)^2+x^2+2y)^2+4x^2(y+b)^2\\
&\leqslant (2y^2+2b^2+x^2+2by)^2+4x^2(2y^2+2b^2)\quad \because b>1\\
&\leqslant (3y^2+3b^2+x^2)^2+8x^2(y^2+b^2)\\
&=9(y^2+b^2)^2+x^4+6(y^2+b^2)x^2+8x^2(y^2+b^2)\\
&=9(y^2+b^2)^2+x^4+14(y^2+b^2)x^2\\
&\leqslant  9(y^2+b^2+x^2)^2\\
\endaligned
\end{equation}

\begin{equation}
\aligned
b_1(x,y)&:=y^2-b^2-x^2\\
b_2(x,y)&:=(y+b)(y+3b)-x^2\\
b_3(x,y)&:=-4(y+b)((y+b)^2-x^2-2y)\\
\endaligned
\end{equation}
\begin{equation}
\aligned
|b_1(x,y-u)|
&\leqslant y^2+b^2+x^2\\
\endaligned
\end{equation}

\begin{equation}
\aligned
|b_2(x,y-u)|
&\leqslant (y+b)(y+3b)+x^2\\
&\leqslant 5(y^2+b^2+x^2)\\
\endaligned
\end{equation}

\begin{equation}
\aligned
|b_3(x,y-u)|
&\leqslant 4(b+y)((y+b)^2+x^2+2y)\\
&\leqslant 4(y+b)(2y^2+2b^2+x^2+2by)\quad \because b>1\\
&\leqslant 4(y+b)(3y^2+3b^2+x^2)\\
&\leqslant 12(y+b)(y^2+b^2+x^2)\\
\endaligned
\end{equation}

\begin{equation}
\aligned
|a_1(x,y-u)|&\leqslant 2(y+b)(y^2+b^2+x^2)\\
|a_2(x,y-u)|&\leqslant 3(y+b)(y^2+b^2+x^2)\\
|a_3(x,y-u)|&\leqslant  9(y^2+b^2+x^2)^2\\
|b_1(x,y-u)|&\leqslant (y^2+b^2+x^2)\\
|b_2(x,y-u)|&\leqslant 5(y^2+b^2+x^2)\\
|b_3(x,y-u)|&\leqslant 12(y+b)(y^2+b^2+x^2)\\
\endaligned
\end{equation}

\begin{equation}
\aligned
&x^2|b_1(x,y-u)b_2(x,y-u)b_3(x,y-u)|\\
&\leqslant x^2\cdot (y^2+b^2+x^2)\cdot 5(y^2+b^2+x^2)\cdot 12(y+b)(y^2+b^2+x^2)\\
&=60 x^2 (y+b)(y^2+b^2+x^2)^3\\
&<60 (y+b)(y^2+b^2+x^2)^4
\endaligned
\end{equation}

\begin{equation}
\aligned
&|b_1(x,y)a_2(x,y)a_3(x,y)|\\
&\leqslant (y^2+b^2+x^2)\cdot 3(y+b)(y^2+b^2+x^2)\cdot 9 (y^2+b^2+x^2)^2\\
&=27 (y+b) (y^2+b^2+x^2)^4\\
\endaligned
\end{equation}

\begin{equation}
\aligned
&|a_1(x,y)a_2(x,y-u)b_3(x,y-u)|\\
&\leqslant 2(y+b)(y^2+b^2+x^2)\cdot 3(y+b)(y^2+b^2+x^2)\cdot 12(y+b)(y^2+b^2+x^2)\\
&\leqslant 72(y+b)^3(y^2+b^2+x^2)^3\\
&\leqslant 144(y+b)(y^2+b^2)(y^2+b^2+x^2)^3\\
&\leqslant 144(y+b)(y^2+b^2+x^2)^4\\
\endaligned
\end{equation}

\begin{equation}
\aligned
&|a_1(x,y-u)b_2(x,y-u)a_3(x,y-u)|\\ 
&\leqslant 2(y+b)(y^2+b^2+x^2)\cdot 5(y^2+b^2+x^2)\cdot 9(y^2+b^2+x^2)^2\\
&=90(y+b)(y^2+b^2+x^2)^4\\
\endaligned
\end{equation}

\begin{equation}
\aligned
|a_4(x,y-u)|
&\leqslant x^2|b_1(x,y-u)b_2(x,y-u)b_3(x,y-u)|\\
&+|b_1(x,y)a_2(x,y-u)a_3(x,y-u)|\\
&+|a_1(x,y)a_2(x,y-u)b_3(x,y-u)|\\
&+|a_1(x,y-u)b_2(x,y-u)a_3(x,y-u)|\\
&<(60 +27+144+90)(y+b)(y^2+b^2+x^2)^4\\
&=321(y+b)(y^2+b^2+x^2)^4\\
\endaligned
\end{equation}

Since $b=5/4$, we have $y^2+b^2-y-b=(y-\frac{1}{4})^2+\frac{25}{16}-\frac{5}{4}-\frac{1}{4}=(y-\frac{1}{4})^2+\frac{1}{16}>0$.
Thus
\begin{equation}
\aligned
|a_4(x,y-u)|
&< 321(y^2+b^2)(y^2+b^2+x^2)^4\\
&\leqslant 321(y^2+b^2+x^2)^5\\
&=:|a_4(x,y)|_{max}
\endaligned
\end{equation}

\begin{equation}
\aligned
&|a_1(x,y-u)a_2(x,y-u)a_3(x,y-u)|\\
&\leqslant 2(y+b)(y^2+b^2+x^2)\cdot 3(y+b)(y^2+b^2+x^2)\cdot 9(y^2+b^2+x^2)^2\\
&\leqslant 54(y+b)^2(y^2+b^2+x^2)^4\\
&\leqslant 108(y^2+b^2+x^2)^5\\
\endaligned
\end{equation}

\begin{equation}
\aligned
&x^2|a_1(x,y-u)b_2(x,y-u)b_3(x,y-u)|\\
&\leqslant x^2\cdot 2(y+b)(y^2+b^2+x^2)\cdot 5(y^2+b^2+x^2)\cdot 12(y+b)(y^2+b^2+x^2)\\
&\leqslant 60x^2(y+b)^2(y^2+x^2)^3\\
&\leqslant 120(y^2+b^2+x^2)^5\\
\endaligned
\end{equation}

\begin{equation}
\aligned
&x^2|b_1(x,y-u)b_2(x,y-u)a_3(x,y-u)|\\
&\leqslant x^2\cdot (y^2+b^2+x^2)\cdot 5(y^2+b^2+x^2)\cdot 9(y^2+b^2+x^2)^2\\
&\leqslant 45 x^2(y^2+b^2+x^2)^4\\
&\leqslant 45 (y^2+b^2+x^2)^5\\
\endaligned
\end{equation}

\begin{equation}
\aligned
&x^2|b_1(x,y-u)a_2(x,y)b_3(x,y-u)|\\
&\leqslant x^2\cdot (y^2+b^2+x^2) \cdot 2(y+b)(y^2+x^2)\cdot 12(y+b)(y^2+b^2+x^2)\\
&\leqslant 24x^2 (y+b)^2 (y^2+b^2+x^2)^3\\
&\leqslant 48(y^2+b^2+x^2)^5\\
\endaligned
\end{equation}

\begin{equation}
\aligned
|b_4(x,y-u)|
&:=|a_1(x,y-u)a_2(x,y-u)a_3(x,y-u)|\\
&+x^2|a_1(x,y-u)b_2(x,y-u)b_3(x,y-u)|\\
&+x^2|b_1(x,y-u)b_2(x,y-u)a_3(x,y-u)|\\
&+x^2|b_1(x,y-u)a_2(x,y)b_3(x,y-u)|\\
&\leqslant (108+120+45+48)(y^2+x^2)^5\\
&=321(y^2+b^2+x^2)^5\\
&=:|b_4(x,y)|_{max}
\endaligned
\end{equation}

\begin{equation}
\aligned
&\frac{((y+b)^2+x^2)}{|x|}|\Im G(x,y)|\\
&\leqslant 2(y^2+b^2+x^2)\int_0^ye^{-u}\frac{|a_4(x,y-u)|+|b_4(x,y-u)|}{c_1(x,y-u)c_3^2(x,y-u)}\,\mathrm{d} u\\
&\leqslant 2(y^2+b^2+x^2)\int_0^ye^{-u}\frac{|a_4(x,y)|_{max}+|b_4(x,y)|_{max}}{c_{1,min}(x,y-u)c_{3,min}^2(x,y-u)}\,\mathrm{d} u\\
&\leqslant 2(y^2+b^2+x^2)\left(321(y^2+b^2+x^2)^5+321(y^2+b^2+x^2)^5\right)\\
 &\times\left(\frac{2^{12}}{(y^2+b^2+x^2)^6}+\frac{e^{-y/2}}{(x^2+b^2)^6}\right)\\
 &\leqslant 1284\cdot 2^{12}+1284\frac{e^{-y/2}(y^2+b^2+x^2)^6}{(x^2+b^2)^6}\\
\endaligned
\end{equation}

Because
\begin{equation}
\aligned
 \frac{e^{-y/2}(y^2+b^2+x^2)^6}{(b^2+x^2)^6}
 &=e^{-y/2}+e^{-y/2}\sum_{1\leqslant k\leqslant 6}\binom{6}{k}\frac{(y^2)^{k}}{(x^2+b^2)^k}\\
 &\leqslant e^{-y/2}+e^{-y/2}\sum_{1\leqslant k\leqslant 6}\binom{6}{k}\frac{(y^2)^{k}}{(b^2)^k}\\
 &\leqslant 1+\sum_{1\leqslant k\leqslant 6}\binom{6}{k}\frac{1}{(b^2)^k}\left((y^2)^{k}e^{-y/2}\right)_{y=4k}\\
 &= 1+\sum_{1\leqslant k\leqslant 6}\binom{6}{k}\frac{(4k)^{2k}}{(eb)^{2k}}\\
 &=:M_{3B}\\
\endaligned
\end{equation}

\begin{equation}
\aligned
&\frac{((y+b)^2+x^2)}{|x|}|\Im G(x,y)|\\
 &\leqslant 1284\cdot 2^{12}+1284M_{3B}=:M_{3C}
\endaligned
\end{equation}

	Thus
	$((y+b)^2+x^2)|\Im G(x,y)|$ is bounded by $|x|M_{3C}$ for all $x\in\R$ and all $y>0$.
	
\end{proof}

\section{\textbf{Appendix B}}

In this appendix we present Zhou's proof for \cref{Zhou1F1BApp} and one more theorem that builds on \cref{Zhou1F1BApp}. For convenience we copy \cref{Zhou1F1BApp} as

\begin{theorem}[=~\cref{Zhou1F1B} ~\cite{Z20171F1}]\label{Zhou1F1BApp}
	
	Let $\mu=\pm;0<2y< m$. Then there exists a positive constant $M$ and a real function $\epsilon(m,y) \in (-1,1)$ such that
	
	\begin{equation}
	{_1}F_1(1;m+2; \mu y)=\left(1-\frac{\mu y}{m}\right)^{-1}+\frac{\epsilon(m,y) M}{m}
	\end{equation}
	holds for all $m>2y>0$.
	
\end{theorem}
\begin{proof}
	
	Let 
	
	\begin{equation}
	\aligned
	\Delta_0(m,y):=m\left(\left(1-\frac{\mu y}{m}\right)^{-1}-{_1}F_1(1;m+2;\mu y)\right)
	\endaligned
	\end{equation}
	
	It is then suffice to prove that $\Delta_0(m,y)$ is bounded by $M$ for all $m>2y>0$.
	
	Since
	\begin{equation}
	\aligned
	\left(1-\frac{\mu y}{m}\right)^{-1}
	&=1+\sum_{k=1}^{\infty}\frac{(\mu y)^k}{m^k}\\
	{_1}F_1(1;m+2;\mu y)
	&=1+\sum_{k=1}^{\infty}\frac{(\mu y)^k}{(m+2)\cdots(m+k+1)}\\
	\endaligned
	\end{equation}
	
	We have
	\begin{equation}
	\aligned
	\Delta_0(m,y)&=\Delta_1(m,y)+\Delta_2(m,y)\\
	\Delta_1(m,y):&=m\sum_{1\leqslant k< y^{1/3}
	}\frac{(\mu y)^k}{m^k}\left(1-\frac{1}{\left( 1+\cfrac{2}{m}\right)\cdots \left(1+\cfrac{k+1}{m}\right)}\right)\\	
	\Delta_2(m,y):&=m\sum_{y^{1/3}\leqslant k< \infty}\frac{(\mu y)^k}{m^k}\left(1-\frac{1}{\left( 1+\cfrac{2}{m}\right)\cdots \left(1+\cfrac{k+1}{m}\right)}\right)\\	\endaligned
	\end{equation}
	
	If $y<8$ then $y^{1/3}<2$	so there is only one term in summation of $\Delta_1(m,y)$. Thus
	
	\begin{equation}
	\aligned
	|\Delta_1(m,y<8)|
	&=m\frac{y}{m}\left(1-\frac{1}{\left( 1+\cfrac{2}{m}\right)}\right)\\
	&=\frac{2y}{m+2}<\frac{m}{m+2}<1\qquad \because m>2y>0\\		
	\endaligned
	\end{equation}
	
	For $y\geqslant 8$, we have
	\begin{equation}
	\aligned
	|\Delta_1(m,y\geqslant 8)|
	&\leqslant m\sum_{1\leqslant k\leqslant y^{1/3}}\frac{y^k}{m^k}\frac{\left|\left( 1+\cfrac{2}{m}\right)\cdots \left(1+\cfrac{k+1}{m}\right)-1\right|}{\left| 1+\cfrac{2}{m}\right|\cdots \left|1+\cfrac{k+1}{m}\right|}\\
	&\leqslant m\sum_{1\leqslant k\leqslant y^{1/3}}\frac{y^k}{m^k}\frac{\left|\cfrac{1}{m}[2+3+\cdots+(k+1)]\right|}{\left| 1+\cfrac{2}{m}\right|\cdots \left|1+\cfrac{k+1}{m}\right|}\\
	&+m\sum_{1\leqslant k\leqslant y^{1/3}}\frac{y^k}{m^k}\frac{\left|\cfrac{1}{m^2}[2\cdot 3+2\cdot 4+\cdots+k\cdot (k+1)]\right|}{\left| 1+\cfrac{2}{m}\right|\cdots \left|1+\cfrac{k+1}{m}\right|}\\
	&+\cdots\\
	&\leqslant m\sum_{1\leqslant k\leqslant y^{1/3}}\frac{y^k}{m^k}\left(\binom{1}{k}\left(\frac{k+1}{m}\right)+\binom{2}{k}\left(\frac{k+1}{m}\right)^2+\cdots+\binom{k}{k}\left(\frac{k+1}{m}\right)^k\right)\\
	&=m\sum_{1\leqslant k\leqslant y^{1/3}}\frac{y^k}{m^k}\left(\left(1+\frac{k+1}{m}\right)^k-1\right)\\
	&=y\sum_{1\leqslant k\leqslant y^{1/3}}\frac{y^k}{m^k}\left(\exp\left(k\log\left(1+\frac{k+1}{m}\right)\right)-1\right)\\
	&\leqslant m\sum_{1\leqslant k\leqslant y^{1/3} }\frac{y^k}{m^k}\left(\exp\left(\frac{k(k+1)}{m}\right)-1\right)\\
	\endaligned
	\end{equation}
	
	When $y\geqslant 8$ we have $y^{1/3}\geqslant 2>(\log 2)^{-1}\approx 1.4427$, thus
	
	\begin{equation}
	\aligned
	\frac{k(k+1)}{m}&=\frac{y}{my^{1/3}}\frac{k}{y^{1/3}}\frac{(k+1)}{y^{1/3}}\\
	&\leqslant\frac{y}{my^{1/3}}\frac{(k+1)}{k}\qquad \because y^{1/3}\geqslant k\\
	&\leqslant\frac{2y}{my^{1/3}}\\
	&\leqslant\frac{2}{2y^{1/3}}\qquad \because m\geqslant
	2y\\
	&<\log 2
	\endaligned
	\end{equation}
	
	Using $e^t-1<2t$ when $0<t<\log 2$, we obtain
	
	\begin{equation}
	\aligned
	|\Delta_1(m,y\geqslant 8)|&\leqslant m\sum_{1\leqslant k\leqslant y^{1/3}}\frac{y^k}{m^k}2\frac{k(k+1)}{m}\\
	&=\sum_{1\leqslant k\leqslant y^{1/3}}2k(k+1)\frac{y^{k}}{m^{k}}\\
	&<\sum_{k=1}^{\infty}4k^2\frac{ y^{k}}{m^{k}}\\
	&=\frac{4\frac{m}{y}(\frac{m}{y}+1)}{\left(\frac{m}{y}-1\right)^3}\\
	&=\frac{4\left(\frac{m}{y}-1+1\right)\left(\frac{m}{y}-1+2\right)}{\left(\frac{m}{y}-1\right)^3}\\
	&=\frac{4}{\left(\frac{m}{y}-1\right)}+\frac{12}{\left(\frac{m}{y}-1\right)^2}+\frac{8}{\left(\frac{m}{y}-1\right)^3}\\
	&\leqslant\frac{4}{2}+\frac{12}{2^2}+\frac{8}{2^3}\qquad \because m\geqslant 2y\\
	&=2+3+1=6
	\endaligned
	\end{equation}
	
	Thus $|\Delta_1(m,y)|<6$.

	\begin{equation}
	\aligned
	|\Delta_2(m,y)|
	&\leqslant m\sum_{k>y^{1/3}}\frac{y^k}{m^k}\left(1+\frac{1}{\left| 1+\cfrac{2}{m}\right|\cdots \left|1+\cfrac{k+1}{m}\right|}\right)\\
	&\leqslant 2m\sum_{k>y^{1/3}}\frac{y^k}{m^k}\\
	&\leqslant 2m\left(\frac{y}{m}\right)^{\lfloor y^{1/3}\rfloor}\left(1+\frac{y}{m}+\frac{y^2}{m^2}+\cdots\right)\\
	&= 2m\left(\frac{y}{m}\right)^{\lfloor y^{1/3}\rfloor}\frac{\frac{m}{y}}{\frac{m}{y}-1}\\
	&< 2m\left(\frac{y}{m}\right)^{\lfloor y^{1/3}\rfloor-1} \qquad \because m> 2y\\
	&< 2y\left(\frac{y}{m}\right)^{\lfloor y^{1/3}\rfloor-2} \qquad \because m> 2y\\
	&\leqslant \frac{2y}{2^{\lfloor y^{1/3}\rfloor-2}}\\
	&<\frac{2y}{2^{y^{1/3}-3}}\\
	&=\frac{16y}{2^{y^{1/3}}}\\
	&\leqslant\frac{432}{(e\log 2)^3}\approx 64.5838
	\endaligned
	\end{equation}
	
	Combining this result with that for $\Delta_1(m,y)$, we conclude that when $m\geqslant 2y>0$,
	\begin{equation}
	\aligned
	|\Delta_0(m,y)|&<6+\frac{432}{(e\log 2)^3}=:M\approx 70.5838
	\endaligned
	\end{equation}
	
	Therefore $\Delta_0(m,y)$ is bounded by $M$. 
	
\end{proof}

\begin{theorem}[=~\cref{Zhou2F2C}]\label{Zhou2F2CApp}
	
	Let $0<q<1;0<2y<m; \mu=\pm; x\in\R$. Then there exists positive number $M=5$ and real functions $\epsilon_1(m,\mu,x,y),\epsilon_2(m,\mu,x,y) \in (-1,1)$ such that
	
	\begin{equation}
	\aligned
	&	{_2}F_2\begin{pmatrix}{\begin{matrix} 1, & m+2+q+ix \\  m+2, & m+3+q+ix  \\  \end{matrix}} & ;\mu y \end{pmatrix}\\
	&={_1}F_1(1,m+2; \mu y)+\frac{(y\epsilon_1+ix\epsilon_2) M}{(m+2+q)^2+x^2}
	\endaligned
	\end{equation}
	holds for all $m>2y>0$ and all $x\in\R$.
	
\end{theorem}

\begin{proof}
	
	Let 
	
	\begin{equation}
	\aligned
	F(m,x,\mu y):&={_2}F_2\begin{pmatrix}{\begin{matrix} 1, & m+2+q+ix \\  m+2, & m+3+q+ix  \\  \end{matrix}} & ;\mu y \end{pmatrix}\\
	\Delta&={_1}F_1(1,m+2;\mu y)-\Re F(m,x,\mu y)\\
	\delta&=\Im F(m,x,\mu y),\\
	\endaligned
	\end{equation}

	It is then suffice to prove that $((m+2+q)^2+x^2)\Delta$ is bound by $y M$  and $((m+2+q)^2+x^2)\delta$ is bounded by $xM$. 
	
	Since
	\begin{equation}
	\aligned
	{_1}F_1(1,m+2;\mu y)
	&=1+\sum_{k=1}^{\infty}\frac{(\mu y)^k}{(m+2)\cdots(m+k+1)}\\
	F(m,x,\mu y)
	&=1+\sum_{k=1}^{\infty}\frac{(\mu y)^k}{(m+2)\cdots(m+k+1)}\frac{(m+2+q+ix)}{(m+2+q+ix+k)}\\
	\Re F(m,x,\mu y)
	&=1+\sum_{k=1}^{\infty}\frac{(\mu y)^k}{(m+2)\cdots(c+k+1)}\frac{((m+2+q)(m+2+q+k)+x^2)}{((m+2+q+k)^2+x^2)}\\
	\Im F(m,x,\mu y)
	&=\sum_{k=1}^{\infty}\frac{(\mu y)^k}{(m+2)\cdots(m+k+1)}\frac{kx}{((m+2+q+k)^2+x^2)}\\
	\endaligned
	\end{equation}
	
	We have
	\begin{equation}
	\aligned
	|\Delta|
	&\leqslant\sum_{k=1}^{\infty}\frac{y^k}{(m+2)\cdots(m+k+1)}\left(1-\frac{((m+2+q)(m+2+q+k)+x^2)}{((m+2+q+k)^2+x^2)}\right)\\
	&=\sum_{k=1}^{\infty}\frac{y^k}{(m+2)\cdots(m+k+1)}\left(\frac{k(m+2+q+k)}{(m+2+q+k)^2+x^2}\right)\\
	&<\sum_{k=1}^{\infty}\frac{y^k}{(m+2)\cdots(m+k+1)}\left(\frac{k(m+2+q+k)}{(m+2+q)^2+x^2}\right)\\
	&<\frac{1}{(m+2+q)^2+x^2}\sum_{k=1}^{\infty}\frac{y^k}{(m+2)\cdots(m+k+1)}\left(2k(m+1+k)\right)\\
	&=\frac{2y}{(m+2+q)^2+x^2}\sum_{k=1}^{\infty}\frac{ky^{k-1}}{(m+2)\cdots(m+k)}\\
	&<\frac{2y}{(m+2+q)^2+x^2}\sum_{k=1}^{\infty}\frac{ky^{k-1}}{m^{k-1}}\\
	&=\frac{2y}{((m+2+q)^2+x^2)}\frac{m^2}{(m-y)^2}\\
	\endaligned
	\end{equation}

	Thus
	\begin{equation}
	\aligned
	y^{-1}((m+2+q)^2+x^2)|\Delta|
	&<\frac{2m^2}{(m-y)^2}=2\left(1-\frac{y}{m}\right)^{-2}\\
	&<2\left(1+\frac{y}{m}\right)^{2}<\frac{9}{2}<5=M\qquad \because m> 2y
	\endaligned
	\end{equation}	
	
	Therefore $((m+2+q)^2+x^2)|\Delta|$ is bounded by $yM$. 
	
	For $x=0$ we have $\delta=0$. For all $|x|>0$ and all $y>0$ we have
	\begin{equation}
	\aligned
	|x|^{-1}|\delta|&=|x|^{-1}|\Im F(m,x,\mu y)|\\
	&\leqslant\sum_{k=1}^{\infty}\frac{y^k}{(m+2)\cdots(m+k+1)}\frac{k}{((m+2+q+k)^2+x^2)}\\
	&< \sum_{k=1}^{\infty}\frac{y^k}{(m+2)\cdots(m+k+1)}\frac{k}{((m+2+q)^2+x^2)}\\
	&=\frac{1}{((m+2+q)^2+x^2)}\sum_{k=1}^{\infty}\frac{ky^k}{(m+2)\cdots(m+k+1)}\\
	&<\frac{1}{((m+2+q)^2+x^2)}\sum_{k=1}^{\infty}\frac{ky^k}{m^k}\\
	&=\frac{1}{((m+2+q)^2+x^2)}\frac{my}{(m-y)^2}\\
	\endaligned
	\end{equation}
	
	\begin{equation}
	\aligned
	|x|^{-1}((m+2+q)^2+x^2)|\delta|
	&<\left(\frac{y}{m}\right)\left(1-\frac{y}{m}\right)^{-2}\\
	&<\left(\frac{y}{m}\right)\left(1+\frac{y}{m}\right)^{2}
	<\frac{9}{8}<5=M\qquad \because m> 2y\\
	\endaligned
	\end{equation}

	Therefore $((m+2+q)^2+x^2)|\delta|$ is also bounded by $|x|M$.
\end{proof}

\pagebreak
%

\end{document}